\documentclass{article}

\usepackage{arxiv}

\usepackage[utf8]{inputenc} 
\usepackage[T1]{fontenc}    
\usepackage{xcolor}         
\usepackage{listings}       
\usepackage{graphicx}
\usepackage{tikz}          
\usetikzlibrary{arrows.meta,calc}
\usepackage[hidelinks]{hyperref} 
\usepackage{url}            
\usepackage{booktabs}       
\usepackage{array}          
\usepackage{amsfonts}       
\usepackage{nicefrac}       
\usepackage{microtype}      
\usepackage{amsmath}
\usepackage{amssymb}
\usepackage{amsthm}
\newif\ifhaslistofitems
\IfFileExists{listofitems.sty}{\haslistofitemstrue}{\haslistofitemsfalse}
\ifhaslistofitems
  \usepackage{stackengine}
  \stackMath
\else
  \providecommand{\stackunder}[2]{\underset{#2}{#1}}
  \providecommand{\setstackgap}[2]{}
  \providecommand{\stackMath}{}
  \providecommand{\stacktype}{L}
\fi

\renewcommand{\baselinestretch}{1.25}

\makeatletter
\renewcommand{\verbatim@font}{%
  \renewcommand{\baselinestretch}{1.1}%
  \small\normalfont\ttfamily}
\makeatother

\definecolor{mflstkeyword}{rgb}{0.13,0.27,0.51}  
\definecolor{mflstcomment}{rgb}{0.27,0.44,0.33}  
\definecolor{mflststring}{rgb}{0.55,0.27,0.16}   
\definecolor{mflstrule}{gray}{0.65}
\definecolor{mflstbreak}{gray}{0.45}
\definecolor{mflstoutput}{gray}{0.28}   
\lstdefinestyle{matlabstyle}{%
  language=Matlab,
  deletestring=[m]',
  morestring=[b]",
  basicstyle=\linespread{1.1}\footnotesize\ttfamily,
  keywordstyle=\color{mflstkeyword},
  commentstyle=\color{mflstcomment},
  stringstyle=\color{mflststring},
  showstringspaces=false,
  columns=fullflexible,  
  keepspaces=true,       
  breaklines=true,
  breakatwhitespace=true,
  breakindent=2em,
  postbreak=\mbox{\textcolor{mflstbreak}{$\hookrightarrow$}\space},
  frame=leftline,      
  framerule=0.8pt,     
  framesep=6pt,        
  xleftmargin=7pt,
  rulecolor=\color{mflstrule},
  captionpos=b,
  aboveskip=\medskipamount,
  belowskip=\medskipamount,
}
\lstdefinestyle{outputstyle}{%
  language={},
  basicstyle=\linespread{1.1}\footnotesize\ttfamily\color{mflstoutput},
  columns=fullflexible,
  keepspaces=true,
  breaklines=true,
  breakatwhitespace=true,
  breakindent=2em,
  postbreak=\mbox{\textcolor{mflstbreak}{$\hookrightarrow$}\space},
  frame=leftline,
  framerule=0.8pt,
  framesep=6pt,
  xleftmargin=7pt,
  rulecolor=\color{mflstrule},
  captionpos=b,
  aboveskip=\medskipamount,
  belowskip=\medskipamount,
}
\newtheorem{theorem}{Theorem}
\newtheorem{lemma}{Lemma}
\newtheorem{corollary}{Corollary}
\newtheorem{proposition}{Proposition}
\theoremstyle{definition}
\newtheorem{definition}{Definition}
\theoremstyle{remark}
\newtheorem{remark}{Remark}
\theoremstyle{plain}

\newcommand{\bmat}[1]{\begin{bmatrix}#1\end{bmatrix}}

\title{A theory of meta-factorization}
\date{2 October 2026}
\hypersetup{pdftitle={A theory of meta-factorization},pdfauthor={Michal P. Karpowicz},pdfkeywords={matrix factorization, oblique projections, generalized inverses, factor replacement}}

\author{
  Micha\l{} P.~Karpowicz\thanks{Earlier versions of this work were developed while the author was affiliated with NASK National Research Institute and Warsaw University of Technology. The present revised and expanded version was developed while the author was affiliated with Samsung AI Center, Warsaw.} \\
  Samsung AI Center\\
  Warsaw, Poland\\
  \texttt{m.karpowicz@samsung.com}
  \\ORCID 0000-0003-1431-3078\\
  }

\begin{document}

\maketitle
\setcounter{footnote}{0} 
\begin{abstract}
Meta-factorization asks how different choices of factors and transforms can describe the same matrix, and how a factor can be replaced while preserving that matrix. We give a common description through two equations that expose the shared projector structure of matrix factorizations and determine exact reconstruction. For $\mathbf A=\mathbf F\mathbf G\mathbf H^*$ with outer factors of full column rank $k$, the projector equation $\mathbf Y^*\mathbf F=\mathbf H^*\mathbf X=\mathbf I_k$ selects generally oblique projectors. The reconstruction equation $\mathbf A=\mathbf F(\mathbf Y^*\mathbf A\mathbf X)\mathbf H^*$ then holds exactly when the outer-factor spaces contain the column spaces of $\mathbf A$ and $\mathbf A^*$, respectively. The mixing matrix is unique, and is nonsingular at $k=\mathrm{rank}(\mathbf A)$.

Two design matrices parametrize every certifying pair under explicit rank conditions. This gives a common constructive route through the SVD, column-pivoted QR, UTV, generalized Nystr\"{o}m, CUR and classical pseudoinverse formulas, while making the additional structural assumptions of each construction explicit. The same description makes changes to an existing factorization testable from the stored factors.

For a supplied replacement $\check{\mathbf F}$ and transform $\check{\mathbf Y}^*$, the core update $\check{\mathbf G}=(\check{\mathbf Y}^*\mathbf F)\mathbf G$ preserves $\mathbf A$ exactly when $[\mathbf F-\check{\mathbf F}(\check{\mathbf Y}^*\mathbf F)]\mathbf G=\mathbf0$, provided the retained right factor has full column rank. Both calculations use stored factors without accessing $\mathbf A$. The same identity returns an exact factored residual when replacement fails, and a rectangular extension permits the stored width to change. With $\check{\mathbf Y}^*=\check{\mathbf F}^{+}$, the updated product is a closest matrix within the candidate's column space in every unitarily invariant norm. Under the stated inverse and rank conditions, the accepted-subspace geometry distinguishes changes of basis from replacements that change the outer-factor space.

Generalized right-hand sides separate changes of certifying coordinates from actions on the core. Nonsingular choices normalize to the original equations. Finite-order choices produce potent operators and periodic reconstruction; classical Fourier components give an exact test for preserving the core at every step, including singular cores. The contribution is a common description together with explicit replacement and reconstruction consequences. The results are exact algebraic relations; numerical acceptance, conditioning and candidate selection require further analysis.
\end{abstract}



\section{Introduction}
\label{section:Introduction}

Matrix factorizations describe a matrix through factors whose structure exposes its spectrum, rank, or fundamental subspaces. Meta-factorization studies the interactions between those factors. Given $\mathbf A=\mathbf F\mathbf G\mathbf H^*$, its starting question is how the mixing matrix $\mathbf G$ and the outer factors cooperate to reconstruct $\mathbf A$, and how those interactions can be changed while preserving the reconstruction.

The key is to expose two hidden projectors. The projector equation
$$
\mathbf Y^*\mathbf F=\mathbf H^*\mathbf X=\mathbf I_k
$$
identifies a certifying pair, while the reconstruction equation
$$
\mathbf A=\mathbf F(\mathbf Y^*\mathbf A\mathbf X)\mathbf H^*
$$
expresses exactness. For outer factors of full column rank, two design matrices $\mathbf B$ and $\mathbf D$ parametrize every solution of the projector equation. Exact reconstruction additionally requires the column spaces of $\mathbf F$ and $\mathbf H$ to contain those of $\mathbf A$ and $\mathbf A^*$, respectively. This is Penrose's consistency condition \cite[Thm.~2]{penrose1955generalized}. Choosing the outer factors, solving the projector equation, and then factoring the resulting mixing matrix gives a constructive route through the SVD, CPQR, UTV, generalized-inverse and sketching examples developed below.

This perspective also makes changes to an existing factorization precise. A basis may require re-orthogonalization, a compressed block may admit a smaller representation, or a different selection of columns may suit a subsequent computation. Given a proposed replacement for $\mathbf F$, one can ask whether the product is preserved and how to update the mixing matrix using the stored factors. These questions matter when $\mathbf A$ is expensive to access or is available only through its factorization. The factor-replacement results in this update develop that consequence of the general framework.

The main operational consequence is a test and repair for a supplied replacement pair $(\check{\mathbf{F}},\check{\mathbf{Y}})$. If $\mathbf{H}$ has full column rank, the update $\check{\mathbf{G}}=(\check{\mathbf{Y}}^*\mathbf{F})\mathbf{G}$ preserves $\mathbf{A}$ exactly when $[\mathbf{F}-\check{\mathbf{F}}(\check{\mathbf{Y}}^*\mathbf{F})]\mathbf{G}=\mathbf{0}$. Both expressions use the existing factors, the candidate, and its transform. They do not require access to $\mathbf{A}$. The same calculation gives an exact factored residual when the test fails. With the pseudoinverse as transform, it returns a closest reconstruction within the candidate's column space. Section~\ref{section:Replacement} develops this result, distinguishes testing a supplied repair from deciding whether any repair exists, and characterizes the accepted column spaces.

\paragraph{Contributions.} The paper contributes a common description of matrix factorizations and explicit consequences for changing their factors and certifying maps. The characterization of certifying pairs organizes classical projector and consistency results; the replacement and generalized-certificate results make the corresponding reconstruction tests, residuals and freedoms explicit. The numbered statements below identify these contributions and the classical results they use.
\begin{enumerate}
\item[(i)] \emph{Characterization and rank regimes.} Proposition~\ref{prop:oblique-projector} describes the effective trial and test spaces at arbitrary rank. For factors satisfying the projector equation, exact reconstruction is equivalent to the two span inclusions (Theorem~\ref{Theorem:meta-factorization}). Proposition~\ref{prop:regimes} and Table~\ref{tab:regimes} distinguish $k<r$, $k=r$, and $k>r$, where $r=\mathrm{rank}(\mathbf{A})$. In particular, above the rank, the square sketched and index-selected constructions with their own design parameters fail the projector equation by the exact Frobenius residual $\sqrt{k-\mathrm{rank}(\mathbf{M})}$, where $\mathbf{M}$ is their compressed matrix, even when reconstruction is exact.
\item[(ii)] \emph{An exhaustive parametrization.} The design matrices $\mathbf{B}$ and $\mathbf{D}$ generate every certifying pair under the rank conditions $\mathrm{rank}(\mathbf{B}^*\mathbf{F})=\mathrm{rank}(\mathbf{H}^*\mathbf{D})=k$ (Theorem~\ref{Theorem:projector-eq-sol}, Proposition~\ref{prop:BD-exhaustive}). This combines the classical one-sided projector construction \cite{cerny2009characterization,sorensen2016deim} with the invariance of Penrose's consistency criterion under admissible choices of inverse (Theorem~\ref{thm:Penrose-oblique}).
\item[(iii)] \emph{CUR families and explicit sampling matrices for rank reduction.} Proposition~\ref{prop:CUR-exactness} recovers the rank criterion for exact CUR and pseudoskeleton reconstruction. Corollary~\ref{cor:CUR-family} places two published CUR identities in one family parametrized by two design matrices. For full-column-rank factors and nonsingular $\mathbf{G}$, Corollary~\ref{thm:wedderburn} expresses the classical condition for subtracting a term of rank $k$ to lower the rank by $k$ \cite{baksalary1989properties}, and supplies the sampling matrices $\mathbf{\Omega}_c=\mathbf{A}^{+}\mathbf{F}$ and $\mathbf{\Omega}_r=(\mathbf{A}^{+})^*\mathbf{H}$, with $\mathbf{G}=(\mathbf{H}^*\mathbf{A}^{+}\mathbf{F})^{-1}$.
\item[(iv)] \emph{Recovery with structural hypotheses.} Propositions~\ref{prop:SVD-gauge}, \ref{prop:CPQR}, and \ref{prop:UTV-completion} specify the matched singular vectors, pivoting rule, and ordered completion needed to recover the SVD, CPQR, and UTV structures. Lemma~\ref{lem:compression} identifies their reconstruction conditions. Section~\ref{subsection:pseudoinverse} re-expresses classical pseudoinverse formulas, including MacDuffee's theorem (Theorem~\ref{thm:MacDuffee}), with $\mathbf{A}^{+}$ supplied as an input.
\item[(v)] \emph{Invariance of the solution formula.} In Corollary~\ref{cor:linear-system}, every admissible $\mathbf{B}$ produces a $\{1,2,4\}$-inverse $\mathbf{Y}^*$ of $\mathbf{A}$. Its particular solution is $\mathbf{A}^{+}\mathbf{c}$, and the whole solution formula is independent of $\mathbf{B}$.
\item[(vi)] \emph{Normalization and finite-order certificates.} For fixed $\mathbf{F}$ and $\mathbf{H}$ and each fixed nonsingular pair $(\mathbf{Z}_c,\mathbf{Z}_r)$, the generalized projector equation is related to the normalized equation by a bijective change of variables on the certifying pair (Proposition~\ref{prop:gauge-reduction}). Finite-order right-hand sides produce $(N{+}1)$-potent operators, and Proposition~\ref{prop:finite-order} describes the finite-order pairs stabilizing the core, including singular cores, through classical cyclic idempotents.
\item[(vii)] \emph{Replacement, residual, and optimal repair.} With the retained right factor of full column rank, Theorem~\ref{thm:replacement} tests the repair induced by an arbitrary candidate and transform using an $m\times k$ mismatch and the mixing matrix. Under the additional $\{1\}$-inverse condition on the transform, this becomes a test of the classical span condition for the existence of a repair (Corollary~\ref{cor:replacement-span}) \cite{penrose1955generalized,douglas1966majorization}. Corollary~\ref{cor:replacement-defect} gives the exact factored residual without rank hypotheses. With the pseudoinverse as transform, the repaired factorization is a closest matrix within the candidate's column space in every unitarily invariant norm (Corollary~\ref{cor:replacement-optimal} and the argument following it). Theorem~\ref{thm:replacement-mirror} gives the right-factor counterpart.
\item[(viii)] \emph{The geometry of accepted replacements.} Suppose both incumbent outer factors have full column rank and each candidate arrives with a $\{1\}$-inverse transform. At $k=r$, every accepted replacement is a change of basis, recovering the classical uniqueness of full-rank factorizations (Corollary~\ref{cor:replacement-full-rank}, Proposition~\ref{prop:replacement-rigidity}(a)). For $r<k<m$, the accepted candidates of rank $k$ have $(k-r)(m-k)$ parameters of column-space freedom beyond that orbit. At $k=m$ this additional freedom vanishes (Proposition~\ref{prop:replacement-rigidity}(b)).
\end{enumerate}

These results separate the freedom to choose a certifying pair from the freedom to change an outer factor. The former organizes the classical constructions. The latter makes a proposed substitution testable and its residual explicit. Section~\ref{section:related} identifies the prior results behind each construction, and Section~\ref{subsection:replacement-prior-art} compares the replacement result with existing orthogonalization, recompression, and updating methods.

\subsection{Outline}

Section~\ref{section:related} places the framework against the rank-reduction, generalized-inverse and randomized-projection literatures, and states its overlap with recent work exactly. Section~\ref{section:Meta-factorization} defines meta-factorization and its two equations, and shows that their solutions define matrix factorizations. Section~\ref{section:Replacement} settles when a factor may be exchanged for another and the mixing matrix repaired, and how much freedom that leaves. Section~\ref{section:How...} applies the theory to the singular value decomposition, the column-pivoted QR decomposition and a family of UTV factorizations. Section~\ref{section:Additional...} relates the framework to generalized matrix inverses, reconstructs the generalized Nystr\"{o}m method and connects it with the CUR decomposition\footnote{A MATLAB code supplement is available here: \url{https://github.com/mkarpowi/mft}}. Section~\ref{section:Generalized...} generalizes the right-hand side of the projector equation and examines the resulting normalization and finite-order operators. Nonsingular right-hand sides reduce by a change of variables to the normalized equation, and finite order characterizes when the certifying operators become $(N{+}1)$-potent. Readers following the constructive route of the earlier version may proceed from Section~\ref{section:Meta-factorization} to the worked recoveries in Sections~\ref{section:How...}--\ref{section:Additional...}; Section~\ref{section:Replacement} develops the added replacement theory.

\section{Related work}
\label{section:related}

Matrix factorizations reveal essential matrix properties --- spectrum, rank, the fundamental subspaces --- which characterize solutions to fundamental problems in physics, signal processing, control engineering, artificial intelligence and data science \cite{strang2019learning,strangCSE}, and they correspond to matrix algorithms \cite{golub2013matrix,stewart1998matrix,demmel1997applied,trefethen1997numerical,townsend2015continuous}. New families of factorizations, together with their unifying algebraic structure, are still being found \cite{edelman2021fifty}. Edelman and Jeong organize structured factorizations systematically, through generalized Cartan decompositions; what follows here does something narrower and different, parametrizing the certifying projector pairs of a supplied three-factor representation and supplying from them a factor-only test and repair. Meta-factorization is related to the theory of generalized matrix inverses: its equations have solutions involving pseudoinverses, define them in special cases, and describe the algebraic formulas for the Moore--Penrose pseudoinverse \cite{penrose1955generalized,langenhop1967generalized,marsaglia1974equalities,ben2003generalized,james1978generalised,strangdrucker}. Its reconstruction equation is also connected to the quadratic matrix equation generating the Lie groups, which gives rise to families of matrix factorizations and, under the appropriate assumptions, can be interpreted as the meta-factorization equation \cite{edelman2021fifty}.

That a left inverse $\mathbf{Y}^*$ of $\mathbf{F}$ makes $\mathbf{F}\mathbf{Y}^*$ a projector onto $\mathcal{C}(\mathbf{F})$ along $\mathcal{N}(\mathbf{Y}^*)$ is textbook \cite{meyer2000matrix,stewart2011oblique}, and the closed form $\mathbf{F}(\mathbf{B}^*\mathbf{F})^{+}\mathbf{B}^*$ includes the coordinate-selection interpolatory projectors used by Sorensen and Embree \cite{sorensen2016deim}. Their technique for deriving the column and row space projectors of the CUR decomposition motivates the projector equation. The projector it produces is the operator $U(VU)^{\dagger}V$ characterized by \v{C}ern\'{y} \cite{cerny2009characterization}. Proposition~\ref{prop:oblique-projector} expresses its effective trial and test spaces explicitly. The sampled inverse in Karpowicz and Strang \cite[Sec.~5.1, Fig.~6]{karpowicz2026generalized} is the same one-sided construction under the correspondence stated there. Designing projectors is a central concern of randomized linear algebra \cite{halko2011finding,mahoney2016lecture,martinsson2019randomized,martinsson2020randomized,shabat2018randomized,kannan2017randomized}. Nakatsukasa \cite{nakatsukasa2020fast} studied the generalized Nystr\"{o}m method, a two-sided projection defined by sampling (sketch) matrices and applicable to general rather than positive semidefinite matrices. With oversampling and suitable stabilization, he obtained near-optimal approximation guarantees and a numerically stable implementation. Fast structured sketches give a cost of $O(mn\log n+r^3)$ for dense matrices, where $r$ is the target approximation rank in that work. Meta-factorization gives rise to that method, and it explains the structure of the CUR decomposition \cite{mahoney2009cur,wang2013improving,sorensen2016deim,hamm2020perspectives}, whose exact square-core form is implicit in Penrose's block representation \cite[p.~18]{penrose1956best}. For selected columns and rows, the approximation $\mathbf{A}\approx\mathbf{F}(\mathbf{Y}^*\mathbf{A}\mathbf{X})\mathbf{H}^*$ has the classical pseudoskeleton form studied by Goreinov, Tyrtyshnikov and Zamarashkin \cite{goreinov1997pseudoskeleton}. Read as an identity, its exactness is characterized for column and row selection by Hamm and Huang \cite{hamm2020perspectives}, in their Theorem~A. Section~\ref{section:Additional...} says precisely which parts of those statements the framework reproduces and which small increment it adds. The same equations reach further back, to the outer-product decomposition produced by rank reduction \cite{wedderburn1934lectures,guttman1944general,egervary1960rank,householder1965theory,cline1979rank,chu1995rank}, and forward to Nakatsukasa and Tropp \cite{nakatsukasa2021accurate}, who study sketched GMRES and Rayleigh--Ritz methods using generally non-orthogonal bases.

Several related theories address different questions. Generalized singular value and QR decompositions treat several matrices together \cite{vanloan1976generalizing,demoor1991tree,demoor1992generalizations}. Here the matrix $\mathbf A$ is fixed, and we vary the factors and the projectors that preserve it. Hubert, Meulman and Heiser \cite{hubert2000two} distinguish factorizations used to solve another problem from factorizations whose components are themselves interpreted. The present equations can serve either purpose; the interpretation of the factors depends on the application.

One comparison must be made explicitly, because the two programmes look alike. Chu, Funderlic and Golub \cite{chu1995rank} showed that a single rank-one reduction formula generates Gram--Schmidt, LU, QR, the SVD, Lanczos and biconjugation, and that is already a unification of the classical factorizations by one algebraic device. The difference here is one of direction rather than of scope. Rank reduction is \emph{sequential and deflationary}: a factorization is built by repeatedly peeling a rank-one piece off $\mathbf{A}$, so the factors are outputs of an algorithm and the theory describes the process. The projector and reconstruction equations are instead a \emph{simultaneous algebraic characterization} of a finished factorization: they say which quadruples $(\mathbf{F},\mathbf{Y},\mathbf{X},\mathbf{H})$ are admissible, without reference to how they were obtained. That is what makes factor replacement tractable. When the right factor has full column rank, Section~\ref{section:Replacement} decides whether the substitution induced by a candidate and the transform accompanying it is exact, and decides that from the old left factor, the mixing matrix, the candidate and that transform alone. It then places the criterion against the classical consistency condition it evaluates and against the recompression and updating literatures. For a proposed term $\mathbf F\mathbf G\mathbf H^*$ with full-rank outer factors and nonsingular $\mathbf G$, Corollary~\ref{thm:wedderburn} gives a direct test for whether subtracting that term lowers the rank of $\mathbf A$ by exactly $k$. When the test passes, it also supplies the sampling matrices that produce the term. We do not claim more: the classical factorizations recovered in Section~\ref{section:How...} are recovered in their algebraic form, and the rank-revealing guarantees that motivate the pivoted variants are properties of the pivoting rules \cite{chan1987rank,gu1996efficient,martinsson2019randutv}, not of the framework.

K\k{e}dzierski's generalized Wedderburn formula \cite{kedzierski2024rank} directly overlaps the reconstruction theorem at the minimal inner dimension. Writing $\mathbf{M}=\mathbf{Y}^*\mathbf{A}\mathbf{X}$, his formula gives $\mathrm{rank}(\mathbf{A}-\mathbf{A}\mathbf{X}\mathbf{M}^{+}\mathbf{Y}^*\mathbf{A})=\mathrm{rank}(\mathbf{A})-\mathrm{rank}(\mathbf{M})$. At $k=r=\mathrm{rank}(\mathbf{A})$, a meta-factorization with full-column-rank outer factors has $\mathbf{M}=\mathbf{G}$ nonsingular. Since $\mathbf{A}\mathbf{X}=\mathbf{F}\mathbf{G}$ and $\mathbf{Y}^*\mathbf{A}=\mathbf{G}\mathbf{H}^*$, substituting its transforms into the rank-reduction formula returns the same rank-$r$ certifying projectors. This correspondence is restricted to $k=r$. For $k>r$, an exact meta-factorization can have rank-$k$ certifying projectors and a singular core, so the corresponding rank-reduction projectors need not coincide with them.

The projector formulas themselves are classical \cite{cerny2009characterization,stewart2011oblique}, and K\k{e}dzierski also gives a universal representation of projectors \cite{kedzierski2024rank}. The present formulation organizes these facts as a two-sided parametrization for supplied outer factors (Theorem~\ref{Theorem:projector-eq-sol}, Proposition~\ref{prop:BD-exhaustive}) and uses it to describe factor replacement and the three inner-dimension regimes. Its additional consequences are the exact residual in Proposition~\ref{prop:regimes}(c), the UTV completion rule, the replacement results, the CUR family, and the explicit rank-reduction witness. Related reverse-order questions are developed in \cite{kedzierski2024rol,kedzierski2025reverse}.

The present paper develops the framework: the two equations, the characterization of their joint solutions, and the recovery of classical factorizations as instances. The pseudoinverse of a product and randomized generalized inverses are developed by the author with G.~Strang elsewhere \cite{karpowicz2023pseudoinverse,karpowicz2026generalized}. Reverse-order laws have a substantial literature of their own \cite{hartwig1986reverse,tian1994reverse,wei1999reverse,kedzierski2025reverse}. The closing example uses a known implication of K\k{e}dzierski \cite[Cor.~5.2]{kedzierski2024rol}; its attribution and hypotheses are stated locally.

\paragraph{Notation.} Capital bold letters denote matrices, $\mathbf{M}^*$ denotes the conjugate transpose of $\mathbf{M}$, and $\mathbf{M}^+$ denotes the Moore--Penrose pseudoinverse of $\mathbf{M}$. The $k\times k$ identity matrix is denoted as $\mathbf{I}_k$. Given ordered subindex sets $I$ and $J$, $\mathbf{A}(I,J)$ denotes the submatrix of $\mathbf{A}$ containing rows and columns of $\mathbf{A}$ indexed by $I$ and $J$, with $\mathbf{A}(:,J)$ extracting the columns of $\mathbf{A}$ indexed by $J$. With $k$ being a positive integer, $1{:}k$ denotes the ordered set $(1,\dots,k)$. Finally, $\mathcal{C}(\mathbf{A})$ denotes the column space and $\mathcal{C}(\mathbf{A}^*)$ denotes the row space for $\mathbf{A}$, and $\mathcal{N}(\mathbf{A})$ denotes the null space. An orthonormal $m\times n$ matrix $\mathbf{U}$ has orthonormal columns: $\mathbf{U}^{*}\mathbf{U} = \mathbf{I}_n$. A unitary matrix is a square orthonormal matrix, and $U(m)$ denotes the group of all $m\times m$ unitary matrices. Three symbols recur throughout. The rank of $\mathbf{A}$ is $r=\mathrm{rank}(\mathbf{A})$. The inner dimension of the factorization $\mathbf{A}=\mathbf{F}\mathbf{G}\mathbf{H}^*$ is $k$, the number of columns of $\mathbf{F}$ and $\mathbf{H}$ and the size of $\mathbf{G}$. The symbol $\ell$ denotes a further target or candidate width, with its bounds stated locally; it satisfies $\ell<r$ only where a lower-rank approximation is specified. The three regimes $k<r$, $k=r$ and $k>r$ behave differently and are distinguished wherever it matters. Proposition~\ref{prop:regimes} collects the differences. Unless stated otherwise, $\mathbb{F}$ denotes either $\mathbb{R}$ or $\mathbb{C}$ and every statement is proved over both. The group of all nonsingular $k\times k$ matrices over $\mathbb{F}$ is written $GL_k(\mathbb{F})$. An \emph{oblique projector} is an idempotent matrix. An \emph{orthogonal projector} is an idempotent Hermitian matrix. The two projectors hidden in a factorization are written $\mathbf{P}_c$ (column space) and $\mathbf{P}_r$ (row space) throughout, so that $\mathbf{R}$ is free for the triangular factor of the QR decomposition and for the selected rows $\mathbf{A}(I,:)$ of the CUR decomposition, the two classical roles in which a reader expects it. A check accent marks a replacement and what it induces: $\check{\mathbf{F}}$ is a factor offered in place of $\mathbf{F}$, $\check{\mathbf{Y}}$ is the transform that arrives with it, and $\check{\mathbf{G}}$ is the repaired mixing matrix. The convention is used in Section~\ref{section:Replacement}.

\section{Meta-factorization}
\label{section:Meta-factorization}

This section characterizes exact factorizations through their certifying projector pairs. The inner dimension $k$ is the number of columns of each outer factor, whereas $r=\mathrm{rank}(\mathbf{A})$ is the rank of the represented matrix. When $\mathbf{F}$ and $\mathbf{H}$ have full column rank $k$, their one-sided inverses produce rank-$k$ projectors. These projectors fix $\mathbf{A}$ precisely when the corresponding column-space inclusions hold. The distinction between $k$ and $r$ permits exact representations with a singular mixing matrix when $k>r$.

Full column rank is essential for the projector equation. For example, $\mathbf{A}=\mathrm{diag}(1,0)$, $\mathbf{F}=\mathbf{H}=\bigl(\begin{smallmatrix}1&1\\0&0\end{smallmatrix}\bigr)$, and $\mathbf{G}=\mathrm{diag}(1,0)$ give an exact factorization with $k=2$ and rank-one outer factors. The equation $\mathbf{Y}^*\mathbf{F}=\mathbf{I}_2$ is unsolvable. The orthogonal projectors $\mathbf{F}\mathbf{F}^{+}$ and $(\mathbf{H}^*)^{+}\mathbf{H}^*$ still fix $\mathbf{A}$, but their rank is $1$. The results below distinguish exact reconstruction from certification by rank-$k$ projectors.

\subsection{Factorizations and hidden projections}

Every matrix factorization studied in this paper has the form
\begin{align}\label{eq:genfactorization}
\setstackgap{L}{15pt}\def\stacktype{L}\def\sz{\scriptstyle}
\begin{aligned}
\stackunder{\mathbf{A}}{\sz (m\times n)} &=
\stackunder{\mathbf{F}}{\sz (m\times k)} &&
\stackunder{\mathbf{G}}{\sz (k\times k)} &&
\stackunder{\mathbf{H}^*,}{\sz (k\times n)}
\end{aligned}
\end{align}
where an arbitrary $m$ by $n$ matrix $\mathbf{A}$ is decomposed into the product of an $m$ by $k$ matrix $\mathbf{F}$, a square $k$ by $k$ matrix $\mathbf{G}$, and an $n$ by $k$ matrix $\mathbf{H}$. These factors may reveal essential properties of $\mathbf{A}$: its spectrum, its rank, or its four fundamental subspaces. The most important factorizations, such as the singular value decomposition (SVD), the column-pivoted QR (CPQR) decomposition, the eigenvalue decomposition (ED), the LU decomposition, or the CUR decomposition, are all described by this pattern.

Exact reconstruction requires the columns of $\mathbf{F}$ and $\mathbf{H}$ to span subspaces containing the column and row spaces of $\mathbf{A}$, respectively. These containing spaces need not equal the data spaces, which permits $k>\mathrm{rank}(\mathbf{A})$. Those spanning sets can be derived directly from the columns and rows of the original matrix, for example by (random) sampling, elimination, or orthogonalizing transformations. The matrix $\mathbf{G}$, which we call the \emph{mixing matrix} or the \emph{core} throughout, is more challenging. It encodes the essential properties of $\mathbf{A}$ isolated by $\mathbf{F}$ and $\mathbf{H}$ in a way that guarantees their reconstruction in the original matrix.

How, then, is the mixing matrix $\mathbf{G}$ to be determined when $\mathbf{F}$ and $\mathbf{H}$ are given? The best way to learn more about a given matrix is to factor it into a particular mixture of other matrices, so let us apply the same trick here and factor $\mathbf{G}$ itself.

If $\mathbf{F}$ and $\mathbf{H}$ have full column rank, the mixing matrix is uniquely determined by multiplying (\ref{eq:genfactorization}) by their one-sided inverses:
\begin{align}\label{eq:Gfactorization}
\mathbf{G} = \mathbf{F}^{+} \mathbf{A} (\mathbf{H}^*)^+.
\end{align}
For rank-deficient outer factors, the expression on the right of \eqref{eq:Gfactorization} gives the canonical solution of minimum Frobenius norm when the factorization is consistent; it need not recover an arbitrary supplied $\mathbf G$. For example, $\mathbf A=\mathbf F=\mathbf H=\operatorname{diag}(1,0)$ admits $\mathbf G=\mathbf I_2$, whereas \eqref{eq:Gfactorization} returns $\operatorname{diag}(1,0)$. Theorem~\ref{thm:Penrose} describes all solutions. Even in the full-column-rank case, direct inversion does not reveal the available choice of certifying transforms. Instead of inverting $\mathbf{F}$ and $\mathbf{H}$, we can introduce two new transforms, $\mathbf{Y}$ and $\mathbf{X}$, and assume that
%
\begin{align}\label{eq:geniotransform}
\setstackgap{L}{15pt}\def\stacktype{L}\def\sz{\scriptstyle}
\begin{aligned}
\stackunder{\mathbf{G}}{\sz (k\times k)} &=
\stackunder{\mathbf{Y}^*}{\sz (k\times m)}&&
\stackunder{\mathbf{A}}{\sz (m\times n)} &&
\stackunder{\mathbf{X}.}{\sz (n\times k)}
\end{aligned}
\end{align}
The key insight follows from substituting (\ref{eq:geniotransform}) into (\ref{eq:genfactorization}):
\begin{align}\label{eq:motivation}
\mathbf{A} =
\mathbf{F} \mathbf{Y}^*
\mathbf{A}
\mathbf{X} \mathbf{H}^*.
\end{align}
The original decomposition (\ref{eq:genfactorization}) can thus be read as a reconstruction formula. We seek transforms for which $\mathbf{F}\mathbf{Y}^*$ and $\mathbf{X}\mathbf{H}^*$ are projectors that fix the column and row data, respectively. Reconstruction alone does not make these products idempotent; the projector equation below supplies that additional structure. Those projectors are determined by $\mathbf{F}$ and $\mathbf{H}$, but also by $\mathbf{Y}$ and $\mathbf{X}$, the factors encoding the properties of $\mathbf{A}$ in $\mathbf{G}$. Therefore, to calculate the mixing matrix when $\mathbf{F}$ and $\mathbf{H}$ are known, we must find $\mathbf{Y}$ and $\mathbf{X}$ that define projectors for the column space and the row space of $\mathbf{A}$. That is the idea behind meta-factorization.

\subsection{The projector equation}

Meta-factorization shapes the internal structure of the reconstruction formula (\ref{eq:motivation}). That is possible only so long as the two products $\mathbf{F}\mathbf{Y}^*$ and $\mathbf{X}\mathbf{H}^*$ it contains are idempotent, and we now write that constraint down as an equation.

Suppose that $\mathbf{F}$ and $\mathbf{H}$ are arbitrary matrices. The projector equation defines $\mathbf{Y}$ and $\mathbf{X}$ such that
\begin{align}
\mathbf{P}_c = \mathbf{F}\mathbf{Y}^* \quad\text{and}\quad \mathbf{P}_r = \mathbf{X}\mathbf{H}^*
\end{align}
are idempotent matrices of rank $k$: $\mathbf{P}_c^2=\mathbf{P}_c$ and $\mathbf{P}_r^2=\mathbf{P}_r$. The distinction between oblique and orthogonal projectors is not cosmetic. Theorem~\ref{Theorem:projector-eq-sol} below builds $\mathbf{Y}$ and $\mathbf{X}$ from a pair of matrices $\mathbf{B}$ and $\mathbf{D}$, which we call the design parameters. The classical factorizations recovered in Section~\ref{section:How...} all arise from the orthogonal choice $\mathbf{B}=\mathbf{F}$, $\mathbf{D}=\mathbf{H}$. The sketched constructions of Section~\ref{section:Additional...} are generally oblique. The index-selected family contains both a generally oblique CUR member ($\mathbf{G}=\mathbf{W}^{+}$, Corollary~\ref{cor:CUR-family}) and one with orthogonal projectors ($\mathbf{G}=\mathbf{C}^{+}\mathbf{A}\mathbf{R}^{+}$). Each case is marked at the point of use.

\begin{theorem}\label{Theorem:projection}
Let $\mathbf{F},\mathbf{Y}\in\mathbb{F}^{m\times k}$ and $\mathbf{H},\mathbf{X}\in\mathbb{F}^{n\times k}$ with $1\le k\le\min(m,n)$. If $\mathbf{F}\mathbf{Y}^*$ and $\mathbf{X}\mathbf{H}^*$ are idempotent and
\begin{align}
\mathrm{rank}(\mathbf{F}\mathbf{Y}^*) = \mathrm{rank}(\mathbf{X}\mathbf{H}^*) = k,
\end{align}
$k$ being the number of columns of the factors, then they satisfy the projector equation,
\begin{align}\label{eq:projector-eq}
\mathbf{Y}^*\mathbf{F} = \mathbf{H}^*\mathbf{X} = \mathbf{I}_{k}.
\end{align}
Conversely, if the projector equation holds, then $\mathbf{F}\mathbf{Y}^*$ and $\mathbf{X}\mathbf{H}^*$ are idempotent of rank $k$, and $\mathbf{F}\mathbf{Y}^*$ is the oblique projector onto $\mathcal{C}(\mathbf{F})$ along $\mathcal{N}(\mathbf{Y}^*)$.
\end{theorem}

\begin{proof}
Consider $\mathbf{P}_c = \mathbf{F}\mathbf{Y}^*$ and $\mathbf{P}_r = \mathbf{X}\mathbf{H}^*$. We prove the second implication first. If (\ref{eq:projector-eq}) holds, then $\mathbf{P}_c^2 = \mathbf{F}\mathbf{Y}^*\mathbf{F}\mathbf{Y}^* = \mathbf{F}\mathbf{Y}^* = \mathbf{P}_c$. Since $\mathbf{Y}^*\mathbf{F}=\mathbf{I}_k$ the matrix $\mathbf{F}$ has a left inverse, so $\mathrm{rank}(\mathbf{F})=k$; and $\mathbf{F}\mathbf{Y}^*\mathbf{F}=\mathbf{F}$ gives $k=\mathrm{rank}(\mathbf{F})=\mathrm{rank}(\mathbf{F}\mathbf{Y}^*\mathbf{F})\le\mathrm{rank}(\mathbf{F}\mathbf{Y}^*)\le\mathrm{rank}(\mathbf{F})=k$, so $\mathrm{rank}(\mathbf{P}_c) = k$. Similarly, $\mathbf{P}_r^2 = \mathbf{X}\mathbf{H}^*\mathbf{X}\mathbf{H}^* = \mathbf{X}\mathbf{H}^* = \mathbf{P}_r$, and $\mathrm{rank}(\mathbf{P}_r) = k$, by the same arguments.

Conversely, suppose $\mathbf{F}\mathbf{Y}^*$ is idempotent with $\mathrm{rank}(\mathbf{F}\mathbf{Y}^*)=k$, $k$ being the number of columns of $\mathbf{F}$ and of $\mathbf{Y}$. Then
\begin{align}\label{eq:cancellation-target}
\mathbf{F}(\mathbf{Y}^*\mathbf{F}-\mathbf{I}_k)\mathbf{Y}^* = \mathbf{0}.
\end{align}
First, $\mathrm{rank}(\mathbf{F}) = \mathrm{rank}(\mathbf{Y}) = k$, by the product-rank sandwich
\begin{align}
k = \mathrm{rank}(\mathbf{F}\mathbf{Y}^*) \le \min\{\mathrm{rank}(\mathbf{F}),\mathrm{rank}(\mathbf{Y})\} \le k .
\end{align}
Both $\mathbf{F}$ and $\mathbf{Y}$ therefore have full column rank, so $\mathbf{F}^{+}\mathbf{F}=\mathbf{I}_k$ and $\mathbf{Y}^*(\mathbf{Y}^*)^{+}=\mathbf{I}_k$. Multiplying \eqref{eq:cancellation-target} on the left by $\mathbf{F}^{+}$ and on the right by $(\mathbf{Y}^*)^{+}$ cancels the outer factors and leaves $\mathbf{Y}^*\mathbf{F}-\mathbf{I}_k=\mathbf{0}$, which is (\ref{eq:projector-eq}).

Similarly, if $\mathbf{X}\mathbf{H}^*$ is idempotent of rank $k$, then
\begin{align}
\mathbf{X}(\mathbf{H}^*\mathbf{X}-\mathbf{I}_k)\mathbf{H}^* = \mathbf{0},
\end{align}
we have $\mathrm{rank}(\mathbf{X}) = \mathrm{rank}(\mathbf{H}) = k$ by the same sandwich, and the same cancellation gives $\mathbf{H}^*\mathbf{X}=\mathbf{I}_k$. Finally, assume \eqref{eq:projector-eq} and put $\mathbf{P}_c=\mathbf{F}\mathbf{Y}^*$. Then $\mathbf{P}_c\mathbf{F}=\mathbf{F}(\mathbf{Y}^*\mathbf{F})=\mathbf{F}$, so $\mathcal{C}(\mathbf{P}_c)=\mathcal{C}(\mathbf{F})$; and if $\mathbf{P}_c\mathbf{x}=\mathbf{0}$ then $\mathbf{Y}^*\mathbf{x}=(\mathbf{Y}^*\mathbf{F})\mathbf{Y}^*\mathbf{x}=\mathbf{Y}^*\mathbf{P}_c\mathbf{x}=\mathbf{0}$, so $\mathcal{N}(\mathbf{P}_c)=\mathcal{N}(\mathbf{Y}^*)$. Since $\mathbf{x}=\mathbf{P}_c\mathbf{x}+(\mathbf{I}_m-\mathbf{P}_c)\mathbf{x}$ and $\mathbf{F}\mathbf{w}\in\mathcal{N}(\mathbf{Y}^*)$ forces $\mathbf{w}=\mathbf{Y}^*\mathbf{F}\mathbf{w}=\mathbf{0}$, we get $\mathbb{F}^m=\mathcal{C}(\mathbf{F})\oplus\mathcal{N}(\mathbf{Y}^*)$: $\mathbf{P}_c$ is the projector onto $\mathcal{C}(\mathbf{F})$ along $\mathcal{N}(\mathbf{Y}^*)$.
\end{proof}

Equation (\ref{eq:projector-eq}) is the projector equation. It ties $\mathbf{Y}^*$ to $\mathbf{F}$ and $\mathbf{X}$ to $\mathbf{H}^*$ so tightly that the products $\mathbf{F}\mathbf{Y}^*$ and $\mathbf{X}\mathbf{H}^*$ are projectors of rank $k$, and everything that follows rests on it. One property of such projectors is used repeatedly below and is recorded here once: an idempotent fixes its own range pointwise, since $\mathbf{x}=\mathbf{P}\mathbf{z}$ gives $\mathbf{P}\mathbf{x}=\mathbf{P}^2\mathbf{z}=\mathbf{P}\mathbf{z}=\mathbf{x}$.

The equation is not new. Theorem~\ref{Theorem:projection} is the textbook characterization of an oblique projector: $\mathbf{P}_c$ is idempotent of rank $k$ if and only if $\mathbf{P}_c=\mathbf{F}(\mathbf{Y}^*\mathbf{F})^{-1}\mathbf{Y}^*$ for matrices $\mathbf{F},\mathbf{Y}$ of full column rank $k$ \emph{with $\mathbf{Y}^*\mathbf{F}$ nonsingular}, and the normalization $\mathbf{Y}^*\mathbf{F}=\mathbf{I}_k$ is a choice of scaling rather than a new object. See Meyer \cite{meyer2000matrix} for the algebra and Stewart \cite{stewart2011oblique} for the numerical theory. \v{C}ern\'{y} \cite{cerny2009characterization} identifies the normalized form with Langenhop's Lemma~2.2 \cite{langenhop1967generalized}. The equation also appears in the generalized-inverse treatment of Ben-Israel and Greville \cite{ben2003generalized}, in the two-space test/trial form familiar from Petrov--Galerkin methods, and in the work of Sorensen and Embree \cite{sorensen2016deim} on the CUR decomposition, whose coordinate-selection interpolatory projectors are instances of the solution \eqref{eq:YX-projector} constructed in Theorem~\ref{Theorem:projector-eq-sol} below.

What is new here is not the object but its use: we treat \eqref{eq:projector-eq} as a \emph{constraint to be solved} jointly with the reconstruction equation, and the pair as a design interface with $\mathbf{B}$ and $\mathbf{D}$ as its knobs. To appreciate the role of the projector equation we must characterize its solutions. The next theorem gives $\mathbf{Y}^*$ and $\mathbf{X}$ in closed form, in terms of two freely chosen matrices, and identifies the projectors they generate.

\begin{theorem}\label{Theorem:projector-eq-sol}
Let $\mathbf{F}\in\mathbb{F}^{m\times k}$ and $\mathbf{H}\in\mathbb{F}^{n\times k}$. The projector equation (\ref{eq:projector-eq}) admits solutions:
\begin{align}\label{eq:YX-projector}
\mathbf{Y}^* = (\mathbf{B}^*\mathbf{F})^+\mathbf{B}^*
\quad\text{and}\quad
\mathbf{X} = \mathbf{D}(\mathbf{H}^*\mathbf{D})^+
\end{align}
for any $\mathbf{B}\in\mathbb{F}^{m\times p}$ and $\mathbf{D}\in\mathbb{F}^{n\times q}$ such that:
\begin{align}\label{eq:rank-preserving cond}
\mathrm{rank}(\mathbf{B}^*\mathbf{F}) = \mathrm{rank}(\mathbf{H}^*\mathbf{D}) = k .
\end{align}
Furthermore, $\mathbf{Y}^*$ and $\mathbf{X}$ satisfy the generalized matrix inverse equations:
\begin{align}\label{eq:YX-PenroseEQs}
\mathbf{F}\mathbf{Y}^*\mathbf{F} = \mathbf{F}
\quad\text{and}\quad
\mathbf{H}^*\mathbf{X}\mathbf{H}^* = \mathbf{H}^*.
\end{align}
Finally, when $\mathbf{B}=\mathbf{F}$ and $\mathbf{D} = \mathbf{H}$, then $\mathbf{Y}^* = \mathbf{F}^+$ and $\mathbf{X} = (\mathbf{H}^*)^+$, and the resulting projectors are orthogonal.
\end{theorem}

The last clause holds for every $\mathbf{F}$, rank-deficient ones included, and it is not a consequence of \eqref{eq:rank-preserving cond}. Proposition~\ref{prop:BD-exhaustive} below shows that the family \eqref{eq:YX-projector} is not merely a family of solutions but \emph{all} of them. The projector $\mathbf{F}(\mathbf{B}^*\mathbf{F})^{+}\mathbf{B}^*$ produced by \eqref{eq:YX-projector} is idempotent for every $\mathbf{B}$, with no rank hypothesis. It is the operator that \v{C}ern\'{y} \cite{cerny2009characterization} writes $U(VU)^{\dagger}V$ and characterizes in full, in the general case where the range of $U$ and the null space of $V$ are not complementary subspaces. The rank condition \eqref{eq:rank-preserving cond} is what places the range, not what creates the projector.

\begin{proof}
Assume (\ref{eq:rank-preserving cond}). The dimensions $p$ and $q$ are free apart from the consequence $p,q\ge k$, which the hypothesis forces: if $p<k$ then $\mathrm{rank}(\mathbf{B}^*\mathbf{F})\le p<k$, and if $q<k$ then $\mathrm{rank}(\mathbf{H}^*\mathbf{D})\le q<k$, so in either case it is unsatisfiable. The two sides rest on different one-sided facts, so we name them separately. On the left, $\mathbf{B}^*\mathbf{F}\in\mathbb{F}^{p\times k}$ has rank $k$ with $p\ge k$, hence full \emph{column} rank, so $(\mathbf{B}^*\mathbf{F})^{+}(\mathbf{B}^*\mathbf{F})=\mathbf{I}_k$. On the right, $\mathbf{H}^*\mathbf{D}\in\mathbb{F}^{k\times q}$ has rank $k$ with $q\ge k$, hence full \emph{row} rank, so $(\mathbf{H}^*\mathbf{D})(\mathbf{H}^*\mathbf{D})^{+}=\mathbf{I}_k$. Substituting (\ref{eq:YX-projector}) into (\ref{eq:projector-eq}) therefore gives
\begin{align}
\mathbf{Y}^*\mathbf{F} = (\mathbf{B}^*\mathbf{F})^+\mathbf{B}^*\mathbf{F} = \mathbf{I}_k = \mathbf{H}^*\mathbf{D}(\mathbf{H}^*\mathbf{D})^+ = \mathbf{H}^*\mathbf{X}.
\end{align}
The generalized inverse equations (\ref{eq:YX-PenroseEQs}) follow immediately: $\mathbf{F}\mathbf{Y}^*\mathbf{F}=\mathbf{F}(\mathbf{Y}^*\mathbf{F})=\mathbf{F}$ and $\mathbf{H}^*\mathbf{X}\mathbf{H}^*=(\mathbf{H}^*\mathbf{X})\mathbf{H}^*=\mathbf{H}^*$. Note that the rank hypothesis is exactly what makes these two cancellations legitimate; without it $\mathbf{F}\mathbf{Y}^*\mathbf{F}\ne\mathbf{F}$ in general, as Proposition~\ref{prop:oblique-projector} makes precise.

For the last clause, take $\mathbf{B}=\mathbf{F}$ and $\mathbf{D}=\mathbf{H}$. Then $\mathbf{Y}^*=(\mathbf{F}^*\mathbf{F})^{+}\mathbf{F}^*=\mathbf{F}^{+}$ and $\mathbf{X}=\mathbf{H}(\mathbf{H}^*\mathbf{H})^{+}=(\mathbf{H}^*)^{+}$: both are instances of the classical identity $(\mathbf{M}^*\mathbf{M})^{+}\mathbf{M}^*=\mathbf{M}^{+}$ of the least-squares g-inverse calculus, which holds for every matrix, with no rank hypothesis; see Ben-Israel and Greville \cite{ben2003generalized}. The resulting projectors $\mathbf{F}\mathbf{F}^{+}$ and $(\mathbf{H}^*)^{+}\mathbf{H}^*$ are Hermitian idempotents, hence orthogonal.
\end{proof}

Matrices of the form (\ref{eq:YX-projector}) therefore solve the projector equation and satisfy the generalized inverse equations, provided the pair $(\mathbf{B},\mathbf{D})$ satisfies (\ref{eq:rank-preserving cond}). In the special case $\mathbf{B} = \mathbf{F}$, $\mathbf{D} = \mathbf{H}$ they give rise to \emph{orthogonal} projectors. That special case is clean algebra, but it is not on its own a recommendation for computing that way. For $\mathbf{F}$ of full column rank, forming $\mathbf{F}^*\mathbf{F}$ squares the spectral condition number, $\kappa_2(\mathbf{F}^*\mathbf{F})=\kappa_2(\mathbf{F})^2$. A worst-case forward-error bound of order $u\,\kappa_2(\mathbf{F})^2$, with $u$ the unit roundoff, therefore reaches order one near $\kappa_2(\mathbf{F})\approx u^{-1/2}$ --- about $10^{8}$ in double precision. How much accuracy is actually lost depends on the algorithm, the rank threshold, the scaling and the structure of $\mathbf{F}$. See Golub and Van Loan \cite{golub2013matrix} for the error analysis of the normal equations. Other admissible choices of the design parameters can give \emph{oblique} projectors, and that is the freedom the rest of the paper exploits. Proposition~\ref{prop:oblique-projector} records what survives when \eqref{eq:rank-preserving cond} fails, and which choices stay orthogonal.

When the rank condition \eqref{eq:rank-preserving cond} fails, something does survive: the construction \eqref{eq:YX-projector} still returns idempotent matrices, but they are no longer of rank $k$ and the projector equation is lost.

The construction can discard some directions even when the columns of $\mathbf F$ span a larger space. Proposition~\ref{prop:oblique-projector} identifies exactly which vectors it keeps and which it sends to zero. It also gives three checks: whether it keeps the whole column space of $\mathbf F$, whether it discards exactly the vectors unseen by $\mathbf B^*$, and whether the projection is orthogonal. This is a reformulation of the range and kernel characterization in \v{C}ern\'y \cite[Thm.~3.1]{cerny2009characterization}.

\begin{proposition}[effective trial and test spaces]\label{prop:oblique-projector}
Let $\mathbf F\in\mathbb F^{m\times k}$ and $\mathbf B\in\mathbb F^{m\times p}$ be arbitrary. Set
$$
\mathbf M=\mathbf B^*\mathbf F,\quad
\mathbf Y^*=\mathbf M^+\mathbf B^*,\quad
\mathbf P_c=\mathbf F\mathbf Y^*,\quad \rho=\operatorname{rank}\mathbf M,
\qquad
\mathcal S=\mathcal C(\mathbf F\mathbf F^*\mathbf B),\quad
\mathcal T=\mathcal C(\mathbf B\mathbf B^*\mathbf F).
$$
Then $\mathbf P_c$ is the projector onto $\mathcal S$ along $\mathcal T^\perp$. In particular,
$$
\mathbb F^m=\mathcal S\oplus\mathcal T^\perp,\qquad
\mathcal C(\mathbf P_c)=\mathcal S\subseteq\mathcal C(\mathbf F),\qquad
\mathcal N(\mathbf P_c)=\mathcal T^\perp=\mathcal N(\mathbf Y^*),
$$
and $\dim\mathcal S=\dim\mathcal T=\rho$. The following tests characterize the projection.
\begin{itemize}
\item The projector leaves every vector in $\mathcal C(\mathbf F)$ unchanged if and only if $\rho=\operatorname{rank}\mathbf F$. Equivalently, either of the following equalities holds precisely under that rank condition:
\begin{equation}\label{eq:oblique-equivalences}
\mathbf P_c\mathbf F=\mathbf F,
\qquad
\mathcal S=\mathcal C(\mathbf F).
\end{equation}

\item The equality $\mathcal N(\mathbf P_c)=\mathcal N(\mathbf B^*)$ holds if and only if $\rho=\operatorname{rank}\mathbf B$. Equivalently, it holds if and only if $\mathcal T=\mathcal C(\mathbf B)$.

\item The projector is orthogonal, $\mathbf P_c=\mathbf P_c^*$, if and only if $\mathcal S=\mathcal T$.

\end{itemize}

The failure of the normalized projector equation is measured exactly by
$$
\mathbf Y^*\mathbf F=\mathbf M^+\mathbf M,
\qquad
\|\mathbf Y^*\mathbf F-\mathbf I_k\|_F=\sqrt{k-\rho}.
$$
The dual statements apply to $\mathbf P_r^*$ with $(\mathbf F,\mathbf B,\mathbf Y)$ replaced by $(\mathbf H,\mathbf D,\mathbf X)$.
\end{proposition}
\begin{proof}
The second Penrose equation gives $\mathbf P_c^2=\mathbf P_c$. Since $\mathbf B^*\mathbf P_c\mathbf F=\mathbf M$ and $\mathbf P_c=\mathbf F\mathbf M^+\mathbf B^*$, rank inequalities give $\operatorname{rank}\mathbf P_c=\rho$. Also $\mathcal C(\mathbf Y^*)\subseteq\mathcal C(\mathbf M^*)$, while $\mathbf Y^*\mathbf F=\mathbf M^+\mathbf M$ has rank $\rho$, so these spaces coincide. Thus
$$
\mathcal C(\mathbf P_c)=\mathbf F\mathcal C(\mathbf M^*)
=\mathcal C(\mathbf F\mathbf F^*\mathbf B)=\mathcal S.
$$
Applying the same argument to $\mathbf P_c^*=\mathbf B(\mathbf F^*\mathbf B)^+\mathbf F^*$ gives $\mathcal C(\mathbf P_c^*)=\mathcal T$, hence $\mathcal N(\mathbf P_c)=\mathcal T^\perp$. The kernels of $\mathbf P_c$ and $\mathbf Y^*$ agree by containment and equality of ranks. The equivalences follow from idempotence and the inclusions of the effective spaces in the corresponding factor spaces. The residual follows because $\mathbf M^+\mathbf M$ is an orthogonal projector of rank $\rho$.
\end{proof}

To use the proposition, form $\mathbf B^*\mathbf F$ and compare its rank with that of $\mathbf F$. Equal ranks mean that every vector represented by $\mathbf F$ is retained; a smaller rank means that some of those directions are lost. The space $\mathcal S$ tells us exactly what remains. Every input has a unique decomposition $\mathbf x=\mathbf s+\mathbf z$ with $\mathbf s\in\mathcal S$ and $\mathbf z\in\mathcal T^\perp$, and the projector returns $\mathbf P_c\mathbf x=\mathbf s$. Comparing $\mathcal S$ and $\mathcal T$ then decides whether the projection is orthogonal. Keeping the whole column space is weaker than satisfying $\mathbf Y^*\mathbf F=\mathbf I_k$: the latter requires $\rho=k$, so it is impossible when $\mathbf F$ itself is rank deficient.

\begin{figure}[tbp]
\centering
\begingroup
\normalfont\normalsize
\definecolor{trial}{HTML}{1764A0}
\definecolor{test}{HTML}{A86616}
\definecolor{null}{HTML}{A74256}
\definecolor{panelbg}{HTML}{F3F5F7}
\definecolor{rule}{HTML}{CDD3D9}
\newcommand{\col}{\mathcal C}
\newcommand{\trialspace}{\mathcal S}
\newcommand{\testspace}{\mathcal T}
\begin{tikzpicture}[
  x=1cm,y=1cm, font=\normalsize,
  line cap=round,line join=round, >=Stealth,
  heading/.style={font=\normalsize\bfseries,anchor=west},
  note/.style={font=\normalsize,align=center},
  space label/.style={font=\normalsize,inner sep=2pt},
  trial line/.style={trial,line width=1.15pt},
  test line/.style={test,line width=.85pt},
  null line/.style={null,line width=.85pt,dash pattern=on 4pt off 3pt},
  vector/.style={black!30,line width=.45pt,->},
  projection/.style={black,line width=1.15pt,-{Stealth[length=5pt,width=4pt]},shorten >=2.1pt},
  stage/.style={align=center,text width=2.7cm,text height=1.7ex,text depth=.6ex,inner sep=0pt},
  map/.style={line width=.8pt,-{Stealth[length=4.5pt,width=3.5pt]}},
  map label/.style={font=\normalsize,above=3pt,inner sep=0pt}
]
\node[font=\normalsize] at (8,12.40) {$P=F(B^*F)^+B^*$};
\node[note] at (8,11.84) {$M=B^*F,\qquad Y^*=M^+B^*,\qquad P=FY^*$};
\node[note,text=trial] at (3.825,11.28) {Effective trial space};
\node[note,text=test] at (12.175,11.28) {Effective test space};
\node at (3.825,10.84) {$\trialspace=\col(FF^*B)=\col(P)$};
\node at (12.175,10.84) {$\testspace=\col(BB^*F)=\col(Y)$};
\node at (8,10.25) {$\mathbb F^m=\trialspace\oplus\testspace^\perp,
  \qquad x=Px+(I-P)x$};

\foreach \offset/\name in {0/left,8.35/right}{
  \begin{scope}[xshift=\offset cm]
    \fill[panelbg,rounded corners=3pt] (0,5.22) rectangle (7.65,9.28);
    \coordinate (\name-origin) at (2.70,6.65);
  \end{scope}
}
\node[heading] at (0,9.65) {(a) A larger design space};
\node[heading] at (8.35,9.65) {(b) Loss of trial-space rank};
\node[note,anchor=north east,text=black!65] at (7.39,9.04) {$\col(B)=\mathbb R^2$};
\node[note,anchor=north east,text=black!65] at (15.74,9.04) {$\col(F)=\mathbb R^2$};

\begin{scope}[shift={(left-origin)},scale=1.4]
  \coordinate (a0) at (0,0);
  \coordinate (ax) at (1,1);
  \coordinate (ap) at (1.5,0);
  \draw[vector] (a0)--(ax);
  \draw[vector] (a0)--(ap);
  \draw[trial line] (-1.62,0)--(3.03,0);
  \draw[test line] (-1.48,-.74)--(2.43,1.215);
  \draw[null line] (-.77,1.54)--(.43,-.86);
  \draw[projection] (ax)--(ap);
  \fill (a0) circle (1.15pt);
  \fill (ax) circle (1.7pt);
  \fill[trial] (ap) circle (1.9pt);
  \node[space label,above left] at (ax) {$x$};
  \node[space label,below=4pt] at (ap) {$Px$};
  \node[space label,below left] at (a0) {$0$};
  \node[space label,trial,anchor=south east] at (3.05,.10) {$\trialspace=\col(F)$};
  \node[space label,test,above] at (2.05,1.13) {$\testspace$};
  \node[space label,null,anchor=south west] at (-.86,1.56) {$\testspace^\perp=\ker P$};
\end{scope}

\begin{scope}[shift={(right-origin)},scale=1.4]
  \coordinate (b0) at (0,0);
  \coordinate (bx) at (1,1);
  \coordinate (bp) at (1.6,.4);
  \draw[vector] (b0)--(bx);
  \draw[vector] (b0)--(bp);
  \draw[trial line] (-1.62,-.405)--(3.03,.7575);
  \draw[test line] (-.90,-.90)--(1.53,1.53);
  \draw[null line] (-1.53,1.53)--(.88,-.88);
  \draw[projection] (bx)--(bp);
  \fill (b0) circle (1.15pt);
  \fill (bx) circle (1.7pt);
  \fill[trial] (bp) circle (1.9pt);
  \node[space label,above left] at (bx) {$x$};
  \node[space label,below=4pt] at (bp) {$Px$};
  \node[space label,below left] at (b0) {$0$};
  \node[space label,trial,anchor=south east] at (3.05,.86) {$\trialspace\subsetneq\col(F)$};
  \node[space label,test,anchor=south east] at (1.43,1.44) {$\testspace=\col(B)$};
  \node[space label,null,anchor=south west] at (-1.58,1.56) {$\testspace^\perp=\ker P$};
\end{scope}

\node at (3.825,4.61) {$F=\begin{pmatrix}1\\0\end{pmatrix},\quad
  B=\begin{pmatrix}1&1\\0&1\end{pmatrix},\quad
  P=\begin{pmatrix}1&\tfrac12\\0&0\end{pmatrix}$};
\node at (12.175,4.61) {$F=\begin{pmatrix}2&0\\0&1\end{pmatrix},\quad
  B=\begin{pmatrix}1\\1\end{pmatrix},\quad
  P=\frac15\begin{pmatrix}4&4\\1&1\end{pmatrix}$};
\node[note] at (3.825,3.88) {$\trialspace=\operatorname{span}\{(1,0)^{\mathsf T}\},\quad
  \testspace=\operatorname{span}\{(2,1)^{\mathsf T}\}$};
\node[note] at (12.175,3.88) {$\trialspace=\operatorname{span}\{(4,1)^{\mathsf T}\},\quad
  \testspace=\operatorname{span}\{(1,1)^{\mathsf T}\}$};
\node[note] at (3.825,3.46) {$\rho=k=1$: all of $\col(F)$ is fixed; $\testspace\subsetneq\col(B)$.};
\node[note] at (12.175,3.46) {$\rho=1<k=2$: $P^2=P$, but $\trialspace\subsetneq\col(F)$.};

\draw[rule] (0,3.07)--(16,3.07);
\node at (8,2.70) {$\dim\trialspace=\dim\testspace=\rho=\operatorname{rank}(M),\qquad
  P=P^*\ \Longleftrightarrow\ \trialspace=\testspace$};
\node[note,text=black!65] at (8,2.26)
  {Both panels use $x=(1,1)^{\mathsf T}$; the black arrow is parallel to $\ker P=\testspace^\perp$.};
\draw[rule] (0,1.91)--(16,1.91);

\node[heading] at (0,1.52) {From sampling to ambient projection};
\node[stage] (input) at (1.70,.76) {$x$};
\node[note] at (1.70,.20) {$\in\mathbb F^m$};
\node[stage] (sample) at (5.90,.76) {$b=B^*x$};
\node[note] at (5.90,.20) {$\in\mathbb F^p$};
\node[stage] (coefficient) at (10.10,.76) {$y=M^+b$};
\node[note] at (10.10,.20) {$\in\col(M^*)$};
\node[stage] (output) at (14.30,.76) {$Px=Fy$};
\node[note] at (14.30,.20) {$\in\trialspace$};
\draw[map] (input.east)--node[map label]{$B^*$}(sample.west);
\draw[map] (sample.east)--node[map label]{$M^+$}(coefficient.west);
\draw[map,trial] (coefficient.east)--node[map label]{$F$}(output.west);
\node[note] at (5.90,-.38) {$\mathbb F^p=\col(M)\mathbin{\oplus_\perp}\ker M^*$};
\node[note] at (10.10,-.38) {$\mathbb F^k=\col(M^*)\mathbin{\oplus_\perp}\ker M$};
\node[note,text=black!75] at (8,-1.10)
  {$M^+$ annihilates $\ker M^*$ and inverts $M:\col(M^*)\to\col(M)$;\\[3pt] $F$ then maps $\col(M^*)$ bijectively onto $\trialspace$.};
\node[note,text=black!65] at (8,-2.06)
  {Karpowicz--Strang (2026), Sec.~5.1, Fig.~6:\\[3pt] $(\bar C,P_{\rm sketch},G_P)=(F,B,Y^*)$;\quad
   \href{https://arxiv.org/abs/2602.00386}{arXiv:2602.00386}.};
\end{tikzpicture}
\endgroup
\space
\caption{Effective trial and test spaces in Proposition~\ref{prop:oblique-projector}. A wide design matrix can have $\mathcal T\subsetneq\mathcal C(\mathbf B)$ even when the full trial space is retained (a). Rank loss can instead leave a proper effective trial space (b). The ambient projector acts onto $\mathcal S$ along $\mathcal T^\perp$ and is orthogonal exactly when $\mathcal S=\mathcal T$. The lower diagram connects the construction to the sampled inverse of Karpowicz--Strang \cite[Sec.~5.1, Fig.~6]{karpowicz2026generalized}.}
\label{fig:prop1-subspaces}
\end{figure}

The sampled inverse $G_P=(P_{\rm sketch}^*\bar C)^+P_{\rm sketch}^*$ of \cite[Sec.~5.1, Fig.~6]{karpowicz2026generalized} becomes $\mathbf Y^*$ under $(\bar C,P_{\rm sketch},G_P)=(\mathbf F,\mathbf B,\mathbf Y^*)$. For $\mathbf b=\mathbf B^*\mathbf x$, the inverse $\mathbf M^+$ annihilates the component of $\mathbf b$ in $\mathcal N(\mathbf M^*)$ and maps the remainder into $\mathcal C(\mathbf M^*)$; multiplication by $\mathbf F$ then maps it bijectively onto $\mathcal S$. Figure~\ref{fig:prop1-subspaces} distinguishes this compressed orthogonal splitting from the ambient oblique projection.

So the construction of Theorem~\ref{Theorem:projector-eq-sol} always yields idempotent matrices, but when $\mathrm{rank}(\mathbf{B}^*\mathbf{F})<\mathrm{rank}(\mathbf{F})$ they project onto a \emph{proper} subspace of $\mathcal{C}(\mathbf{F})$ and $\mathbf{F}\mathbf{Y}^*\mathbf{F}\ne\mathbf{F}$. Whenever $\mathrm{rank}(\mathbf{B}^*\mathbf{F})<k$ the projector equation fails, with the exact residual $\|\mathbf{Y}^*\mathbf{F}-\mathbf{I}_k\|_F=\sqrt{k-\mathrm{rank}(\mathbf{B}^*\mathbf{F})}$. Finally, a word on type: the projectors of \eqref{eq:YX-projector} are \emph{oblique} in general, and \emph{orthogonal} precisely when $\mathcal{C}(\mathbf{Y})=\mathcal{C}(\mathbf{P}_c)$, which under $\mathrm{rank}(\mathbf{B}^*\mathbf{F})=\mathrm{rank}(\mathbf{F})$ reads $\mathcal{C}(\mathbf{B}\mathbf{B}^*\mathbf{F})=\mathcal{C}(\mathbf{F})$. Any $\mathbf{B}$ with $\mathcal{C}(\mathbf{B})=\mathcal{C}(\mathbf{F})$ satisfies it, the default choice $\mathbf{B}=\mathbf{F}$ included, so $\mathcal{C}(\mathbf{B})=\mathcal{C}(\mathbf{F})$ is sufficient for orthogonality. It is not necessary, the condition being strictly weaker. Take $\mathbf{F}=\mathbf{e}_1\in\mathbb{F}^{3\times1}$ with $\mathbf{B}=[\mathbf{e}_1\ \ \mathbf{e}_2]$: here $\mathbf{P}_c=\mathrm{diag}(1,0,0)$ is orthogonal although $\mathcal{C}(\mathbf{B})\ne\mathcal{C}(\mathbf{F})$. This characterization of idempotents is classical. See Meyer \cite{meyer2000matrix} and, for the numerical theory, Stewart \cite{stewart2011oblique}.

\subsection{The reconstruction equation and meta-factorization procedure}

The projector equation is only half of what a reconstruction needs. Given $\mathbf{F}$ and $\mathbf{H}$, it supplies matrices $\mathbf{Y}$ and $\mathbf{X}$ that keep $\mathbf{F}\mathbf{Y}^*$ and $\mathbf{X}\mathbf{H}^*$ projectors of rank $k$. It identifies their ranges, $\mathcal C(\mathbf P_c)=\mathcal C(\mathbf F)$ and $\mathcal C(\mathbf P_r^*)=\mathcal C(\mathbf H)$, but does not ensure that those spaces contain the column and row data of $\mathbf A$. Those inclusions must be enforced separately. If the projector equation holds, $\mathbf{A}$ is transformed into some
\begin{align}
\mathbf{M} =
(\mathbf{F}\mathbf{Y}^*)\mathbf{A}(\mathbf{X}\mathbf{H}^*).
\end{align}
But when $\mathbf{F}$ and $\mathbf{H}$ are arbitrary, we cannot expect $\mathbf{M}$ to approximate~$\mathbf{A}$. For the reconstruction to succeed we must also require the columns of $\mathbf{F}$ to span a subspace containing the column space of $\mathbf{A}$, and those of $\mathbf{H}$ a subspace containing its row space. The next theorem shows that this requirement is not merely necessary but exactly right: given the projector equation, the two span conditions are all that reconstruction needs.

\begin{figure}[!ht]
\centering
\begin{tikzpicture}[
  x=1cm,y=1cm,font=\normalsize,
  certificate/.style={draw=black!45,fill=black!3,rounded corners=2pt,
    text width=6.4cm,minimum width=6.8cm,minimum height=0.75cm,align=center},
  consequence/.style={draw=black!60,fill=white,rounded corners=2pt,
    text width=6.4cm,minimum width=6.8cm,minimum height=1.45cm,align=center},
  reconstruction/.style={draw=black,fill=black!8,line width=0.8pt,
    rounded corners=2pt,minimum width=4.8cm,minimum height=0.80cm,
    align=center},
  freedom/.style={draw=black!40,dashed,rounded corners=2pt,
    text width=13.4cm,minimum width=13.8cm,minimum height=0.75cm,align=center},
  flow/.style={->,line width=0.7pt,black!70}]

\node[font=\normalsize\bfseries] at (-3.70,3.65) {Column-space certificate};
\node[certificate] (colspan) at (-3.70,2.90)
  {$\mathcal{C}(\mathbf A)\subseteq\mathcal{C}(\mathbf P_c)=\mathcal{C}(\mathbf F)$};
\node[consequence] (colfixed) at (-3.70,1.22)
  {$\mathbf P_c\mathbf A=\mathbf A$\\[2pt]
   every column of $\mathbf A$\\is fixed};

\node[font=\normalsize\bfseries] at (3.70,3.65) {Row-space certificate};
\node[certificate] (rowspan) at (3.70,2.90)
  {$\mathcal{C}(\mathbf A^{*})\subseteq\mathcal{C}(\mathbf P_r^{*})=\mathcal{C}(\mathbf H)$};
\node[consequence] (rowfixed) at (3.70,1.22)
  {$\mathbf A\mathbf P_r=\mathbf A$\\[2pt]
   equivalently, every column of $\mathbf A^{*}$\\is fixed};

\node[reconstruction] (recon) at (0,-0.65) {$\mathbf P_c\mathbf A\mathbf P_r=\mathbf A$};
\node[freedom] (free) at (0,-2.10)
  {The action away from these two data subspaces remains design-dependent.};

\draw[flow] (colspan) -- (colfixed);
\draw[flow] (rowspan) -- (rowfixed);
\draw[flow] (colfixed.south) -- (recon.north west);
\draw[flow] (rowfixed.south) -- (recon.north east);
\node[fill=white,inner sep=2pt] at (0,0.23) {combine};
\draw[black!45,dashed,line width=0.6pt] (recon) -- (free);
\end{tikzpicture}
\caption{The two fixed-subspace certificates imply exact reconstruction. Since
$\mathbf{P}_c$ is idempotent and
$\mathcal{C}(\mathbf{A})\subseteq\mathcal{C}(\mathbf{P}_c)=\mathcal{C}(\mathbf{F})$,
it fixes every column of $\mathbf{A}$, giving $\mathbf{P}_c\mathbf{A}=\mathbf{A}$.
Likewise, $\mathcal{C}(\mathbf{A}^{*})\subseteq\mathcal{C}(\mathbf{P}_r^{*})=\mathcal{C}(\mathbf{H})$
gives $\mathbf{A}\mathbf{P}_r=\mathbf{A}$. Combining the two identities yields
$\mathbf{P}_c\mathbf{A}\mathbf{P}_r=\mathbf{A}$. Only the action away from these
data subspaces depends on the design matrices.}
\label{fig:sandwich}
\end{figure}
\begin{theorem}\label{Theorem:meta-factorization}
Let $\mathbf{A}\in\mathbb{F}^{m\times n}$, $\mathbf{F}\in\mathbb{F}^{m\times k}$, $\mathbf{H}\in\mathbb{F}^{n\times k}$, and let $\mathbf{Y}\in\mathbb{F}^{m\times k}$, $\mathbf{X}\in\mathbb{F}^{n\times k}$ solve the projector equation (\ref{eq:projector-eq}). Then the reconstruction equation
\begin{align}\label{eq:meta-factorization-eq}
\mathbf{A} =
\mathbf{F} \mathbf{Y}^*
\mathbf{A}
\mathbf{X} \mathbf{H}^*
\end{align}
holds if and only if
\begin{align}\label{eq:span-inclusions}
\mathcal{C}(\mathbf{A})\subseteq\mathcal{C}(\mathbf{F})
\quad\text{and}\quad
\mathcal{C}(\mathbf{A}^*)\subseteq\mathcal{C}(\mathbf{H}).
\end{align}
The implication from \eqref{eq:meta-factorization-eq} to \eqref{eq:span-inclusions} holds without the projector equation.
\end{theorem}

\begin{proof}
Put $\mathbf{P}_c=\mathbf{F}\mathbf{Y}^*$ and $\mathbf{P}_r=\mathbf{X}\mathbf{H}^*$. By Theorem~\ref{Theorem:projection} both are idempotent of rank $k$. We first identify their ranges, the step on which everything rests. Since every column of $\mathbf{P}_c=\mathbf{F}\mathbf{Y}^*$ is a linear combination of columns of $\mathbf{F}$, we have $\mathcal{C}(\mathbf{P}_c)\subseteq\mathcal{C}(\mathbf{F})$, and since $\mathbf{P}_c\mathbf{F}=\mathbf{F}(\mathbf{Y}^*\mathbf{F})=\mathbf{F}$ we have $\mathcal{C}(\mathbf{F})\subseteq\mathcal{C}(\mathbf{P}_c)$; hence
\begin{align}\label{eq:range-identification}
\mathcal{C}(\mathbf{P}_c) = \mathcal{C}(\mathbf{F}),
\qquad\text{and dually}\qquad
\mathcal{C}(\mathbf{P}_r^*) = \mathcal{C}(\mathbf{H}),
\end{align}
the second from $\mathbf{P}_r^*\mathbf{H}=\mathbf{H}(\mathbf{H}^*\mathbf{X})^*=\mathbf{H}$.

Suppose \eqref{eq:span-inclusions} holds. An idempotent fixes its own range pointwise, so $\mathbf{P}_c$ fixes every vector of $\mathcal{C}(\mathbf{A})\subseteq\mathcal{C}(\mathbf{F})=\mathcal{C}(\mathbf{P}_c)$, and $\mathbf{P}_r^*$ every vector of $\mathcal{C}(\mathbf{A}^*)\subseteq\mathcal{C}(\mathbf{H})=\mathcal{C}(\mathbf{P}_r^*)$; Figure~\ref{fig:sandwich} sets out the whole implication. As a~result $\mathbf{P}_c\mathbf{A} = \mathbf{A}$ and $\mathbf{P}_r^*\mathbf{A}^* = \mathbf{A}^*$, which gives $\mathbf{P}_c \mathbf{A} \mathbf{P}_r = \mathbf{A}$.

Conversely, suppose $\mathbf{A}=\mathbf{P}_c\mathbf{A}\mathbf{P}_r$. Then $\mathcal{C}(\mathbf{A})=\mathcal{C}(\mathbf{P}_c\mathbf{A}\mathbf{P}_r)\subseteq\mathcal{C}(\mathbf{P}_c)=\mathcal{C}(\mathbf{F})$ and $\mathcal{C}(\mathbf{A}^*)=\mathcal{C}(\mathbf{P}_r^*\mathbf{A}^*\mathbf{P}_c^*)\subseteq\mathcal{C}(\mathbf{P}_r^*)=\mathcal{C}(\mathbf{H})$, which is \eqref{eq:span-inclusions}. This direction uses only that $\mathbf{A}$ factors through $\mathbf{F}$ and $\mathbf{H}^*$, not the projector equation.
\end{proof}

The proof is valid over $\mathbb{R}$ and over $\mathbb{C}$ alike. It takes adjoints only where the statement itself does, and it contains no step passing from $\mathbf{M}^*\mathbf{M}=\mathbf{0}$ to $\mathbf{M}=\mathbf{0}$.

Under the column-space inclusions $\mathcal{C}(\mathbf{A})\subseteq\mathcal{C}(\mathbf{F})$ and $\mathcal{C}(\mathbf{A}^*)\subseteq\mathcal{C}(\mathbf{H})$, equivalently when $\mathbf{A}=\mathbf{F}\mathbf{C}\mathbf{H}^*$ for some $\mathbf{C}\in\mathbb{F}^{k\times k}$, the projector equation is sufficient for the reconstruction, but not necessary. Fix a scalar $c\ne0$ and take $\mathbf{Y}^*\mathbf{F}=c\mathbf{I}_k$ and $\mathbf{H}^*\mathbf{X}=c^{-1}\mathbf{I}_k$. That pair reconstructs every $\mathbf{A}$ of the form $\mathbf{A}=\mathbf{F}\mathbf{C}\mathbf{H}^*$ with $\mathbf{C}\in\mathbb{F}^{k\times k}$, while for $c\ne0,1$ equation (\ref{eq:projector-eq}) fails and $\mathbf{F}\mathbf{Y}^*$ is not idempotent. Excluding $c=1$ is not a technicality. From $(\mathbf{F}\mathbf{Y}^*)^2-\mathbf{F}\mathbf{Y}^*=(c-1)\mathbf{F}\mathbf{Y}^*$ and $\mathbf{F}\mathbf{Y}^*\ne\mathbf{0}$ whenever $c\ne0$, the value $c=1$ is exactly the case in which the rescaled pair is the projector equation again. What suffices, under those same inclusions, is that $\mathbf{Y}^*$ be a $\{1\}$-inverse of $\mathbf{F}$ and $\mathbf{X}$ a $\{1\}$-inverse of $\mathbf{H}^*$. A $\{1\}$-inverse of $\mathbf{F}$ is any matrix satisfying the first Penrose equation $\mathbf{F}\mathbf{Y}^*\mathbf{F}=\mathbf{F}$, and in general there are many. The projector equation is the normalization that also pins the rank.

Basic meta-factorization proceeds according to the following steps (illustrated with MATLAB code):
\begin{enumerate}
\item[\textbf{Step 1}:] Given $\mathbf A$ of rank $r$, choose $r\le k\le\min(m,n)$ and matrices $\mathbf F\in\mathbb F^{m\times k}$, $\mathbf H\in\mathbb F^{n\times k}$ with linearly independent columns, so that $\operatorname{rank}\mathbf F=\operatorname{rank}\mathbf H=k$ and the inclusions \eqref{eq:span-inclusions} hold. At $k=r$ their columns are bases of the data spaces; at $k>r$ they are bases of larger containing spaces. The independence requirement makes the normalized equation in Step 2 solvable.

\item[\textbf{Step 2}:] Choose design matrices satisfying \eqref{eq:rank-preserving cond} and solve the projector equation $\mathbf{Y}^*\mathbf{F} = \mathbf{H}^*\mathbf{X} = \mathbf{I}_k$ for the transforms $\mathbf{Y}$ and $\mathbf{X}$. Factoring the adjoint on the right gives square triangular solves even for rectangular design matrices:
\begin{lstlisting}
[QY,RY] = qr(B'*F,0);       % B'*F is p-by-k, full column rank.
[QX,RX] = qr((H'*D)',0);    % (H'*D)' is q-by-k, full column rank.
Y       = (RY\(QY'*B'))';   % Select admissible B and D;
X       = (D*QX)/RX';       % B = F and D = H are the defaults.
\end{lstlisting}

\item[\textbf{Step 3}:] Calculate the mixing matrix $\mathbf{G} = \mathbf{Y}^*\mathbf{A}\mathbf{X}$ and (if needed) perform its factorization to obtain the desired decomposition $\mathbf{A} = \mathbf{F}\mathbf{G}\mathbf{H}^*$:
\begin{lstlisting}
G       = Y'*A*X;
% If needed, perform additional factorization of G, e.g.,
% [U,S,V] = svd(G), [Q,R] = qr(G) or [L,U] = lu(G), etc.
Ar      = F*G*H';
\end{lstlisting}
\end{enumerate}
A separate rank-deficient variant uses $\mathbf Y^*=\mathbf F^+$ and $\mathbf X=(\mathbf H^*)^+$. It reconstructs exactly under the two span inclusions but generally does not solve the normalized projector equation. With the full-rank choices in Steps 1--3, the procedure finds solutions of the meta-factorization equations,
\begin{align}\label{eq:generator}
\setstackgap{L}{15pt}\def\stacktype{L}\def\sz{\scriptstyle}
\begin{aligned}
\stackunder{\mathbf{G}}{\sz (k\times k)} &=
\stackunder{\mathbf{Y}^*}{\sz (k\times m)}&&
\stackunder{\mathbf{A}}{\sz (m\times n)} &&
\stackunder{\mathbf{X},}{\sz (n\times k)}
\\
\stackunder{\mathbf{A}}{\sz (m\times n)} &=
\stackunder{\mathbf{F}}{\sz (m\times k)} &&
\stackunder{\mathbf{G}}{\sz (k\times k)} &&
\stackunder{\mathbf{H}^*.}{\sz (k\times n)}
\end{aligned}
\end{align}

The outer factors $\mathbf{F}$ and $\mathbf{H}$ choose spaces containing the column and row spaces of $\mathbf{A}$. The admissible design matrices then choose the certifying transforms $\mathbf{Y}^*$ and $\mathbf{X}$, giving $\mathbf{G}=\mathbf{Y}^*\mathbf{A}\mathbf{X}$ and $\mathbf{A}=\mathbf{F}\mathbf{G}\mathbf{H}^*$. For the full-column-rank outer factors used above, the exact core is unique even though the certifying transforms need not be. Factoring this core exposes additional structure that can be absorbed into the outer factors, as in Section~\ref{section:How...}. Merely refactoring a fixed core does not change the supplied certifying projectors.

The proof of Theorem~\ref{Theorem:meta-factorization} uses only that $\mathbf{F}\mathbf{Y}^*$ and $\mathbf{X}\mathbf{H}^*$ are idempotent \emph{of rank $k$}, which by Theorem~\ref{Theorem:projection} is the projector equation itself. Bare idempotence is not enough, as $\mathbf{A}=\mathbf{I}_2$ with $\mathbf{Y}^*=\mathrm{diag}(1,0)$ shows, where the reconstruction error is exactly~$1$. Combining the theorem with the explicit solutions of Theorem~\ref{Theorem:projector-eq-sol} gives the form in which the framework is used in practice, with $\mathbf{B}$ and $\mathbf{D}$ as design parameters.

\begin{corollary}\label{Theorem:generalized-meta-factorization}
Let $\mathbf{A}$, $\mathbf{F}$, $\mathbf{H}$ be as in Theorem~\ref{Theorem:meta-factorization}, and let $\mathbf{B}\in\mathbb{F}^{m\times p}$ and $\mathbf{D}\in\mathbb{F}^{n\times q}$ satisfy
\begin{align}\label{eq:BD-admissible}
\mathrm{rank}(\mathbf{B}^*\mathbf{F}) = \mathrm{rank}(\mathbf{H}^*\mathbf{D}) = k,
\end{align}
which forces $p,q\ge k$, and put
\begin{align}
\mathbf{Y}^* = (\mathbf{B}^*\mathbf{F})^+\mathbf{B}^*
\quad\text{and}\quad
\mathbf{X} = \mathbf{D}(\mathbf{H}^*\mathbf{D})^+.
\end{align}
If $\mathcal{C}(\mathbf{A})\subseteq\mathcal{C}(\mathbf{F})$ and $\mathcal{C}(\mathbf{A}^*)\subseteq\mathcal{C}(\mathbf{H})$, then the reconstruction equation holds.
\end{corollary}

\begin{proof}
Condition \eqref{eq:BD-admissible} makes $\mathbf{B}^*\mathbf{F}\in\mathbb{F}^{p\times k}$ of full column rank, so $(\mathbf{B}^*\mathbf{F})^{+}(\mathbf{B}^*\mathbf{F})=\mathbf{I}_k$ and $\mathbf{Y}^*\mathbf{F}=\mathbf{I}_k$; dually $\mathbf{H}^*\mathbf{D}\in\mathbb{F}^{k\times q}$ is of full row rank, so $(\mathbf{H}^*\mathbf{D})(\mathbf{H}^*\mathbf{D})^{+}=\mathbf{I}_k$ and $\mathbf{H}^*\mathbf{X}=\mathbf{I}_k$. The projector equation therefore holds and Theorem~\ref{Theorem:meta-factorization} applies.
\end{proof}

Corollary~\ref{Theorem:generalized-meta-factorization} supplies the certifying pair in closed form. The following proposition shows that varying the admissible design matrices reaches every solution of the projector equation.

\begin{proposition}[exhaustiveness of the $(\mathbf{B},\mathbf{D})$ parametrization]\label{prop:BD-exhaustive}
Let $\mathbf{F}\in\mathbb{F}^{m\times k}$ and $\mathbf{H}\in\mathbb{F}^{n\times k}$. On the right-hand side below, the design matrices range over $\mathbf B\in\mathbb F^{m\times p}$ and $\mathbf D\in\mathbb F^{n\times q}$ of arbitrary widths $p,q\ge k$. Then
\begin{align}\label{eq:BD-is-PE}
\Bigl\{(\mathbf{Y},\mathbf{X}) : \mathbf{Y}^*\mathbf{F}=\mathbf{H}^*\mathbf{X}=\mathbf{I}_k\Bigr\}
=
\Bigl\{\bigl(\bigl((\mathbf{B}^*\mathbf{F})^{+}\mathbf{B}^*\bigr)^*,\ \mathbf{D}(\mathbf{H}^*\mathbf{D})^{+}\bigr) : (\mathbf{B},\mathbf{D}) \text{ satisfying } \eqref{eq:BD-admissible}\Bigr\}.
\end{align}
\end{proposition}

\begin{proof}
Every pair on the right is given by \eqref{eq:YX-projector}, and \eqref{eq:BD-admissible} is the rank condition \eqref{eq:rank-preserving cond} of Theorem~\ref{Theorem:projector-eq-sol}. That theorem gives \eqref{eq:projector-eq}, proving the inclusion $\supseteq$. Conversely, take any solution $(\mathbf Y,\mathbf X)$ of \eqref{eq:projector-eq} and choose $\mathbf B=\mathbf Y$, $\mathbf D=\mathbf X$, with $p=q=k$. Then $\mathbf B^*\mathbf F=\mathbf H^*\mathbf D=\mathbf I_k$, so \eqref{eq:BD-admissible} holds. Since $\mathbf I_k^+=\mathbf I_k$, the formulas \eqref{eq:YX-projector} return $\mathbf Y^*$ and $\mathbf X$ unchanged, proving $\subseteq$.
\end{proof}

The parametrization is surjective but not unique. For example, multiplying $\mathbf{B}$ by a nonzero scalar leaves $\mathbf{Y}$ unchanged, and the same holds for $\mathbf{D}$ and $\mathbf{X}$. Thus the design matrices provide a complete encoding of certifying pairs with redundancy in the choice of parameters.

The two equations do not travel together as the inner dimension $k$ moves relative to $r=\mathrm{rank}(\mathbf{A})$, and several constructions below live in different regimes. The next proposition records the three cases once and for all. Taken together, they show that the two equations decouple, neither implying the other.

\begin{figure}[!ht]
\centering
\begin{tikzpicture}[x=1cm,y=1cm,font=\normalsize,
    panel/.style={align=center},
    zone/.style={draw=black!50,line width=.6pt}]
\fill[black!12] (0,.45) rectangle (3.5,2.25);
\fill[black!5] (3.5,.45) rectangle (14,2.25);
\draw[zone] (0,.45) rectangle (14,2.25);
\draw[black,line width=.8pt] (0,0)--(14,0);
\draw[black,line width=.8pt] (3.5,-.12)--(3.5,2.25);
\draw[black,line width=.8pt] (14,-.12)--(14,.45);
\node[panel,text width=3.1cm] at (1.75,1.35) {No exact core\\can reconstruct $\mathbf A$};
\node[panel,text width=10cm] at (8.75,1.35) {Full-rank containing factors and normalized transforms\\give exact reconstruction throughout this region};
\node[panel,anchor=south] at (1.75,2.30) {$k<r$};
\node[panel,anchor=south] at (8.75,2.30) {$r\le k\le\min(m,n)$};
\node[panel,anchor=north] at (3.5,-.18) {$r$};
\node[panel,anchor=north] at (14,-.18) {$\min(m,n)$};
\node[panel,anchor=north] at (7,-.85) {At $k=r$: nonsingular core.\quad At $k>r$: singular core of rank $r$.};
\end{tikzpicture}
\caption{The minimum feasible inner dimension is $r$, but exact reconstruction with normalized projectors is possible at every $r\le k\le\min(m,n)$. Full-column-rank outer factors must contain the two data spaces. The special rank-deficient sketch construction of Proposition~\ref{prop:regimes}(c) has a separate normalization obstruction, given by \eqref{eq:PE-residual}; that obstruction is not a reconstruction-error bound.}
\label{fig:regimes-line}
\end{figure}
\begin{proposition}\label{prop:regimes}
Let $\mathbf{A}\in\mathbb{F}^{m\times n}$ with $r=\mathrm{rank}(\mathbf{A})$ and let the inner dimension satisfy $1\le k\le\min(m,n)$.
\begin{enumerate}
\item[(a)] \emph{Below the rank, $k<r$.} The reconstruction equation is unattainable: $\mathbf{A}\ne\mathbf{F}\mathbf{G}\mathbf{H}^*$ for every $\mathbf{F}\in\mathbb{F}^{m\times k}$, $\mathbf{G}\in\mathbb{F}^{k\times k}$, $\mathbf{H}\in\mathbb{F}^{n\times k}$. The projector equation is nevertheless solvable.
\item[(b)] \emph{At the rank, $k=r$.} Assume the two inclusions $\mathcal{C}(\mathbf{A})\subseteq\mathcal{C}(\mathbf{F})$ and $\mathcal{C}(\mathbf{A}^*)\subseteq\mathcal{C}(\mathbf{H})$, and let $\mathbf{Y}$, $\mathbf{X}$ be any pair satisfying the projector equation, for example one produced from admissible design matrices by Theorem~\ref{Theorem:projector-eq-sol}. Then both equations hold, the two inclusions are equalities of spans, and the mixing matrix is unique and nonsingular, $\mathbf{G}=\mathbf{F}^{+}\mathbf{A}(\mathbf{H}^*)^{+}$.
\item[(c)] \emph{Above the rank, $k>r$.} There exists a core $\mathbf G$ with $\mathbf A=\mathbf F\mathbf G\mathbf H^*$ if and only if the two span inclusions hold. If $\mathbf F$ and $\mathbf H$ also have full column rank, that core is unique, has rank $r<k$, and equals $\mathbf Y^*\mathbf A\mathbf X$ for every pair satisfying $\mathbf Y^*\mathbf F=\mathbf H^*\mathbf X=\mathbf I_k$. The inclusions are then strict and the core is singular. This existence statement is distinct from reconstruction by arbitrary supplied transforms. For the following special sketched construction, normalization has a different obstruction. For any $\mathbf{\Omega}_c\in\mathbb{F}^{n\times k}$ and $\mathbf{\Omega}_r\in\mathbb{F}^{m\times k}$, set $\mathbf{F}=\mathbf{A}\mathbf{\Omega}_c$, $\mathbf{H}^*=\mathbf{\Omega}_r^*\mathbf{A}$, $\mathbf{B}=\mathbf{\Omega}_r$, $\mathbf{D}=\mathbf{\Omega}_c$ and $\mathbf{M}=\mathbf{\Omega}_r^*\mathbf{A}\mathbf{\Omega}_c$. Then $\mathbf{Y}^*\mathbf{F}=\mathbf{M}^{+}\mathbf{M}$ and $\mathbf{H}^*\mathbf{X}=\mathbf{M}\mathbf{M}^{+}$, and the projector equation fails with the exact residual
\begin{align}\label{eq:PE-residual}
\|\mathbf{Y}^*\mathbf{F}-\mathbf{I}_k\|_F = \|\mathbf{H}^*\mathbf{X}-\mathbf{I}_k\|_F = \sqrt{k-\mathrm{rank}(\mathbf{M})}\ \ge\ 1
\end{align}
on both sides, for every choice of sketches or index sets.
\end{enumerate}
\end{proposition}

\begin{proof}
(a) is the rank bound: $\mathrm{rank}(\mathbf{F}\mathbf{G}\mathbf{H}^*)\le k<r$. The projector equation is solvable there, as $\mathbf{F}=\mathbf{Y}=\mathbf{I}_m(:,1{:}k)$ and $\mathbf{H}=\mathbf{X}=\mathbf{I}_n(:,1{:}k)$ show, these being legitimate choices because $k<r\le\min(m,n)$. To verify the existence clause of (c), necessity follows from the column spaces of a product. Under the two inclusions, $\mathbf F\mathbf F^+\mathbf A=\mathbf A$ and $\mathbf A(\mathbf H^*)^+\mathbf H^*=\mathbf A$, so $\mathbf F^+\mathbf A(\mathbf H^*)^+$ is a core even when the outer factors are deficient. For (b) and the full-column-rank clause of (c), the projector equation forces $\mathrm{rank}(\mathbf{F})=\mathrm{rank}(\mathbf{H})=k$, since $k=\mathrm{rank}(\mathbf{I}_k)=\mathrm{rank}(\mathbf{Y}^*\mathbf{F})\le\mathrm{rank}(\mathbf{F})\le k$. Theorem~\ref{Theorem:meta-factorization} then gives the reconstruction exactly under the inclusions, which are available if and only if $k\ge r$, and full column rank of $\mathbf{F}$ together with full row rank of $\mathbf{H}^*$ makes the solution unique. If $k>r$ then $\mathrm{rank}(\mathbf{G})=\mathrm{rank}(\mathbf{F}\mathbf{G}\mathbf{H}^*)=r<k$, so $\mathbf{G}$ is singular. For \eqref{eq:PE-residual}, $\mathbf{M}^{+}\mathbf{M}$ is an orthogonal projector of rank $\mathrm{rank}(\mathbf{M})$, so $\mathbf{I}_k-\mathbf{M}^{+}\mathbf{M}$ is an orthogonal projector of rank $k-\mathrm{rank}(\mathbf{M})$ and $\|\mathbf{I}_k-\mathbf{M}^{+}\mathbf{M}\|_F^2=\mathrm{tr}(\mathbf{I}_k-\mathbf{M}^{+}\mathbf{M})=k-\mathrm{rank}(\mathbf{M})$; the same computation applies on the other side, and $\mathrm{rank}(\mathbf{M})\le\mathrm{rank}(\mathbf{A})=r<k$.
\end{proof}

The inclusions in (b) cannot be dropped. Take $\mathbf{A}=\mathrm{diag}(1,0)$ and $\mathbf{F}=\mathbf{H}=\mathbf{e}_2$: the projector equation holds with $\mathbf{B}=\mathbf{F}$ and $\mathbf{D}=\mathbf{H}$, while $\|\mathbf{A}-\mathbf{F}\mathbf{Y}^*\mathbf{A}\mathbf{X}\mathbf{H}^*\|_F=1$. The matrices $\mathbf{F}=\mathbf{A}\mathbf{\Omega}_c$ and $\mathbf{H}^*=\mathbf{\Omega}_r^*\mathbf{A}$ of part (c) are equally concrete: they are the sketched and index-selected constructions of Section~\ref{section:Additional...}, taken with their own design parameters. That matters for the residual \eqref{eq:PE-residual}, which is governed by $\mathbf{M}$ only for that design pair. For a different pair it is $\sqrt{k-\mathrm{rank}(\mathbf{B}^*\mathbf{F})}$, and that need not equal $\sqrt{k-\mathrm{rank}(\mathbf{M})}$. Within the design pair of part (c), on the other hand, the residual is at least one for every choice of sketches or index sets, so increasing the unrecompressed square sketch width above $r$ cannot restore this normalization. The residual measures failure of the normalized projector equation, not failure of reconstruction: the sketched core $\mathbf M^+$ reconstructs exactly when $\operatorname{rank}\mathbf M=r$, as Section~\ref{subsection:NystroemCUR} explains.

\begin{table}[tbp]
\centering
\begin{tabular}{@{}l p{0.54\textwidth} p{0.22\textwidth}@{}}
\toprule
Inner dimension & Full-rank outer factors with normalized transforms & Exact core\\
\midrule
$k<r$ & The inverse equations can hold, but containing factor spaces are impossible. & None\\
$k=r$ & Both equations hold under the span inclusions. & Unique, nonsingular\\
$r<k\le\min(m,n)$ & Both equations hold for larger containing factor spaces. & Unique, rank $r<k$\\
\bottomrule
\end{tabular}
\caption{General rank regimes. The special sketch obstruction \eqref{eq:PE-residual} concerns deficient sampled outer factors and is separate from this table. It is a defect in the normalized equation, not a lower bound on reconstruction error.}
\label{tab:regimes}
\end{table}

Table~\ref{tab:regimes} and Figure~\ref{fig:regimes-line} separate reconstruction from the structure of the core. For arbitrary transforms, the computed core $\mathbf Y^*\mathbf A\mathbf X$ always has rank at most $\min(k,r)$. A supplied exact core with deficient outer factors need not have rank $r$ or be singular: $\mathbf F=\mathbf H=[\,1\ \ 0\,]$, $\mathbf G=\mathbf I_2$, $\mathbf A=1$ is an example. With full-column-rank outer factors, every exact factorization at $k>r$ has a singular mixing matrix. An SVD can still have its characteristic diagonal form, with zero singular values appended, whereas the exact CPQR construction can no longer have $\mathbf{G}=\mathbf{I}_k$ (Proposition~\ref{prop:CPQR}). For the unrecompressed square sketched construction with its own design pair, Proposition~\ref{prop:regimes}(c) gives a projector-equation residual of at least one above the rank. The oversampled generalized Nystr\"{o}m construction in Section~\ref{subsection:NystroemCUR} recompresses its factors and is treated separately.

We can adjust the meta-factorization procedure to admit low-rank bases as well. Given $\mathbf{A}$ of rank $r$ and a target inner dimension $\ell<r$, we find rank-$\ell$ projectors by solving the projector equation (\ref{eq:projector-eq}) with $\mathbf{I}_\ell$ for the low-rank bases $\mathbf{F}_\ell$ and $\mathbf{H}_\ell$. Setting $\mathbf{B} = \mathbf{F}_\ell$ and $\mathbf{D} = \mathbf{H}_\ell$ is preferred in this case. Here the reconstruction equation must be replaced by an approximation, and the substitution is not a matter of caution but of arithmetic: $\mathrm{rank}(\mathbf{F}_\ell\mathbf{G}\mathbf{H}_\ell^*)\le\ell<r$, so
\begin{align}
\mathbf{A} \ne \mathbf{F}_\ell\mathbf{G}\mathbf{H}_\ell^*
\qquad\text{for every choice of }\ \mathbf{F}_\ell,\ \mathbf{G},\ \mathbf{H}_\ell,
\end{align}
and the best attainable error is the truncated-SVD optimum \cite{golub2013matrix} $\|\mathbf{A}-\mathbf{F}_\ell\mathbf{G}\mathbf{H}_\ell^*\|_F\ge(\sum_{i>\ell}\sigma_i^2)^{1/2}$, attained by the leading $\ell$ singular vectors. The procedure also admits designing one projector only, in which case it is enough to work with either $\mathbf{F}$ or $\mathbf{H}$.

\subsection{Factorizations as solutions of linear matrix equations}

For fixed outer factors $\mathbf F$ and $\mathbf H$, the question is whether a mixing matrix $\mathbf G$ can reproduce $\mathbf A$, and, if so, which mixing matrices do so. Penrose's theorem answers both questions \cite[Thm.~2]{penrose1955generalized}. The condition is that the outer factors contain the required column and row spaces. If either outer factor is rank deficient, different mixing matrices may give the same product; the theorem describes all of them.

We state this classical result to identify the foundation of the construction. Meta-factorization makes the choice of certifying transforms explicit through $\mathbf B$ and $\mathbf D$, connects those choices to the familiar factorizations developed below, and uses the same equations to test changes to an existing factorization. Changing a certifying transform for fixed full-column-rank factors leaves the exact mixing matrix unchanged. Replacing an outer factor is a further question, treated in Section~\ref{section:Replacement}. Only the first Penrose equation is needed for the consistency and all-solutions formulas, so the theorem also holds with the corresponding $\{1\}$-inverses in place of the pseudoinverses \cite{penrose1955generalized,ben2003generalized}.

\begin{theorem}[Penrose \cite{penrose1955generalized}, Theorem 2]\label{thm:Penrose}
Let $\mathbf{A}\in\mathbb{F}^{m\times n}$, $\mathbf{F}\in\mathbb{F}^{m\times k}$ and $\mathbf{H}\in\mathbb{F}^{n\times k}$ be arbitrary.\footnote{Penrose states the result for a rectangular unknown $\mathbf{G}$; we quote the square case $\mathbf{G}\in\mathbb{F}^{k\times k}$, which is the one meta-factorization uses. Corollary~\ref{cor:linear-system} below is an instance of Penrose's rectangular form, not of the square case quoted here.} A necessary and sufficient condition for the equation
\begin{align}\label{eq:PenroseLME}
\mathbf{F}\mathbf{G}\mathbf{H}^* = \mathbf{A}
\end{align}
to have a solution $\mathbf{G}$ is
\begin{align}\label{eq:PenroseNSC}
\mathbf{F}\mathbf{F}^+\mathbf{A}(\mathbf{H}^*)^+\mathbf{H}^* = \mathbf{A},
\end{align}
in which case the general solution is
\begin{align}\label{eq:PenroseProj}
\mathbf{G} =
\mathbf{F}^+\mathbf{A}(\mathbf{H}^*)^+
+
\mathbf{E} - \mathbf{F}^+\mathbf{F}\mathbf{E}\mathbf{H}^*(\mathbf{H}^*)^+,
\end{align}
where $\mathbf{E}\in\mathbb{F}^{k\times k}$ is arbitrary.
\end{theorem}

The arbitrary term in \eqref{eq:PenroseProj} records changes in the mixing matrix that the outer factors erase: if $\mathbf G_0$ is one solution, every solution is $\mathbf G_0+\mathbf K$ with $\mathbf F\mathbf K\mathbf H^*=\mathbf0$. Retaining this term matters for rank-deficient outer factors; the pseudoinverse choice alone would give only one possible core. The next lemma states exactly when the core is unique. In particular, a singular core can still be unique when both outer factors have full column rank.

\begin{lemma}\label{lem:singleton}
Let $\mathbf{A}\in\mathbb{F}^{m\times n}$, $\mathbf{F}\in\mathbb{F}^{m\times k}$ and $\mathbf{H}\in\mathbb{F}^{n\times k}$ satisfy the criterion \eqref{eq:PenroseNSC}, so that \eqref{eq:PenroseLME} is solvable. Then its solution set is a singleton if and only if $\mathrm{rank}(\mathbf{F})=\mathrm{rank}(\mathbf{H})=k$.
\end{lemma}

\begin{proof}
The solution set is a coset of the kernel of the linear map $\mathbf{G}\mapsto\mathbf{F}\mathbf{G}\mathbf{H}^*$, so it is a singleton exactly when that kernel is trivial. If $\mathrm{rank}(\mathbf{F})=\mathrm{rank}(\mathbf{H})=k$ then $\mathbf{F}$ is left-invertible and $\mathbf{H}^*$ right-invertible, so $\mathbf{F}\mathbf{G}\mathbf{H}^*=\mathbf{0}$ forces $\mathbf{G}=\mathbf{0}$. Conversely, suppose $\mathrm{rank}(\mathbf{F})<k$ and take $\mathbf{v}\in\mathcal{N}(\mathbf{F})$, $\mathbf{v}\ne\mathbf{0}$, together with any $\mathbf{w}\ne\mathbf{0}$; then $\mathbf{G}=\mathbf{v}\mathbf{w}^*\ne\mathbf{0}$ and $\mathbf{F}\mathbf{G}\mathbf{H}^*=(\mathbf{F}\mathbf{v})(\mathbf{w}^*\mathbf{H}^*)=\mathbf{0}$. If instead $\mathrm{rank}(\mathbf{H})<k$, take $\mathbf{w}\in\mathcal{N}(\mathbf{H})$, $\mathbf{w}\ne\mathbf{0}$, and any $\mathbf{v}\ne\mathbf{0}$; then $\mathbf{G}=\mathbf{v}\mathbf{w}^*\ne\mathbf{0}$ and $\mathbf{G}\mathbf{H}^*=\mathbf{v}(\mathbf{H}\mathbf{w})^*=\mathbf{0}$. In either case the kernel is nontrivial.
\end{proof}

Penrose's condition \eqref{eq:PenroseNSC} uses orthogonal projectors. The design construction allows other projectors, so we must identify what this freedom changes. For fixed full-column-rank outer factors, every admissible choice of $\mathbf B$ and $\mathbf D$ gives the same solvability decision and, when a solution exists, the same mixing matrix. The next theorem records this guarantee explicitly: one may choose the certifying transforms without changing the exact factorization or checking solvability afresh for each choice. It connects the classical solution formula to the whole family of transforms generated by the design matrices. The family consists of transforms and their projectors; the core itself remains unique by Lemma~\ref{lem:singleton}.

\begin{theorem}\label{thm:Penrose-oblique}
Let $\mathbf{A}\in\mathbb{F}^{m\times n}$, $\mathbf{F}\in\mathbb{F}^{m\times k}$ and $\mathbf{H}\in\mathbb{F}^{n\times k}$ with $\mathrm{rank}(\mathbf{F})=\mathrm{rank}(\mathbf{H})=k$. Then for \emph{every} pair $(\mathbf{B},\mathbf{D})$ with $\mathbf{B}\in\mathbb{F}^{m\times p}$, $\mathbf{D}\in\mathbb{F}^{n\times q}$ and $\mathrm{rank}(\mathbf{B}^*\mathbf{F})=\mathrm{rank}(\mathbf{H}^*\mathbf{D})=k$, a~necessary and sufficient condition for the equation
\begin{align}\label{eq:PenroseLME-oblique}
\mathbf{F}\mathbf{G}\mathbf{H}^* = \mathbf{A}
\end{align}
to have a solution $\mathbf{G}$ is
\begin{align}\label{eq:PenroseNSC-oblique}
\mathbf{F}\mathbf{Y}^*
\mathbf{A}
\mathbf{X}\mathbf{H}^* = \mathbf{A},
\end{align}
where
\begin{align}
\mathbf{Y}^* = (\mathbf{B}^*\mathbf{F})^+\mathbf{B}^*
\quad\text{and}\quad
\mathbf{X} = \mathbf{D}(\mathbf{H}^*\mathbf{D})^+.
\end{align}
The truth value of this criterion does not depend on the choice of $(\mathbf{B},\mathbf{D})$: it holds for one admissible pair if and only if it holds for all of them. When it holds, the solution is \emph{unique},
\begin{align}\label{eq:PenroseProj-oblique}
\mathbf{G} =
\mathbf{Y}^*
\mathbf{A}
\mathbf{X}
=
\mathbf{F}^{+}\mathbf{A}(\mathbf{H}^*)^{+}.
\end{align}
\end{theorem}

\begin{proof}
The proof reduces both conditions to the same pair of span inclusions, which mention neither $\mathbf{B}$ nor $\mathbf{D}$; the asserted independence of the design parameters then follows. Since $\mathrm{rank}(\mathbf{B}^*\mathbf{F})=k$ and $\mathbf{B}^*\mathbf{F}\in\mathbb{F}^{p\times k}$, the matrix $\mathbf{B}^*\mathbf{F}$ has full column rank, so $(\mathbf{B}^*\mathbf{F})^{+}(\mathbf{B}^*\mathbf{F})=\mathbf{I}_k$ and $\mathbf{Y}^*\mathbf{F}=\mathbf{I}_k$; dually $\mathbf{H}^*\mathbf{D}$ has full row rank and $\mathbf{H}^*\mathbf{X}=\mathbf{I}_k$. Both hypotheses force $\mathrm{rank}(\mathbf{F})=\mathrm{rank}(\mathbf{H})=k$. Hence $\mathbf{P}_c=\mathbf{F}\mathbf{Y}^*$ is idempotent with $\mathcal{C}(\mathbf{P}_c)=\mathcal{C}(\mathbf{F})$, and $\mathbf{P}_r=\mathbf{X}\mathbf{H}^*$ is idempotent with $\mathcal{C}(\mathbf{P}_r^*)=\mathcal{C}(\mathbf{H})$.

Suppose (\ref{eq:PenroseNSC-oblique}) holds, i.e.\ $\mathbf{P}_c\mathbf{A}\mathbf{P}_r=\mathbf{A}$. Then $\mathbf{G}=\mathbf{Y}^*\mathbf{A}\mathbf{X}$ solves (\ref{eq:PenroseLME-oblique}). Conversely, if $\mathbf{F}\mathbf{G}\mathbf{H}^*=\mathbf{A}$ for some $\mathbf{G}$, then $\mathcal{C}(\mathbf{A})\subseteq\mathcal{C}(\mathbf{F})=\mathcal{C}(\mathbf{P}_c)$ and $\mathcal{C}(\mathbf{A}^*)\subseteq\mathcal{C}(\mathbf{H})=\mathcal{C}(\mathbf{P}_r^*)$; an idempotent fixes its own range, so $\mathbf{P}_c\mathbf{A}=\mathbf{A}$ and $\mathbf{A}\mathbf{P}_r=\mathbf{A}$, whence $\mathbf{P}_c\mathbf{A}\mathbf{P}_r=\mathbf{A}$.

Both conditions are therefore equivalent to the pair of inclusions $\mathcal{C}(\mathbf{A})\subseteq\mathcal{C}(\mathbf{F})$ and $\mathcal{C}(\mathbf{A}^*)\subseteq\mathcal{C}(\mathbf{H})$, which mention neither $\mathbf{B}$ nor $\mathbf{D}$; this is the asserted independence of the design parameters. Finally, $\mathrm{rank}(\mathbf{F})=k$ makes $\mathbf{F}$ left-invertible and $\mathrm{rank}(\mathbf{H})=k$ makes $\mathbf{H}^*$ right-invertible, so the solution of $\mathbf{F}\mathbf{G}\mathbf{H}^*=\mathbf{A}$ is unique. It equals $\mathbf{F}^{+}\mathbf{A}(\mathbf{H}^*)^{+}$, because under the criterion the two inclusions give $\mathbf{F}\mathbf{F}^{+}\mathbf{A}=\mathbf{A}$ and $\mathbf{A}(\mathbf{H}^*)^{+}\mathbf{H}^*=\mathbf{A}$, so that $\mathbf{F}\bigl(\mathbf{F}^{+}\mathbf{A}(\mathbf{H}^*)^{+}\bigr)\mathbf{H}^*=\mathbf{A}$ exhibits it as a solution; evaluating $\mathbf{Y}^*\mathbf{A}\mathbf{X}$ at any admissible $(\mathbf{B},\mathbf{D})$ must therefore return the same matrix.
\end{proof}

For fixed full-column-rank outer factors, the design matrices therefore choose how to certify the factorization. They do not choose a different exact mixing matrix. In generalized-inverse terminology, a $\{1\}$-inverse of a full-column-rank $\mathbf F$ is exactly a left inverse. Proposition~\ref{prop:BD-exhaustive} shows that admissible $\mathbf B$ generate every such inverse, and the right-hand construction does the same for $\mathbf H^*$. The theorem identifies what is constant across this parametrized family: the solvability decision and the unique exact core. This invariance is a specialization of Penrose's consistency theorem \cite{penrose1955generalized,ben2003generalized}. The individual projectors also belong to the classical construction \cite{cerny2009characterization}, including the coordinate-selection projectors of \cite{sorensen2016deim}.

The invariance theorem concerns exact reconstruction of the given matrix. Its different certifying transforms can nevertheless respond differently to perturbed data. A two-dimensional example shows why this distinction matters. Let $\mathbf e_1=(1,0)^{\mathsf T}$ and $\mathbf e_2=(0,1)^{\mathsf T}$ be the coordinate vectors in $\mathbb R^2$. Take the one-column matrix $\mathbf A=\mathbf F=\mathbf e_1$, with $\mathbf H=\mathbf D=\mathbf X=[1]$, so the exact core is the scalar $\mathbf G=[1]$. For every real $t$, the design $\mathbf B_t=(1,t)^{\mathsf T}$ satisfies $\mathbf B_t^*\mathbf F=1$ and gives
$$
\mathbf Y_t^*=[\,1\ \ t\,],\qquad
\mathbf P_t=\mathbf F\mathbf Y_t^*
=\begin{bmatrix}1&t\\0&0\end{bmatrix},\qquad
\|\mathbf P_t\|_2=\sqrt{1+t^2}.
$$
Every choice recovers the same core, $\mathbf Y_t^*\mathbf A\mathbf X=[1]$, and satisfies $\mathbf P_t\mathbf e_1=\mathbf e_1$: each projector leaves the original data vector unchanged. This is what it means to say that the projector fixes $\mathbf e_1$. Now add an error $\varepsilon\mathbf e_2$ in the second coordinate, with $\varepsilon\in\mathbb R$. Direct multiplication gives
$$
\mathbf P_t(\mathbf e_1+\varepsilon\mathbf e_2)
=\begin{bmatrix}1&t\\0&0\end{bmatrix}
 \begin{bmatrix}1\\\varepsilon\end{bmatrix}
=\begin{bmatrix}1+t\varepsilon\\0\end{bmatrix}
=(1+t\varepsilon)\mathbf e_1.
$$
Relative to the original data $\mathbf e_1$, an input error of size $|\varepsilon|$ produces an output error of size $|t\varepsilon|$. For $\varepsilon\ne0$, the perturbed vector lies outside $\mathcal C(\mathbf F)$, so the theorem makes no exact-reconstruction claim for it. The example shows that exactness on the original data leaves error amplification uncontrolled: every $t$ passes the rank check, however large $|t|$ becomes. Numerical robustness therefore requires information beyond admissibility. For fixed two-sided projectors, the corresponding bound is

$$
\|\mathbf P_c(\mathbf A+\boldsymbol\Delta)\mathbf P_r-\mathbf A\|_2
\le \|\mathbf P_c\|_2\,\|\boldsymbol\Delta\|_2\,\|\mathbf P_r\|_2,
$$
whenever $\mathbf P_c\mathbf A\mathbf P_r=\mathbf A$.

Stewart \cite{stewart2011oblique} studies the numerical behavior of oblique projectors. Here the amplification follows directly from their geometry. For an admissible left transform, put $\mathbf P=\mathbf F\mathbf Y^*$ and choose orthonormal bases $\mathbf Q_F$ and $\mathbf Q_Y$ of $\mathcal C(\mathbf F)$ and $\mathcal C(\mathbf Y)$. Then
$$
\mathbf P=\mathbf Q_F(\mathbf Q_Y^*\mathbf Q_F)^{-1}\mathbf Q_Y^*,
\qquad
\|\mathbf P\|_2
=\frac{1}{\sigma_{\min}(\mathbf Q_Y^*\mathbf Q_F)}
=\frac{1}{\cos\theta},
$$
where $\theta$ is the largest principal angle between those two spaces. The norm equality follows because multiplication by the two orthonormal bases preserves the nonzero singular values. Proposition~\ref{prop:oblique-projector} identifies the effective test space as $\mathcal C(\mathbf Y)=\mathcal C(\mathbf B\mathbf B^*\mathbf F)$. It agrees with $\mathcal C(\mathbf B)$ when $\mathrm{rank}(\mathbf B)=k$, but it can be smaller when the design matrix has more than $k$ independent columns. For example, $\mathbf F=\mathbf e_1$ and $\mathbf B=\bigl(\begin{smallmatrix}1&1\\0&1\end{smallmatrix}\bigr)$ give $\mathcal C(\mathbf B)=\mathbb R^2$, while $\mathcal C(\mathbf Y)=\mathrm{span}\{(1,1/2)^{\mathsf T}\}$ and $\|\mathbf P\|_2=\sqrt5/2$. A rule for choosing designs with controlled amplification would require numerical analysis beyond the exact identities established here.

The projectors so built are generally oblique: for admissible $(\mathbf{B},\mathbf{D})$, $\mathbf{F}\mathbf{Y}^*$ is Hermitian if and only if $\mathcal{C}(\mathbf{B}\mathbf{B}^*\mathbf{F})=\mathcal{C}(\mathbf{F})$, equivalently $\mathbf{Y}^*=\mathbf{F}^{+}$. The simpler condition $\mathcal{C}(\mathbf{B})=\mathcal{C}(\mathbf{F})$ is sufficient, and necessary only when $p=k$. Already $\mathbf{B}=\mathbf{I}_m$ returns the orthogonal projector for every $\mathbf{F}$. Section~\ref{section:Additional...} puts this freedom to work. Specializing it to column and row selection returns the known exactness condition for the CUR decomposition.

The single-right-hand-side case makes another limit of the design freedom explicit. For a consistent system $\mathbf A\mathbf x=\mathbf c$, every admissible $\mathbf B$ below gives the same minimum-norm particular solution and the same family of homogeneous solutions. Thus choosing $\mathbf B$ can change the inverse operator but does not supply additional solutions to the equation. The following specialization of Penrose's classical formula records this fact in the framework's notation and retains the null-space term needed when $\mathrm{rank}(\mathbf A)<n$ \cite{penrose1955generalized,ben2003generalized}.
\begin{corollary}\label{cor:linear-system}
Let $\mathbf{A}\in\mathbb{F}^{m\times n}$, $\mathbf{B}\in\mathbb{F}^{m\times p}$ with $\mathrm{rank}(\mathbf{B}^*\mathbf{A}) = \mathrm{rank}(\mathbf{A})$, and put $\mathbf{Y}^* = (\mathbf{B}^*\mathbf{A})^+\mathbf{B}^*$. If $\mathbf{c}\in\mathcal{C}(\mathbf{A})$, then the general solution of the vector equation $\mathbf{A}\mathbf{x} = \mathbf{c}$ is
\begin{align}\label{eq:linear-system-solution}
\mathbf{x} = \mathbf{Y}^*\mathbf{c}+(\mathbf{I}_n-\mathbf{Y}^*\mathbf{A})\mathbf{y},
\qquad \mathbf{y}\in\mathbb{F}^{n} \text{ arbitrary},
\end{align}
and the parametrization is non-degenerate: $\mathrm{rank}(\mathbf{I}_n-\mathbf{Y}^*\mathbf{A}) = n-\mathrm{rank}(\mathbf{A})$.
\end{corollary}

\begin{proof}
The rank hypothesis gives $\mathcal{C}(\mathbf{A}^*\mathbf{B})=\mathcal{C}(\mathbf{A}^*)$, so the orthogonal projector $(\mathbf{B}^*\mathbf{A})^{+}(\mathbf{B}^*\mathbf{A})$ onto $\mathcal{C}(\mathbf{A}^*\mathbf{B})$ fixes the rows of $\mathbf{A}$, i.e.\ $\mathbf{A}\mathbf{Y}^*\mathbf{A}=\mathbf{A}$: the matrix $\mathbf{Y}^*$ is a $\{1\}$-inverse of $\mathbf{A}$. Writing $\mathbf{c}=\mathbf{A}\mathbf{z}$, every vector \eqref{eq:linear-system-solution} solves the equation, because $\mathbf{A}\mathbf{x}=\mathbf{A}\mathbf{Y}^*\mathbf{A}\mathbf{z}+\mathbf{A}\mathbf{y}-\mathbf{A}\mathbf{Y}^*\mathbf{A}\mathbf{y}=\mathbf{c}$. Conversely every solution is captured, by taking $\mathbf{y}=\mathbf{x}$: then $\mathbf{Y}^*\mathbf{c}+(\mathbf{I}_n-\mathbf{Y}^*\mathbf{A})\mathbf{x}=\mathbf{x}+\mathbf{Y}^*(\mathbf{c}-\mathbf{A}\mathbf{x})=\mathbf{x}$. The rank statement holds because $\mathbf{Y}^*\mathbf{A}$ is idempotent of rank $\mathrm{rank}(\mathbf{A})$.
\end{proof}

Every solution is represented by the formula for any matrix $\mathbf{Y}^*$. The rank condition is needed to ensure that every represented vector solves the equation.

The rank hypothesis also fixes the type of inverse. Under it, $\mathbf{Y}^*\mathbf{A}=(\mathbf{B}^*\mathbf{A})^{+}(\mathbf{B}^*\mathbf{A})$ is the orthogonal projector onto $\mathcal{C}(\mathbf{A}^*\mathbf{B})=\mathcal{C}(\mathbf{A}^*)$, which is $\mathbf{A}^{+}\mathbf{A}$. The first and fourth Penrose equations follow at once. The second needs one further step, namely $\mathcal{C}(\mathbf{Y}^*)\subseteq\mathcal{C}\bigl((\mathbf{B}^*\mathbf{A})^{+}\bigr)=\mathcal{C}(\mathbf{A}^*\mathbf{B})=\mathcal{C}(\mathbf{A}^*)$, which gives $\mathbf{Y}^*\mathbf{A}\mathbf{Y}^*=\mathbf{A}^{+}\mathbf{A}\mathbf{Y}^*=\mathbf{Y}^*$. That step is not dispensable. Take $\mathbf{A}=\mathrm{diag}(1,0)$ with $\mathbf{Z}=\mathbf{I}_2$: it satisfies $\mathbf{Z}\mathbf{A}=\mathbf{A}^{+}\mathbf{A}$ and the first and fourth equations, but not the second. So $\mathbf{Y}^*$ is a $\{1,2,4\}$-inverse of $\mathbf{A}$ for every admissible $\mathbf{B}$.

That settles what the design parameter can and cannot change. By the classical minimum-norm property of $\{1,4\}$-inverses (Ben-Israel and Greville \cite{ben2003generalized}), the particular solution $\mathbf{Y}^*\mathbf{c}$ is the minimum-norm solution $\mathbf{A}^{+}\mathbf{c}$, whatever $\mathbf{B}$ is. More is true: since $\mathbf{Y}^*\mathbf{A}=\mathbf{A}^{+}\mathbf{A}$ as matrices, \emph{both} terms of \eqref{eq:linear-system-solution} are the same for every admissible $\mathbf{B}$. The design parameter is therefore invisible in the solution set of a consistent system. Where it does act is on the map $\mathbf{Y}^*$ itself, on right-hand sides outside $\mathcal{C}(\mathbf{A})$: there $\mathbf{Y}^*\mathbf{c}$ can depend on $\mathbf{B}$, and need not be a least-squares solution. It is one for every $\mathbf{c}$ exactly when the third Penrose equation holds as well. That equation, and with it $\mathbf{Y}^*=\mathbf{A}^{+}$, holds exactly when $\mathcal{C}(\mathbf{B}\mathbf{B}^*\mathbf{A})=\mathcal{C}(\mathbf{A})$ --- in particular for $\mathcal{C}(\mathbf{B})=\mathcal{C}(\mathbf{A})$ and for $\mathbf{B}=\mathbf{I}_m$, and with $\mathbf{B}=\mathbf{A}$ it reduces to $\mathbf{A}^{+}$. When $\mathrm{rank}(\mathbf A)=n$, the null-space term vanishes and the consistent solution is unique. When $\mathrm{rank}(\mathbf A)<n$, that term describes precisely the existing freedom to add a vector in $\mathcal N(\mathbf A)$; changing $\mathbf B$ changes neither that freedom nor the minimum-norm representative. The design choice can affect the response to data outside $\mathcal C(\mathbf A)$, which is where the distinction between an oblique inverse and the Moore--Penrose inverse matters.

Section~\ref{section:Replacement} uses the construction for a different task: changing an outer factor of an already stored factorization while preserving its product. Given a candidate factor and its transform, the framework supplies
\begin{itemize}
\item a test,
\item a core update,
\item and a factored residual without reconstructing or accessing $\mathbf A$.
\end{itemize}
The test concerns the update induced by the supplied transform; the next section states when it also decides whether any exact update exists. Its solvability condition remains the classical one.

\section{Replacing a factor}
\label{section:Replacement}

Suppose $\mathbf{A}=\mathbf{F}\mathbf{G}\mathbf{H}^*$ is already stored and a candidate $\check{\mathbf{F}}$ is offered in place of $\mathbf{F}$. The columns of $\mathbf{A}$ are linear combinations of the columns of $\mathbf{F}$, with coefficients supplied by $\mathbf{G}\mathbf{H}^*$. Replacing $\mathbf{F}$ changes the vectors used in these combinations. To preserve $\mathbf{A}$ while keeping $\mathbf{H}$ fixed, we must check whether the mixing matrix $\mathbf{G}$ can be adjusted to compensate.

Such candidates arise in basis orthogonalization, recompression, and column selection. The candidate may come from a separate routine, and the original matrix may be stored out of core or never assembled. The practical question is therefore how to test the replacement and update the mixing matrix using the stored factors.

The simplest case is that the new columns can reproduce every old column: $\mathbf{F}=\check{\mathbf{F}}\mathbf{S}$ for some coefficient matrix $\mathbf{S}$. The same coefficients can then be absorbed into the mixing matrix:
$$
\mathbf{A}=\mathbf{F}\mathbf{G}\mathbf{H}^*
=\check{\mathbf{F}}\mathbf{S}\mathbf{G}\mathbf{H}^*
=\check{\mathbf{F}}(\mathbf{S}\mathbf{G})\mathbf{H}^*.
$$
Thus $\check{\mathbf{G}}=\mathbf{S}\mathbf{G}$ repairs the factorization. Reproducing all of $\mathbf{F}$, however, can ask for more than the product needs. It is enough that $\check{\mathbf{F}}\check{\mathbf{G}}=\mathbf{F}\mathbf{G}$: only the combinations of columns used by $\mathbf{G}$ have to be preserved. When $\mathbf{H}$ has full column rank, this condition is also necessary, because $\mathbf{H}^*$ has a right inverse. If $\mathbf{G}$ is singular, an exact repair may therefore be possible even when the new factor cannot reproduce every column of $\mathbf{F}$.

A supplied transform determines the proposed coefficients and hence a specific repair. The construction below tests that repair and gives its residual from the stored factors. It begins with a fixed number of columns in the factors.

For the fixed-width construction, let $k\ge1$ and $\mathbf{A}=\mathbf{F}\mathbf{G}\mathbf{H}^*$ be a factorization in the sense of \eqref{eq:genfactorization}, with $\mathbf{F}\in\mathbb{F}^{m\times k}$, $\mathbf{G}\in\mathbb{F}^{k\times k}$ and $\mathbf{H}\in\mathbb{F}^{n\times k}$. No rank hypothesis is imposed at this stage. A candidate replacement is $\check{\mathbf{F}}\in\mathbb{F}^{m\times k}$, accompanied by a transform $\check{\mathbf{Y}}\in\mathbb{F}^{m\times k}$, often supplied by the candidate's construction. Applying $\check{\mathbf{Y}}^*$ to the old factor gives the proposed coefficient matrix, and subtracting the resulting representation gives the mismatch:

\begin{align}\label{eq:replacement-interface}
\mathbf{S}_c = \check{\mathbf{Y}}^{*}\mathbf{F}
\qquad\text{and}\qquad
\mathbf{E}_c = \mathbf{F}-\check{\mathbf{F}}\mathbf{S}_c .
\end{align}
We call $\mathbf{S}_c$ the \emph{interface} and $\mathbf{E}_c$ the \emph{mismatch}. Each column of $\mathbf{S}_c$ contains the coefficients assigned by $\check{\mathbf{Y}}^*$ to the corresponding column of $\mathbf{F}$. Thus $\check{\mathbf{F}}\mathbf{S}_c$ reproduces the whole old factor exactly when $\mathbf{E}_c=\mathbf{0}$. The proposed repair is $\check{\mathbf{G}}=\mathbf{S}_c\mathbf{G}$, and it reproduces $\mathbf{F}\mathbf{G}$ exactly when $\mathbf{E}_c\mathbf{G}=\mathbf{0}$. This is the test developed below. Both matrices are built from $\mathbf{F}$, $\check{\mathbf{F}}$ and $\check{\mathbf{Y}}$: $\mathbf{S}_c$ is $k\times k$ and $\mathbf{E}_c$ is $m\times k$. If the test passes, the mixing matrix changes while the product $\mathbf{A}$ is preserved.

\paragraph{Three hypotheses, named once.}
Three conditions recur below. Each is named here and is assumed only where it is named:
\begin{align}\label{eq:replacement-hypotheses}
\text{(PE-c)}\quad \mathbf{Y}^{*}\mathbf{F}=\mathbf{I}_k,
\qquad
\text{(PE-r)}\quad \mathbf{H}^{*}\mathbf{X}=\mathbf{I}_k,
\qquad
\text{(ADM)}\quad \check{\mathbf{F}}\check{\mathbf{Y}}^{*}\check{\mathbf{F}}=\check{\mathbf{F}},
\end{align}
the first for some $\mathbf{Y}\in\mathbb{F}^{m\times k}$, the second for some $\mathbf{X}\in\mathbb{F}^{n\times k}$. The first two say that the left and the right half of the projector equation \eqref{eq:projector-eq} are solvable, equivalently that $\mathrm{rank}(\mathbf{F})=k$ and that $\mathrm{rank}(\mathbf{H})=k$. The third says that $\check{\mathbf{Y}}^{*}$ is a $\{1\}$-inverse of $\check{\mathbf{F}}$. That is weaker than $\check{\mathbf{Y}}^{*}\check{\mathbf{F}}=\mathbf{I}_k$, the form the projector design of Corollary~\ref{Theorem:generalized-meta-factorization} supplies, and unlike that form it permits $\mathrm{rank}(\check{\mathbf{F}})<k$. No statement below assumes a rank hypothesis it does not name.

\paragraph{What is given, and what it costs to have it.} Two things are given. The first is the incumbent factorization $\mathbf{A}=\mathbf{F}\mathbf{G}\mathbf{H}^{*}$, the object to be preserved. The second is the pair $(\check{\mathbf{F}},\check{\mathbf{Y}})$, the candidate together with the transform that arrives with it. Where that pair comes from is worth stating, because a saving in the repair would be illusory if obtaining the pair cost what recomputing the factorization costs.

Candidate construction is a separate operation. When another routine already supplies the candidate, the count below measures the additional work needed to test and repair its substitution. Any cost of selecting or constructing a candidate specifically for this purpose must be added.

The transform is the one that same construction supplies, and three cases cover the practice. When $\check{\mathbf{F}}$ has orthonormal columns, the commonest case, $\check{\mathbf{Y}}=\check{\mathbf{F}}$ serves at no cost. The interface is then $\mathbf{S}_c=\check{\mathbf{F}}^{*}\mathbf{F}$; when the candidate is the orthonormal factor of a thin $QR$ factorization of $\mathbf{F}$ itself, which is the re-orthogonalization case, that interface is the triangular factor the factorization has already returned. For a candidate not built from $\mathbf{F}$ it is neither triangular nor free, and costs one $O(mk^{2})$ product. That choice also satisfies \textup{(ADM)}, so the reading of Corollary~\ref{cor:replacement-span} applies to it. When the candidate is built from a design pair in the sense of Theorem~\ref{Theorem:projector-eq-sol}, that construction returns its transform in closed form.

If neither construction supplies a transform, one may compute $\check{\mathbf Y}^{*}=\check{\mathbf F}^{+}$, which always satisfies \textup{(ADM)}. For $k\le m$, its standard dense cost follows from a thin $QR$ factorization $\check{\mathbf F}=\mathbf Q\mathbf R$. Forming the $m\times k$ factor $\mathbf Q$ and the $k\times k$ factor $\mathbf R$ costs $O(mk^{2})$, and computing the SVD of $\mathbf R$ costs $O(k^{3})$ \cite{golub2013matrix}. Forming $\check{\mathbf F}^{+}$ from the retained singular vectors adds at most $O(mk^{2})$ operations. Adding these costs gives $O(mk^{2}+k^{3})$, including when the candidate is rank deficient.

For $m<k$, the same construction applied to $\check{\mathbf F}^{*}$ costs $O(km^{2}+m^{3})$, which also satisfies the stated upper bound. These are standard dense floating-point counts at fixed precision: the exact pseudoinverse uses the nonzero singular values, while a numerical implementation must choose a rank threshold. A left inverse is more than \textup{(ADM)} requires and exists only when $\mathrm{rank}(\check{\mathbf F})=k$.

Each route constructs the transform from the candidate and the available factor data. The resulting repair remains confined to the mixing matrix.

\subsection{When a replacement is exact}

Under \textup{(PE-r)}, the replacement is exact precisely when the mismatch is annihilated by the mixing matrix.

\begin{theorem}[replacement criterion]\label{thm:replacement}
Assume \textup{(PE-r)}, let $\check{\mathbf{F}},\check{\mathbf{Y}}\in\mathbb{F}^{m\times k}$ be arbitrary, let $\mathbf{S}_c$, $\mathbf{E}_c$ be as in \eqref{eq:replacement-interface}, and put $\check{\mathbf{G}}=\mathbf{S}_c\mathbf{G}$. Then
\begin{align}\label{eq:replacement-criterion}
\mathbf{A}=\check{\mathbf{F}}\check{\mathbf{G}}\mathbf{H}^*
\qquad\text{if and only if}\qquad
\mathbf{E}_c\mathbf{G}=\mathbf{0}.
\end{align}
\end{theorem}

Nothing is assumed of $\check{\mathbf{F}}$, of $\check{\mathbf{Y}}$, or of their ranks. The test $\mathbf{E}_c\mathbf{G}=\mathbf{0}$ and the repaired mixing matrix $\check{\mathbf{G}}=\mathbf{S}_c\mathbf{G}$ are both formed from $\mathbf{F}$, $\mathbf{G}$, $\check{\mathbf{F}}$ and $\check{\mathbf{Y}}$ alone: $\mathbf{A}$ is never read.

For an arbitrary transform, passing the test certifies the specified repair $\check{\mathbf{G}}=\mathbf{S}_c\mathbf{G}$. Failure need not rule out a different repair. For example, a zero transform rejects every nonzero $\mathbf{A}$ even if the candidate spans its column space. Under \textup{(ADM)}, the test is equivalent to $\mathcal{C}(\mathbf{A})\subseteq\mathcal{C}(\check{\mathbf{F}})$ and therefore also decides existence (Corollary~\ref{cor:replacement-span}).

\begin{proof}
Substituting $\mathbf{A}=\mathbf{F}\mathbf{G}\mathbf{H}^*$ and regrouping,
\begin{align}\label{eq:replacement-defect}
\mathbf{A}-\check{\mathbf{F}}(\mathbf{S}_c\mathbf{G})\mathbf{H}^*
=\mathbf{F}\mathbf{G}\mathbf{H}^*-\check{\mathbf{F}}\mathbf{S}_c\mathbf{G}\mathbf{H}^*
=(\mathbf{F}-\check{\mathbf{F}}\mathbf{S}_c)\mathbf{G}\mathbf{H}^*
=\mathbf{E}_c\mathbf{G}\mathbf{H}^*.
\end{align}
If $\mathbf{E}_c\mathbf{G}=\mathbf{0}$ the right-hand side vanishes, and with it the defect on the left. This direction uses the factorization alone. Conversely, if that defect vanishes then $\mathbf{E}_c\mathbf{G}\mathbf{H}^*=\mathbf{0}$, and multiplying on the right by the $\mathbf{X}$ of \textup{(PE-r)} gives $\mathbf{E}_c\mathbf{G}=\mathbf{E}_c\mathbf{G}(\mathbf{H}^*\mathbf{X})=\mathbf{0}$.
\end{proof}

The proof rearranges products and cancels once. It never relates $\check{\mathbf{Y}}^{*}$ to $\check{\mathbf{Y}}$, never forms a pseudoinverse, and contains no step passing from $\mathbf{M}^*\mathbf{M}=\mathbf{0}$ to $\mathbf{M}=\mathbf{0}$. The theorem therefore holds not only over $\mathbb{R}$ and $\mathbb{C}$ but over any field carrying an involution, with $(\cdot)^*$ read as that involution applied entrywise to the transpose.

The hypothesis on $\mathbf{H}$ is consumed exactly once, in the last step, and it is what turns the honest condition $\mathbf{E}_c\mathbf{G}\mathbf{H}^*=\mathbf{0}$, an $m\times n$ identity, into $\mathbf{E}_c\mathbf{G}=\mathbf{0}$, an $m\times k$ one. It cannot be dropped. Take $k=2$, $n=1$, $\mathbf{F}=\mathbf{G}=\mathbf{I}_2$ and $\mathbf{H}^*=[\,1\ \ 0\,]^{\mathsf{T}}$, so that $\mathbf{A}=[\,1\ \ 0\,]^{\mathsf{T}}$, and offer $\check{\mathbf{F}}=\check{\mathbf{Y}}^{*}=\mathrm{diag}(1,0)$. Then $\mathbf{S}_c=\mathrm{diag}(1,0)$ and $\mathbf{E}_c\mathbf{G}=\mathrm{diag}(0,1)\ne\mathbf{0}$, while $\check{\mathbf{F}}(\mathbf{S}_c\mathbf{G})\mathbf{H}^*=\mathbf{A}$: the test rejects an exact replacement. Here $\mathrm{rank}(\mathbf{H})=1<k$, so $\mathbf{H}^*\mathbf{X}=\mathbf{I}_2$ has no solution and the theorem does not apply.

Computing $\check{\mathbf{Y}}^{*}\mathbf{F}$, $\check{\mathbf{F}}\mathbf{S}_c$, and $\mathbf{E}_c\mathbf{G}$ costs $O(mk^2)$ operations, and computing $\mathbf{S}_c\mathbf{G}$ costs $O(k^3)$. Thus test and repair together cost $O(mk^2+k^3)$, or $O(mk^2)$ when $k\le m$. By comparison, directly evaluating $\check{\mathbf{Y}}^{*}\mathbf{A}\mathbf{X}$ through Corollary~\ref{Theorem:generalized-meta-factorization} requires a dense multiplication involving $\mathbf{A}$, with cost $\Theta(mnk)$ when $k\le\min(m,n)$. This comparison concerns that explicit reconstruction formula. Section~\ref{subsection:replacement-prior-art} discusses methods that already operate on factors, and Section~\ref{subsection:not-claimed} records the numerical limitations of the exact criterion.

\subsection{How wrong a rejected replacement is}

The same calculation supplies the residual of an inexact substitution. Once $\mathbf{E}_c\mathbf{G}$ has been computed, the pair $(\mathbf{E}_c\mathbf{G},\mathbf{H}^*)$ represents the residual exactly, with rank at most $k$. Forming its full $m\times n$ array would cost an additional $\Theta(mnk)$ operations, so the factored representation preserves the computational benefit of the test.

\begin{corollary}[exact repair defect]\label{cor:replacement-defect}
Let $\mathbf{A}=\mathbf{F}\mathbf{G}\mathbf{H}^*$ and let $\mathbf{S}_c$, $\mathbf{E}_c$ be as in \eqref{eq:replacement-interface}, with $\check{\mathbf{F}}$ and $\check{\mathbf{Y}}$ arbitrary. Then
\begin{align}\label{eq:replacement-residual}
\mathbf{A}-\check{\mathbf{F}}(\mathbf{S}_c\mathbf{G})\mathbf{H}^*=(\mathbf{E}_c\mathbf{G})\mathbf{H}^*,
\end{align}
so that $\|\mathbf{A}-\check{\mathbf{F}}(\mathbf{S}_c\mathbf{G})\mathbf{H}^*\|_F=\|\mathbf{E}_c\mathbf{G}\mathbf{H}^*\|_F$, which equals $\|\mathbf{E}_c\mathbf{G}\|_F$ when $\mathbf{H}^*\mathbf{H}=\mathbf{I}_k$.
\end{corollary}

\begin{proof}
The identity is \eqref{eq:replacement-defect}, which consumed only $\mathbf{A}=\mathbf{F}\mathbf{G}\mathbf{H}^*$. For the norms, $\|\mathbf{M}\mathbf{H}^*\|_F^2=\mathrm{tr}(\mathbf{H}\mathbf{M}^*\mathbf{M}\mathbf{H}^*)=\mathrm{tr}(\mathbf{M}^*\mathbf{M}\mathbf{H}^*\mathbf{H})$, which is $\|\mathbf{M}\|_F^2$ when $\mathbf{H}^*\mathbf{H}=\mathbf{I}_k$.
\end{proof}

\begin{remark}\label{rem:defect-norms}
Only the matrix identity \eqref{eq:replacement-residual} is free of hypotheses on $\mathbf{H}$; the last equality is not. Orthonormality of $\mathbf{H}$ is what makes right multiplication by $\mathbf{H}^*$ preserve the Frobenius norm, and without it the two norms genuinely differ. Already $\mathbf{H}=2[\,\mathbf{e}_1\ \ \mathrm{i}\mathbf{e}_2\,]$ gives $\mathbf{H}^*\mathbf{H}=4\mathbf{I}_2$ and $\|\mathbf{M}\mathbf{H}^*\|_F=2\|\mathbf{M}\|_F$.
\end{remark}

For full-column-rank $\mathbf H$, the same trace calculation also gives
$$
\sigma_{\min}(\mathbf H)\|\mathbf E_c\mathbf G\|_F
\le \|\mathbf E_c\mathbf G\mathbf H^*\|_F
\le \|\mathbf H\|_2\|\mathbf E_c\mathbf G\|_F.
$$
These exact bounds expose the effect of scaling the retained factor. They do not control errors made while computing the mismatch.

\paragraph{A compact residual norm at any rank.}
Let $\mathbf K=\mathbf E_c\mathbf G$ and take thin factorizations $\mathbf K=\mathbf Q_K\mathbf R_K$ and $\mathbf H=\mathbf Q_H\mathbf R_H$, with orthonormal columns in the two $\mathbf Q$ factors. They may come from rank-revealing QR or SVD and need not have $k$ columns. Zero rank is allowed: empty factors represent the zero matrix and its residual norm is zero. Then
$$
\mathbf K\mathbf H^*=\mathbf Q_K(\mathbf R_K\mathbf R_H^*)\mathbf Q_H^*.
$$
The residual's nonzero singular values are those of the core $\mathbf R_K\mathbf R_H^*$. Hence its Frobenius and spectral norms can be evaluated on that core, and exactness is equivalent to $\mathbf R_K\mathbf R_H^*=\mathbf0$, even for deficient $\mathbf H$. This test reads the retained factor; Theorem~\ref{thm:replacement} avoids it. These are exact identities and do not bound rounding errors in forming the factors.

\subsection{When the repair is the best one its factor allows}
\label{subsection:replacement-optimal}

When the incoming transform is $\check{\mathbf{F}}^{+}$, the repair minimizes the reconstruction error over the candidate's column space. This is the classical orthogonal projection property, evaluated here using the existing factors.

Regroup \eqref{eq:replacement-defect} the other way. Consuming again only $\mathbf{A}=\mathbf{F}\mathbf{G}\mathbf{H}^*$,
\begin{align}\label{eq:replacement-projection}
\check{\mathbf{F}}(\mathbf{S}_c\mathbf{G})\mathbf{H}^*
=\check{\mathbf{F}}\check{\mathbf{Y}}^{*}\mathbf{F}\mathbf{G}\mathbf{H}^*
=\check{\mathbf{F}}\check{\mathbf{Y}}^{*}\mathbf{A}.
\end{align}
The repair applies $\check{\mathbf{F}}\check{\mathbf{Y}}^{*}$ to $\mathbf{A}$, and Figure~\ref{fig:optimal} is the resulting picture. Under \textup{(ADM)} that matrix is idempotent with column space $\mathcal{C}(\check{\mathbf{F}})$, so what the repair returns is a projection of $\mathbf{A}$ onto the space the new factor spans --- and it returns it, by the left-hand side of \eqref{eq:replacement-projection}, without $\mathbf{A}$ being read. Which projection depends on the transform. One transform is distinguished.

\begin{figure}[!ht]
\centering
\begin{tikzpicture}[x=1cm,y=1cm,font=\normalsize,
    sub/.style={fill=black!7,draw=black!45,line width=0.5pt},
    panel/.style={},
    resid/.style={black,line width=1.1pt},
    comp/.style={gray!65,line width=0.6pt,dashed}]

\draw[sub] (0,0) -- (6.6,0) -- (8.6,1.5) -- (2.0,1.5) -- cycle;
\node[panel,anchor=north east] at (8.6,-0.12)
     {$\{\mathbf{M}:\mathcal{C}(\mathbf{M})\subseteq\mathcal{C}(\check{\mathbf{F}})\}$};

\coordinate (A)  at (4.3,3.1);
\coordinate (PA) at (4.3,0.82);
\fill (A) circle (1.4pt);  \node[panel,anchor=south] at (4.3,3.22) {$\mathbf{A}$};
\fill (PA) circle (1.4pt);
\node[panel,anchor=north east] at (4.14,0.75) {$\mathbf{P}\mathbf{A}$};
\draw[resid] (A) -- (PA);
\node[panel,anchor=east,align=right] at (4.12,2.0)
     {$(\mathbf{I}-\mathbf{P})\mathbf{A}$};
\draw (4.3,0.82) ++(0,0.24) -- ++(0.24,0) -- ++(0,-0.24);

\coordinate (Z) at (6.8,1.0);
\fill[gray!70] (Z) circle (1.2pt);
\node[panel,anchor=west,text=black] at (6.95,1.02) {$\check{\mathbf{F}}\mathbf{Z}$};
\draw[comp] (A) -- (Z);
\node[panel,anchor=north,align=center] at (4.3,-0.90)
     {the repair is the foot of the perpendicular: no $\check{\mathbf{F}}\mathbf{Z}$ is nearer,\\
      and the criterion is the case $\mathbf{A}=\mathbf{P}\mathbf{A}$};
\end{tikzpicture}
\caption{Orthogonal repair in Frobenius geometry. With $\mathbf{P}=\check{\mathbf{F}}\check{\mathbf{F}}^{+}$ and $\mathbf{S}_c=\check{\mathbf{F}}^{+}\mathbf{F}$, the repaired factorization $\check{\mathbf{F}}(\mathbf{S}_c\mathbf{G})\mathbf{H}^{*}=\mathbf{P}\mathbf{A}$ is the unique closest matrix to $\mathbf{A}$ in the shaded set for the Frobenius norm. The solid segment is the orthogonal residual. The core $\mathbf{S}_c\mathbf{G}$ is computed from the stored factors and the candidate. The accompanying text proves optimality for every unitarily invariant norm, for which the closest matrix need not be unique.}
\label{fig:optimal}
\end{figure}
\begin{corollary}[optimal repair]\label{cor:replacement-optimal}
Let $\mathbf{A}=\mathbf{F}\mathbf{G}\mathbf{H}^*$, let $\check{\mathbf{F}}\in\mathbb{F}^{m\times k}$ be arbitrary, take $\check{\mathbf{Y}}^{*}=\check{\mathbf{F}}^{+}$, and write $\mathbf{P}=\check{\mathbf{F}}\check{\mathbf{F}}^{+}$. Then
\begin{align}\label{eq:replacement-optimal}
\check{\mathbf{F}}(\mathbf{S}_c\mathbf{G})\mathbf{H}^*=\mathbf{P}\mathbf{A}
\qquad\text{and}\qquad
\mathbf{A}-\check{\mathbf{F}}(\mathbf{S}_c\mathbf{G})\mathbf{H}^*=(\mathbf{I}_m-\mathbf{P})\mathbf{A},
\end{align}
and for every $\mathbf{Z}\in\mathbb{F}^{k\times n}$,
\begin{align}\label{eq:replacement-optimal-bound}
\|\mathbf{A}-\check{\mathbf{F}}(\mathbf{S}_c\mathbf{G})\mathbf{H}^*\|_F\le\|\mathbf{A}-\check{\mathbf{F}}\mathbf{Z}\|_F .
\end{align}
The repaired factorization is a closest matrix to $\mathbf{A}$ among all matrices whose column space lies in $\mathcal{C}(\check{\mathbf{F}})$, and the substitution is exact precisely when that distance is zero.
\end{corollary}

\begin{proof}
Both identities in \eqref{eq:replacement-optimal} are \eqref{eq:replacement-projection} at $\check{\mathbf{Y}}^{*}=\check{\mathbf{F}}^{+}$. For the inequality, split
\begin{align*}
\mathbf{A}-\check{\mathbf{F}}\mathbf{Z}=\mathbf{M}_1+\mathbf{M}_2,
\qquad
\mathbf{M}_1=\mathbf{P}\mathbf{A}-\check{\mathbf{F}}\mathbf{Z},
\qquad
\mathbf{M}_2=(\mathbf{I}_m-\mathbf{P})\mathbf{A}.
\end{align*}
Since $\mathbf{P}$ is Hermitian and idempotent and $\check{\mathbf{F}}^{*}\mathbf{P}=(\mathbf{P}\check{\mathbf{F}})^{*}=\check{\mathbf{F}}^{*}$, we have $\mathbf{M}_1^{*}\mathbf{M}_2=\mathbf{0}$, so $\|\mathbf{A}-\check{\mathbf{F}}\mathbf{Z}\|_F^2=\|\mathbf{M}_1\|_F^2+\|\mathbf{M}_2\|_F^2\ge\|\mathbf{M}_2\|_F^2$, with equality exactly when $\check{\mathbf{F}}\mathbf{Z}=\mathbf{P}\mathbf{A}$. By \eqref{eq:replacement-optimal}, $\mathbf{M}_2$ is the matrix on the left of \eqref{eq:replacement-optimal-bound}.
\end{proof}

The Frobenius norm is not special here. The same splitting gives $\mathbf{M}^{*}\mathbf{M}=\mathbf{M}_1^{*}\mathbf{M}_1+\mathbf{M}_2^{*}\mathbf{M}_2\succeq\mathbf{M}_2^{*}\mathbf{M}_2$ for $\mathbf{M}=\mathbf{A}-\check{\mathbf{F}}\mathbf{Z}$, so $\sigma_i(\mathbf{M})\ge\sigma_i(\mathbf{M}_2)$ for every $i$ by Weyl's monotonicity theorem, and \eqref{eq:replacement-optimal-bound} therefore holds in every unitarily invariant norm.

The classical Frobenius-norm part of this is Penrose's: the best approximate solution of $\check{\mathbf{F}}\mathbf{Z}=\mathbf{A}$ --- the least-squares solution of minimum norm --- is $\mathbf{Z}=\check{\mathbf{F}}^{+}\mathbf{A}$ \cite{penrose1956best}. What \eqref{eq:replacement-optimal} adds is that the repair already computed supplies that solution in factored form, $\check{\mathbf{F}}^{+}\mathbf{A}=(\mathbf{S}_c\mathbf{G})\mathbf{H}^{*}$, whose $k\times k$ left factor $\mathbf{S}_c\mathbf{G}=\check{\mathbf{F}}^{+}\mathbf{F}\mathbf{G}$ reads $\check{\mathbf{F}}$, $\mathbf{F}$ and $\mathbf{G}$, and nothing else. The optimum is characterized by $\mathbf{A}$ and computed without it.

A rejected candidate therefore still yields a best reconstruction in its column space. When $\check{\mathbf{F}}$ has orthonormal columns, its adjoint is already the pseudoinverse, so the orthogonal transform is available directly. Otherwise, computing $\check{\mathbf{F}}^{+}$ costs $O(mk^2+k^3)$ operations. Optimality concerns the fixed candidate and does not select its column space or guarantee numerical conditioning.

The scope is exactly the core. With $\check{\mathbf{F}}$ given, no core does better; the corollary does not say that $\check{\mathbf{F}}$ was well chosen, and it says nothing about conditioning. Choosing $\mathcal{C}(\check{\mathbf{F}})$ is the design question Section~\ref{subsection:not-claimed} still declines --- though \eqref{eq:replacement-optimal} does supply the objective such a rule would have to minimize, $\|(\mathbf{I}_m-\mathbf{P})\mathbf{A}\|_F$.

\subsection{What the criterion says about the span of the new factor}

The test $\mathbf{E}_c\mathbf{G}=\mathbf{0}$ is an algebraic condition on four matrices, and it is natural to ask what it asserts about the subspace the new factor spans. Under \textup{(PE-r)} and \textup{(ADM)} it asserts exactly the span inclusion of Theorem~\ref{Theorem:meta-factorization}, transported to the replaced factor.

\begin{corollary}[span reading]\label{cor:replacement-span}
Assume \textup{(PE-r)} and \textup{(ADM)}. Then
\begin{align}
\mathbf{E}_c\mathbf{G}=\mathbf{0}
\qquad\text{if and only if}\qquad
\mathcal{C}(\mathbf{A})\subseteq\mathcal{C}(\check{\mathbf{F}}).
\end{align}
\end{corollary}

\begin{proof}
If $\mathbf{E}_c\mathbf{G}=\mathbf{0}$, Theorem~\ref{thm:replacement} gives $\mathbf{A}=\check{\mathbf{F}}(\mathbf{S}_c\mathbf{G})\mathbf{H}^*$, so $\mathcal{C}(\mathbf{A})\subseteq\mathcal{C}(\check{\mathbf{F}})$. Conversely, suppose $\mathcal{C}(\mathbf{A})\subseteq\mathcal{C}(\check{\mathbf{F}})$ and write $\mathbf{A}=\check{\mathbf{F}}\mathbf{M}$. Then, using $\mathbf{S}_c=\check{\mathbf{Y}}^{*}\mathbf{F}$ and applying $\mathbf{F}\mathbf{G}\mathbf{H}^*=\mathbf{A}$ twice,
\begin{align}
\mathbf{E}_c\mathbf{G}\mathbf{H}^*
=\mathbf{A}-\check{\mathbf{F}}\check{\mathbf{Y}}^{*}\mathbf{A}
=\check{\mathbf{F}}\mathbf{M}-\check{\mathbf{F}}\check{\mathbf{Y}}^{*}\check{\mathbf{F}}\mathbf{M}
=(\check{\mathbf{F}}-\check{\mathbf{F}}\check{\mathbf{Y}}^{*}\check{\mathbf{F}})\mathbf{M}
=\mathbf{0},
\end{align}
where the last equality is \textup{(ADM)}. Multiplying on the right by the $\mathbf{X}$ of \textup{(PE-r)} then gives $\mathbf{E}_c\mathbf{G}=\mathbf{0}$.
\end{proof}

So the criterion is the framework's own span condition, not a new one: a replacement is exact precisely when it still spans a subspace containing the column space of $\mathbf{A}$, which is what the reconstruction equation asked of $\mathbf{F}$ in the first place. What the criterion adds is that this is decidable from $\mathbf{F}$, $\mathbf{G}$, $\check{\mathbf{F}}$ and $\check{\mathbf{Y}}$, without computing a column space and without touching $\mathbf{A}$. No rank hypothesis is placed on $\check{\mathbf{F}}$. One further consequence of the equivalence is used later. Call $\check{\mathbf{F}}$ \emph{accepted} when it passes the test. Whether a candidate is accepted then depends on $\check{\mathbf{F}}$ alone, and not on which $\{1\}$-inverse transform accompanies it, so the accepted set is well defined.

The hypothesis on $\mathbf{H}$ is load-bearing here too, and the witness after Theorem~\ref{thm:replacement} shows it. In that witness $\check{\mathbf{F}}\check{\mathbf{Y}}^{*}\check{\mathbf{F}}=\check{\mathbf{F}}$ and $\mathcal{C}(\mathbf{A})=\mathcal{C}(\check{\mathbf{F}})$, so the inclusion holds, while $\mathbf{E}_c\mathbf{G}=\mathrm{diag}(0,1)\ne\mathbf{0}$. The hypothesis \textup{(ADM)} on $\check{\mathbf{Y}}$ is load-bearing too. Taking $\mathbf{F}=\check{\mathbf{F}}=\mathbf{I}_2$, $\mathbf{G}=\mathbf{H}=\mathbf{I}_2$ and $\check{\mathbf{Y}}=\mathbf{0}$ leaves the inclusion true and $\mathbf{E}_c\mathbf{G}=\mathbf{I}_2\ne\mathbf{0}$.

\subsection{Replacing the right factor instead}

Everything above is stated for the left factor. The right factor obeys the adjoint statement, which we call the mirror, and the two halves of the projector equation exchange roles in it. One could obtain the mirror by applying the criterion to the adjoint data. It is stated separately because the interface moves to the other side of the mixing matrix and the cancellation consumes the other half of the projector equation. The direct proof below is shorter than the translation.

\begin{theorem}[adjoint mirror]\label{thm:replacement-mirror}
Let $\mathbf{A}=\mathbf{F}\mathbf{G}\mathbf{H}^*$ and suppose the left half of the projector equation \eqref{eq:projector-eq} is solvable: $\mathbf{Y}^*\mathbf{F}=\mathbf{I}_k$ for some $\mathbf{Y}\in\mathbb{F}^{m\times k}$. For arbitrary $\check{\mathbf{H}},\check{\mathbf{X}}\in\mathbb{F}^{n\times k}$ put
\begin{align}\label{eq:replacement-interface-r}
\mathbf{S}_r=\mathbf{H}^*\check{\mathbf{X}}
\qquad\text{and}\qquad
\mathbf{E}_r=\mathbf{H}^*-\mathbf{S}_r\check{\mathbf{H}}^{*}.
\end{align}
Then
\begin{align}\label{eq:replacement-mirror}
\mathbf{A}=\mathbf{F}(\mathbf{G}\mathbf{S}_r)\check{\mathbf{H}}^{*}
\qquad\text{if and only if}\qquad
\mathbf{G}\mathbf{E}_r=\mathbf{0}.
\end{align}
\end{theorem}

\begin{proof}
Substituting $\mathbf{A}=\mathbf{F}\mathbf{G}\mathbf{H}^*$ and regrouping, $\mathbf{A}-\mathbf{F}(\mathbf{G}\mathbf{S}_r)\check{\mathbf{H}}^{*}=\mathbf{F}\mathbf{G}(\mathbf{H}^*-\mathbf{S}_r\check{\mathbf{H}}^{*})=\mathbf{F}(\mathbf{G}\mathbf{E}_r)$, which gives the implication from right to left. For the converse, that same identity turns the vanishing of the left-hand side into $\mathbf{F}(\mathbf{G}\mathbf{E}_r)=\mathbf{0}$. Left-multiplying by $\mathbf{Y}^*$ and using $\mathbf{Y}^*\mathbf{F}=\mathbf{I}_k$ then leaves $\mathbf{G}\mathbf{E}_r=\mathbf{0}$.
\end{proof}

Nothing is assumed of $\check{\mathbf{H}}$, of $\check{\mathbf{X}}$, or of their ranks. This proof, like the proof of Theorem~\ref{thm:replacement}, rearranges products and cancels once: it never relates $\check{\mathbf{H}}^{*}$ to $\check{\mathbf{H}}$, never forms a pseudoinverse, and contains no step passing from $\mathbf{M}^*\mathbf{M}=\mathbf{0}$ to $\mathbf{M}=\mathbf{0}$. The mirror therefore holds, again, over any field carrying an involution.

Which half of the projector equation is required is not a matter of taste. Replacing the left factor consumes solvability on the right; replacing the right factor consumes solvability on the left. Drop the latter and the mirror fails: with $\mathbf{F}=\mathrm{diag}(1,0)$, $\mathbf{G}=\mathbf{H}^*=\mathbf{I}_2$, $\check{\mathbf{H}}^{*}=\mathrm{diag}(1,0)$ and $\check{\mathbf{X}}=(\check{\mathbf{H}}^{*})^{+}$ one has $\mathbf{A}=\mathbf{F}(\mathbf{G}\mathbf{S}_r)\check{\mathbf{H}}^{*}$ while $\mathbf{G}\mathbf{E}_r=\mathrm{diag}(0,1)\ne\mathbf{0}$. Here $\mathrm{rank}(\mathbf{F})=1<k$.

\paragraph{Changing stored widths.}
The fixed-width theorem has a direct rectangular extension. For $\check{\mathbf F},\check{\mathbf Y}\in\mathbb F^{m\times\ell}$, use
$$
\mathbf S_c=\check{\mathbf Y}^*\mathbf F\in\mathbb F^{\ell\times k},\qquad
\check{\mathbf G}=\mathbf S_c\mathbf G\in\mathbb F^{\ell\times k}.
$$
The mismatch still has size $m\times k$, and the proof gives the same residual and, for full-column-rank $\mathbf H$, the same exactness test. Under (ADM) the span criterion remains valid, and the pseudoinverse still gives the projection optimum. This extension permits an actual reduction of stored width, whereas a deficient $m\times k$ candidate alone does not. Replacing both sides gives an $\ell_c\times\ell_r$ core. If each incoming transform is a $\{1\}$-inverse and both incoming spaces contain the corresponding data spaces, the two projectors fix $\mathbf A$ separately and the simultaneous repair is exact.

\subsection{Replacement in a full-rank factorization}

At the minimal inner dimension $r=\mathrm{rank}(\mathbf{A})$, write a full-rank factorization as $\mathbf{A}=\mathbf{K}\mathbf{L}$, where $\mathbf{K}\in\mathbb{F}^{m\times r}$ has full column rank and $\mathbf{L}\in\mathbb{F}^{r\times n}$ has full row rank. The mirror then gives a complete description of right-factor replacements. Exact candidates span the row space of $\mathbf{A}$, and the corresponding change of basis determines the unique repaired left factor.

Over $\mathbb{R}$ and $\mathbb{C}$, the right inverse required below can be chosen as $\mathbf{E}=(\check{\mathbf{L}})^{+}$. Section~\ref{subsection:pseudoinverse} later uses the same full-rank factorization structure in the pseudoinverse formulas.

\begin{corollary}[replacement in a full-rank factorization]\label{cor:replacement-full-rank}
Let $\mathbf{A}=\mathbf{K}\mathbf{L}$, where $\mathbf{K}\in\mathbb{F}^{m\times r}$ and $\mathbf{L}\in\mathbb{F}^{r\times n}$ satisfy $\mathrm{rank}(\mathbf{K})=\mathrm{rank}(\mathbf{L})=r=\mathrm{rank}(\mathbf{A})$. Let the candidate $\check{\mathbf{L}}\in\mathbb{F}^{r\times n}$ have a right inverse $\mathbf{E}\in\mathbb{F}^{n\times r}$, so that $\check{\mathbf{L}}\mathbf{E}=\mathbf{I}_r$, and define $\mathbf{S}_r=\mathbf{L}\mathbf{E}$.
\begin{enumerate}
\item[(a)] \emph{Testing the computed repair.} The repair is exact if and only if it reproduces the original right factor:
$$
\mathbf{A}=\mathbf{K}\mathbf{S}_r\check{\mathbf{L}}
\quad\Longleftrightarrow\quad
\mathbf{L}=\mathbf{S}_r\check{\mathbf{L}}.
$$
\item[(b)] \emph{Equivalent descriptions of an accepted candidate.} The following four conditions are equivalent:
\begin{enumerate}
\item[(i)] $\mathbf{A}=\check{\mathbf{K}}\check{\mathbf{L}}$ for some $\check{\mathbf{K}}\in\mathbb{F}^{m\times r}$.
\item[(ii)] $\mathbf{L}=\mathbf{S}\check{\mathbf{L}}$ for some $\mathbf{S}\in\mathbb{F}^{r\times r}$.
\item[(iii)] $\mathbf{L}=\mathbf{S}_r\check{\mathbf{L}}$ for the computed matrix $\mathbf{S}_r=\mathbf{L}\mathbf{E}$.
\item[(iv)] $\mathcal{C}((\check{\mathbf{L}})^*)=\mathcal{C}(\mathbf{A}^*)$, so the candidate spans the row space of $\mathbf{A}$.
\end{enumerate}
Whenever these conditions hold, the matrix $\mathbf{S}$ in (ii) is unique, equals $\mathbf{S}_r$, and is nonsingular.
\item[(c)] \emph{Constructing the replacement.} When the conditions in (b) hold, $\check{\mathbf{K}}=\mathbf{K}\mathbf{S}_r$ gives an exact factorization $\mathbf{A}=\check{\mathbf{K}}\check{\mathbf{L}}$. Both $\mathbf{S}_r$ and $\check{\mathbf{K}}$ are computed from $\mathbf{K}$, $\mathbf{L}$, $\check{\mathbf{L}}$ and $\mathbf{E}$ alone.
\item[(d)] \emph{Uniqueness.} The exact left factor $\check{\mathbf{K}}$, when it exists, is unique.
\end{enumerate}
\end{corollary}

\begin{proof}
Clause (a) is Theorem~\ref{thm:replacement-mirror} at $k=r$, $\mathbf{F}=\mathbf{K}$, $\mathbf{G}=\mathbf{I}_r$, $\mathbf{H}^*=\mathbf{L}$, $\check{\mathbf{H}}^{*}=\check{\mathbf{L}}$ and $\check{\mathbf{X}}=\mathbf{E}$: then $\mathbf{S}_r=\mathbf{L}\mathbf{E}$ and $\mathbf{E}_r=\mathbf{L}-\mathbf{S}_r\check{\mathbf{L}}$, and $\mathrm{rank}(\mathbf{K})=r$ supplies the hypothesis that theorem needs.

For (b), write $\mathbf{P}$ for a left inverse of $\mathbf{K}$ and $\mathbf{R}$ for a right inverse of $\mathbf{L}$, both of which exist by the rank hypotheses. If $\mathbf{A}=\check{\mathbf{K}}\check{\mathbf{L}}$ then $\mathbf{L}=\mathbf{P}\mathbf{A}=(\mathbf{P}\check{\mathbf{K}})\check{\mathbf{L}}$, and conversely $\mathbf{L}=\mathbf{S}\check{\mathbf{L}}$ gives $\mathbf{A}=\mathbf{K}\mathbf{L}=(\mathbf{K}\mathbf{S})\check{\mathbf{L}}$, so the first two conditions are equivalent. Given $\mathbf{L}=\mathbf{S}\check{\mathbf{L}}$, right-multiplying by $\mathbf{E}$ gives $\mathbf{S}=\mathbf{L}\mathbf{E}=\mathbf{S}_r$, which collapses the second condition onto the third. The matrix $\mathbf{S}$ is also nonsingular, since $\mathbf{S}(\check{\mathbf{L}}\mathbf{R})=\mathbf{L}\mathbf{R}=\mathbf{I}_r$ exhibits a square matrix with a one-sided inverse. It remains to match the third condition with the fourth. From $\mathbf{A}=\mathbf{K}\mathbf{L}$ and $\mathbf{L}=\mathbf{P}\mathbf{A}$ we get $\mathcal{C}(\mathbf{L}^*)=\mathcal{C}(\mathbf{A}^*)$, so $\mathbf{L}=\mathbf{S}\check{\mathbf{L}}$ gives $\mathcal{C}(\mathbf{A}^*)=\mathcal{C}(\mathbf{L}^*)\subseteq\mathcal{C}((\check{\mathbf{L}})^*)$. For the reverse inclusion, $\mathbf{S}$ is nonsingular by the previous step, so $\check{\mathbf{L}}=\mathbf{S}^{-1}\mathbf{L}$ and $\mathcal{C}((\check{\mathbf{L}})^*)\subseteq\mathcal{C}(\mathbf{L}^*)=\mathcal{C}(\mathbf{A}^*)$. Conversely, $\mathcal{C}((\check{\mathbf{L}})^*)=\mathcal{C}(\mathbf{A}^*)=\mathcal{C}(\mathbf{L}^*)$ says in particular that every row of $\mathbf{L}$ is a combination of rows of $\check{\mathbf{L}}$, which is $\mathbf{L}=\mathbf{S}\check{\mathbf{L}}$ for some $\mathbf{S}$.

Clause (c) is (a) together with $\mathbf{A}=\mathbf{K}\mathbf{L}$. For (d), if $\mathbf{K}_1\check{\mathbf{L}}=\mathbf{K}_2\check{\mathbf{L}}$, right-multiplication by $\mathbf{E}$ gives $\mathbf{K}_1=\mathbf{K}_2$.
\end{proof}

Clause (b) is where the full-rank hypothesis on $\mathbf{L}$ is spent, and it is spent once: in $\mathbf{S}(\check{\mathbf{L}}\mathbf{R})=\mathbf{L}\mathbf{R}=\mathbf{I}_r$, where the right inverse $\mathbf{R}$ of $\mathbf{L}$ is what exhibits $\mathbf{S}$ as nonsingular. The rest of the clause runs on the left-invertibility of $\mathbf{K}$ and on the right inverse $\mathbf{E}$. The reverse inclusion is then read off $\check{\mathbf{L}}=\mathbf{S}^{-1}\mathbf{L}$, but it does not need that route: $\mathbf{A}=\check{\mathbf{K}}\check{\mathbf{L}}$ with $\mathrm{rank}(\mathbf{A})=r$ forces $\mathrm{rank}(\check{\mathbf{K}})=r$, so $\check{\mathbf{K}}$ has a left inverse $\check{\mathbf{P}}$ and $\check{\mathbf{L}}=\check{\mathbf{P}}\mathbf{A}$ --- the two lines used just above for $(\mathbf{K},\mathbf{L})$, applied to $(\check{\mathbf{K}},\check{\mathbf{L}})$. The two routes cost the same, since $\mathrm{rank}(\mathbf{K})=r$ makes $\mathrm{rank}(\mathbf{L})=\mathrm{rank}(\mathbf{A})$. What neither route can dispense with is the inner dimension. At $k=r$ the equality is forced, and above the rank it is not: in the letters of Theorem~\ref{thm:replacement-mirror}, take $\mathbf{A}=\check{\mathbf{F}}=[\,1\ \ 0\,]$ and $\check{\mathbf{H}}^{*}=\mathbf{I}_2$, so that $k=2>r=1$. The factorization $\mathbf{A}=\check{\mathbf{F}}\check{\mathbf{H}}^{*}$ holds with $\check{\mathbf{H}}^{*}$ of full row rank, while $\mathcal{C}(\check{\mathbf{H}})\ne\mathcal{C}(\mathbf{A}^*)$. Clauses (a), (c) and (d) do not need the hypothesis at all: (a) and (c) run on $\mathbf{A}=\mathbf{K}\mathbf{L}$ with $\mathbf{K}$ left-invertible, and (d) on the right inverse $\mathbf{E}$ alone. In clause (c), $\mathbf{A}$ is not read.

For a candidate constructed as $\check{\mathbf{L}}=\mathbf{T}^{-1}\mathbf{L}$ with $\mathbf{T}$ nonsingular, choose $\mathbf{E}=\mathbf{R}\mathbf{T}$, where $\mathbf{R}$ is any right inverse of $\mathbf{L}$. Then $\mathbf{S}_r=\mathbf{L}\mathbf{E}=\mathbf{T}$ and $\check{\mathbf{K}}=\mathbf{K}\mathbf{T}$. Conversely, clause (b) shows that every accepted candidate has this form. This recovers the classical uniqueness of full-rank factorizations up to change of basis and explains why the additional column-space freedom appears only above the minimal inner dimension.

\subsection{How much freedom the criterion leaves}

Now assume that both incumbent outer factors have full column rank and that each candidate transform satisfies \textup{(ADM)}. Write $r=\mathrm{rank}(\mathbf{A})$. The rigidity just found at $k=r$ is a feature of the inner dimension, and Figure~\ref{fig:dichotomy} shows the three cases side by side. At $k=r$ the accepted set is exactly the change-of-basis orbit of $\mathbf{F}$. At $k=m$ so is the set of accepted candidates of rank $k$. For $r<k<m$, the accepted candidates of rank $k$ exceed that orbit by $(k-r)(m-k)$ parameters of column-space freedom.

The rank qualifier is needed only above the rank. At $k=r$ acceptance already forces it. The mixing matrix $\mathbf{G}$ is nonsingular there, so $\mathbf{E}_c\mathbf{G}=\mathbf{0}$ gives $\mathbf{E}_c=\mathbf{0}$ and hence $\mathbf{F}=\check{\mathbf{F}}\mathbf{S}_c$, which leaves $\check{\mathbf{F}}$ no room to drop rank. Above the rank, $\mathbf{G}$ is singular and acceptance no longer forces it. Take $\mathbf{A}=\mathbf{G}=\mathrm{diag}(1,0)$ with $\mathbf{F}=\mathbf{H}=\mathbf{I}_2$, so $k=m=2$ and $r=1$, and offer $\check{\mathbf{F}}=\check{\mathbf{Y}}=\mathrm{diag}(1,0)$. The transform satisfies \textup{(ADM)}, the mismatch is $\mathbf{E}_c=\mathrm{diag}(0,1)$, the test passes because $\mathbf{E}_c\mathbf{G}=\mathbf{0}$, and the repaired factorization reproduces $\mathbf{A}$ exactly. Yet $\check{\mathbf{F}}$ has rank $1$, so it is no change of basis. Rank-deficient candidates of this kind are what the qualifier excludes.

The notation $GL_k(\mathbb{F})$ denotes the \emph{general linear group}: all nonsingular $k\times k$ matrices over $\mathbb{F}$, with matrix multiplication as the group operation. Multiplying a full-column-rank basis matrix on the right by a member of this group changes its basis but preserves its column space. All bases of the same subspace form one $GL_k(\mathbb{F})$ \emph{orbit}. The \emph{Grassmannian}, named after Hermann Grassmann, is the set of all subspaces of a fixed dimension in a given vector space. Thus taking basis matrices modulo $GL_k(\mathbb{F})$ means identifying all bases of each subspace and counting that subspace only once \cite{absil2008optimization}.

For $r\le k\le m$, write $\mathcal{R}_k(\mathbf{A})$ for the candidates of rank $k$ whose column space contains $\mathcal{C}(\mathbf{A})$, and $\mathcal{V}_k(\mathbf{A})$ for the $k$-dimensional subspaces containing $\mathcal{C}(\mathbf{A})$. Under \textup{(PE-r)} and \textup{(ADM)}, Corollary~\ref{cor:replacement-span} identifies $\mathcal{R}_k(\mathbf{A})$ with the accepted rank-$k$ candidates. An accepted candidate therefore involves two choices: a subspace $\mathcal{V}\in\mathcal{V}_k(\mathbf{A})$, and a basis of that subspace. All choices of basis for the same $\mathcal{V}$ form one $GL_k(\mathbb{F})$ orbit.

The column space $\mathcal{C}(\mathbf{A})$ is the part every accepted subspace must contain. It already accounts for $r$ of the $k$ dimensions, leaving $k-r$ independent directions to choose beyond it. Adding a vector from $\mathcal{C}(\mathbf{A})$ to any one of those extra directions does not change their combined span with $\mathcal{C}(\mathbf{A})$. The quotient $\mathbb{F}^m/\mathcal{C}(\mathbf{A})$ records exactly this identification: two vectors represent the same element when their difference lies in $\mathcal{C}(\mathbf{A})$. The mandatory $r$ dimensions become zero in the quotient, so the ambient dimension is $m-r$ and the image $\mathcal{V}/\mathcal{C}(\mathbf{A})$ has dimension $k-r$. Conversely, a $(k-r)$-dimensional subspace of the quotient determines one accepted column space: take all vectors of $\mathbb{F}^m$ whose quotient classes belong to it. Thus $\mathcal{V}\mapsto\mathcal{V}/\mathcal{C}(\mathbf{A})$ is a one-to-one correspondence, with no choice of basis or complementary subspace required.

This identifies the family with the Grassmannian of $(k-r)$-dimensional subspaces in $m-r$ dimensions. Its number of local parameters is $(k-r)((m-r)-(k-r))=(k-r)(m-k)$. The lemma below makes that count concrete: in each coordinate chart the choice is encoded by an arbitrary $(m-k)\times(k-r)$ matrix $\mathbf{Y}_2$. These are parameters for the column space, in addition to the choice of basis within it. They are real parameters over $\mathbb{R}$ and complex parameters over $\mathbb{C}$, giving real dimension $2(k-r)(m-k)$ in the latter case. At $k=r$ there are no extra directions to choose, and at $k=m$ the only possible column space is all of $\mathbb{F}^m$.

For example, if $\mathcal{C}(\mathbf{A})=\mathrm{span}(\mathbf{e}_1)$ in $\mathbb{R}^3$ and $k=2$, an accepted column space is any plane through that fixed line. Modulo $\mathrm{span}(\mathbf{e}_1)$, the plane becomes a line in a two-dimensional quotient. The vectors $\mathbf{e}_2+t\mathbf{e}_3$ and $\mathbf{e}_2+t\mathbf{e}_3+a\mathbf{e}_1$ represent the same quotient vector and determine the same plane together with $\mathbf{e}_1$. This is the one-parameter chart used in the example below, with other charts covering the remaining planes.

\begin{lemma}[the $k$-dimensional subspaces containing $\mathcal{C}(\mathbf{A})$]\label{lem:replacement-family}
Let $r=\mathrm{rank}(\mathbf{A})<k\le m$, with $k\ge1$. In coordinates in which $\mathcal{C}(\mathbf{A})=\mathrm{span}(\mathbf{e}_1,\dots,\mathbf{e}_r)$, a $k$-dimensional subspace on which the first $k$ coordinates are independent lies in $\mathcal{V}_k(\mathbf{A})$ if and only if it is
\begin{align}\label{eq:replacement-dimension}
\mathcal{V}=\mathcal{C}\!\left(\begin{bmatrix}\mathbf{I}_r & \mathbf{0}\\ \mathbf{0} & \mathbf{I}_{k-r}\\ \mathbf{0} & \mathbf{Y}_2\end{bmatrix}\right),
\qquad
\mathbf{Y}_2\in\mathbb{F}^{(m-k)\times(k-r)}\ \text{arbitrary},
\end{align}
and $\mathbf{Y}_2$ is then unique. Charts of this form, one for each choice of $k$ independent coordinates, cover $\mathcal{V}_k(\mathbf{A})$, which therefore carries $(k-r)(m-k)$ free entries on every chart.
\end{lemma}

\begin{proof}
Sending $\mathcal{V}$ to its image in $\mathbb{F}^m/\mathcal{C}(\mathbf{A})$ is a bijection of $\mathcal{V}_k(\mathbf{A})$ onto the $(k-r)$-dimensional subspaces of that quotient, a subspace containing $\mathcal{C}(\mathbf{A})$ being recovered from its image and the image having dimension $k-r$. So the family is the Grassmannian of $(k-r)$-planes in $m-r$ dimensions, of dimension $(k-r)(m-k)$ and covered by charts of that many entries \cite{absil2008optimization}. It remains to say which chart \eqref{eq:replacement-dimension} is. Taking the coordinates of the statement costs nothing, since $\check{\mathbf{F}}\mapsto\mathbf{P}\check{\mathbf{F}}$ preserves rank and moves column spaces by $\mathbf{P}$. A basis matrix of $\mathcal{V}$ whose top $k\times k$ block $\check{\mathbf{F}}_1$ is nonsingular gives $\mathcal{V}=\mathcal{C}([\,\mathbf{I}_k;\mathbf{Y}\,])$ with $\mathbf{Y}=\check{\mathbf{F}}_2\check{\mathbf{F}}_1^{-1}$, and $\mathbf{Y}$ is the only such block. The inclusion $\mathcal{C}(\mathbf{A})\subseteq\mathcal{V}$ holds exactly when the first $r$ columns of $\mathbf{Y}$ vanish, since $\mathbf{e}_i$ has bottom block $\mathbf{0}$ for $i\le r$. What is left of $\mathbf{Y}$ is the $\mathbf{Y}_2$ of \eqref{eq:replacement-dimension}, free and unique, which is the standard chart around $\mathcal{V}$ read in these coordinates.
\end{proof}

\begin{remark}\label{rem:charts}
A single chart does not reach the whole family: for $r=1$, $k=2$ and $m=3$ it misses $\mathrm{span}(\mathbf{e}_1,\mathbf{e}_3)$, whose first $k$ coordinates are dependent. What the atlas supplies is that finitely many charts suffice, each carrying the same free entries, which is what makes the count well defined.
\end{remark}

\begin{figure}[!ht]
\centering
\begin{tikzpicture}[x=1cm,y=1cm,
    amb/.style={draw=black!45,line width=0.5pt,fill=black!3},
    acc/.style={draw=black!55,line width=0.7pt,fill=black!16},
    ca/.style={black,line width=1.5pt},
    panel/.style={font=\normalsize}]

\begin{scope}
  \draw[amb] (0,0) rectangle (3.8,2.2);
  \node[panel,anchor=north west] at (0.08,2.15) {$\mathbb{F}^m$};
  \draw[acc] (0.55,0.76) rectangle (3.25,0.84);
  \draw[ca]  (0.55,0.80) -- (3.25,0.80);
  \node[panel,anchor=north] at (1.90,0.58) {$\mathcal{C}(\mathbf{A})$};
  \node[panel,anchor=south] at (1.90,2.24) {$k=r$};
  \node[panel,anchor=north,align=center] at (1.90,-0.10) {one subspace};
\end{scope}

\begin{scope}[xshift=4.8cm]
  \draw[amb] (0,0) rectangle (3.8,2.2);
  \node[panel,anchor=north west] at (0.08,2.15) {$\mathbb{F}^m$};
  \foreach \t in {0.98, 0.70, 0.45, 0.24}{
    \fill[black!12,draw=black!55,line width=0.5pt]
         (0.55,0.80) -- (3.25,0.80) -- (3.25,{0.80+\t}) -- (0.55,{0.80+0.30*\t})
         -- cycle;}
  \draw[ca] (0.55,0.80) -- (3.25,0.80);
  \fill (0.55,0.80) circle (1.3pt);
  \node[panel,anchor=east] at (0.45,0.80) {$\mathbf{0}$};
  \node[panel,anchor=north] at (1.90,0.58) {$\mathcal{C}(\mathbf{A})$};
  \node[panel,anchor=south] at (1.90,2.24) {$r<k<m$};
  \node[panel,anchor=north,align=center] at (1.90,-0.10)
       {a family of\\ $(k-r)(m-k)$ parameters};
\end{scope}

\begin{scope}[xshift=9.6cm]
  \draw[acc] (0,0) rectangle (3.8,2.2);
  \node[panel,anchor=north west] at (0.08,2.15) {$\mathbb{F}^m$};
  \draw[ca] (0.55,0.80) -- (3.25,0.80);
  \node[panel,anchor=north] at (1.90,0.58) {$\mathcal{C}(\mathbf{A})$};
  \node[panel,anchor=south] at (1.90,2.24) {$k=m$};
  \node[panel,anchor=north,align=center] at (1.90,-0.10) {one subspace};
\end{scope}
\end{tikzpicture}
\caption{Column-space freedom for accepted rank-$k$ candidates, assuming full column rank of both incumbent outer factors and \textup{(ADM)} for each candidate transform. At $k=r$ and $k=m$, there is one admissible $k$-dimensional column space, so every accepted rank-$k$ candidate is a change of basis of $\mathbf{F}$. For $r<k<m$, the spaces form a family with $(k-r)(m-k)$ parameters. Candidates of lower rank are outside this count and can also be accepted, as the network example shows.}
\label{fig:dichotomy}
\end{figure}
\begin{samepage}
The next proposition determines when replacing a factor can change its column space while preserving $\mathbf{A}$. At the minimal width $k=r$, every accepted replacement is a change of basis of the original factor. Additional column spaces become available to rank-$k$ candidates only when $r<k<m$. The proposition separates these cases and turns the preceding parameter count into a description of all accepted candidates of rank $k$.

\begin{proposition}[rigidity dichotomy]\label{prop:replacement-rigidity}
Assume $k\ge1$, \textup{(PE-c)} and \textup{(PE-r)}, and suppose each candidate arrives with a transform satisfying \textup{(ADM)}. Write $r=\mathrm{rank}(\mathbf{A})$.
\begin{enumerate}
\item[(a)] \emph{At the rank: $k=r$.} Equivalently, $\mathbf{G}$ is nonsingular. The following four conditions are equivalent:
\begin{enumerate}
\item[(i)] $\mathbf{E}_c\mathbf{G}=\mathbf{0}$: the candidate passes the acceptance test.
\item[(ii)] $\mathbf{E}_c=\mathbf{0}$: the mismatch itself vanishes.
\item[(iii)] $\mathcal{C}(\check{\mathbf{F}})=\mathcal{C}(\mathbf{F})$: the two factors span the same subspace.
\item[(iv)] $\check{\mathbf{F}}=\mathbf{F}\mathbf{S}_0^{-1}$ for exactly one $\mathbf{S}_0\in GL_k(\mathbb{F})$: the replacement is a change of basis.
\end{enumerate}
When these conditions hold, $\mathbf{S}_0=\mathbf{S}_c$ and the repaired mixing matrix is $\check{\mathbf{G}}=\mathbf{S}_c\mathbf{G}$.
\item[(b)] \emph{Above the rank: $k>r$.} Restrict attention to accepted candidates of rank $k$. Beyond the changes of basis
$$
\{\mathbf{F}\mathbf{S}^{-1}:\mathbf{S}\in GL_k(\mathbb{F})\},
$$
the freedom is the choice of a column space in $\mathcal{V}_k(\mathbf{A})$. Lemma~\ref{lem:replacement-family} gives exactly $(k-r)(m-k)$ additional parameters locally. When $k=m$, the column space is necessarily $\mathbb{F}^m$ and every accepted rank-$k$ candidate is a change of basis of $\mathbf{F}$.
\end{enumerate}
Accepted candidates of rank $k$ that no change of basis reaches exist exactly when $r<k<m$.
\end{proposition}
\end{samepage}

\begin{proof}
(a) Nonsingularity of $\mathbf{G}$ makes $\mathbf{E}_c\mathbf{G}=\mathbf{0}$ equivalent to $\mathbf{E}_c=\mathbf{0}$, that is, to $\mathbf{F}=\check{\mathbf{F}}\mathbf{S}_c$. Let $\mathbf{L}$ be a left inverse of $\mathbf{F}$, supplied by $\mathrm{rank}(\mathbf{F})=k$. From $\mathbf{F}=\check{\mathbf{F}}\mathbf{S}_c$ we get $(\mathbf{L}\check{\mathbf{F}})\mathbf{S}_c=\mathbf{I}_k$; a square matrix with a one-sided inverse is nonsingular, so $\mathbf{S}_c\in GL_k(\mathbb{F})$, $\mathbf{S}_c^{-1}=\mathbf{L}\check{\mathbf{F}}$ and $\check{\mathbf{F}}=\mathbf{F}\mathbf{S}_c^{-1}$, whence $\mathcal{C}(\check{\mathbf{F}})=\mathcal{C}(\mathbf{F})$. Moreover $\mathbf{S}_c\mathbf{L}$ is then a left inverse of $\check{\mathbf{F}}$, so $\mathbf{S}_c$ is the only matrix with $\mathbf{F}=\check{\mathbf{F}}\mathbf{S}_c$, and $\mathbf{S}_0$ is unique and equal to $\mathbf{S}_c=\check{\mathbf{Y}}^{*}\mathbf{F}$. Conversely, $\mathcal{C}(\mathbf{F})\subseteq\mathcal{C}(\check{\mathbf{F}})$ alone suffices, and in particular $\mathcal{C}(\check{\mathbf{F}})=\mathcal{C}(\mathbf{F})$ does: writing $\mathbf{F}=\check{\mathbf{F}}\mathbf{P}$,
\begin{align}
\mathbf{E}_c=\check{\mathbf{F}}\mathbf{P}-\check{\mathbf{F}}\check{\mathbf{Y}}^{*}\check{\mathbf{F}}\mathbf{P}=(\check{\mathbf{F}}-\check{\mathbf{F}}\check{\mathbf{Y}}^{*}\check{\mathbf{F}})\mathbf{P}=\mathbf{0},
\end{align}
by \textup{(ADM)} and nothing else. The closing formula $\check{\mathbf{G}}=\mathbf{S}_c\mathbf{G}$ is Theorem~\ref{thm:replacement}. Finally, \textup{(PE-c)} and \textup{(PE-r)} make $\mathbf{F}$ left-invertible and $\mathbf{H}^*$ right-invertible, so $\mathrm{rank}(\mathbf{A})=\mathrm{rank}(\mathbf{G})$ and $\mathbf{G}$ is nonsingular exactly when $k=r$.

(b) Two matrices of rank $k$ have the same column space exactly when they differ by a right factor in $GL_k(\mathbb{F})$, which is the Grassmannian presented as the matrices of full column rank modulo that group \cite{absil2008optimization}. So $\mathcal{R}_k(\mathbf{A})$ is the disjoint union, over $\mathcal{V}_k(\mathbf{A})$, of one right $GL_k(\mathbb{F})$ orbit per member of the family. The orbit over $\mathcal{C}(\mathbf{F})$ is the change-of-basis orbit $\{\mathbf{F}\mathbf{S}^{-1}:\mathbf{S}\in GL_k(\mathbb{F})\}$, since $\mathbf{F}$ itself lies in $\mathcal{R}_k(\mathbf{A})$ by $\mathbf{A}=\mathbf{F}\mathbf{G}\mathbf{H}^*$ and $\mathrm{rank}(\mathbf{F})=k$. By Lemma~\ref{lem:replacement-family} the family carries $(k-r)(m-k)$ free entries on every chart, and that same $GL_k(\mathbb{F})$ sits over every member. So the accepted candidates of rank $k$ exceed the orbit by exactly those $(k-r)(m-k)$ parameters, and coincide with it exactly when the family is a singleton, which by the count is $(k-r)(m-k)=0$ and so, for $k>r$, is $k=m$. If $r<k<m$, setting the free entry of \eqref{eq:replacement-dimension} to $0$ and to $1$ gives two members of the family, at most one of them $\mathcal{C}(\mathbf{F})$, so a basis matrix of the other is accepted, has rank $k$, and lies outside the orbit. With (a) at $k=r$ and equality at $k=m$, that is the closing sentence.
\end{proof}

\paragraph{An accepted replacement outside the original column space.} The interior regime can be seen in three dimensions. Over $\mathbb{R}$, take $\mathbf{F}=\mathbf{H}=[\,\mathbf{e}_1\ \mathbf{e}_2\,]$, $\mathbf{G}=\mathrm{diag}(1,0)$, and $\mathbf{A}=\mathrm{diag}(1,0,0)$, so $r=1<k=2<m=3$. For any $t\in\mathbb{R}$, offer
\begin{align*}
\check{\mathbf{F}}(t)=
\begin{bmatrix}1&0\\0&1\\0&t\end{bmatrix},
\qquad
\check{\mathbf{Y}}^*=\mathbf{F}^{\mathsf{T}},
\qquad
\mathbf{S}_c=\mathbf{I}_2,
\qquad
\mathbf{E}_c=
\begin{bmatrix}0&0\\0&0\\0&-t\end{bmatrix}.
\end{align*}
Both incumbent factors have full column rank, and $\check{\mathbf{Y}}^*\check{\mathbf{F}}(t)=\mathbf{I}_2$, so \textup{(ADM)} holds. Although $\mathbf{E}_c\ne\mathbf{0}$ for $t\ne0$, its product with $\mathbf{G}$ is zero. Every candidate is accepted with $\check{\mathbf{G}}=\mathbf{G}$. For $t\ne0$, its column space differs from $\mathcal{C}(\mathbf{F})$, because its second column has a nonzero third component. This realizes the single parameter $(k-r)(m-k)=1$ in the chart of Lemma~\ref{lem:replacement-family}.

The example also distinguishes exactness from conditioning. The candidate's certifying projector and its spectral norm are
\begin{align*}
\check{\mathbf{F}}(t)\check{\mathbf{Y}}^*
=\begin{bmatrix}1&0&0\\0&1&0\\0&t&0\end{bmatrix},
\qquad
\|\check{\mathbf{F}}(t)\check{\mathbf{Y}}^*\|_2=\sqrt{1+t^2}.
\end{align*}
The norm follows by taking the largest eigenvalue of the projector's Gram matrix, $\mathrm{diag}(1,1+t^2,0)$. Thus this family preserves $\mathbf{A}$ exactly while the chosen certifying projector has unbounded norm. The choice of transform matters: taking $\check{\mathbf{Y}}^*=\check{\mathbf{F}}(t)^+$ gives the orthogonal projector onto the same accepted column space.

Three remarks on how to read \eqref{eq:replacement-dimension} and the dichotomy drawn from it. First, the count is of entries of $\mathbb{F}$. The maps that produce those entries --- a change of basis and one block inverse --- are rational in the entries of $\check{\mathbf{F}}$ and involve no conjugation and no pseudoinverse. So over $\mathbb{C}$ the parameters are complex and their number is unchanged, while counted over $\mathbb{R}$ there are twice as many, exactly as there are twice as many of the $k^2$ entries of the change of basis. The change of basis contributes that same $k^2$ over every member of the family, and that is what lets those $k^2$ cancel. It is why the excess of $\mathcal{R}_k(\mathbf{A})$ over the orbit is the count of the family and not something larger. The dichotomy $(k-r)(m-k)=0$ is unaffected, and so is the conclusion drawn from it.

Second, the two parts do partition the possibilities. Under the stated hypotheses $\mathrm{rank}(\mathbf{A})=\mathrm{rank}(\mathbf{G})$. So $r\le k$ always, $\mathbf{G}$ is nonsingular exactly at $k=r$, and the cases $k=r$ and $k>r$ are complementary. Both hypotheses earn their place there. Without solvability on the left the chain in (a) fails: for $\mathbf{F}=\mathrm{diag}(1,0)$, $\mathbf{G}=\mathbf{H}=\mathbf{X}=\mathbf{I}_2$ and $\check{\mathbf{F}}=\check{\mathbf{Y}}=\mathbf{I}_2$ every remaining hypothesis holds and $\mathbf{E}_c=\mathbf{0}$, while $\mathcal{C}(\check{\mathbf{F}})\ne\mathcal{C}(\mathbf{F})$. Without solvability on the right the reading of $\mathcal{R}_k(\mathbf{A})$ as the accepted candidates of rank $k$ fails, since that reading is Corollary~\ref{cor:replacement-span}.

Third, \eqref{eq:replacement-dimension} counts one rank at a time. Corollary~\ref{cor:replacement-span} carries no rank hypothesis, so the test also accepts candidates of rank below $k$. Every $\check{\mathbf{F}}$ whose column space contains $\mathcal{C}(\mathbf{A})$ is accepted, whatever \textup{(ADM)} transform accompanies it. The accepted candidates of rank $k'$ span the $k'$-dimensional subspaces containing $\mathcal{C}(\mathbf{A})$, a family of $(k'-r)(m-k')$ parameters, and $(k-r)(m-k)$ is that count at $k'=k$ alone. It does not bound the others. At $r=1$, $k=4$ and $m=5$ the accepted candidates of rank three already span a four-parameter family, against three at rank four. This is why $\mathcal{R}_k(\mathbf{A})$ fixes the rank, and why the counts quoted elsewhere in the paper say which candidates they count.

The dichotomy is what the change of variables of Section~\ref{section:Generalized...} cannot supply. Proposition~\ref{prop:gauge-reduction} rescales the certifying pair with $\mathbf{F}$ and $\mathbf{H}$ held fixed, so it never moves $\mathcal{C}(\mathbf{F})$ and never quantifies over replacements. Part (b) of the proposition exhibits, for every $r<k<m$, accepted candidates of rank $k$ exceeding the orbit by $(k-r)(m-k)$ parameters, and every accepted replacement outside the orbit changes the column space of the left factor. Part (a) proves that at $k=r$ there is nothing above the orbit at all. Factor replacement at $k=r$ is the $GL_k$ change of basis and nothing besides.

\subsection{Relation to earlier work}
\label{subsection:replacement-prior-art}

The replacement criterion connects classical matrix-equation consistency with computations performed on an existing factorization. The following comparisons identify which conditions and updates are inherited and what the present formulation makes explicit.

Substituting $\mathbf{A}=\mathbf{F}\mathbf{G}\mathbf{H}^*$ into the mismatch and regrouping once more,
\begin{align}\label{eq:replacement-projector}
\mathbf{E}_c\mathbf{G}\mathbf{H}^*=\mathbf{A}-\check{\mathbf{F}}\check{\mathbf{Y}}^{*}\mathbf{A},
\end{align}
so the criterion of Theorem~\ref{thm:replacement} asserts that the operator $\check{\mathbf{F}}\check{\mathbf{Y}}^{*}$ fixes $\mathbf{A}$. That operator is a projector onto $\mathcal{C}(\check{\mathbf{F}})$ exactly when $\check{\mathbf{F}}\check{\mathbf{Y}}^{*}\check{\mathbf{F}}=\check{\mathbf{F}}$, and only then does fixing $\mathbf{A}$ amount to a statement about $\mathcal{C}(\check{\mathbf{F}})$ alone. When $\check{\mathbf{Y}}^{*}=(\check{\mathbf{F}})^{+}$ this is the consistency condition for the linear matrix equation $\check{\mathbf{F}}\mathbf{Z}=\mathbf{A}$, which is Penrose's \cite[Thm.~2]{penrose1955generalized}, and it is equivalent to the range inclusion $\mathcal{C}(\mathbf{A})\subseteq\mathcal{C}(\check{\mathbf{F}})$ whose operator form is Douglas's lemma \cite{douglas1966majorization}. Corollary~\ref{cor:replacement-span} is therefore not a new statement about subspaces. It is the classical one, and it is recorded here to identify the accepted set, not to claim it. What Theorem~\ref{thm:replacement} adds is that the same condition is an identity among $\mathbf{F}$, $\mathbf{G}$, $\check{\mathbf{F}}$ and $\check{\mathbf{Y}}$, so that it can be decided --- and the repaired mixing matrix produced --- without forming $\mathbf{A}$, and that the equivalence requires no hypothesis at all on the candidate or on its transform; reading that same test as a condition on $\mathcal{C}(\check{\mathbf{F}})$ alone is Corollary~\ref{cor:replacement-span}, and that reading does consume \textup{(ADM)}.

Hierarchical matrix methods routinely absorb basis transformations into coupling matrices through updates of the form $\mathbf{S}\mapsto\mathbf{T}\mathbf{S}\check{\mathbf{T}}^{*}$ \cite[\S5.2]{boukaram2019hierarchical}. Their orthogonalization and recompression machinery also operates on factored representations \cite[Chs.~5--6]{borm2010efficient}\cite[Ch.~8]{hackbusch2015hierarchical}, including projection-error calculations \cite[\S5.6]{borm2010efficient}. In an exact orthogonalization, the new basis preserves the old span by construction. Rank truncation can change that span and introduce approximation error. The present result treats an arbitrary supplied candidate and transform, including oblique and rank-deficient choices, and states the exact acceptance condition and residual in one form. The operation count above compares this form with explicit reconstruction through $\mathbf{A}$, rather than claiming an improvement over methods that already update stored factors.

Corollary~\ref{cor:replacement-full-rank} likewise recovers, rather than establishes, the classical picture: that the full-rank factorizations of a matrix form a single orbit of $GL_r$ is part of the standard theory of such factorizations \cite{piziak1999fullrank,ben2003generalized}. The point of stating it here is that it arrives as a specialization of the mirror, and that Proposition~\ref{prop:replacement-rigidity} shows it to be a property of the inner dimension alone. The count is standard too, and Lemma~\ref{lem:replacement-family} takes it from the Grassmannian rather than establishing it \cite{absil2008optimization}. What the section adds is the reading of that count as the excess of the accepted candidates of rank $k$ over the change-of-basis orbit.

Finally, replacing part of a factorization rather than recomputing it is the premise of the updating literature --- of rank-revealing $ULV$ and $URV$ updating \cite{stewart1993updating}, of incremental singular value decompositions \cite{brand2006fast}, and of the updating and downdating supported by generalized Nystr\"{o}m \cite{nakatsukasa2020fast}. Those methods modify a factorization when the matrix changes. The question here is the complementary one, with the matrix held fixed and the factor exchanged, and it is answered by a criterion rather than by an algorithm.

The same shape turns up in iterative solvers. The vocabulary there invites a confusion, so it is best to mark where the resemblance stops. Krylov subspace recycling carries a subspace from one solve to the next, and the recycling method of Parks and co-authors maintains an invariant of exactly the present kind: a basis $\mathbf{U}$ together with its image $\mathbf{C}=\mathbf{A}\mathbf{U}$ normalized so that $\mathbf{C}^{*}\mathbf{C}=\mathbf{I}$. When the matrix changes, that pair is repaired rather than recomputed. A thin $QR$ of the new image supplies $\mathbf{Q}$ and $\mathbf{R}$, the image basis becomes $\mathbf{Q}$, and the subspace basis becomes $\mathbf{U}\mathbf{R}^{-1}$ \cite{parks2006recycling}. That is a $GL_k$ change of basis restoring a named invariant, which is the role $\mathbf{S}_c$ plays here, and it uses the same $GL_k$ change-of-basis freedom described in Proposition~\ref{prop:replacement-rigidity}(a), here to maintain a solver invariant. Deflation is the neighbouring device, and in the survey of that literature it is called deflation preconditioning \cite{soodhalter2020survey}. Preconditioning replaces $\mathbf{A}\mathbf{x}=\mathbf{b}$ by an equivalent system, and nothing of that kind happens here. The matrix is untouched, the factorization still reproduces it exactly, and what changes is only which factors express it. A preconditioner is judged by how fast an iteration converges; the criterion of this section decides whether a substitution is exact. The analogy is one of structure, not of purpose.

Two recent papers meet this section from two different directions, and the relation to each should be stated plainly.

The first replaces a factor \emph{gradually} rather than in one step. Everything above is discrete: $\mathbf{F}$ is exchanged for $\check{\mathbf{F}}$ and the core is corrected once. Suppose instead the three factors are allowed to drift continuously, as $\mathbf{F}(t)$, $\mathbf{G}(t)$, $\mathbf{H}(t)$, and ask which velocities keep the product equal to the same $\mathbf{A}$ at every instant. Differentiating $\mathbf{A}=\mathbf{F}(t)\mathbf{G}(t)\mathbf{H}(t)^{*}$ answers it: the velocities must satisfy
\begin{align}\label{eq:tangent-constraint}
\dot{\mathbf{F}}\mathbf{G}\mathbf{H}^{*}+\mathbf{F}\dot{\mathbf{G}}\mathbf{H}^{*}+\mathbf{F}\mathbf{G}\dot{\mathbf{H}}^{*}=\mathbf{0}.
\end{align}
The left change-of-basis action of Proposition~\ref{prop:replacement-rigidity}(a), together with its mirror, supplies a family of such curves. Pick any two $k\times k$ matrices $\mathbf{K}_c$ and $\mathbf{K}_r$ --- they are free, and they generate the motion --- and set
\begin{align}\label{eq:gauge-curve}
\mathbf{F}(t)=\mathbf{F}e^{t\mathbf{K}_c},
\qquad
\mathbf{H}(t)=\mathbf{H}e^{t\mathbf{K}_r^{*}},
\qquad
\mathbf{G}(t)=e^{-t\mathbf{K}_c}\mathbf{G}e^{-t\mathbf{K}_r}.
\end{align}
The product telescopes back to $\mathbf{A}$ for every $t$, since the two exponentials on each side cancel, and $\mathbf{F}(t)$ is a change of basis of $\mathbf{F}$ because $e^{t\mathbf{K}_c}$ is invertible. So \eqref{eq:gauge-curve} is the orbit traced out continuously, and its velocity at $t=0$ is $\dot{\mathbf{F}}=\mathbf{F}\mathbf{K}_c$, $\dot{\mathbf{H}}=\mathbf{H}\mathbf{K}_r^{*}$, $\dot{\mathbf{G}}=-\mathbf{K}_c\mathbf{G}-\mathbf{G}\mathbf{K}_r$. The signs are worth keeping straight: $\mathbf{K}_c$ and $\mathbf{K}_r$ generate the \emph{outer factors}, while the interfaces they induce run the other way. Because $\mathbf{F}(t)$ and $\mathbf{H}(t)$ have full column rank, \textup{(ADM)} makes their accompanying transforms one-sided inverses, so the induced interfaces are $\mathbf{S}_c(t)=e^{-t\mathbf{K}_c}$ and $\mathbf{S}_r(t)=e^{-t\mathbf{K}_r}$, whose generators are $-\mathbf{K}_c$ and $-\mathbf{K}_r$. Benzi and Viviani \cite{benzi2025lie} linearize multiplicative group factorizations of a \emph{varying} matrix near the identity, and obtain additive decompositions of its tangent --- the classical splittings $\mathbf{A}=\mathbf{M}-\mathbf{N}$ from which iterative methods are built. The curve \eqref{eq:gauge-curve} does something different: it holds $\mathbf{A}$ fixed, so its three derivatives sum to zero rather than decomposing anything. It describes the reparametrizations that leave $\mathbf{A}$ fixed, and it carries two independent $GL_k(\mathbb{F})$ actions rather than one: $\mathbf{K}_c$ generates the left action of Proposition~\ref{prop:replacement-rigidity}(a) and $\mathbf{K}_r$ the right action supplied by Theorem~\ref{thm:replacement-mirror}, the interfaces themselves being generated by $-\mathbf{K}_c$ and $-\mathbf{K}_r$. The connection is one of differential form; the two constructions serve different purposes.

The second connection concerns two-sided preconditioning. Jiang and Yin \cite{jiang2026twosided} study two-sided preconditioned CGLS for factorized linear systems through an augmented formulation. A related algebraic change of factors, rather than a description of their particular preconditioner, can be written in the notation of this section as follows. Let $\mathbf{A}=\mathbf{L}\mathbf{R}^{\mathsf{T}}$ and let $\mathbf{C}_L,\mathbf{C}_R$ be nonsingular. Set
\begin{align}\label{eq:two-sided-precond}
\check{\mathbf{F}}=\mathbf{L}\mathbf{C}_L,
\qquad
\check{\mathbf{H}}=\mathbf{R}\mathbf{C}_R^{\mathsf{T}},
\qquad
\check{\mathbf{G}}=\mathbf{C}_L^{-1}\mathbf{C}_R^{-1},
\end{align}
and $\check{\mathbf{F}}\check{\mathbf{G}}\check{\mathbf{H}}^{\mathsf{T}}=\mathbf{A}$ for every such pair, again by cancellation. Both outer factors move and the core absorbs the change. Read as an application of Theorem~\ref{thm:replacement} and its mirror, this asks more than the cancellation does: those statements assume \textup{(PE-c)} and \textup{(PE-r)}, so $\mathbf{L}$ and $\mathbf{R}$ must have full column rank. The identity itself holds without that.

Neither observation is a result of this paper, and the boundary should be said out loud. What this section contributes is the finite test and the closed-form repair. What it does not contribute is any reason to prefer one admissible choice over another --- and for a preconditioner that is the whole question, since $\mathbf{C}_L$ and $\mathbf{C}_R$ are chosen to make an iteration converge quickly. That is a conditioning matter, and Section~\ref{subsection:not-claimed} says this paper does not settle it. A different line of work picks a representative by minimizing a norm instead, obtaining factorizations that are optimal and in some cases unique \cite{christensen2023unique}.

\subsection{A worked example: replacing a spanning-tree factor}
\label{subsection:replacement-application}

Suppose the Laplacian of a resistor network is already stored as $\mathbf A=\mathbf F\mathbf G\mathbf H^{\mathsf T}$. We want to replace $\mathbf F$ by a factor built from a different selection of edges, keep $\mathbf H$ fixed, and update the core so that the product still represents the same network. The question is whether the proposed factor contains all the directions needed to reconstruct $\mathbf A$. The example below applies the replacement test to two edge selections and then to two factors with an extra column. The network and all its conductances remain fixed throughout. Selecting edges for a factor does not mean removing the other resistors from the physical network.

\begin{figure}[!ht]
\centering
\begin{tikzpicture}[x=1.2cm,y=1.4cm,
    nd/.style={circle,draw,fill=white,inner sep=1.1pt,font=\normalsize},
    off/.style={gray!40,line width=0.6pt},
    on/.style={black,line width=1.6pt},
    panel/.style={font=\normalsize}]
%
\begin{scope}
  \coordinate(p1)at(0,1);\coordinate(p2)at(0,0);\coordinate(p3)at(1,.5);
  \coordinate(p4)at(2,1);\coordinate(p5)at(2,0);
  \draw[off](p1)--(p2) node[midway,left,font=\normalsize,text=black]{2};
  \draw[off](p2)--(p3) node[midway,below,font=\normalsize,text=black]{1};
  \draw[off](p1)--(p3) node[midway,above,font=\normalsize,text=black]{3};
  \draw[off](p3)--(p4) node[midway,above,font=\normalsize,text=black]{1};
  \draw[off](p4)--(p5) node[midway,right,font=\normalsize,text=black]{4};
  \draw[off](p3)--(p5) node[midway,below,font=\normalsize,text=black]{2};
  \foreach\i in{1,...,5}{\node[nd]at(p\i){\i};}
  \node[panel]at(1,-.62){(a) the network};
\end{scope}
\begin{scope}[xshift=3.8cm]
  \coordinate(p1)at(0,1);\coordinate(p2)at(0,0);\coordinate(p3)at(1,.5);
  \coordinate(p4)at(2,1);\coordinate(p5)at(2,0);
  \draw[off](p1)--(p2);\draw[off](p2)--(p3);\draw[off](p1)--(p3);
  \draw[off](p3)--(p4);\draw[off](p4)--(p5);\draw[off](p3)--(p5);
  \draw[on](p1)--(p2);\draw[on](p2)--(p3);\draw[on](p3)--(p4);\draw[on](p4)--(p5);
  \foreach\i in{1,...,5}{\node[nd]at(p\i){\i};}
  \node[panel]at(1,-.62){(b) the factor $\mathbf{F}$};
\end{scope}
\begin{scope}[xshift=7.6cm]
  \coordinate(p1)at(0,1);\coordinate(p2)at(0,0);\coordinate(p3)at(1,.5);
  \coordinate(p4)at(2,1);\coordinate(p5)at(2,0);
  \draw[off](p1)--(p2);\draw[off](p2)--(p3);\draw[off](p1)--(p3);
  \draw[off](p3)--(p4);\draw[off](p4)--(p5);\draw[off](p3)--(p5);
  \draw[on](p1)--(p3);\draw[on](p2)--(p3);\draw[on](p3)--(p5);\draw[on](p4)--(p5);
  \foreach\i in{1,...,5}{\node[nd]at(p\i){\i};}
  \node[panel]at(1,-.62){(c) accepted};
\end{scope}
\begin{scope}[xshift=11.4cm]
  \coordinate(p1)at(0,1);\coordinate(p2)at(0,0);\coordinate(p3)at(1,.5);
  \coordinate(p4)at(2,1);\coordinate(p5)at(2,0);
  \draw[off](p1)--(p2);\draw[off](p2)--(p3);\draw[off](p1)--(p3);
  \draw[off](p3)--(p4);\draw[off](p4)--(p5);\draw[off](p3)--(p5);
  \draw[on](p1)--(p2);\draw[on](p2)--(p3);\draw[on](p1)--(p3);\draw[on](p4)--(p5);
  \foreach\i in{1,...,5}{\node[nd]at(p\i){\i};}
  \node[panel]at(1,-.62){(d) refused};
\end{scope}
\end{tikzpicture}
\caption{The network and the edge selections used in the first two replacement tests. (a) The five-node network, with conductances written beside its six resistors. Its Laplacian maps node voltages to the net currents supplied at those nodes. (b) The original spanning tree: its four oriented edge vectors form the columns of the incidence matrix $\mathbf F=\mathbf H$. (c) A second spanning tree gives an accepted replacement for $\mathbf F$, with an updated core and the original $\mathbf H$ retained. (d) These four edges leave two disconnected components, so their incidence matrix cannot replace $\mathbf F$ while preserving the original Laplacian. For these edge-based candidates with pseudoinverse transforms, acceptance is equivalent to connecting all five nodes. The thin gray edges show the unchanged original network.}
\label{fig:network}
\end{figure}

Figure~\ref{fig:network}(a) has $n=5$ nodes and six resistors, all with positive conductance. Let $\mathbf v$ contain the node voltages. A resistor joining nodes $i$ and $j$ carries current $c_{ij}(v_i-v_j)$ from $i$ to $j$. Summing these outgoing currents at node $i$ gives the net external current $b_i$ that must be supplied there to maintain the voltages. Thus $\mathbf b=\mathbf A\mathbf v$, with $b_i=\sum_j c_{ij}(v_i-v_j)$ and
\begin{align}\label{eq:laplacian}
\mathbf{A}=\sum_{\text{resistors}}c_{ij}\,
(\mathbf{e}_i-\mathbf{e}_j)(\mathbf{e}_i-\mathbf{e}_j)^{*},
\end{align}
the weighted Laplacian of the network. A positive $b_i$ means that current enters the network from an external source at node $i$, and a negative $b_i$ means that current is withdrawn there. Collecting the vectors $\mathbf e_i-\mathbf e_j$ as columns and the conductances on a diagonal gives the usual incidence-matrix assembly \cite{strangCSE}. The same assembly also describes spring networks, pipe networks and steady diffusion on a graph, with a different interpretation of the edge weights.

For an edge oriented from $i$ to $j$, its incidence column is $\mathbf e_i-\mathbf e_j$: it has a $+1$ at node $i$, a $-1$ at node $j$, and zeros elsewhere. Because the network is connected, these columns span the four-dimensional space of vectors whose entries sum to zero. This is $\mathcal C(\mathbf A)$, so $r=\mathrm{rank}(\mathbf A)=n-1=4$. The one missing direction is the constant vector $\mathbf 1$. Adding the same constant to every node voltage changes no voltage difference and hence no current.

A spanning tree connects all five nodes using four edges and no cycle. Its four incidence columns are independent and form a basis for $\mathcal C(\mathbf A)$. Let $\mathbf F$ be the $5\times4$ incidence matrix of the tree in Figure~\ref{fig:network}(b). Since the Laplacian is symmetric, we use the same factor on the right, $\mathbf H=\mathbf F$, and set $\mathbf G=\mathbf F^+\mathbf A(\mathbf H^+)^{\mathsf T}$. The product $\mathbf F\mathbf G\mathbf H^{\mathsf T}$ represents all six resistors, including those outside the tree. Their contributions are encoded in $\mathbf G$.

For a proposed left factor $\check{\mathbf F}$, choose its pseudoinverse as the transform, $\check{\mathbf Y}^*=\check{\mathbf F}^+$. Compute $\mathbf S_c=\check{\mathbf F}^+\mathbf F$ and update the core to $\check{\mathbf G}=\mathbf S_c\mathbf G$. We call the replacement \emph{accepted} when $\check{\mathbf F}\check{\mathbf G}\mathbf H^{\mathsf T}=\mathbf A$. Since the retained $\mathbf H$ has full column rank, Theorem~\ref{thm:replacement} tests this by $\mathbf E_c\mathbf G=\mathbf0$, where $\mathbf E_c=\mathbf F-\check{\mathbf F}\mathbf S_c$. Corollary~\ref{cor:replacement-span} says that acceptance here is equivalent to $\mathcal C(\mathbf A)\subseteq\mathcal C(\check{\mathbf F})$. For incidence columns of edges on these five nodes, that condition holds exactly when the selected edges connect every node. A disconnected selection cannot represent current injections that transfer current between its separate components.

The code tests four proposed left factors. The first two replace the original $\mathbf F$ while retaining $\mathbf H=\mathbf F$. For the last two, the starting factorization instead uses $\mathbf F_m=\mathbf H_m=[\mathbf F\ \mathbf1]$, with a $5\times5$ core $\mathbf G_m$ calculated in the same way. The added constant column makes $\mathbf F_m$ invertible, although $\mathbf A$ still has rank four.
\begin{enumerate}
\item The incidence matrix of the second spanning tree in Figure~\ref{fig:network}(c), denoted by \texttt{T2} in the code.
\item The incidence matrix of the disconnected edge selection in Figure~\ref{fig:network}(d), denoted by \texttt{D}.
\item The second tree's incidence matrix with the constant column appended.
\item The original tree's incidence matrix with a zero column appended, replacing the constant column of $\mathbf F_m$.
\end{enumerate}

The MATLAB script first forms $\mathbf A$ and the two starting factorizations. Its helper \texttt{inc} stores one edge per row, so transposing its output gives the node-by-edge factors described above. In the loop, \texttt{base} and \texttt{core} select the starting factor and core for each case, and \texttt{Fc} is the proposed replacement. The value \texttt{test} is $\|\mathbf E_c\mathbf G\|_F$ (using $\mathbf G_m$ in the last two cases). Computing this test and the updated core uses the stored factors without reading $\mathbf A$. The separate diagnostic \texttt{err} deliberately uses $\mathbf A$ to check the reconstructed product. The remaining columns report the candidate rank and $\det(\mathbf S_c)$. The formula for \texttt{G} uses \texttt{H} explicitly to match $\mathbf G=\mathbf F^+\mathbf A(\mathbf H^+)^{\mathsf T}$, although \texttt{H=F} makes the two pseudoinverses identical here.

The diagram in Figure~\ref{fig:network} shows exactly the network and edge lists used in this script. Panels (c) and (d) correspond to cases 1 and 2. Cases 3 and 4 add a constant or zero column, not an edge, so they do not require additional graph panels. The script prints the following four rows in order.
\begin{lstlisting}[style=matlabstyle]
%% the network, and the Laplacian it defines
n = 5;
E = [1 2; 2 3; 1 3; 3 4; 4 5; 3 5];   g = [2 1 3 1 4 2]';
B = inc(E,n);   A = B'*diag(g)*B;            % rank r = n-1 = 4

%% at k = r: F = H = the edge vectors of a spanning tree
T1 = [1 2; 2 3; 3 4; 4 5];   T2 = [1 3; 2 3; 3 5; 4 5];
D  = [1 2; 2 3; 1 3; 4 5];               % triangle + lone edge: disconnected
F  = inc(T1,n)';   H = F;    G  = pinv(F)*A*pinv(H)';

%% at k = m: adjoin the constant vector, the null direction
Fm = [F ones(n,1)];          Gm = pinv(Fm)*A*pinv(Fm)';

%% the criterion, on four candidates
cand = {inc(T2,n)', inc(D,n)', [inc(T2,n)' ones(n,1)], [F zeros(n,1)]};
base = {F, F, Fm, Fm};       core = {G, G, Gm, Gm};
for c = 1:4
    Fc = cand{c};   Yc = pinv(Fc);           % a {1}-inverse is admissible
    Sc = Yc*base{c};   Ec = base{c} - Fc*Sc; % interface and mismatch
    test = norm(Ec*core{c},'fro');           % the criterion
    err  = norm(A - Fc*(Sc*core{c})*base{c}','fro');
    fprintf('rank %d   test %11.6e   err %11.6e   det %6.3f\n', ...
            rank(Fc), test, err, det(Sc));
end

function B = inc(E,n)                        % incidence matrix of an edge list
    B = zeros(size(E,1),n);
    for e = 1:size(E,1),  B(e,E(e,1)) = 1;  B(e,E(e,2)) = -1;  end
end
\end{lstlisting}
\begin{lstlisting}[style=outputstyle]
rank 4   test 3.114894e-14   err 3.752010e-14   det  1.000
rank 3   test 3.291403e+00   err 3.415650e+00   det -0.000
rank 5   test 3.044879e-14   err 4.435485e-14   det  1.000
rank 4   test 6.895672e-15   err 1.815507e-14   det  0.000
\end{lstlisting}

These are the MATLAB R2024b results, with one row for each of the four cases listed above. The displayed small nonzero residuals are floating-point effects. Exact arithmetic gives zero for both residuals in cases 1, 3 and 4. In case 2, the criterion norm is $\|\mathbf E_c\mathbf G\|_F=\sqrt{390}/6\approx3.291403$, and the reconstruction residual is $\|(\mathbf E_c\mathbf G)\mathbf H^{\mathsf T}\|_F=\sqrt{105}/3\approx3.415650$. The printed negative zero in the determinant column is also a floating-point effect.

Case 1 is accepted because the second tree connects all five nodes and its incidence columns span $\mathcal C(\mathbf A)$. Case 2 is rejected because its columns sum to zero separately on each of its two components. For example, they cannot represent $\mathbf e_1-\mathbf e_5$, which injects current in one component and withdraws it from the other. That vector belongs to $\mathcal C(\mathbf A)$ for the original connected network. The failed span condition therefore explains why updating the core cannot make this candidate reconstruct the original Laplacian.

Case 1 also illustrates Proposition~\ref{prop:replacement-rigidity}(a). At $k=r=4$, both the original and accepted replacement columns are bases for the same four-dimensional space. Consequently, $\mathbf F=\check{\mathbf F}\mathbf S_c$ with $\mathbf S_c$ invertible: the replacement changes the basis used to represent the network, while preserving that space. Here $\det(\mathbf S_c)=1$. More generally, the change between two oriented spanning-tree incidence bases has integer entries and determinant $\pm1$. This follows from total unimodularity of incidence matrices, whose square submatrices have determinant $0$ or $\pm1$. Thus integer entries in the stored core remain integer under this tree-basis update.

Cases 3 and 4 explain the rank qualification in Proposition~\ref{prop:replacement-rigidity}(b). Here $k=m=5$, and the original factor $\mathbf F_m$ spans all of $\mathbb R^5$. Every other rank-five factor also spans $\mathbb R^5$, so replacing it by another full-rank factor can only change the basis of that space. Case 3 does precisely this, using the second tree and the same constant column. Its interface is invertible, and the test accepts it.
Case 4 is accepted as well, but its candidate has rank four, because the fifth column is zero. No invertible change of basis can turn the rank-five $\mathbf F_m$ into this candidate. Nevertheless, the four tree columns already span everything needed by $\mathbf A$. The constant direction is in the nullspace of the Laplacian, and its row and column in $\mathbf G_m$ are zero. Replacing that unused factor column by zero therefore leaves the product unchanged. In practice, a representation may contain a direction that can be discarded without losing the matrix it stores. This is why the classification of accepted rank-$k$ factors does not cover every possible accepted replacement.

For this connected five-node network, the two possible full-rank stored widths are $k=r=4$ and $k=m=5$. There is no integer $k$ strictly between them. At width four, an accepted rank-four factor must span the zero-sum node space. At width five, every rank-five factor spans the whole node space. The example therefore shows changes of basis at both endpoints, together with a rank-deficient replacement in case 4, but it cannot show accepted rank-$k$ factors with different column spaces.

To see that additional freedom, consider instead a five-node network with two connected components and positive conductances. Its Laplacian has rank $r=3$, leaving two independent constant-on-component directions in its nullspace. At width $k=4$, a full-rank factor can contain the three required directions and one chosen mixture of those two unused directions. Different mixtures can give different four-dimensional column spaces, each containing the Laplacian's column space and hence allowing exact reconstruction. The count in Lemma~\ref{lem:replacement-family} is then $(k-r)(m-k)=(4-3)(5-4)=1$. This is the practical distinction: extra stored columns can sometimes be redirected into a different subspace without changing the matrix, whereas changing a minimal tree basis only changes coordinates within one fixed subspace.

\subsection{Changing both tree factors and preserving effective resistance}
\label{subsection:replacement-resistance}

Return to the minimal four-column tree factors of the connected network, and write $\mathbf L=\mathbf A$. In the preceding tests only the left factor was replaced. To retain a symmetric representation, we now replace both outer factors by the second tree's incidence matrix $\check{\mathbf F}$. Starting from $\mathbf L=\mathbf F\mathbf G\mathbf F^{\mathsf T}$, use the left interface $\mathbf S_c=\check{\mathbf F}^+\mathbf F$ and the right interface $\mathbf S_r=\mathbf S_c^{\mathsf T}$. Updating the core on both sides gives the congruence transformation
\begin{align}\label{eq:replacement-congruence}
\mathbf{L}=\check{\mathbf{F}}\check{\mathbf{G}}\check{\mathbf{F}}^{\mathsf{T}},
\qquad
\check{\mathbf{G}}=\mathbf{S}\mathbf{G}\mathbf{S}^{\mathsf{T}},
\qquad
\mathbf{S}=\check{\mathbf{F}}^{+}\mathbf{F},
\end{align}
which is symmetric whenever $\mathbf{G}$ is. On the network of Figure~\ref{fig:network} all three matrices are integral:
\begin{align}
\mathbf{S}=\begin{bmatrix}1&0&0&0\\-1&1&0&0\\0&0&1&0\\0&0&-1&1\end{bmatrix},
\quad
\mathbf{G}=\begin{bmatrix}5&3&0&0\\3&4&0&0\\0&0&3&2\\0&0&2&6\end{bmatrix},
\quad
\check{\mathbf{G}}=\begin{bmatrix}5&-2&0&0\\-2&3&0&0\\0&0&3&-1\\0&0&-1&5\end{bmatrix}.
\end{align}
The two cores use different tree coordinates but reconstruct exactly the same Laplacian. The interface $\mathbf S$ is the one with determinant $1$ in case 1 above. Its integer entries explain why the two cores also have integer entries.

Preserving the Laplacian preserves the network's current--voltage relation $\mathbf b=\mathbf L\mathbf v$. The same node voltages require the same supplied currents, and the same current injections produce the same voltage differences. One useful quantity is the effective resistance between nodes $i$ and $j$: supply one unit of current at $i$ and withdraw it at $j$, and measure the resulting voltage difference $v_i-v_j$. Changing the tree factors must leave this value unchanged.

For the minimal tree basis used here, $\mathbf G=\mathbf F^+\mathbf L(\mathbf F^+)^{\mathsf T}$ is symmetric positive definite. Indeed, $\mathbf L$ is positive semidefinite, and for nonzero $\mathbf x$ the vector $(\mathbf F^+)^{\mathsf T}\mathbf x$ is a nonzero element of $\mathcal C(\mathbf L)$, where the Laplacian is positive definite. Thus $\mathbf G^{-1}$ exists and $\mathbf L^+=(\mathbf F^+)^{\mathsf T}\mathbf G^{-1}\mathbf F^+$. The effective resistance is $R_{ij}=(\mathbf e_i-\mathbf e_j)^{\mathsf T}\mathbf L^+(\mathbf e_i-\mathbf e_j)$ \cite{klein1993resistance}. Writing the unit-current injection in tree coordinates as $\mathbf z_{ij}=\mathbf F^+(\mathbf e_i-\mathbf e_j)$ gives
\begin{align}\label{eq:resistance-form}
R_{ij}=\mathbf{z}_{ij}^{\mathsf{T}}\mathbf{G}^{-1}\mathbf{z}_{ij}.
\end{align}
Replacing the tree sends $\mathbf{z}_{ij}$ to $\mathbf{S}\mathbf{z}_{ij}$ and $\mathbf{G}$ to $\mathbf{S}\mathbf{G}\mathbf{S}^{\mathsf{T}}$, and the value does not move, since $(\mathbf{S}\mathbf{z})^{\mathsf{T}}(\mathbf{S}\mathbf{G}\mathbf{S}^{\mathsf{T}})^{-1}(\mathbf{S}\mathbf{z})=\mathbf{z}^{\mathsf{T}}\mathbf{G}^{-1}\mathbf{z}$. Here $R_{12}=4/11$, $R_{15}=97/154$, $R_{45}=3/14$ and $R_{24}=68/77$, and both trees return them exactly. Summing over pairs gives the Kirchhoff index \cite{gutman1996quasi}
\begin{align}\label{eq:kirchhoff}
\sum_{i<j}R_{ij}=n\,\mathrm{tr}\bigl(\mathbf{G}^{-1}(\mathbf{F}^{\mathsf{T}}\mathbf{F})^{-1}\bigr)=\frac{394}{77},
\end{align}
whose right-hand side reads the stored factors and nothing else: $\mathbf{L}$ is never formed.

To use this calculation after choosing another spanning tree, form its incidence matrix $\check{\mathbf F}$, compute $\mathbf S=\check{\mathbf F}^+\mathbf F$, and replace the stored core by $\check{\mathbf G}=\mathbf S\mathbf G\mathbf S^{\mathsf T}$. Resistances can then be evaluated from \eqref{eq:resistance-form} using the new tree coordinates and core, or retained from the old representation because their values are unchanged. Neither this core update nor the resistance calculation requires rebuilding the full Laplacian.

The resistance identities \eqref{eq:resistance-form} and \eqref{eq:kirchhoff} are classical, as is their expression in spanning-tree coordinates \cite{zelazo2014definiteness,zelazo2018graph}. The role of meta-factorization here is to give the explicit update between the two stored representations. This is an exact change of representation. By comparison, the companion work \cite{karpowicz2026generalized} estimates effective resistance from principal submatrices and bounds the approximation error. The present example makes no approximation claim. Its displayed cores and resistances are exact integers or rationals, and the rejected residual norms above are exact radicals. The MATLAB output illustrates how the same calculation appears in floating-point arithmetic.

\subsection{What is not claimed}
\label{subsection:not-claimed}

For a supplied factor and transform, this section gives an exact replacement test, a core update and a factored residual, using the stored factors. Turning these identities into a numerical method requires three further ingredients.

\begin{itemize}
\item \emph{A reliable acceptance rule.} Floating-point arithmetic needs a tolerance and an error bound that account for the conditioning and scaling of the factors and transform. The exact residual formulas do not by themselves justify treating a computed small mismatch as zero.
\item \emph{A choice of candidate.} The pseudoinverse repair is optimal for a fixed candidate space; choosing that space still requires balancing approximation error, storage and conditioning. The parameter count \eqref{eq:replacement-dimension} covers rank-$k$ candidates only. Lower-rank candidates can also be exact under the span criterion, but belong to separate families.
\item \emph{Performance evidence.} The test and repair counts measure arithmetic after the candidate and transform are available. Candidate construction has its own cost, and the paper reports no timing comparison with refactorization.
\end{itemize}

\section{How meta-factorization describes known factorizations}
\label{section:How...}

The projector and reconstruction equations recover the algebraic forms of the SVD, CPQR, and UTV factorizations. Each construction specifies the inner dimensions for exact reconstruction and the additional choices that produce its characteristic core. These hypotheses also identify where the replacement criterion of Section~\ref{section:Replacement} can be applied.

The subsections take the SVD, the CPQR decomposition and the UTV family in turn: the first reconstructs its target, the second derives the form of the mixing matrix, and the third factors the mixing matrix of the QR-based decompositions and so develops a family of UTV-like factorizations.

\subsection{Singular value decomposition (SVD)}
\label{subsection:SVD}

We begin by running the procedure on the reduced SVD. The exercise turns on one hypothesis that is easy to overlook: the two bases must be a \emph{matched} pair of singular vector systems, and not merely eigenvector bases of $\mathbf{A}\mathbf{A}^*$ and $\mathbf{A}^*\mathbf{A}$, or the mixing matrix is not $\mathrm{diag}(\sigma_1,\dots,\sigma_k)$.

For every rectangular $m\times n$ matrix $\mathbf{A}$, the SVD represents $\mathbf{A}$ by a singular value matrix with unitary changes of coordinates on the two sides:
\begin{align}
\setstackgap{L}{15pt}\def\stacktype{L}\def\sz{\scriptstyle}
\begin{aligned}
\stackunder{\mathbf{S}}{\sz (m\times n)} =
\stackunder{\mathbf{U}^*}{\sz (m\times m)} &&
\stackunder{\mathbf{A}}{\sz (m\times n)} &&
\stackunder{\mathbf{V}.}{\sz (n\times n)}
\end{aligned}
\end{align}

The matrices $\mathbf{U}$ and $\mathbf{V}$ are complex unitary, and $\mathbf{S}$ is a rectangular diagonal matrix with (real non-negative) singular values on the diagonal. The fundamental subspaces for $\mathbf{A}$ --- the column space, the row space and the corresponding nullspaces --- and the rank of $\mathbf{A}$ all emerge from this factorization.

For $k=r=\mathrm{rank}(\mathbf{A})$ the reduced singular value decomposition factors $\mathbf{A}$ into the following product:
\begin{align}
\setstackgap{L}{15pt}\def\stacktype{L}\def\sz{\scriptstyle}
\begin{aligned}
\stackunder{\mathbf{A}}{\sz (m\times n)} &=&
\stackunder{\mathbf{U}(:,1{:}k)}{\sz (m\times k)} &&
\stackunder{\mathbf{S}(1{:}k,1{:}k)}{\sz (k\times k)} &&
\stackunder{\mathbf{V}(:,1{:}k)^*.}{\sz (k\times n)}
\end{aligned}
\end{align}

We can reconstruct the reduced singular value decomposition by following the steps of meta-factorization. The goal is to build the diagonal mixing matrix $\mathbf{G} = \mathbf{S}(1{:}k,1{:}k) = \mathrm{diag}(\sigma_1,\dots,\sigma_k)$.

The first step is to select bases for $\mathcal{C}(\mathbf{A})$ and $\mathcal{C}(\mathbf{A}^*)$. The required choice is
\begin{align}
\mathbf{F} = \mathbf{U}(:,1{:}k)
\quad\text{and}\quad
\mathbf{H}^* = \mathbf{V}(:,1{:}k)^*,
\end{align}
where, for $k\le r$, the columns of $\mathbf{U}(:,1{:}k)$ and $\mathbf{V}(:,1{:}k)$ are orthonormal eigenvectors corresponding to the $k$ largest eigenvalues of $\mathbf{A}\mathbf{A}^*$ and $\mathbf{A}^*\mathbf{A}$. More generally, $\mathbf{U}$ and $\mathbf{V}$ satisfy the (Hermitian) eigenvalue equations
\begin{align}\label{eq:svd-eig}
\begin{aligned}
(\mathbf{A}\mathbf{A}^*)\mathbf{U} = \mathbf{U}(\mathbf{S}\mathbf{S}^*)
\quad\text{and}\quad
(\mathbf{A}^*\mathbf{A})\mathbf{V} = \mathbf{V}(\mathbf{S}^*\mathbf{S}).
\end{aligned}
\end{align}
Two words in the previous sentence are chosen carefully. The matrices \emph{satisfy} \eqref{eq:svd-eig}. They are not \emph{defined} by it, and the distinction is what the next paragraph is about. We also fix a choice of $\mathbf{U}(:,1{:}k)$ once and for all: when $\sigma_k=\sigma_{k+1}$ the notation $\mathbf{U}(:,1{:}k)$ names a selection rather than a well-defined object, and the gap condition $\sigma_k>\sigma_{k+1}$ is what removes the ambiguity by making the rank-$k$ truncation unique. Nothing below consumes that gap. Proposition~\ref{prop:SVD-gauge} holds for whichever selection is made, and repeated singular values --- straddling the cut or lying below it --- leave both its statement and its proof intact.
The structure of the SVD is imposed by the choice of $\mathbf{F}$ and $\mathbf{H}$, but not by \eqref{eq:svd-eig} alone. Those two eigenvalue equations leave each column of $\mathbf{U}$ and of $\mathbf{V}$ free up to a sign over $\mathbb{R}$, or a phase over $\mathbb{C}$, \emph{independently on the two sides}, and free up to a whole unitary block inside a repeated singular value. We therefore require in addition that $\mathbf{U}$ and $\mathbf{V}$ be a \emph{matched} pair, that is, that the columns of $\mathbf{H} = [\mathbf{v}_1\mid\dots\mid\mathbf{v}_k]$ be connected with the columns of $\mathbf{F} = [\mathbf{u}_1\mid\dots\mid\mathbf{u}_k]$ by
\begin{align}\label{eq:svd bounds}
\mathbf{A}\mathbf{v}_i=\sigma_i\mathbf{u}_i
\quad\text{and}\quad
\mathbf{A}^*\mathbf{u}_i=\sigma_i\mathbf{v}_i
\quad\text{for}\quad
i = 1,\dots,k,
\end{align}
equivalently $\mathbf{A}\mathbf{V}(:,1{:}k)=\mathbf{U}(:,1{:}k)\mathbf{S}(1{:}k,1{:}k)$, which for $k=\min(m,n)$ is $\mathbf{A}=\mathbf{U}\mathbf{S}\mathbf{V}^*$. Given \eqref{eq:svd-eig}, the first equation of \eqref{eq:svd bounds} implies the second, so only the first is a genuine requirement: for $\sigma_i>0$ it gives $\mathbf{A}^*\mathbf{u}_i=\mathbf{A}^*\mathbf{A}\mathbf{v}_i/\sigma_i=\sigma_i\mathbf{v}_i$, and for $\sigma_i=0$ one has $\|\mathbf{A}^*\mathbf{u}_i\|_2^2=\mathbf{u}_i^*\mathbf{A}\mathbf{A}^*\mathbf{u}_i=0$. It is a hypothesis, not a consequence, and Proposition~\ref{prop:SVD-gauge} shows that it is exactly the right one: it is necessary as well as sufficient, so nothing weaker will serve, and no separate phase convention need be imposed on top of it. For the indices with $\sigma_i=0$ it holds automatically, so it constrains only $i\le\min(k,r)$.

\begin{proposition}\label{prop:SVD-gauge}
Let $\mathbf{A}\in\mathbb{C}^{m\times n}$, let $\mathbf{S}$ carry the singular values $\sigma_1\ge\dots\ge\sigma_{\min(m,n)}\ge0$, let $\mathbf{U}\in U(m)$ and $\mathbf{V}\in U(n)$ satisfy \eqref{eq:svd-eig}, fix $1\le k\le\min(m,n)$ and put $\mathbf{G}=\mathbf{U}(:,1{:}k)^*\mathbf{A}\mathbf{V}(:,1{:}k)$. Then
\begin{align}
\mathbf{G} = \mathrm{diag}(\sigma_1,\dots,\sigma_k)
\qquad\text{if and only if}\qquad
\mathbf{A}\mathbf{v}_i=\sigma_i\mathbf{u}_i \quad\text{for } i=1,\dots,k,
\end{align}
and under the right-hand condition $\mathbf{G}$ is diagonal \emph{and} real non-negative.
\end{proposition}

\begin{proof}
Both directions are computations with the entries $\mathbf{u}_i^*\mathbf{A}\mathbf{v}_j$ of $\mathbf{G}$. $(\Leftarrow)$ Assume $\mathbf{A}\mathbf{v}_i=\sigma_i\mathbf{u}_i$ for $i=1,\dots,k$. Then the $(i,j)$ entry of $\mathbf{G}$ is $\mathbf{u}_i^*\mathbf{A}\mathbf{v}_j=\sigma_j\mathbf{u}_i^*\mathbf{u}_j=\sigma_j\delta_{ij}$, which is the asserted diagonal form. $(\Rightarrow)$ Fix $j\le k$. The second equation of \eqref{eq:svd-eig} gives $\mathbf{A}^*\mathbf{A}\mathbf{v}_j=\sigma_j^2\mathbf{v}_j$, hence $\|\mathbf{A}\mathbf{v}_j\|_2^2=\mathbf{v}_j^*\mathbf{A}^*\mathbf{A}\mathbf{v}_j=\sigma_j^2$. Expanding $\mathbf{A}\mathbf{v}_j$ in the orthonormal basis $\mathbf{u}_1,\dots,\mathbf{u}_m$ of $\mathbb{C}^m$, Parseval's identity reads $\sum_{i=1}^{m}|\mathbf{u}_i^*\mathbf{A}\mathbf{v}_j|^2=\sigma_j^2$, while the hypothesis supplies $\mathbf{u}_j^*\mathbf{A}\mathbf{v}_j=\mathbf{G}_{jj}=\sigma_j$. Every remaining coordinate therefore vanishes, and $\mathbf{A}\mathbf{v}_j=\sigma_j\mathbf{u}_j$. No gap between $\sigma_k$ and $\sigma_{k+1}$ is used: the coordinates with $i>k$ are killed by Parseval, not by orthogonality. They cannot be killed by orthogonality --- for $\mathbf{A}=\mathbf{I}_2$, $\mathbf{U}=\mathbf{I}_2$, $\mathbf{V}=\bigl(\begin{smallmatrix}3/5&-4/5\\ 4/5&3/5\end{smallmatrix}\bigr)$ and $k=1$ both equations of \eqref{eq:svd-eig} hold identically while $\mathbf{u}_2^*\mathbf{A}\mathbf{v}_1=4/5\ne0$.
\end{proof}

The implication $(\Rightarrow)$ needs less than the displayed equivalence offers: it uses only the second equation of \eqref{eq:svd-eig}, and only the diagonal of $\mathbf{G}$ --- $\mathbf{G}_{jj}=\sigma_j$ for $j=1,\dots,k$ already implies both sides of the equivalence.

The hypothesis cannot be dropped. Take $\mathbf{A}=\mathbf{I}_2$, $\mathbf{U}=\mathbf{I}_2$ and $\mathbf{V}=\bigl(\begin{smallmatrix}c&-s\\ s&c\end{smallmatrix}\bigr)$ with $c^2+s^2=1$: both equations of \eqref{eq:svd-eig} hold identically, yet $\mathbf{G}=\mathbf{V}$, whose off-diagonal entry is $-s$. Over $\mathbb{C}$ a single phase suffices to break diagonality of the \emph{real non-negative} kind: with $\mathbf{A}=\mathrm{diag}(3,1)$, $\mathbf{U}=\mathrm{diag}(\mathrm{i},1)$ and $\mathbf{V}=\mathbf{I}_2$ one gets $\mathbf{G}=\mathrm{diag}(-3\mathrm{i},1)$.

Thus, under \eqref{eq:svd bounds}, and consuming beyond it only the orthonormality of the columns $\mathbf{u}_1,\dots,\mathbf{u}_k$ --- neither eigenvalue equation is needed for this step --- we have
\begin{align}\label{eq:SVDdiagonalizationMs}
\mathbf{u}^*_i\mathbf{A}\mathbf{v}_i &= \sigma_i,\ i = 1,\dots,k,
\\
\mathbf{u}^*_i\mathbf{A}\mathbf{v}_j &= 0,\ i \neq j.
\end{align}
The second step of meta-factorization should confirm this fundamental fact and, indeed, it does. Solving the projector equation (\ref{eq:projector-eq}) for $\mathbf{Y}$ and $\mathbf{X}$, we obtain
\begin{align}
\mathbf{Y}^* = \mathbf{U}(:,1{:}k)^*
\quad\text{and}\quad
\mathbf{X} = \mathbf{V}(:,1{:}k).
\end{align}
These matched transforms produce the diagonal core of singular values. By Proposition~\ref{prop:SVD-gauge} the matched-pair hypothesis \eqref{eq:svd bounds} is exactly what is needed, and the final third step of the meta-factorization procedure yields
\begin{align}
\begin{aligned}
\mathbf{G} = \mathbf{S}(1{:}k,1{:}k) = \mathbf{U}(:,1{:}k)^*\mathbf{A}\mathbf{V}(:,1{:}k) = \mathrm{diag}(\sigma_1,\dots,\sigma_k).
\end{aligned}
\end{align}
As desired, we have reconstructed the reduced SVD for $\mathbf{A}$: for $k\ge r$,
\begin{align}
\begin{aligned}
\mathbf{A}
&= \mathbf{U}(:,1{:}k)\mathbf{U}(:,1{:}k)^*\mathbf{A}\mathbf{V}(:,1{:}k)\mathbf{V}(:,1{:}k)^*
\\
&= \mathbf{U}(:,1{:}k)\mathbf{S}(1{:}k,1{:}k)\mathbf{V}(:,1{:}k)^*.
\end{aligned}
\end{align}
For $k<r$ the right-hand side is the rank-$k$ truncation of $\mathbf{A}$ rather than $\mathbf{A}$ itself, and the equality must be replaced by the Eckart--Young bound. No choice of projectors can do better, since $\mathrm{rank}(\mathbf{F}\mathbf{G}\mathbf{H}^*)\le k$. This is the $k<r$ row of Proposition~\ref{prop:regimes}. Nothing here is new about the SVD: the subsection validates the procedure against a decomposition whose answer is known, and its one contribution is Proposition~\ref{prop:SVD-gauge}, which says which hypothesis on $\mathbf{U}$ and $\mathbf{V}$ the diagonal mixing matrix really needs.

\subsection{Column-Pivoted QR decomposition (CPQR)}
\label{subsection:CPQR}

Run on the CPQR, the procedure returns a mixing matrix that is simply the identity: the decomposition has already placed $\mathbf{A}$ in the coordinates given by its factors. Exactly when that holds, and exactly when the reconstruction holds, are two different questions with two different answers, and this subsection keeps them apart.

The CPQR decomposition factors any rectangular $m\times n$ matrix $\mathbf{A}$ into the product of a unitary $m\times m$ matrix $\mathbf{Q}$, an upper rectangular $m\times n$ matrix $\mathbf{R}$ and an $n\times n$ permutation matrix $\mathbf{\Pi}$:
\begin{align}
\setstackgap{L}{15pt}\def\stacktype{L}\def\sz{\scriptstyle}
\begin{aligned}
\stackunder{\mathbf{A}}{\sz (m\times n)} &=&
\stackunder{\mathbf{Q}}{\sz (m\times m)} &&
\stackunder{\mathbf{R}}{\sz (m\times n)} &&
\stackunder{\mathbf{\Pi}^*.}{\sz (n\times n)}
\end{aligned}
\end{align}
The CPQR also reveals matrix rank. See Chan, Gu and Eisenstat \cite{chan1987rank,gu1996efficient} for more details. Throughout this section $\mathbf{A}\mathbf{\Pi}=\mathbf{Q}\mathbf{R}$ is computed with Businger--Golub greedy column pivoting \cite{businger1965linear}, or with any rule satisfying $|R_{jj}|\ge\|\mathbf{R}(j{:}m,l)\|_2$ for $l\ge j$; the greedy rule satisfies that inequality under any tie-break. Without such a rule the block form displayed next fails, and so do the conclusions that Proposition~\ref{prop:CPQR}(c) and the second sentence of Proposition~\ref{prop:CPQR}(d) draw from it; those two statements carry the rule as a hypothesis, so dropping it leaves them vacuous rather than false. Parts (a), (b) and the first sentence of (d) hold without it. For $\mathbf{A}$ of rank $r$ and $k=r$ we then obtain
\begin{align}
\begin{aligned}
\mathbf{R} =
\begin{bmatrix}
\mathbf{R}_{11} & \mathbf{R}_{12}
\\
\mathbf{0} & \mathbf{0}
\end{bmatrix},
\end{aligned}
\end{align}
where $\mathbf{R}_{11}=\mathbf{R}(1{:}r,1{:}r)$ is a nonsingular upper triangular $r\times r$ matrix, the nonsingularity being a consequence of the pivoting rule rather than of the block shape. The thin QR decomposition is then given by
\begin{align}
\setstackgap{L}{15pt}\def\stacktype{L}\def\sz{\scriptstyle}
\begin{aligned}
\stackunder{\mathbf{A}}{\sz (m\times n)} &=&
\stackunder{\mathbf{Q}(:,1{:}k)}{\sz (m\times k)} &&
\stackunder{\mathbf{R}(1{:}k,:)}{\sz (k\times n)} &&
\stackunder{\mathbf{\Pi}^*.}{\sz (n\times n)}
\end{aligned}
\end{align}
Now meta-factorize the decomposition itself. Suppose that $\mathbf{Q}$, $\mathbf{R}$ and $\mathbf{\Pi}$ are known for $\mathbf{A}$, and ask for the form of the mixing matrix $\mathbf{G}$.

Define $\mathbf{F}$ and $\mathbf{H}$ by the factors of the thin QR decomposition,
\begin{align}
\mathbf{F} = \mathbf{Q}(:,1{:}k) \quad\text{and}\quad
\mathbf{H}^* = \mathbf{R}(1{:}k,:)\mathbf{\Pi}^*.
\end{align}
The default design parameters $\mathbf{B}=\mathbf{F}$ and $\mathbf{D}=\mathbf{H}$ give, through \eqref{eq:YX-projector}, $\mathbf{Y}^*=\mathbf{Q}(:,1{:}k)^*$ and $\mathbf{X}=(\mathbf{H}^*)^{+}$ --- a pair that solves the projector equation precisely when $\mathrm{rank}(\mathbf{R}(1{:}k,:))=k$ (Proposition~\ref{prop:CPQR}(b) and the example following it) --- hence
\begin{align}\label{eq:CPQRsemisimilarity}
\begin{aligned}
\mathbf{G} &=
\mathbf{Q}(:,1{:}k)^*\mathbf{A}(\mathbf{R}(1{:}k,:)\mathbf{\Pi}^*)^+.
\end{aligned}
\end{align}
A remark for implementers: $\mathbf{X}=(\mathbf{H}^*)^{+}$ as displayed and $\mathbf{X}=\mathbf{D}(\mathbf{H}^*\mathbf{D})^{+}$ with $\mathbf{D}=\mathbf{H}$, as Theorem~\ref{Theorem:projector-eq-sol} supplies it, are algebraically identical but not numerically so. The second forms $\mathbf{H}^*\mathbf{H}$ and squares the condition number. Use the displayed form.
However, by the properties of the thin QR decomposition,
\begin{align}
\begin{aligned}
\mathbf{Q}(:,1{:}k)^*\mathbf{A} = \mathbf{R}(1{:}k,:)\mathbf{\Pi}^*.
\end{aligned}
\end{align}
Hence $\mathbf{G}=\mathbf{H}^*(\mathbf{H}^*)^{+}$ unconditionally, and the mixing matrix for the CPQR decomposition is
\begin{align}\label{eq:CPQR-G}
\mathbf{G} = \mathbf{I}_k
\end{align}
precisely when $\mathbf{R}(1{:}k,:)$ has rank $k$, which is what the pivoting rule secures for every $k\le r$. Collecting these statements with their exact hypotheses, and adding the range of $k$ in which the reconstruction itself holds:

\begin{proposition}\label{prop:CPQR}
Let $\mathbf{A}\in\mathbb{C}^{m\times n}$ with $r=\mathrm{rank}(\mathbf{A})$, let $\mathbf{A}\mathbf{\Pi}=\mathbf{Q}\mathbf{R}$ be any column-pivoted QR factorization, fix $1\le k\le\min(m,n)$ and set $\mathbf{F}=\mathbf{Q}(:,1{:}k)$, $\mathbf{H}^*=\mathbf{R}(1{:}k,:)\mathbf{\Pi}^*$, $\mathbf{B}=\mathbf{F}$, $\mathbf{D}=\mathbf{H}$. Then
\begin{enumerate}
\item[(a)] $\mathbf{G}=\mathbf{H}^*(\mathbf{H}^*)^{+}$, with no hypothesis on rank, pivoting or conditioning;
\item[(b)] $\mathbf{G}=\mathbf{I}_k$ if and only if $\mathrm{rank}(\mathbf{R}(1{:}k,:))=k$, which forces $k\le r$;
\item[(c)] if the pivoting rule satisfies $|R_{jj}|\ge\|\mathbf{R}(j{:}m,l)\|_2$ for all $l\ge j$, then $R_{jj}\ne0$ for every $j\le r$, so $\mathbf{R}_{11}$ is nonsingular and $\mathbf{G}=\mathbf{I}_k$ whenever $k\le r$;
\item[(d)] $\mathbf{A}=\mathbf{F}\mathbf{G}\mathbf{H}^*$ if and only if $\mathbf{R}((k{+}1){:}m,:)=\mathbf{0}$, and this forces $k\ge r$. Hence, \emph{under} (c): $\mathbf{A}=\mathbf{F}\mathbf{G}\mathbf{H}^*$ if and only if $k\ge r$, and $\mathbf{G}=\mathbf{I}_k$ and $\mathbf{A}=\mathbf{F}\mathbf{G}\mathbf{H}^*$ hold together exactly when $k=r$.
\end{enumerate}
\end{proposition}

\begin{proof}
(a) $\mathbf{Q}(:,1{:}k)^*\mathbf{A}=\mathbf{R}(1{:}k,:)\mathbf{\Pi}^*=\mathbf{H}^*$ is the display above, and $\mathbf{X}=(\mathbf{H}^*)^{+}$, so $\mathbf{G}=\mathbf{Y}^*\mathbf{A}\mathbf{X}=\mathbf{H}^*(\mathbf{H}^*)^{+}$. (b) For $\mathbf{M}\in\mathbb{C}^{k\times n}$, $\mathbf{M}\mathbf{M}^{+}=\mathbf{I}_k$ if and only if $\mathbf{M}$ has full row rank; and $\mathrm{rank}(\mathbf{R}(1{:}k,:))\le\mathrm{rank}(\mathbf{R})=r$. (c) If $R_{jj}=0$ for some $j\le r$ then the rule forces $\mathbf{R}(j{:}m,l)=\mathbf{0}$ for every $l\ge j$, and $\mathbf{R}(j{:}m,l)=\mathbf{0}$ for $l<j$ because $\mathbf{R}$ is upper trapezoidal; hence $\mathbf{R}(j{:}m,:)=\mathbf{0}$ and $\mathrm{rank}(\mathbf{R})\le j-1<r$, a contradiction. (d) $\mathbf{F}\mathbf{G}\mathbf{H}^*=\mathbf{Q}(:,1{:}k)\mathbf{H}^*(\mathbf{H}^*)^{+}\mathbf{H}^*=\mathbf{Q}(:,1{:}k)\mathbf{R}(1{:}k,:)\mathbf{\Pi}^*$ by the first Penrose equation, and this equals $\mathbf{A}=\mathbf{Q}\mathbf{R}\mathbf{\Pi}^*$ if and only if the trailing rows of $\mathbf{R}$ vanish, which gives $\mathrm{rank}(\mathbf{R})\le k$, i.e.\ $k\ge r$. For the converse assume (c) and $k\ge r$, and suppose $\mathbf{R}((k{+}1){:}m,:)\ne\mathbf{0}$. Since $\mathbf{R}$ is upper trapezoidal there is a nonzero entry $R_{i,l}$ with $i>k\ge r$ and $l\ge i$, so $\|\mathbf{R}((r{+}1){:}m,l)\|_2>0$ with $l\ge r+1$, and (c) forces $|R_{r+1,r+1}|>0$. Together with $R_{jj}\ne0$ for $j\le r$ this makes $\mathbf{R}(1{:}(r{+}1),1{:}(r{+}1))$ nonsingular upper triangular, so $\mathrm{rank}(\mathbf{R})\ge r+1$ --- a contradiction.
\end{proof}

Without a pivoting rule the condition of (b) can fail where one wants it to hold: for $\mathbf{A}=\mathrm{diag}(0,1)$ with $r=k=1$ and $\mathbf{\Pi}=\mathbf{I}_2$ the factorization $\mathbf{A}=\mathbf{Q}\mathbf{R}$ with $\mathbf{Q}=\mathbf{I}_2$ is a legal CPQR, $\mathbf{R}_{11}=[0]$, $\mathbf{H}^*=[0\ \ 0]$, the projector equation is unsolvable and $\mathbf{G}=[0]$. Nonsingularity of $\mathbf{R}_{11}$ is sufficient but not necessary for (b): $\mathbf{R}(1{:}2,:)=\bigl(\begin{smallmatrix}0&1&0\\0&0&1\end{smallmatrix}\bigr)$ has singular $\mathbf{R}_{11}$ and $\mathbf{G}=\mathbf{I}_2$.

Notice what this does and does not say. The identity $\mathbf{G}=\mathbf{I}_k$ records that the CPQR has already placed $\mathbf{A}$ in the coordinates given by $\mathbf{F}$ and $\mathbf{H}$, so that no further mixing is required. Beyond the rank bound $k\le r$ that (b) implies, it carries no information about $\mathbf{A}$ itself, and in particular it is not a diagonalization or a similarity relation.

The final step of the meta-factorization then yields the thin column-pivoted QR decomposition,
\begin{align}
\begin{aligned}
\mathbf{A}
&= \mathbf{Q}(:,1{:}k)\mathbf{G}\mathbf{R}(1{:}k,:)\mathbf{\Pi}^*
\\
&= \mathbf{Q}(:,1{:}k)\mathbf{R}(1{:}k,:)\mathbf{\Pi}^*,
\end{aligned}
\end{align}
as expected --- but only under the pivoting rule of Proposition~\ref{prop:CPQR}(c). Under that rule the mixing matrix is the identity when $k\le r$ and the reconstruction holds when $k\ge r$, so the two coincide exactly at $k=r$, which is the case treated above. For $k<r$ we then have $\mathbf{G}=\mathbf{I}_k$ exactly while $\mathbf{A}\ne\mathbf{F}\mathbf{G}\mathbf{H}^*$, since $\mathrm{rank}(\mathbf{F}\mathbf{G}\mathbf{H}^*)\le k<r$. The identity $\mathbf{G}=\mathbf{I}_k$ does not imply reconstruction. Without a pivoting rule neither half survives. The converse implication $k\ge r\Rightarrow\mathbf{R}((k{+}1){:}m,:)=\mathbf{0}$ of Proposition~\ref{prop:CPQR}(d) needs the pivoting rule of (c) and is false without it, and both $\mathbf{G}=\mathbf{I}_k$ and the reconstruction can then fail at $k=r$: for the legal CPQR of $\mathbf{A}=\mathrm{diag}(0,1)$ above, $k=r=1$ gives $\mathbf{G}=[0]$ and $\mathbf{A}\ne\mathbf{F}\mathbf{G}\mathbf{H}^*$. Here too the framework returns a known object, and what it adds is Proposition~\ref{prop:CPQR}, which keeps the two conditions apart instead of reading them off a single picture of the block structure of $\mathbf{R}$.

\subsection{UTV decompositions}
\label{subsection:UTV}

We have already seen how meta-factorization describes selected factorizations, the SVD and the CPQR. It can also design them, and the lever is simple: whatever structure one imposes on the factors of the mixing matrix $\mathbf{G}$ is inherited by the factors of $\mathbf{A}$. Factoring $\mathbf{G}$ three ways --- by the SVD, by the CPQR and by the LU decomposition --- therefore produces three UTV-like factorizations from one construction, and the first two guarantee orthonormal factors on both sides. We demonstrate this by reconstructing the UTV decomposition and then introducing its modifications.

As described by Golub and Van Loan and by Stewart \cite{golub2013matrix,stewart1998matrix}, the matrix $\mathbf{A}$ can be factored into a product of two unitary matrices, $\mathbf{U}$ and $\mathbf{V}$, and an upper-triangular or lower-triangular mixing matrix $\mathbf{T}$. The resulting UTV decomposition\footnote{The UTV decomposition with upper-triangular $\mathbf{T}$ is also known as the URV decomposition, whereas the one with lower-triangular $\mathbf{T}$ is also known as the ULV decomposition.} is a~generalization of the SVD and the CPQR decomposition. The SVD is the case with diagonal $\mathbf{T}$ and the CPQR is the case with $\mathbf{V}$ being the permutation matrix. The symbols $\mathbf{U}$, $\mathbf{T}$ and $\mathbf{V}$ in this section denote the UTV factors and are unrelated to the SVD factors $\mathbf{U}$, $\mathbf{S}$, $\mathbf{V}$ of Section~\ref{subsection:SVD}. The SVD of the mixing matrix is written with bars, $\bar{\mathbf{U}}\bar{\mathbf{S}}\bar{\mathbf{V}}^*$.

Throughout this section $1\le k\le\min(m,n)$, so that $\mathbf{Q}_c(:,1{:}k)$ and $\mathbf{Q}_r(:,1{:}k)$ are defined. The unitary factors come from the column-pivoted QR decompositions
\begin{align}
\mathbf{A} = \mathbf{Q}_c\mathbf{R}_c\mathbf{\Pi}^*_c
\quad \text{and}\quad
\mathbf{A}^* = \mathbf{Q}_r\mathbf{R}_r\mathbf{\Pi}^*_r.
\end{align}
To reconstruct the UTV decomposition we need only the row space projector. Define the row space basis by the leading $k$ columns of $\mathbf{Q}_r$, which we require to span $\mathcal{C}(\mathbf{A}^*)$:
\begin{align}
\mathbf{H}^* = \mathbf{Q}_r(:,1{:}k)^*.
\end{align}
The projector equation yields
\begin{align}
\mathbf{X} = \mathbf{Q}_r(:,1{:}k).
\end{align}
As a result we reach the decomposition
\begin{align}\label{eq:one-sided-instance}
\mathbf{F} = \mathbf{A}\mathbf{Q}_r(:,1{:}k),
\qquad
\mathbf{H} = \mathbf{Q}_r(:,1{:}k),
\qquad
\mathbf{G} = \mathbf{I}_k,
\end{align}
which is a legitimate instance of \eqref{eq:genfactorization}: the mixing matrix is $k\times k$, as the signature requires. Notice that this is the \emph{adjoint dual} of the CPQR instance of Section~\ref{subsection:CPQR}: there $\mathbf{F}$ carries the orthonormal basis and $\mathbf{H}^*$ the coordinates, here $\mathbf{H}^*$ carries the orthonormal basis and $\mathbf{F}$ the coordinates. The two therefore lose rank on opposite sides: above the rank it is $\mathbf{H}^*$ that drops below $k$ there, and $\mathbf{F}$ that drops below $k$ here. The reconstruction
\begin{align}\label{eq:one-sided-RE}
\mathbf{A} = \mathbf{A}\mathbf{Q}_r(:,1{:}k)\mathbf{Q}_r(:,1{:}k)^*
\end{align}
holds if and only if $\mathcal{C}(\mathbf{A}^*)\subseteq\mathcal{C}(\mathbf{Q}_r(:,1{:}k))$, that is, if and only if the leading $k$ columns of $\mathbf{Q}_r$ span the row space of $\mathbf{A}$, as required above. It is not an identity for an arbitrary unitary $\mathbf{Q}_r$. With rank-revealing pivoting in $\mathbf{A}^*=\mathbf{Q}_r\mathbf{R}_r\mathbf{\Pi}_r^*$ the condition reads $\mathbf{R}_r((k{+}1){:}n,:)=\mathbf{0}$, equivalently $k\ge\mathrm{rank}(\mathbf{A})$, and under it the residual block $\mathbf{A}\mathbf{Q}_r(:,(k{+}1){:}n)$ vanishes, since $\mathbf{A}\mathbf{Q}_r(:,(k{+}1){:}n)=\mathbf{A}\mathbf{Q}_r(:,1{:}k)\mathbf{Q}_r(:,1{:}k)^*\mathbf{Q}_r(:,(k{+}1){:}n)=\mathbf{0}$ for a unitary $\mathbf{Q}_r$, so the coordinate matrix built below is rectangular diagonal and the construction returns a full SVD of $\mathbf{A}$. The row-span assumption is confined to \eqref{eq:one-sided-RE}: nothing below is derived from it. The identity \eqref{eq:UTV-ground}, the decomposition \eqref{eq:UTV}, Proposition~\ref{prop:UTV-completion} and the example following it are all stated for an arbitrary unitary $\mathbf{Q}_r$ and any $k$, where the residual block need not vanish and its ordered completion is what secures triangularity. There the reconstruction rests not on \eqref{eq:one-sided-RE} but on the unconditional identity $\mathbf{Q}_r\mathbf{Q}_r^*=\mathbf{I}_n$. Partitioning the full basis into its leading and trailing blocks gives
\begin{align}\label{eq:UTV-ground}
\begin{aligned}
\mathbf{A}
&= \mathbf{A}\mathbf{Q}_r\mathbf{Q}_r^*
\\
&=
\bmat{\mathbf{A}\mathbf{Q}_r(:,1{:}k) & \mathbf{A}\mathbf{Q}_r(:,(k{+}1){:}n)}
\begin{bmatrix}
\mathbf{Q}_r(:,1{:}k)^*
\\
\mathbf{Q}_r(:,(k{+}1){:}n)^*
\end{bmatrix}
\\
&=
\bmat{\mathbf{F} & \widetilde{\mathbf{W}}}
\mathbf{Q}_r^*,
\qquad \widetilde{\mathbf{W}}:=\mathbf{A}\mathbf{Q}_r(:,(k{+}1){:}n).
\end{aligned}
\end{align}

This is the factorization we wish to transform into the UTV decomposition. The key step is the characteristic move of meta-factorization: factor a factor. Here the mixing matrix is trivial by \eqref{eq:one-sided-instance}, so the factor to be split is the column basis $\mathbf{F}$. In this case the required factors are provided by the \emph{full} SVD of $\mathbf{F}$, with $\bar{\mathbf{U}}\in\mathbb{C}^{m\times m}$ unitary:
\begin{align}\label{eq:SVDofG}
\begin{aligned}
\mathbf{F} = \mathbf{A}\mathbf{Q}_r(:,1{:}k) = \bar{\mathbf{U}}\bar{\mathbf{S}}\bar{\mathbf{V}}^*.
\end{aligned}
\end{align}
When $k<m$, a thin SVD does not in general suffice. Its $m\times k$ factor $\bar{\mathbf{U}}$ satisfies $\bar{\mathbf{U}}\bar{\mathbf{U}}^*\ne\mathbf{I}_m$ and $\mathbf{A}-\mathbf{U}\mathbf{T}\mathbf{V}^*=(\mathbf{I}_m-\bar{\mathbf{U}}\bar{\mathbf{U}}^*)\widetilde{\mathbf{W}}\mathbf{Q}_r(:,(k{+}1){:}n)^*$, which is nonzero unless $\mathcal{C}(\widetilde{\mathbf{W}})\subseteq\mathcal{C}(\bar{\mathbf{U}})$, a subspace that reduces to $\mathcal{C}(\mathbf{F})$ exactly when $\mathbf{F}$ has full column rank. Thus a thin SVD can lose part of the reconstruction. At the boundary $k=m$, the thin left factor is square and unitary, so this obstruction disappears.
Inserting (\ref{eq:SVDofG}) into (\ref{eq:UTV-ground}) and factoring out $\bar{\mathbf{U}}$ and $\bar{\mathbf{V}}^*$, we obtain the desired UTV decomposition, still with no span hypothesis:
\begin{align}\label{eq:UTV}
\begin{aligned}
\mathbf{A} &=
\bar{\mathbf{U}}
\bmat{\bar{\mathbf{S}} & \bar{\mathbf{U}}^*\widetilde{\mathbf{W}}}
\begin{bmatrix}
\bar{\mathbf{V}}^* & \mathbf{0}
\\
\mathbf{0} & \mathbf{I}_{n-k}
\end{bmatrix}\mathbf{Q}_r^*
= \mathbf{U}\mathbf{T}\mathbf{V}^*,
\\
\mathbf{U} &= \bar{\mathbf{U}},
\\
\mathbf{T} &=
\bmat{\bar{\mathbf{S}} & \bar{\mathbf{U}}^*\widetilde{\mathbf{W}}},
\\
\mathbf{V} &= \mathbf{Q}_r\begin{bmatrix}
\bar{\mathbf{V}} & \mathbf{0}
\\
\mathbf{0} & \mathbf{I}_{n-k}
\end{bmatrix}.
\end{aligned}
\end{align}
The matrices $\mathbf{U}$ and $\mathbf{V}$ are indeed unitary, being square and orthonormal, and the identity $\mathbf{A}=\mathbf{U}\mathbf{T}\mathbf{V}^*$ holds for every unitary $\mathbf{Q}_r$, every $k$ and every full SVD of $\mathbf{F}$, with no rank hypothesis at all. Triangularity of $\mathbf{T}$, however, is not automatic, and what secures it must be said exactly. Only the right-hand block $\bar{\mathbf{U}}^*\widetilde{\mathbf{W}}$ can carry mass below the diagonal. A full SVD of $\mathbf{F}$ pins only the first $\rho=\mathrm{rank}(\mathbf{F})$ columns of $\bar{\mathbf{U}}$. The remaining columns are an arbitrary orthonormal basis of $\mathcal{C}(\mathbf{F})^{\perp}$, and an unlucky choice destroys triangularity. The following rule removes the freedom.

\begin{proposition}\label{prop:UTV-completion}
Let $\mathbf{F}=\mathbf{A}\mathbf{Q}_r(:,1{:}k)$ with $\rho=\mathrm{rank}(\mathbf{F})$, let $\widetilde{\mathbf{W}}=\mathbf{A}\mathbf{Q}_r(:,(k{+}1){:}n)$, and construct a full SVD $\mathbf{F}=\bar{\mathbf{U}}\bar{\mathbf{S}}\bar{\mathbf{V}}^*$ as follows: apply Gram--Schmidt to the columns of $\widetilde{\mathbf{W}}$, taken in order, after projecting out $\mathcal{C}(\mathbf{F})$, discarding any column whose projection is zero; let $\bar{\mathbf{u}}_{\rho+1},\dots,\bar{\mathbf{u}}_{\rho+t}$ be the $t=\mathrm{rank}(\mathbf{A})-\rho$ vectors it produces, in that order, and let $\bar{\mathbf{u}}_{\rho+t+1},\dots,\bar{\mathbf{u}}_m$ be an arbitrary orthonormal basis of $\mathcal{C}(\mathbf{A})^{\perp}$. Then $\mathbf{T}=\bmat{\bar{\mathbf{S}} & \bar{\mathbf{U}}^*\widetilde{\mathbf{W}}}$ is upper triangular, for every unitary $\mathbf{Q}_r$ and every $k$, with no rank hypothesis. Moreover, for an \emph{arbitrary} full SVD $\mathbf{F}=\bar{\mathbf{U}}\bar{\mathbf{S}}\bar{\mathbf{V}}^*$ --- not necessarily the one constructed above --- $\mathbf{T}$ is upper triangular if and only if $(\bar{\mathbf{U}}^*\widetilde{\mathbf{W}})_{i,l}=0$ whenever $i>k+l$; this holds for \emph{every} such $\bar{\mathbf{U}}$ whenever $m\le k+1$ or $\mathcal{C}(\widetilde{\mathbf{W}})\subseteq\mathcal{C}(\mathbf{F})$.
\end{proposition}

\begin{proof}
The left block $\bar{\mathbf{S}}$ is rectangular diagonal and contributes nothing below the diagonal, so $\mathbf{T}$ is upper triangular exactly when $(\bar{\mathbf{U}}^*\widetilde{\mathbf{W}})_{i,l}=0$ for $i>k+l$, i.e.\ exactly when $\widetilde{\mathbf{W}}\mathbf{e}_l\in\mathrm{span}(\bar{\mathbf{u}}_1,\dots,\bar{\mathbf{u}}_{k+l})$ for every $l$. The stated completion adds at most one new vector per column of $\widetilde{\mathbf{W}}$, and a discarded column adds none, so $\widetilde{\mathbf{W}}\mathbf{e}_l\in\mathrm{span}(\bar{\mathbf{u}}_1,\dots,\bar{\mathbf{u}}_{\rho+l})\subseteq\mathrm{span}(\bar{\mathbf{u}}_1,\dots,\bar{\mathbf{u}}_{k+l})$ since $\rho\le k$; the arbitrary tail is never used. The two sufficient conditions are immediate: if $m\le k+1$ the index set $\{(i,k+l):i>k+l,\ i\le m\}$ is empty; if $\mathcal{C}(\widetilde{\mathbf{W}})\subseteq\mathcal{C}(\mathbf{F})$ then every $\widetilde{\mathbf{W}}\mathbf{e}_l$ already lies in $\mathrm{span}(\bar{\mathbf{u}}_1,\dots,\bar{\mathbf{u}}_\rho)$.
\end{proof}

\begin{remark}\label{rem:utv-tail}
The arbitrary tail of the construction is needed whenever $\mathrm{rank}(\mathbf{A})<m$, since $[\mathbf{F}\ \widetilde{\mathbf{W}}]=\mathbf{A}\mathbf{Q}_r$ and Gram--Schmidt cannot supply more than $\mathrm{rank}(\mathbf{A})-\rho$ vectors. The equivalence in the second half of the proposition says that the ordered completion is sufficient for triangularity but not necessary. Without such a rule the claim is false. Take
$$
\mathbf{A}=\bmat{0&0&3\\0&0&0\\5&0&0\\0&3&0},
$$
$k=1$ and $\mathbf{Q}_r=\mathbf{I}_3$, which is a legal Businger--Golub factor for $\mathbf{A}^*$ under one admissible tie-break. Then $\mathbf{F}=(0,0,5,0)^*$ has $\rho=1$, and $\bar{\mathbf{U}}=\bmat{\mathbf{e}_3&\mathbf{e}_1&\mathbf{e}_4&\mathbf{e}_2}$ gives a legal full SVD of $\mathbf{F}$: $\mathbf{U}$ and $\mathbf{V}$ are exactly unitary and $\mathbf{A}=\mathbf{U}\mathbf{T}\mathbf{V}^*$ holds identically, yet $T_{3,2}=3$ sits below the diagonal. Ordering the same two vectors the other way, $\bar{\mathbf{U}}=\bmat{\mathbf{e}_3&\mathbf{e}_4&\mathbf{e}_1&\mathbf{e}_2}$, is what the completion rule prescribes, and it returns the upper triangular
$$
\mathbf{T}=\bmat{5&0&0\\0&3&0\\0&0&3\\0&0&0}.
$$
The completion rule of Proposition~\ref{prop:UTV-completion} is one way to remove this freedom. Practical algorithms remove it blockwise instead. The randomized UTV algorithm of Martinsson et al.~\cite{martinsson2019randutv} takes a full SVD $\mathbf{A}\mathbf{V}_1=\mathbf{U}\mathbf{D}\mathbf{W}^*$ of the leading block of $b$ columns, forms $\bmat{\mathbf{D} & \mathbf{U}^*\mathbf{A}\mathbf{V}_2}$, and recurses on the trailing block, driving $b$ columns at a time to upper triangular form.
\end{remark}

Note also that if $\widetilde{\mathbf{W}}=\mathbf{0}$ then $\mathbf{T}=\bmat{\bar{\mathbf{S}} & \mathbf{0}}$ is rectangular diagonal and the output is a full SVD of $\mathbf{A}$. The converse fails: $\widetilde{\mathbf{W}}\ne\mathbf{0}$ is necessary but not sufficient for a non-diagonal $\mathbf{T}$, because $\mathbf{T}$ is diagonal exactly when $\bar{\mathbf{U}}^*\widetilde{\mathbf{W}}$ is supported on the positions $(k+l,l)$. For $\mathbf{A}=\mathbf{I}_3$, $k=1$ and $\mathbf{Q}_r=\mathbf{I}_3$ one has $\widetilde{\mathbf{W}}=\bmat{\mathbf{e}_2&\mathbf{e}_3}\ne\mathbf{0}$, the completion rule gives $\bar{\mathbf{U}}=\mathbf{I}_3$, and $\mathbf{T}=\mathbf{I}_3$. Nothing here establishes rank revelation or updatability, which are what the UTV family exists for.

We can now turn to the more general case of factorizations with two-sided projectors, built on the two column-pivoted QR decompositions introduced at the start of this section. Selecting the leading $k$ columns of $\mathbf{Q}_c$ to span the column space and the leading $k$ columns of $\mathbf{Q}_r$ to span the row space of $\mathbf{A}$, we have
\begin{align}
\mathbf{F} = \mathbf{Q}_c(:,1{:}k) \quad\text{and}\quad
\mathbf{H} = \mathbf{Q}_r(:,1{:}k),
\end{align}
and solving the projector equation, we get
\begin{align}
\mathbf{Y}^* = \mathbf{Q}_c(:,1{:}k)^*
\quad\text{and}\quad
\mathbf{X} = \mathbf{Q}_r(:,1{:}k).
\end{align}
From that we immediately obtain the mixing matrix and the two-sided compression,
\begin{align}\label{eq:CPQR-family}
\begin{aligned}
\mathbf{G} &= \mathbf{Q}_c(:,1{:}k)^*\mathbf{A}\mathbf{Q}_r(:,1{:}k),
\\
\mathbf{P}_c\mathbf{A}\mathbf{P}_r &= \mathbf{Q}_c(:,1{:}k)\mathbf{Q}_c(:,1{:}k)^*\mathbf{A}\mathbf{Q}_r(:,1{:}k)\mathbf{Q}_r(:,1{:}k)^*,
\end{aligned}
\end{align}
where $\mathbf{P}_c=\mathbf{F}\mathbf{F}^*$ and $\mathbf{P}_r=\mathbf{H}\mathbf{H}^*$ are the orthogonal projectors onto $\mathcal{C}(\mathbf{F})$ and $\mathcal{C}(\mathbf{H})$. It is convenient to separate what is unconditional here from what is not, because the three designs below differ only in the factorization applied to $\mathbf{G}$.

\begin{lemma}\label{lem:compression}
Let $\mathbf{A}\in\mathbb{F}^{m\times n}$, let $\mathbf{Q}_c\in\mathbb{F}^{m\times m}$ and $\mathbf{Q}_r\in\mathbb{F}^{n\times n}$ be any unitary matrices, let $1\le k\le\min(m,n)$, and set $\mathbf{F}=\mathbf{Q}_c(:,1{:}k)$, $\mathbf{H}=\mathbf{Q}_r(:,1{:}k)$, $\mathbf{G}=\mathbf{F}^*\mathbf{A}\mathbf{H}$, $\mathbf{P}_c=\mathbf{F}\mathbf{F}^*$ and $\mathbf{P}_r=\mathbf{H}\mathbf{H}^*$. For any $\mathbf{M}_1,\mathbf{M}_2,\mathbf{M}_3\in\mathbb{F}^{k\times k}$ with $\mathbf{M}_1\mathbf{M}_2\mathbf{M}_3=\mathbf{G}$, put $\mathbf{U}=\mathbf{F}\mathbf{M}_1$, $\mathbf{T}=\mathbf{M}_2$ and $\mathbf{V}=\mathbf{H}\mathbf{M}_3^*$. Then
\begin{align}\label{eq:compression}
\mathbf{U}\mathbf{T}\mathbf{V}^* = \mathbf{P}_c\mathbf{A}\mathbf{P}_r
\qquad\text{identically,}
\end{align}
and consequently $\mathbf{A}=\mathbf{U}\mathbf{T}\mathbf{V}^*$ if and only if $\mathcal{C}(\mathbf{A})\subseteq\mathcal{C}(\mathbf{F})$ and $\mathcal{C}(\mathbf{A}^*)\subseteq\mathcal{C}(\mathbf{H})$. In particular, if $k<\mathrm{rank}(\mathbf{A})$ then $\mathbf{A}\ne\mathbf{U}\mathbf{T}\mathbf{V}^*$ for every $\mathbf{Q}_c$, $\mathbf{Q}_r$ and every factorization of $\mathbf{G}$.
\end{lemma}

\begin{proof}
Substituting the three factors and using $\mathbf{M}_1\mathbf{M}_2\mathbf{M}_3=\mathbf{G}$ gives $\mathbf{U}\mathbf{T}\mathbf{V}^*=\mathbf{F}\mathbf{M}_1\mathbf{M}_2\mathbf{M}_3\mathbf{H}^*=\mathbf{F}\mathbf{G}\mathbf{H}^*=\mathbf{F}\mathbf{F}^*\mathbf{A}\mathbf{H}\mathbf{H}^*$. Since an idempotent fixes its own range pointwise, the two span conditions give $\mathbf{P}_c\mathbf{A}\mathbf{P}_r=\mathbf{A}$; conversely, if $\mathbf{P}_c\mathbf{A}\mathbf{P}_r=\mathbf{A}$ then $\mathcal{C}(\mathbf{A})=\mathcal{C}(\mathbf{P}_c\mathbf{A}\mathbf{P}_r)\subseteq\mathcal{C}(\mathbf{P}_c)=\mathcal{C}(\mathbf{F})$ and, taking adjoints, $\mathcal{C}(\mathbf{A}^*)\subseteq\mathcal{C}(\mathbf{H})$. The rank bound $\mathrm{rank}(\mathbf{U}\mathbf{T}\mathbf{V}^*)\le k$ is immediate from $\mathbf{U}\in\mathbb{F}^{m\times k}$.
\end{proof}

Lemma~\ref{lem:compression} is bookkeeping rather than a new result --- two-sided orthonormal compression followed by a factorization of the small core is the defining pattern of randomized low-rank algorithms \cite{martinsson2020randomized,martinsson2019randutv} --- but stating it once, unconditionally, is shorter and safer than asserting an unconditional ``$=$'' three times below. With rank-revealing pivoting on both sides the two span conditions coincide and both read $k\ge\mathrm{rank}(\mathbf{A})$.

The matched-pair condition (\ref{eq:svd bounds}) need not hold for these bases, so the mixing matrix $\mathbf{G}$ is in general no longer diagonal. It is not excluded either: should the leading pivoted columns happen to form a matched pair, the $(\Leftarrow)$ half of Proposition~\ref{prop:SVD-gauge} --- which consumes only the orthonormality of the columns of $\mathbf{Q}_c$ --- returns a diagonal $\mathbf{G}$ again. So it does for $\mathbf{A}=\mathrm{diag}(3,2,1)$, where greedy pivoting leaves $\mathbf{Q}_c=\mathbf{Q}_r=\mathbf{I}_3$ and $\mathbf{G}=\mathrm{diag}(3,2)$ at $k=2$. Diagonal or not, $\mathbf{G}$ stores the whole of $\mathbf{P}_c\mathbf{A}\mathbf{P}_r$ in the system of coordinates given by $\mathbf{F}$ and $\mathbf{H}$, and hence, under the two span conditions of Lemma~\ref{lem:compression}, the whole of $\mathbf{A}$. This is the observation we can exploit. Since the identity \eqref{eq:compression} is inherited from the projectors while the \emph{structure} of $\mathbf{T}$ is inherited from the factorization of $\mathbf{G}$, we may work with thin factors $\mathbf{U}$ and $\mathbf{V}$ having orthonormal columns rather than with square unitary ones. The result is the design of UTV-like factorizations.

The key step is factoring the mixing matrix. Consider again performing the SVD of $\mathbf{G}$:
\begin{align}
\mathbf{G} = \mathbf{Q}_c(:,1{:}k)^*\mathbf{A}\mathbf{Q}_r(:,1{:}k) = \bar{\mathbf{U}}\bar{\mathbf{S}}\bar{\mathbf{V}}^*.
\end{align}
Introducing that result into (\ref{eq:UTV}) defines the following factorization, displayed under the column-span condition $\mathcal{C}(\mathbf{A})\subseteq\mathcal{C}(\mathbf{Q}_c(:,1{:}k))$ and exactly there. Since $\bar{\mathbf{U}}\bar{\mathbf{U}}^*=\mathbf{I}_k$, the product $\mathbf{U}\mathbf{T}\mathbf{V}^*$ below equals $\mathbf{P}_c\mathbf{A}$ identically, so without that condition it returns $\mathbf{P}_c\mathbf{A}$, not $\mathbf{A}$. The row-span condition is not consumed, the trailing block of $\mathbf{Q}_r$ being retained on the right:
\begin{align}\label{eq:UTV1}
\begin{aligned}
\mathbf{A} &=
\mathbf{Q}_c(:,1{:}k)\bar{\mathbf{U}}
\bmat{\bar{\mathbf{S}} & \bar{\mathbf{U}}^*\mathbf{Q}_c(:,1{:}k)^*\mathbf{A}\mathbf{Q}_r(:,(k{+}1){:}n)}
\begin{bmatrix}
\bar{\mathbf{V}}^* & \mathbf{0}
\\
\mathbf{0} & \mathbf{I}_{n-k}
\end{bmatrix}\mathbf{Q}_r^*
= \mathbf{U}\mathbf{T}\mathbf{V}^*,
\\
\mathbf{U} &= \mathbf{Q}_c(:,1{:}k)\bar{\mathbf{U}},
\\
\mathbf{T} &=
\bmat{\bar{\mathbf{S}} & \bar{\mathbf{U}}^*\mathbf{Q}_c(:,1{:}k)^*\mathbf{A}\mathbf{Q}_r(:,(k{+}1){:}n)},
\\
\mathbf{V} &= \mathbf{Q}_r\begin{bmatrix}
\bar{\mathbf{V}} & \mathbf{0}
\\
\mathbf{0} & \mathbf{I}_{n-k}
\end{bmatrix}.
\end{aligned}
\end{align}
One caveat belongs here rather than in a footnote. Whenever both span conditions of Lemma~\ref{lem:compression} hold, the residual block $\bar{\mathbf{U}}^*\mathbf{Q}_c(:,1{:}k)^*\mathbf{A}\mathbf{Q}_r(:,(k{+}1){:}n)$ vanishes, $\mathbf{T}$ is diagonal, and \eqref{eq:UTV1} is the thin SVD of $\mathbf{A}$. The SVD member of the family is therefore a validation of the framework rather than a new factorization. The CPQR and LU members below are the ones that produce a genuinely triangular $\mathbf{T}$.
The following MATLAB code illustrates the design:
\begin{lstlisting}
    F       = Qc(:,1:k);        B       = F;
    H       = Qr(:,1:k);        D       = H;
    [QY,RY] = qr(B'*F,0);       [QX,RX] = qr(H'*D,0);
    Y       = (RY\(QY'*B'))';   X       = (D/RX)*QX';
    G       = Y'*A*X;
    [Ubar,Sbar,Vbar] = svd(G);
    U       = Y*Ubar;
    T       = [Sbar,Ubar'*Y'*A*Qr(:,k+1:end)];
    V       = Qr*[Vbar,zeros(k,n-k);zeros(k,n-k)',eye(n-k)];
    Ar      = U*T*V';
\end{lstlisting}
For $k<m$, $\mathbf{U}$ has orthonormal columns but is not square unitary. Since $\bar{\mathbf{U}}$ is unitary,
\begin{align}
\mathbf{U}\mathbf{U}^* = \mathbf{Q}_c(:,1{:}k)\bar{\mathbf{U}}\bar{\mathbf{U}}^*\mathbf{Q}_c(:,1{:}k)^* = \mathbf{Q}_c(:,1{:}k)\mathbf{Q}_c(:,1{:}k)^*,
\end{align}
which is the orthogonal projector onto $\mathcal{C}(\mathbf{Q}_c(:,1{:}k))$ and equals $\mathbf{I}_m$ only when $k=m$. In every case $\mathbf{U}^*\mathbf{U}=\mathbf{I}_k$, and $\mathbf{U}$ is square unitary precisely when $k=m$. Padding $\mathbf{T}$ with $m-k$ zero rows and completing $\mathbf{U}$ to a unitary matrix recovers a genuine square-unitary UTV, so no new class of matrices is reached by dropping squareness.

By following the same technique of meta-factorization, we can design another UTV-like decomposition. Instead of the SVD, performing the CPQR decomposition of the mixing matrix yields
\begin{align}
\begin{aligned}
\mathbf{G} = \bar{\mathbf{Q}}\bar{\mathbf{R}}\bar{\mathbf{\Pi}}^*.
\end{aligned}
\end{align}
The desired factorization can now be computed as, again under the two span conditions of Lemma~\ref{lem:compression},
\begin{align}
\begin{aligned}
\mathbf{A} &= \mathbf{Q}_c(:,1{:}k)\bar{\mathbf{Q}}\bar{\mathbf{R}}\bar{\mathbf{\Pi}}^*\mathbf{Q}_r(:,1{:}k)^* = \mathbf{U}\mathbf{T}\mathbf{V}^*,
\\
\mathbf{U} &= \mathbf{Q}_c(:,1{:}k)\bar{\mathbf{Q}},
\\
\mathbf{T} &= \bar{\mathbf{R}},
\\
\mathbf{V} &= \mathbf{Q}_r(:,1{:}k)\bar{\mathbf{\Pi}}.
\end{aligned}
\end{align}
In this case $\mathbf{U}$ and $\mathbf{V}$ have orthonormal columns, $\mathbf{U}^*\mathbf{U}=\mathbf{V}^*\mathbf{V}=\mathbf{I}_k$, and $\mathbf{T}=\bar{\mathbf{R}}$ is upper triangular. The pivoting choice inside the CPQR of $\mathbf{G}$ contributes only that triangularity: any $\bar{\mathbf{\Pi}}$, including $\bar{\mathbf{\Pi}}=\mathbf{I}_k$, gives the same identity. When $k>\mathrm{rank}(\mathbf{A})$ the construction is still exact and $\mathbf{U},\mathbf{V}$ still have full column rank, but $\mathbf{T}$ is singular and the factorization has rank $\mathrm{rank}(\mathbf{A})<k$. The MATLAB code illustrates the design:
\begin{lstlisting}
    F       = Qc(:,1:k);      B       = F;
    H       = Qr(:,1:k);      D       = H;
    [QY,RY] = qr(B'*F,0);     [QX,RX] = qr(H'*D,0);
    Y       = (RY\(QY'*B'))'; X       = (D/RX)*QX';
    G       = Y'*A*X;
    [Qbar,Rbar,Pbar] = qr(G);
    U       = Qc(:,1:k)*Qbar;
    T       = Rbar;
    V       = Qr(:,1:k)*Pbar;
    Ar      = U*T*V';
\end{lstlisting}

Similarly, to obtain the decomposition with the lower-triangular mixing matrix $\mathbf{T}$, consider the permuted LU decomposition of the mixing matrix. This gives rise to the following equality,
\begin{align}
\begin{aligned}
\mathbf{G} = \tilde{\mathbf{\Pi}}^*\tilde{\mathbf{L}}\tilde{\mathbf{U}}.
\end{aligned}
\end{align}
By following the same procedure as before, and under the same two span conditions $\mathcal{C}(\mathbf{A})\subseteq\mathcal{C}(\mathbf{Q}_c(:,1{:}k))$ and $\mathcal{C}(\mathbf{A}^*)\subseteq\mathcal{C}(\mathbf{Q}_r(:,1{:}k))$, we obtain
\begin{align}
\begin{aligned}
\mathbf{A} &= \mathbf{Q}_c(:,1{:}k)\tilde{\mathbf{\Pi}}^*\tilde{\mathbf{L}}\tilde{\mathbf{U}}\mathbf{Q}_r(:,1{:}k)^* = \mathbf{U}\mathbf{T}\mathbf{V}^* ,
\\
\mathbf{U} &= \mathbf{Q}_c(:,1{:}k)\tilde{\mathbf{\Pi}}^*,
\\
\mathbf{T} &= \tilde{\mathbf{L}},
\\
\mathbf{V} &= \mathbf{Q}_r(:,1{:}k)\tilde{\mathbf{U}}^*.
\end{aligned}
\end{align}
As before, the MATLAB code illustrates the design:
\begin{lstlisting}
    F       = Qc(:,1:k);      B       = F;
    H       = Qr(:,1:k);      D       = H;
    [QY,RY] = qr(B'*F,0);     [QX,RX] = qr(H'*D,0);
    Y       = (RY\(QY'*B'))'; X       = (D/RX)*QX';
    G       = Y'*A*X;
    [Ltilde,Utilde,Ptilde] = lu(G);
    U       = Qc(:,1:k)*Ptilde';
    T       = Ltilde;
    V       = Qr(:,1:k)*Utilde';
    Ar      = U*T*V';
\end{lstlisting}
In that case, $\mathbf{T}=\tilde{\mathbf{L}}$ is \emph{unit} lower triangular by design, so $\det(\mathbf{T})=1$. This fixes the product of its singular values, not their individual values. Since $\mathbf{V}$ generally lacks orthonormal columns, the singular values of $\mathbf{T}$ need not equal those of $\mathbf{A}$. The left factor $\mathbf{U}$ always has orthonormal columns, while $\mathbf{V}^*\mathbf{V}=\mathbf{I}_k$ holds if and only if $\tilde{\mathbf{U}}$ is diagonal with every diagonal entry of modulus one. This member illustrates the design pattern. When $\mathbf{G}$ is singular, $\mathbf{V}$ has rank below $k$ and no $\mathbf{X}$ solves $\mathbf{V}^*\mathbf{X}=\mathbf{I}_k$. The product still reconstructs $\mathbf{A}$ under the span conditions above, but its factors cannot satisfy the normalized projector equation.

Meta-factorization reconstructs the algebraic form of these factorizations, exposes their internal structure, and admits modifications. Two limits of that claim should be stated explicitly. First, what is reproduced is the algebraic identity. The rank-revealing bounds that make the CPQR and UTV families useful are properties of the pivoting rule and of the algorithm \cite{chan1987rank,gu1996efficient,martinsson2019randutv}, and are not derived here. Second, the ambition of generating several classical factorizations from one algebraic device is not new: Chu, Funderlic and Golub \cite{chu1995rank} generate Gram--Schmidt, LU, QR, the SVD and more from a single rank-one reduction formula. The introduction explains how the present direction differs. Corollary~\ref{thm:wedderburn} starts from prescribed factors and gives both a direct form of the classical rank-reduction test and sampling matrices that realize it.

\section{Additional benefits of meta-factorization}
\label{section:Additional...}

Three further constructions show what the design freedom is for. We relate the two equations to the theory of generalized matrix inverses, reconstruct the generalized Nystr\"{o}m method for computing randomized low-rank matrix approximations, and exhibit the hidden projectors of the CUR decomposition together with the exact condition under which it reproduces $\mathbf{A}$. The last of these is a known characterization, due to Hamm and Huang \cite{hamm2020perspectives}. What the framework contributes is a short derivation of it and the two-design-matrix family of Corollary~\ref{cor:CUR-family} in which their two identities appear as two choices of the design parameters $(\mathbf{B},\mathbf{D})$.

\subsection{Pseudoinverse and its formulas}
\label{subsection:pseudoinverse}

The reconstruction equation also describes reflexive generalized inverses. Specialize to $\mathbf{F}=\mathbf{H}^*=\mathbf{A}$ and form $\mathbf{G}=\mathbf{Y}^*\mathbf{A}\mathbf{X}$ from arbitrary $\{1\}$-inverses $\mathbf{Y}^*$ and $\mathbf{X}$. The resulting cores are exactly the $\{1,2\}$-inverses of $\mathbf{A}$. For a nonzero singular square matrix this class has infinitely many members, including $\mathbf{A}^{+}$. For $\mathbf{A}=\mathbf{0}$, reflexivity forces $\mathbf{G}=\mathbf{0}$. The second part of the subsection re-expresses classical explicit formulas for the Moore--Penrose pseudoinverse in the same notation.

The Moore--Penrose pseudoinverse satisfies the two equations
\begin{align}
\begin{aligned}
\mathbf{A}\mathbf{A}^{\!+}\mathbf{A} =\mathbf{A}
\quad\text{and}\quad
\mathbf{A}^{\!+}\mathbf{A}\mathbf{A}^{\!+}\! = \mathbf{A}^{\!+}.
\end{aligned}
\end{align}
Take
\begin{align}
\mathbf{F} = \mathbf{A} \quad\text{and}\quad
\mathbf{H}^* = \mathbf{A}.
\end{align}
The signature of \eqref{eq:genfactorization} forces $\mathbf{A}$ to be square: $\mathbf{F}\in\mathbb{F}^{m\times k}$ with $\mathbf{F}=\mathbf{A}$ gives $k=n$, while $\mathbf{H}^*\in\mathbb{F}^{k\times n}$ with $\mathbf{H}^*=\mathbf{A}$ gives $k=m$, so $m=n=k$, the dimensions being incompatible otherwise. Assume then $\mathbf{A}\in\mathbb{F}^{n\times n}$ and $k=n$, and let $\mathbf{Y}^*$ and $\mathbf{X}$ range over the matrices supplied by the generalized inverse equations \eqref{eq:YX-PenroseEQs}, that is, over the $\{1\}$-inverses of $\mathbf{A}$, written $\mathbf{A}\{1\}$. Then
\begin{align}\label{eq:12-class}
\bigl\{ \mathbf{G}=\mathbf{Y}^*\mathbf{A}\mathbf{X}\ |\ \mathbf{Y}^*,\mathbf{X}\in\mathbf{A}\{1\} \bigr\}
=
\bigl\{ \mathbf{Z}\ |\ \mathbf{A}\mathbf{Z}\mathbf{A}=\mathbf{A},\ \mathbf{Z}\mathbf{A}\mathbf{Z}=\mathbf{Z} \bigr\},
\end{align}
the set of reflexive $\{1,2\}$-inverses of $\mathbf{A}$. Indeed, if $\mathbf{Y}^*$ and $\mathbf{X}$ are $\{1\}$-inverses then $\mathbf{A}\mathbf{G}\mathbf{A}=(\mathbf{A}\mathbf{Y}^*\mathbf{A})\mathbf{X}\mathbf{A}=\mathbf{A}$ and $\mathbf{G}\mathbf{A}\mathbf{G}=\mathbf{Y}^*(\mathbf{A}\mathbf{X}\mathbf{A})\mathbf{Y}^*\mathbf{A}\mathbf{X}=\mathbf{Y}^*\mathbf{A}\mathbf{X}=\mathbf{G}$. Conversely, a reflexive inverse $\mathbf{Z}$ is recovered by $\mathbf{Y}^*=\mathbf{X}=\mathbf{Z}$. In particular the reconstruction equation holds throughout,
\begin{align}
\mathbf{A} = \mathbf{A}\mathbf{G}\mathbf{A},
\end{align}
and $\mathbf{A}^{\!+}$ is one member of \eqref{eq:12-class}. It is not distinguished by the meta-factorization equations: the two symmetry conditions $(\mathbf{A}\mathbf{Z})^*=\mathbf{A}\mathbf{Z}$ and $(\mathbf{Z}\mathbf{A})^*=\mathbf{Z}\mathbf{A}$ that single it out are not encoded by them. For $\mathbf{A}=\mathrm{diag}(1,0)$ the choice $\mathbf{Y}^*=\bigl(\begin{smallmatrix}1&0\\1&0\end{smallmatrix}\bigr)$, $\mathbf{X}=\bigl(\begin{smallmatrix}1&1\\0&0\end{smallmatrix}\bigr)$ gives $\mathbf{G}=\bigl(\begin{smallmatrix}1&1\\1&1\end{smallmatrix}\bigr)$, which satisfies the first two Penrose equations exactly, the last two not at all, and differs from $\mathbf{A}^{\!+}$.

The \emph{projector} equation does not transfer to the choice $\mathbf{F}=\mathbf{H}^{*}=\mathbf{A}$, and that is what delimits the correspondence. Imposing $\mathbf{Y}^*\mathbf{F}=\mathbf{H}^*\mathbf{X}=\mathbf{I}_n$ here forces $\mathbf{A}^{\!+}\mathbf{A}=\mathbf{A}\mathbf{A}^{\!+}=\mathbf{I}_n$, i.e.\ $\mathbf{A}$ nonsingular, in which case $\mathbf{A}^{\!+}=\mathbf{A}^{-1}$ and the whole construction degenerates to $\mathbf{A}=\mathbf{A}\mathbf{A}^{-1}\mathbf{A}$. It is the reconstruction equation alone that carries over to the Penrose setting. The $\{1\}$-, $\{1,2\}$-, $\{1,3\}$- and $\{1,4\}$-inverse classes used here specify which of Penrose's four equations \cite{penrose1955generalized} are imposed. See Ben-Israel and Greville \cite{ben2003generalized} for this standard terminology and the corresponding parametrizations.
We can now use meta-factorization to describe the internal structure of the pseudoinverse. One term, used throughout this subsection, needs fixing first.

\begin{definition}\label{def:full-rank-factorization}
A \emph{full-rank factorization} of $\mathbf{A}\in\mathbb{F}^{m\times n}$ with $r=\mathrm{rank}(\mathbf{A})>0$ is a pair $(\mathbf{K},\mathbf{L})$ with $\mathbf{K}\in\mathbb{F}^{m\times r}$, $\mathbf{L}\in\mathbb{F}^{r\times n}$, $\mathbf{A}=\mathbf{K}\mathbf{L}$ and $\mathrm{rank}(\mathbf{K})=\mathrm{rank}(\mathbf{L})=r$; that is, the inner dimension equals the rank of $\mathbf{A}$ and both factors attain it.
\end{definition}

Such factorizations exist for every nonzero matrix; see Piziak and Odell \cite{piziak1999fullrank}. The inner dimension is not free: if it exceeds $r$ then the factors cannot both have full rank, and the formulas below, which invert $\mathbf{C}^*\mathbf{C}$, $\mathbf{R}\mathbf{R}^*$ and $\mathbf{B}_{\mathrm{MD}}^*\mathbf{A}\mathbf{D}_{\mathrm{MD}}^*$, fail.

To illustrate the general idea, we use the full-rank CR construction described by Piziak and Odell \cite[Section~3, p.~194]{piziak1999fullrank},
\begin{align}
\mathbf{A} = \mathbf{C}\mathbf{R}.
\end{align}
When $\mathbf{A}$ has rank $r$, it can be factored into the product of the $m\times r$ matrix $\mathbf{C}$, containing the first $r$ independent columns of $\mathbf{A}$, and the $r\times n$ matrix $\mathbf{R}$, containing the $r$ nonzero rows of the row-reduced echelon form of $\mathbf{A}$. Since both $\mathbf{C}$ and $\mathbf{R}$ have full rank $r$, the pair $(\mathbf{C},\mathbf{R})$ is a full-rank factorization and the pseudoinverse of $\mathbf{A}$ is
\begin{align}
\mathbf{A}^{\!+}= \mathbf{R}^{+}\mathbf{C}^{+} = \mathbf{R}^*
(\mathbf{R}\mathbf{R}^*)^{-1}(\mathbf{C}^*\mathbf{C})^{-1}
\mathbf{C}^*.
\end{align}
However, we also have
\begin{align}\label{eq:CR-gram}
(\mathbf{R}\mathbf{R}^*)^{-1}(\mathbf{C}^*\mathbf{C})^{-1} =
(\mathbf{C}^*\mathbf{C}\mathbf{R}\mathbf{R}^*)^{-1} =
(\mathbf{C}^*\mathbf{A}\mathbf{R}^*)^{-1},
\end{align}
where the middle equality is the Gram factorization $\mathbf{C}^*\mathbf{A}\mathbf{R}^*=(\mathbf{C}^*\mathbf{C})(\mathbf{R}\mathbf{R}^*)$, both factors being nonsingular because the factorization is of full rank.
Therefore, we obtain the following formula for the pseudoinverse:
\begin{align}\label{eq:strangPI}
\mathbf{A}^{\!+}= \mathbf{R}^*(\mathbf{C}^*\mathbf{A}\mathbf{R}^*)^{-1}\mathbf{C}^*.
\end{align}

The CR decomposition above is a full-rank factorization, and so is $\mathbf{A}=\mathbf{B}_{\mathrm{MD}}\mathbf{D}_{\mathrm{MD}}$ in Theorem~\ref{thm:MacDuffee}. We can now re-express formula \eqref{eq:strangPI} as a meta-factorization, thereby revealing the projectors hidden in it. Selecting
\begin{align}
\mathbf{F} = \mathbf{R}^* \quad\text{and}\quad
\mathbf{H} = \mathbf{C},
\end{align}
and solving the projector equation for $\mathbf{Y}$ and $\mathbf{X}$, yields
\begin{align}
\mathbf{Y}^* = (\mathbf{R}\mathbf{R}^*)^{-1}\mathbf{R} \quad\text{and}\quad
\mathbf{X} = \mathbf{C}(\mathbf{C}^*\mathbf{C})^{-1}.
\end{align}
This brings us to the decomposition
\begin{align}
\begin{aligned}
\mathbf{G} =
(\mathbf{R}\mathbf{R}^*)^{-1}\mathbf{R}
\mathbf{A}^{+}
\mathbf{C}(\mathbf{C}^*\mathbf{C})^{-1}
\quad\text{and}\quad
\mathbf{A}^{\!+}=
\mathbf{R}^*(\mathbf{R}\mathbf{R}^*)^{-1}\mathbf{R}
\mathbf{A}^{+}
\mathbf{C}(\mathbf{C}^*\mathbf{C})^{-1}\mathbf{C}^*.
\end{aligned}
\end{align}
To show that this is equivalent to (\ref{eq:strangPI}) we use the full-rank case of Greville's reverse-order law \cite{greville1966note}, given explicitly by Piziak and Odell \cite[Theorem~1]{piziak1999fullrank}. It applies because $\mathbf{A}=\mathbf{C}\mathbf{R}$ is a full-rank factorization in the sense of Definition~\ref{def:full-rank-factorization}:
\begin{align}
\mathbf{A}^{\!+} = \mathbf{R}^+\mathbf{C}^+.
\end{align}
Substituting and cancelling gives
\begin{align}
\begin{aligned}
\mathbf{A}^{\!+}&=
\mathbf{R}^*(\mathbf{R}\mathbf{R}^*)^{-1}\mathbf{R}
\mathbf{A}^{+}
\mathbf{C}(\mathbf{C}^*\mathbf{C})^{-1}\mathbf{C}^*
\\&=
\mathbf{R}^*(\mathbf{R}\mathbf{R}^*)^{-1}
\mathbf{R}
\mathbf{R}^{+}\mathbf{C}^{+}
\mathbf{C}
(\mathbf{C}^*\mathbf{C})^{-1}\mathbf{C}^*
\\&=
\mathbf{R}^*(\mathbf{R}\mathbf{R}^*)^{-1}
(\mathbf{C}^*\mathbf{C})^{-1}\mathbf{C}^*
\\&=
\mathbf{R}^*(\mathbf{C}^*\mathbf{A}\mathbf{R}^*)^{-1}\mathbf{C}^*,
\end{aligned}
\end{align}
as required. The second equality uses $\mathbf{R}\mathbf{R}^{+}=\mathbf{I}_r$ and $\mathbf{C}^{+}\mathbf{C}=\mathbf{I}_r$, which hold because $\mathbf{R}$ has full row rank and $\mathbf{C}$ full column rank. The third is \eqref{eq:CR-gram}. The matrix being meta-factorized in this passage is $\mathbf{A}^{\!+}$, not $\mathbf{A}$, and a mixing matrix for it is therefore given by
\begin{align}
\begin{aligned}
\mathbf{G} &= (\mathbf{C}^*\mathbf{A}\mathbf{R}^*)^{-1},
\end{aligned}
\end{align}
which is invertible for the reason already displayed in \eqref{eq:CR-gram}: its inverse $\mathbf{C}^*\mathbf{A}\mathbf{R}^*$ equals $(\mathbf{C}^*\mathbf{C})(\mathbf{R}\mathbf{R}^*)$, a product of two nonsingular Gram matrices.

The following classical result, presented by Ben-Israel and Greville in \cite{ben2003generalized}, shows that formula (\ref{eq:strangPI}) is a~special case of the general formula of MacDuffee. It is quoted here, not derived: the material of this subsection re-expresses known pseudoinverse formulas inside the framework, and the credit for them belongs to Penrose \cite{penrose1955generalized}, Greville \cite{greville1966note} and MacDuffee.

\begin{theorem}[MacDuffee, see Ben-Israel and Greville \cite{ben2003generalized}]\label{thm:MacDuffee}
If a complex $m\times n$ matrix $\mathbf{A}$ of rank $r>0$ has a full-rank factorization
\begin{align}
\begin{aligned}
\mathbf{A} =  \mathbf{B}_{\mathrm{MD}}\mathbf{D}_{\mathrm{MD}},
\qquad
\mathbf{B}_{\mathrm{MD}}\in\mathbb{C}^{m\times r},
\quad
\mathbf{D}_{\mathrm{MD}}\in\mathbb{C}^{r\times n},
\end{aligned}
\end{align}
then
\begin{align}\label{eq:macduffeePI}
\begin{aligned}
\mathbf{A}^{\!+} = \mathbf{D}_{\mathrm{MD}}^*(\mathbf{B}_{\mathrm{MD}}^*\mathbf{A}\mathbf{D}_{\mathrm{MD}}^*)^{-1}\mathbf{B}_{\mathrm{MD}}^*.
\end{aligned}
\end{align}
\end{theorem}

Let us re-express the formula by following the steps of meta-factorization. To keep the two roles apart, write $\mathbf{B}_{\mathrm{MD}}$ and $\mathbf{D}_{\mathrm{MD}}$ for the factors of the full-rank factorization $\mathbf{A}=\mathbf{B}_{\mathrm{MD}}\mathbf{D}_{\mathrm{MD}}$, reserving $\mathbf{B}$ and $\mathbf{D}$ for the design parameters of Theorem~\ref{Theorem:projector-eq-sol}. Meta-factorize $\mathbf{A}^{\!+}\in\mathbb{F}^{n\times m}$ with
\begin{align}
\mathbf{F} = \mathbf{D}_{\mathrm{MD}}^*
\quad\text{and}\quad
\mathbf{H}^* = \mathbf{B}_{\mathrm{MD}}^*,
\end{align}
both of rank $r$. Taking the default design parameters $\mathbf{B}=\mathbf{F}$ and $\mathbf{D}=\mathbf{H}$ in Theorem~\ref{Theorem:projector-eq-sol} gives the transforms
\begin{align}
\mathbf{Y}^*=(\mathbf{D}_{\mathrm{MD}}^*)^+
\quad\text{and}\quad
\mathbf{X}=(\mathbf{B}_{\mathrm{MD}}^*)^+,
\end{align}
for which the two hidden projectors $\mathbf{F}\mathbf{Y}^*=\mathbf{D}_{\mathrm{MD}}^*(\mathbf{D}_{\mathrm{MD}}^*)^{+}$ and $\mathbf{X}\mathbf{H}^*=(\mathbf{B}_{\mathrm{MD}}^*)^{+}\mathbf{B}_{\mathrm{MD}}^*$ are \emph{orthogonal}. Therefore, a mixing matrix for the pseudoinverse is
\begin{align}\label{eq:macduffee-G}
\mathbf{G} = (\mathbf{D}_{\mathrm{MD}}^*)^+\mathbf{A}^{\!+}(\mathbf{B}_{\mathrm{MD}}^*)^+ =
(\mathbf{B}_{\mathrm{MD}}^*\mathbf{A}\mathbf{D}_{\mathrm{MD}}^*)^+.
\end{align}
The matrix $\mathbf{B}_{\mathrm{MD}}^*\mathbf{A}\mathbf{D}_{\mathrm{MD}}^*$ is nonsingular, and this is where the full-rank hypothesis of Definition~\ref{def:full-rank-factorization} is consumed most visibly. It is also what makes \eqref{eq:CR-gram}, the cancellations $\mathbf{R}\mathbf{R}^{+}=\mathbf{I}_r$, $\mathbf{C}^{+}\mathbf{C}=\mathbf{I}_r$ and the second equality of \eqref{eq:macduffee-G} available above. Substituting $\mathbf{A}=\mathbf{B}_{\mathrm{MD}}\mathbf{D}_{\mathrm{MD}}$ and regrouping gives the Gram factorization
\begin{align}\label{eq:macduffee-gram}
\mathbf{B}_{\mathrm{MD}}^*\mathbf{A}\mathbf{D}_{\mathrm{MD}}^*
= (\mathbf{B}_{\mathrm{MD}}^*\mathbf{B}_{\mathrm{MD}})(\mathbf{D}_{\mathrm{MD}}\mathbf{D}_{\mathrm{MD}}^*),
\end{align}
a product of two $r\times r$ Gram matrices, each Hermitian positive definite because $\mathbf{B}_{\mathrm{MD}}$ has full column rank and $\mathbf{D}_{\mathrm{MD}}$ full row rank. Hence $\mathbf{B}_{\mathrm{MD}}^*\mathbf{A}\mathbf{D}_{\mathrm{MD}}^*$ is invertible and
\begin{align}
\mathbf{A}^{\!+} &= \mathbf{D}_{\mathrm{MD}}^*(\mathbf{B}_{\mathrm{MD}}^*\mathbf{A}\mathbf{D}_{\mathrm{MD}}^*)^{-1}\mathbf{B}_{\mathrm{MD}}^*,
\end{align}
giving MacDuffee's explicit formula. Two things about this passage should be said plainly. First, the matrix being meta-factorized here is $\mathbf{A}^{\!+}$, not $\mathbf{A}$. Second, the argument re-expresses a classical formula inside the framework rather than deriving it from first principles, since $\mathbf{A}^{\!+}$ enters as an input at \eqref{eq:macduffee-G}. A derivation that never takes $\mathbf{A}^{\!+}$ as an input is available and is short: verify the four Penrose equations directly for $\mathbf{Z}=\mathbf{D}_{\mathrm{MD}}^*(\mathbf{D}_{\mathrm{MD}}\mathbf{D}_{\mathrm{MD}}^*)^{-1}(\mathbf{B}_{\mathrm{MD}}^*\mathbf{B}_{\mathrm{MD}})^{-1}\mathbf{B}_{\mathrm{MD}}^*$, using \eqref{eq:macduffee-gram} and the fact that the inverse of a Hermitian matrix is Hermitian. The first two Penrose equations are cancellations that consume only the invertibility of the two Gram matrices, while the third and fourth reduce to the statement that $(\mathbf{B}_{\mathrm{MD}}^*\mathbf{B}_{\mathrm{MD}})^{-1}$ and $(\mathbf{D}_{\mathrm{MD}}\mathbf{D}_{\mathrm{MD}}^*)^{-1}$ are Hermitian. That derivation is not what the present section does. Finally, the mixing matrix depends on the choice of full-rank factorization: replacing $(\mathbf{B}_{\mathrm{MD}},\mathbf{D}_{\mathrm{MD}})$ by $(\mathbf{B}_{\mathrm{MD}}\mathbf{T},\mathbf{T}^{-1}\mathbf{D}_{\mathrm{MD}})$ with $\mathbf{T}$ nonsingular sends $\mathbf{G}$ to $\mathbf{T}^{*}\mathbf{G}\mathbf{T}^{-*}$, while $\mathbf{A}^{\!+}$ is unchanged. We therefore write ``a mixing matrix'', not ``the mixing matrix'', in this passage. This change of basis is the special case of factor replacement recovered in Section~\ref{section:Replacement}. In the factorization of $\mathbf{A}^{\!+}$, the simultaneous change is $\mathbf{F}\mapsto\mathbf{F}\mathbf{T}^{-*}$ and $\mathbf{H}\mapsto\mathbf{H}\mathbf{T}$. Applying the left and right replacement rules in succession gives the two interfaces $\mathbf{T}^{*}$ and $\mathbf{T}^{-*}$ and hence the displayed core $\mathbf{T}^{*}\mathbf{G}\mathbf{T}^{-*}$. Each replacement preserves its incumbent factor's space and therefore passes the criterion identically. At this minimal inner dimension, the accepted right factors have the row space of $\mathbf{A}^{\!+}$, equivalently the column space of $\mathbf{A}$, and are precisely changes of basis of $\mathbf{B}_{\mathrm{MD}}^*$. The corresponding left-factor statement concerns the column space of $\mathbf{A}^{\!+}$, equivalently the row space of $\mathbf{A}$.

Everything in this subsection is a re-expression of classical material --- Penrose's equations \cite{penrose1955generalized}, Greville's reverse-order law \cite{greville1966note}, MacDuffee's formula via Ben-Israel and Greville \cite{ben2003generalized}, and the full-rank factorization of Piziak and Odell \cite{piziak1999fullrank}. Its purpose is to show that these formulas are what the meta-factorization equations produce when the design parameters are chosen in the obvious way, not to add to them. Deeper results on generalized inverses of matrix products, including the cases where the reverse-order law fails, are developed elsewhere by the author with G.~Strang: the $\mathbf{A}=\mathbf{C}\mathbf{R}$ pseudoinverse question in \cite{karpowicz2023pseudoinverse}, and the universal formula together with a sketched formula whose exactness clause is a rank-preservation condition in \cite{karpowicz2026generalized}. Where those papers write $\mathbf{A}=\mathbf{C}\mathbf{R}$ for a full-rank factorization, the present notation is $\mathbf{A}=\mathbf{F}\mathbf{G}\mathbf{H}^*$ with $\mathbf{G}=\mathbf{I}_k$. The present paper owns the framework; those papers own the pseudoinverse theory. One kinship deserves a sentence. The rank-preservation clause there, that the sketches must preserve $\mathrm{rank}(\mathbf{A})$, and the exactness locus $\mathrm{rank}(\mathbf{\Omega}_r^*\mathbf{A}\mathbf{\Omega}_c)=\mathrm{rank}(\mathbf{A})$ of Proposition~\ref{prop:CUR-exactness} below are the same phenomenon met in two settings. The two constructions are relatives, and their exactness condition is the same rank-preservation requirement.

\subsection{Generalized Nystr\"{o}m's method, the CUR and the outer-product decomposition}
\label{subsection:NystroemCUR}

The generalized Nystr\"{o}m method and the CUR decomposition share a sketch-based construction: coordinate sketches select the original columns and rows used by CUR. For these fixed outer factors, two choices of the design parameters $\mathbf{B}$ and $\mathbf{D}$ produce the orthogonal-projector and intersection-core forms of CUR. This subsection derives both from the same three steps and relates them to the outer-product decomposition.

Nakatsukasa \cite{nakatsukasa2020fast} analyzed the generalized Nystr\"{o}m method, obtaining approximation guarantees and a fast, numerically stable implementation with oversampling and suitable stabilization. We first study its unregularized algebraic form with equal sketch widths. It approximates $\mathbf{A}$ by $\widehat{\mathbf{A}}$, a two-sided projection defined by two sampling (sketch) matrices, an $n\times k$ matrix $\mathbf{\Omega}_c$ and an $m\times k$ matrix $\mathbf{\Omega}_r$:
\begin{align}
\begin{aligned}
\widehat{\mathbf{A}} &= \mathbf{A}\mathbf{\Omega}_c(\mathbf{\Omega}_r^*\mathbf{A}\mathbf{\Omega}_c)^{+}\mathbf{\Omega}_r^*\mathbf{A}.
\end{aligned}
\end{align}
Two projectors are involved. Write $\mathbf{M}=\mathbf{\Omega}_r^*\mathbf{A}\mathbf{\Omega}_c$, the $k\times k$ matrix that governs both. Throughout this paragraph $\mathbf{\Omega}_c$ and $\mathbf{\Omega}_r$ are assumed to have full column rank $k$. Only the ``only if'' halves below use it, and each uses one half of it: the first needs $\mathrm{rank}(\mathbf{\Omega}_r)=k$ (indeed $\mathrm{rank}(\mathbf{\Omega}_r)=k$ or $\mathrm{rank}(\mathbf{A}\mathbf{\Omega}_c)=k$ is enough) and the second needs $\mathrm{rank}(\mathbf{\Omega}_c)=k$. The ``if'' halves need nothing, since $\mathrm{rank}(\mathbf{M})=k$ already forces both. Without the assumption the ``exactly when'' fails in the ``only if'' direction, since the two range-and-kernel identifications can hold although $\mathrm{rank}(\mathbf{M})<k$, as $\mathbf{A}=\mathbf{I}_3$, $\mathbf{\Omega}_c=\mathbf{\Omega}_r=[\mathbf{e}_1,\mathbf{e}_1]$ shows. The first projector,
\begin{align}
\begin{aligned}
\mathbf{P}_c &= \mathbf{A}\mathbf{\Omega}_c(\mathbf{\Omega}_r^*\mathbf{A}\mathbf{\Omega}_c)^{+}\mathbf{\Omega}_r^*,
\end{aligned}
\end{align}
is idempotent of rank $\mathrm{rank}(\mathbf{M})$ for every choice of sketches, and its range is contained in $\mathcal{C}(\mathbf{A}\mathbf{\Omega}_c)$. It is the generally oblique projection \emph{onto} $\mathcal{C}(\mathbf{A}\mathbf{\Omega}_c)$ along $\mathcal{N}(\mathbf{\Omega}_r^*)$ exactly when $\mathrm{rank}(\mathbf{M})=k$. This is Proposition~\ref{prop:oblique-projector} with $\mathbf{F}=\mathbf{A}\mathbf{\Omega}_c$ and $\mathbf{B}=\mathbf{\Omega}_r$, and the qualification is not idle: with $\mathbf{A}=\mathbf{I}_3$, $\mathbf{\Omega}_c=\mathbf{I}_3(:,(1,2))$ and $\mathbf{\Omega}_r=\mathbf{I}_3(:,(1,3))$ one has $\mathrm{rank}(\mathbf{A}\mathbf{\Omega}_c)=k=2$ while $\mathrm{rank}(\mathbf{M})=1$, so $\mathbf{P}_c$ projects onto a proper subspace of $\mathcal{C}(\mathbf{A}\mathbf{\Omega}_c)$. The second,
\begin{align}
\begin{aligned}
\mathbf{P}_r &= \mathbf{\Omega}_c(\mathbf{\Omega}_r^*\mathbf{A}\mathbf{\Omega}_c)^{+}\mathbf{\Omega}_r^*\mathbf{A},
\end{aligned}
\end{align}
is idempotent of the same rank $\mathrm{rank}(\mathbf{M})$, its range is contained in $\mathcal{C}(\mathbf{\Omega}_c)$ and its null space contains $\mathcal{N}(\mathbf{\Omega}_r^*\mathbf{A})$. It is the generally oblique projection \emph{onto} $\mathcal{C}(\mathbf{\Omega}_c)$ along $\mathcal{N}(\mathbf{\Omega}_r^*\mathbf{A})$, acting from the right, exactly when $\mathrm{rank}(\mathbf{M})=k$. The constructions of this subsection are stated with that hypothesis in force wherever the projector equation is used.

Neither projector reproduces $\mathbf{A}$ for every choice of sketches. Exactness holds on one locus only, and it pays to identify that locus now, before any specialization, because the same single condition will govern the CUR decomposition at the end of this section.

\begin{proposition}\label{prop:CUR-exactness}
Let $\mathbf{A}\in\mathbb{F}^{m\times n}$ with $r=\mathrm{rank}(\mathbf{A})$, let $\mathbf{\Omega}_c\in\mathbb{F}^{n\times k}$ and $\mathbf{\Omega}_r\in\mathbb{F}^{m\times k}$ be arbitrary, put $\mathbf{M}=\mathbf{\Omega}_r^*\mathbf{A}\mathbf{\Omega}_c$, and let $\mathbf{M}^{-}$ be any $\{1\}$-inverse of $\mathbf{M}$. Then
\begin{align}\label{eq:GN-exactness}
\mathbf{A} = \mathbf{A}\mathbf{\Omega}_c\,\mathbf{M}^{-}\,\mathbf{\Omega}_r^*\mathbf{A}
\qquad\text{if and only if}\qquad
\mathrm{rank}(\mathbf{M}) = \mathrm{rank}(\mathbf{A}) = r,
\end{align}
equivalently $\mathcal{C}(\mathbf{A}\mathbf{\Omega}_c)=\mathcal{C}(\mathbf{A})$ and $\mathcal{C}(\mathbf{A}^*\mathbf{\Omega}_r)=\mathcal{C}(\mathbf{A}^*)$.
\end{proposition}

\begin{proof}
$(\Leftarrow)$ Assume $\mathrm{rank}(\mathbf{M})=r$ and write $\mathbf{F}=\mathbf{A}\mathbf{\Omega}_c$. Since $\mathrm{rank}(\mathbf{M})\le\mathrm{rank}(\mathbf{F})\le r$, the hypothesis forces $\mathrm{rank}(\mathbf{F})=r$ and hence $\mathcal{C}(\mathbf{F})=\mathcal{C}(\mathbf{A})$; write $\mathbf{A}=\mathbf{F}\mathbf{S}$. From $\mathbf{M}=\mathbf{\Omega}_r^*\mathbf{F}$ we get $\mathcal{N}(\mathbf{F})\subseteq\mathcal{N}(\mathbf{M})$, and both null spaces have dimension $k-r$, so they are equal. The columns of $\mathbf{I}_k-\mathbf{M}^{-}\mathbf{M}$ lie in $\mathcal{N}(\mathbf{M})=\mathcal{N}(\mathbf{F})$ because $\mathbf{M}(\mathbf{I}_k-\mathbf{M}^{-}\mathbf{M})=\mathbf{0}$, whence $\mathbf{F}(\mathbf{I}_k-\mathbf{M}^{-}\mathbf{M})=\mathbf{0}$. Finally $\mathbf{\Omega}_r^*\mathbf{A}=\mathbf{M}\mathbf{S}$, so $\mathbf{F}\mathbf{M}^{-}\mathbf{\Omega}_r^*\mathbf{A}=\mathbf{F}\mathbf{M}^{-}\mathbf{M}\mathbf{S}=\mathbf{F}\mathbf{S}=\mathbf{A}$. $(\Rightarrow)$ The identity gives $\mathcal{C}(\mathbf{A})\subseteq\mathcal{C}(\mathbf{A}\mathbf{\Omega}_c)\subseteq\mathcal{C}(\mathbf{A})$ and $\mathcal{C}(\mathbf{A}^*)\subseteq\mathcal{C}(\mathbf{A}^*\mathbf{\Omega}_r)\subseteq\mathcal{C}(\mathbf{A}^*)$, so $\mathbf{\Omega}_r^*$ is injective on $\mathcal{C}(\mathbf{A})=\mathcal{C}(\mathbf{A}\mathbf{\Omega}_c)$ and $\mathrm{rank}(\mathbf{M})=\mathrm{rank}(\mathbf{\Omega}_r^*\mathbf{A}\mathbf{\Omega}_c)=r$. Note that this direction uses no Penrose equation at all.
\end{proof}

Three consequences deserve notice. Since $\mathrm{rank}(\mathbf{M})\le k$, the criterion forces $k\ge r$, so \eqref{eq:GN-exactness} is false for every choice of sketches when $k<r$. Only the first Penrose equation is used in $(\Leftarrow)$, and none at all in $(\Rightarrow)$, which is why the proposition is stated for an arbitrary $\{1\}$-inverse $\mathbf{M}^{-}$ rather than for $\mathbf{M}^{+}$: the Moore--Penrose inverse is a convenient representative here, not a requirement. And with $\mathbf{\Omega}_c=\mathbf{I}_n(:,J)$ and $\mathbf{\Omega}_r=\mathbf{I}_m(:,I)$ the criterion reads $\mathrm{rank}(\mathbf{A}(I,J))=\mathrm{rank}(\mathbf{A})$, which is the exactness condition for CUR published by Hamm and Huang \cite{hamm2020perspectives}, Theorem~A. For arbitrary sketches, the criterion above expresses the same rank-preservation principle. The classical pseudoskeleton setting uses selected rows and columns \cite{goreinov1997pseudoskeleton}, and our column/row-selection specialization recovers the published characterization.

Before $\mathbf{B}$ and $\mathbf{D}$ are specialized, however, the framework says slightly more than either of those statements, and it says which choice produces which published identity. The setting is column and row selection: $\mathbf{C}$ collects $k$ columns of $\mathbf{A}$, $\mathbf{R}$ collects $k$ rows, and $\mathbf{W}$ is the submatrix where they cross. These are the classical CUR symbols. In particular $\mathbf{R}$ here is a set of rows of $\mathbf{A}$, not the echelon factor of the CR decomposition of Section~\ref{subsection:pseudoinverse}. This is the CUR decomposition, which the end of this section treats in full.

\begin{corollary}\label{cor:CUR-family}
Let $\mathbf{A}\in\mathbb{F}^{m\times n}$, let $I\subseteq\{1,\dots,m\}$ and $J\subseteq\{1,\dots,n\}$ with $|I|=|J|=k$, put $\mathbf{C}=\mathbf{A}(:,J)$, $\mathbf{R}=\mathbf{A}(I,:)$ and $\mathbf{W}=\mathbf{A}(I,J)$, and let $\mathbf{B}\in\mathbb{F}^{m\times p}$ and $\mathbf{D}\in\mathbb{F}^{n\times q}$ be arbitrary. Then
\begin{align}\label{eq:CUR-family}
\mathbf{A} = \mathbf{C}(\mathbf{B}^*\mathbf{C})^{+}\mathbf{B}^*\mathbf{A}\,\mathbf{D}(\mathbf{R}\mathbf{D})^{+}\mathbf{R}
\end{align}
if and only if
\begin{align}\label{eq:CUR-family-cond}
\mathrm{rank}(\mathbf{C})=\mathrm{rank}(\mathbf{R})=\mathrm{rank}(\mathbf{A}),
\quad
\mathrm{rank}(\mathbf{B}^*\mathbf{C})=\mathrm{rank}(\mathbf{C}),
\quad
\mathrm{rank}(\mathbf{R}\mathbf{D})=\mathrm{rank}(\mathbf{R}).
\end{align}
The choice $\mathbf{B}=\mathbf{C}$, $\mathbf{D}=\mathbf{R}^*$ satisfies the last two conditions of \eqref{eq:CUR-family-cond} for every $I$ and $J$ and gives $\mathbf{G}=\mathbf{C}^{+}\mathbf{A}\mathbf{R}^{+}$, i.e.\ $\mathbf{A}=\mathbf{C}\mathbf{C}^{+}\mathbf{A}\mathbf{R}^{+}\mathbf{R}$. The choice $\mathbf{B}=\mathbf{I}_m(:,I)$, $\mathbf{D}=\mathbf{I}_n(:,J)$ makes $\mathbf{B}^*\mathbf{C}=\mathbf{R}\mathbf{D}=\mathbf{W}$, collapses \eqref{eq:CUR-family-cond} to $\mathrm{rank}(\mathbf{W})=\mathrm{rank}(\mathbf{A})$, and gives $\mathbf{G}=\mathbf{W}^{+}$, i.e.\ $\mathbf{A}=\mathbf{C}\mathbf{W}^{+}\mathbf{R}$.
\end{corollary}

\begin{proof}
Write $\mathbf{P}_c=\mathbf{C}(\mathbf{B}^*\mathbf{C})^{+}\mathbf{B}^*$ and $\mathbf{P}_r=\mathbf{D}(\mathbf{R}\mathbf{D})^{+}\mathbf{R}$. The claim is $\mathbf{P}_c\mathbf{A}\mathbf{P}_r=\mathbf{A}$. Both factors are idempotent, $\mathbf{P}_c^2=\mathbf{P}_c$ and $\mathbf{P}_r^2=\mathbf{P}_r$, by the second Penrose equation for $(\mathbf{B}^*\mathbf{C})^{+}$ and $(\mathbf{R}\mathbf{D})^{+}$. Consequently $\mathbf{P}_c\mathbf{A}\mathbf{P}_r=\mathbf{A}$ already implies $\mathbf{P}_c\mathbf{A}=\mathbf{P}_c(\mathbf{P}_c\mathbf{A}\mathbf{P}_r)=\mathbf{A}$ and dually $\mathbf{A}\mathbf{P}_r=\mathbf{A}$, so the two factors may be treated separately. For the left factor, $\mathbf{P}_c\mathbf{A}=\mathbf{A}$ requires $\mathcal{C}(\mathbf{A})\subseteq\mathcal{C}(\mathbf{C})$, hence, the reverse inclusion being automatic for a column subset, $\mathrm{rank}(\mathbf{C})=\mathrm{rank}(\mathbf{A})$. Given that, $\mathbf{P}_c\mathbf{A}=\mathbf{A}$ holds if and only if $\mathbf{P}_c\mathbf{C}=\mathbf{C}$, which holds if and only if $\mathcal{N}(\mathbf{B}^*\mathbf{C})\subseteq\mathcal{N}(\mathbf{C})$, and that in turn holds if and only if $\mathrm{rank}(\mathbf{B}^*\mathbf{C})=\mathrm{rank}(\mathbf{C})$. The reverse inclusion of null spaces is again automatic. For the right factor, $\mathbf{A}\mathbf{P}_r=\mathbf{A}$ requires $\mathcal{C}(\mathbf{A}^{*})\subseteq\mathcal{C}(\mathbf{R}^{*})$, hence, the reverse inclusion being automatic for a row subset, $\mathrm{rank}(\mathbf{R})=\mathrm{rank}(\mathbf{A})$. Given that, put $\mathbf{S}=\mathbf{A}\mathbf{R}^{+}$, so $\mathbf{S}\mathbf{R}=\mathbf{A}$ and $\mathbf{S}$ maps $\mathcal{C}(\mathbf{R})$ bijectively onto $\mathcal{C}(\mathbf{A})$. Then $\mathbf{A}\mathbf{P}_r=\mathbf{A}$ if and only if the orthogonal projector onto $\mathcal{C}(\mathbf{R}\mathbf{D})^{\perp}$ annihilates $\mathbf{R}$, i.e.\ if and only if $\mathcal{C}(\mathbf{R}\mathbf{D})=\mathcal{C}(\mathbf{R})$. The two specializations are immediate substitutions.
\end{proof}

The two named members are precisely items (iii) and (ii) of Hamm and Huang's Theorem~A \cite{hamm2020perspectives}. Both take $p=q=k$, but the equivalence constrains neither $p$ nor $q$: \eqref{eq:CUR-family-cond} compares $\mathrm{rank}(\mathbf{B}^*\mathbf{C})$ with $\mathrm{rank}(\mathbf{C})$ and $\mathrm{rank}(\mathbf{R}\mathbf{D})$ with $\mathrm{rank}(\mathbf{R})$, not with $k$. Corollary~\ref{cor:CUR-family} is therefore a modest, corollary-strength generalization of a published equivalence rather than a new theorem: every member of the family follows from that equivalence in a few lines. Its value here is explanatory. It exhibits the two published identities as the two natural choices of the design parameters $(\mathbf{B},\mathbf{D})$ in one construction, and it makes visible which of the three rank conditions each choice is paying for --- which is exactly the ``hidden projectors'' thesis of this section. Hamm's generalized pseudoskeleton decompositions \cite{hamm2023generalized} narrow that remainder further, extending CUR exactness to arbitrary fields and to generalized inverses, so that what a family of this kind contributes is the parametrization rather than the exactness characterizations.

We now return to the generalized Nystr\"{o}m method itself, which the meta-factorization procedure produces once the two generally oblique projectors are designed. By Theorem~\ref{Theorem:projector-eq-sol}, with $\mathbf{\Omega}_r^*\mathbf{A}\mathbf{\Omega}_c$ of rank $k$,
\begin{align}
\mathbf{F} = \mathbf{A}\mathbf{\Omega}_c
\quad&\text{and}\quad
\mathbf{H}^* = \mathbf{\Omega}_r^*\mathbf{A}
\\\label{eq:Nystroem-B-D}
\mathbf{B} = \mathbf{\Omega}_r
\quad&\text{and}\quad
\mathbf{D} = \mathbf{\Omega}_c
\end{align}
yield the following solutions to the projector equation:
\begin{align}
\mathbf{Y}^*= (\mathbf{\Omega}_r^*\mathbf{A}\mathbf{\Omega}_c)^+\mathbf{\Omega}_r^*
\quad\text{and}\quad
\mathbf{X}=\mathbf{\Omega}_c(\mathbf{\Omega}_r^*\mathbf{A}\mathbf{\Omega}_c)^+.
\end{align}
As a result, the mixing matrix becomes:
\begin{align}\label{eq:GN-G}
\begin{aligned}
\mathbf{G}
&= (\mathbf{\Omega}_r^*\mathbf{A}\mathbf{\Omega}_c)^{+}\mathbf{\Omega}_r^*\mathbf{A}\mathbf{\Omega}_c(\mathbf{\Omega}_r^*\mathbf{A}\mathbf{\Omega}_c)^{+}
= (\mathbf{\Omega}_r^*\mathbf{A}\mathbf{\Omega}_c)^{+},
\end{aligned}
\end{align}
where the only property of $(\cdot)^{+}$ used is the second Penrose equation. Consequently \eqref{eq:GN-G} holds verbatim with $(\mathbf{\Omega}_r^*\mathbf{A}\mathbf{\Omega}_c)^{+}$ replaced by any $\{1,2\}$-inverse of $\mathbf{\Omega}_r^*\mathbf{A}\mathbf{\Omega}_c$. Reflexivity is what this display needs. The exactness criterion of Proposition~\ref{prop:CUR-exactness} needs less, and holds for any $\{1\}$-inverse.
Assembling the three factors gives the generalized Nystr\"{o}m method:
\begin{align}
\begin{aligned}
\widehat{\mathbf{A}}
&= \mathbf{F}\mathbf{G}\mathbf{H}^*
\\
&= \mathbf{A}\mathbf{\Omega}_c(\mathbf{\Omega}_r^*\mathbf{A}\mathbf{\Omega}_c)^{+}\mathbf{\Omega}_r^*\mathbf{A}.
\end{aligned}
\end{align}
For the plain generalized Nystr\"{o}m method, Nakatsukasa \cite[Section~5.1]{nakatsukasa2020fast} recommends computing $\widehat{\mathbf{A}}$ through the QR factorization $\mathbf{M}=\mathbf{\Omega}_r^*\mathbf{A}\mathbf{\Omega}_c = \mathbf{Q}_{M}\mathbf{R}_{M}$, whose subscripts keep these factors apart from the selected rows $\mathbf{R}$ of the CUR decomposition below, and then take
\begin{align}
\begin{aligned}
\widehat{\mathbf{A}}
&= (\mathbf{A}\mathbf{\Omega}_c\mathbf{R}_{M}^{-1})(\mathbf{Q}_{M}^*\mathbf{\Omega}_r^*\mathbf{A}),
\end{aligned}
\end{align}
which completes the third step of the meta-factorization procedure. This is the plain QR implementation, not a finite-precision stability guarantee. Nakatsukasa's stabilized implementation truncates or regularizes numerically singular directions of the compressed matrix \cite[Section~5.1.1]{nakatsukasa2020fast}. The approximation error can be substantially reduced with row-space oversampling, and this costs one point of bookkeeping, because in the present instantiation the row sketch plays two roles at once: it is the design parameter $\mathbf{B}=\mathbf{\Omega}_r$ \emph{and} it builds the right factor $\mathbf{H}^*=\mathbf{\Omega}_r^*\mathbf{A}$. Take $\mathbf{\Omega}_r\in\mathbb{F}^{m\times p}$ with $p>k$. Keep $\mathrm{rank}(\mathbf{M})=k$ in force, as everywhere in this subsection. Without it $\mathbf{R}_{M}$ is singular and the display below is undefined. Then $\mathbf{M}=\mathbf{\Omega}_r^*\mathbf{A}\mathbf{\Omega}_c$ is $p\times k$, $\mathbf{H}^*$ is $p\times n$, and the raw triple $(\mathbf{F},\mathbf{M}^{+},\mathbf{H}^*)$ has a \emph{rectangular} core $\mathbf{M}^{+}\in\mathbb{F}^{k\times p}$. The left half of the projector equation still reads $\mathbf{Y}^*\mathbf{F}=\mathbf{I}_k$, but the right half becomes $\mathbf{H}^*\mathbf{X}=\mathbf{M}\mathbf{M}^{+}\in\mathbb{F}^{p\times p}$, of rank $k<p$. Absorbing the factors of $\mathbf{M}=\mathbf{Q}_{M}\mathbf{R}_{M}$ into the outer factors restores the form \eqref{eq:genfactorization}: with
\begin{align}\label{eq:GN-oversampled}
\widetilde{\mathbf{F}} = \mathbf{A}\mathbf{\Omega}_c\mathbf{R}_{M}^{-1}\in\mathbb{F}^{m\times k},
\quad
\widetilde{\mathbf{H}}^* = \mathbf{Q}_{M}^*\mathbf{\Omega}_r^*\mathbf{A}\in\mathbb{F}^{k\times n},
\quad
\mathbf{G} = \mathbf{I}_k,
\end{align}
the pair $(\mathbf{B},\mathbf{D})=(\mathbf{\Omega}_r,\mathbf{\Omega}_c)$ is admissible for $(\widetilde{\mathbf{F}},\widetilde{\mathbf{H}})$, both halves of the projector equation hold, and the mixing matrix is $k\times k$. Oversampling therefore enlarges the design parameter and the raw right factor alike. The inner dimension of \eqref{eq:genfactorization} is restored by the compression, not preserved by it. The MATLAB code below computes exactly \eqref{eq:GN-oversampled}:
\begin{lstlisting}
    D       = randn(n,k);   B  = randn(m,ceil(2*k)); % row-space
    F       = A*D;          Hr = A'*B;               %   oversampling
    [Q,R]   = qr(B'*A*D,0);                          % R is k-by-k
    Ft      = F/R;          Ht = Q'*Hr';             % compressed factors,
    Ar      = Ft*Ht;                                 %   inner dimension k
\end{lstlisting}

The generalized Nystr\"{o}m method is closely related to the outer-product decompositions developed in the classical rank-reduction literature by Wedderburn \cite{wedderburn1934lectures}, Egerv{\'a}ry \cite{egervary1960rank}, Householder \cite{householder1965theory}, Cline and Funderlic \cite{cline1979rank}, and Chu et al. \cite{chu1995rank}, and the rank-reduction conditions of that line take the following form in our notation. The following corollary expresses the classical criterion in terms of proposed factors $\mathbf F,\mathbf G,\mathbf H$. It also gives explicit sampling matrices, so the criterion can be applied without searching for them.

\emph{Rank subtractivity} and \emph{rank additivity} name one and the same condition here, namely \eqref{eq:rank-additivity} below, read either as the rank that subtracting $\mathbf{F}\mathbf{G}\mathbf{H}^*$ removes from $\mathbf{A}$ or as the two ranks adding up to $\mathrm{rank}(\mathbf{A})$. Both names occur in the literature quoted next. The first equivalence of the corollary below --- rank subtractivity if and only if the sketch conditions \eqref{eq:rank-reduction-conds} hold --- is the block rank-reduction theorem stated by Chu, Funderlic and Golub \cite{chu1995rank}, who attribute it to Cline and Funderlic \cite{cline1979rank}. The equivalence is already in that statement. Block rank subtractivity also has a name and a literature of its own: the Wedderburn--Guttman theorem, going back to Guttman \cite{guttman1944general}; see the historical account in \cite[\S2.2]{hubert2000two}. Takane and Yanai \cite{takane2005wedderburn} study the modern form of that theorem --- with the middle matrix generalized to a $\{1\}$-inverse, a setting broader than the nonsingular $\mathbf{G}$ treated here --- in \emph{Linear Algebra and its Applications}, and give a necessary and sufficient condition for it. See Gal{\'a}ntai \cite{galantai2010rank} for a survey of the Egerv{\'a}ry rank-reduction line.

The same rank condition is written $\mathbf N\le^{-}\mathbf A$ in the literature on the \emph{minus partial order}, with $\mathbf N=\mathbf F\mathbf G\mathbf H^*$. This notation means that subtracting $\mathbf N$ lowers the rank by exactly $\mathrm{rank}(\mathbf N)$; it is not an entrywise comparison of the two matrices. A classical test for this relation consists of $\mathcal{C}(\mathbf{N})\subseteq\mathcal{C}(\mathbf{A})$, $\mathcal{C}(\mathbf{N}^*)\subseteq\mathcal{C}(\mathbf{A}^*)$ and $\mathbf{N}\mathbf{A}^{-}\mathbf{N}=\mathbf{N}$ for some, hence every, $\{1\}$-inverse of $\mathbf{A}$. See Baksalary, Pukelsheim and Styan \cite{baksalary1989properties}, Eq.~(1.21), with (1.12) and (1.13).

\begin{corollary}[of the minus-order criterion {\cite[Eq.~(1.21)]{baksalary1989properties}}
and the block rank-reduction theorem {\cite[Thm.~1.3]{chu1995rank}}]\label{thm:wedderburn} Let $\mathbf{A}\in\mathbb{F}^{m\times n}$, $\mathbf{F}\in\mathbb{F}^{m\times k}$ and $\mathbf{H}\in\mathbb{F}^{n\times k}$ with $\mathrm{rank}(\mathbf{F})=\mathrm{rank}(\mathbf{H})=k$, and let $\mathbf{G}\in\mathbb{F}^{k\times k}$ be nonsingular. Then
\begin{align}\label{eq:rank-additivity}
\begin{aligned}
\mathrm{rank}(\mathbf{A} - \mathbf{F}\mathbf{G}\mathbf{H}^*) =
\mathrm{rank}(\mathbf{A})-\mathrm{rank}(\mathbf{F}\mathbf{G}\mathbf{H}^*)
\end{aligned}
\end{align}
holds if and only if there exist matrices $\mathbf{\Omega}_c$ and $\mathbf{\Omega}_r$ such that
\begin{align}\label{eq:rank-reduction-conds}
\begin{aligned}
\mathbf{F} = \mathbf{A}\mathbf{\Omega}_c
\quad\text{and}\quad
\mathbf{G}^{-1} = \mathbf{\Omega}_r^*\mathbf{A}\mathbf{\Omega}_c
\quad\text{and}\quad
\mathbf{H} = \mathbf{A}^*\mathbf{\Omega}_r,
\end{aligned}
\end{align}
and this in turn holds if and only if
\begin{align}\label{eq:rank-reduction-closed}
\mathcal{C}(\mathbf{F})\subseteq\mathcal{C}(\mathbf{A}),
\quad
\mathcal{C}(\mathbf{H})\subseteq\mathcal{C}(\mathbf{A}^*)
\quad\text{and}\quad
\mathbf{G} = (\mathbf{H}^*\mathbf{A}^{+}\mathbf{F})^{-1},
\end{align}
in which case one may choose $\mathbf{\Omega}_c = \mathbf{A}^{+}\mathbf{F}$ and $\mathbf{\Omega}_r = (\mathbf{A}^{+})^*\mathbf{H}$.
\end{corollary}

\begin{proof}
Put $\mathbf{N}=\mathbf{F}\mathbf{G}\mathbf{H}^*$. Rank additivity forces $\mathrm{rank}(\mathbf{A})$ to attain the bound $\mathrm{rank}(\mathbf{A}-\mathbf{N})+\mathrm{rank}(\mathbf{N})$, hence $\mathcal{C}(\mathbf{A})=\mathcal{C}(\mathbf{A}-\mathbf{N})+\mathcal{C}(\mathbf{N})\supseteq\mathcal{C}(\mathbf{N})$ and dually for rows. Given the inclusions, write the thin SVD $\mathbf{A}=\mathbf{U}_{\!A}\mathbf{\Sigma}\mathbf{V}_{\!A}^*$ with an $r\times r$ core, so that $\mathbf{N}=\mathbf{U}_{\!A}\mathbf{N}_1\mathbf{V}_{\!A}^*$; additivity becomes $\mathrm{rank}(\mathbf{I}_r-\mathbf{J})=r-\mathrm{rank}(\mathbf{J})$ for the square matrix $\mathbf{J}=\mathbf{\Sigma}^{-1}\mathbf{N}_1$, which holds if and only if $\mathbf{J}^2=\mathbf{J}$, i.e.\ if and only if $\mathbf{N}\mathbf{A}^{+}\mathbf{N}=\mathbf{N}$. With $\mathbf{F}$ of full column rank, $\mathbf{H}^*$ of full row rank and $\mathbf{G}$ nonsingular, that equation cancels to $\mathbf{G}(\mathbf{H}^*\mathbf{A}^{+}\mathbf{F})\mathbf{G}=\mathbf{G}$, i.e.\ to \eqref{eq:rank-reduction-closed}. For the equivalence of \eqref{eq:rank-reduction-conds} and \eqref{eq:rank-reduction-closed}, note that $\mathbf{F}=\mathbf{A}\mathbf{\Omega}_c$ and $\mathbf{H}=\mathbf{A}^*\mathbf{\Omega}_r$ are solvable exactly under the two inclusions, and that the product $\mathbf{\Omega}_r^*\mathbf{A}\mathbf{\Omega}_c$ is then the \emph{same} for every admissible pair. The solution sets are $\mathbf{\Omega}_c+\mathbf{V}_c$ with $\mathbf{A}\mathbf{V}_c=\mathbf{0}$ and $\mathbf{\Omega}_r+\mathbf{V}_r$ with $\mathbf{A}^*\mathbf{V}_r=\mathbf{0}$, and expanding $(\mathbf{\Omega}_r+\mathbf{V}_r)^*\mathbf{A}(\mathbf{\Omega}_c+\mathbf{V}_c)$ kills the three cross terms. Evaluating at the Moore--Penrose choice gives $\mathbf{\Omega}_r^*\mathbf{A}\mathbf{\Omega}_c=\mathbf{H}^*\mathbf{A}^{+}\mathbf{F}$.
\end{proof}

The first step of that proof is classical and deserves its name. Condition \eqref{eq:rank-additivity} says exactly that $\mathbf{N}=\mathbf{F}\mathbf{G}\mathbf{H}^*$ lies below $\mathbf{A}$ in the minus partial order, and, under the two span inclusions, its equivalence with $\mathbf{N}\mathbf{A}^{+}\mathbf{N}=\mathbf{N}$ is a standard characterization of that order \cite{baksalary1989properties}. For a broader treatment of rank additivity and subtractivity, see Marsaglia and Styan \cite{marsaglia1974equalities}. Given that equivalence, \eqref{eq:rank-reduction-closed} is two lines: substitute $\mathbf{N}=\mathbf{F}\mathbf{G}\mathbf{H}^*$, cancel the full-rank outer factors to get $\mathbf{G}(\mathbf{H}^*\mathbf{A}^{+}\mathbf{F})\mathbf{G}=\mathbf{G}$, and invert. For using this criterion with prescribed factors, the explicit choice $\mathbf{\Omega}_c=\mathbf{A}^{+}\mathbf{F}$, $\mathbf{\Omega}_r=(\mathbf{A}^{+})^*\mathbf{H}$ is useful: whenever any sampling matrices satisfy \eqref{eq:rank-reduction-conds}, these particular matrices do. Thus one checks the span inclusions and the displayed formula for $\mathbf G$, then obtains the sampling matrices directly. This is an application of the classical criterion to the variables used here. The invariance of $\mathbf{\Omega}_r^*\mathbf{A}\mathbf{\Omega}_c$ across the admissible pairs that makes this work is K\k{e}dzierski's Lemma~5.9 in \cite{kedzierski2024rank}. We re-derive it here for completeness.

Both extra hypotheses are needed. With $\mathbf{A}=\mathbf{I}_2$, $\mathbf{F}=\mathbf{H}=\mathrm{diag}(1,0)$ and $\mathbf{G}=\mathbf{I}_2$, rank additivity holds --- the rank drops by exactly one --- yet $\mathbf{A}$ is invertible, so the pair is unique, $\mathbf{\Omega}_r^*\mathbf{A}\mathbf{\Omega}_c=\mathrm{diag}(1,0)$ is singular, and no admissible pair exists. Here $\mathrm{rank}(\mathbf{F})=1<k=2$. With $\mathbf{A}=\mathbf{F}=\mathbf{H}=\mathbf{I}_2$ and $\mathbf{G}=\mathrm{diag}(1,0)$, rank additivity again holds while $\mathbf{G}^{-1}$ does not exist and \eqref{eq:rank-reduction-conds} is not even well formed.

The second counterexample marks exactly the boundary that K\k{e}dzierski \cite{kedzierski2024rank} crosses: allowing the core $\mathbf G=(\mathbf{\Omega}_r^*\mathbf{A}\mathbf{\Omega}_c)^{+}$ keeps a rank-reduction identity even when the compressed matrix is singular, where the inverse condition \eqref{eq:rank-reduction-conds} cannot be written down. His Theorem~5.1 is therefore strictly stronger than the sufficiency direction of Corollary~\ref{thm:wedderburn}. What Corollary~\ref{thm:wedderburn} adds is the converse and the closed form \eqref{eq:rank-reduction-closed}, both of which need $\mathbf{G}$ invertible.

The closed form \eqref{eq:rank-reduction-closed} computes the mixing matrix of a rank-subtractive term from $(\mathbf{A},\mathbf{F},\mathbf{H})$ alone, without reference to any sketch. That term represents the whole of $\mathbf{A}$ exactly when $k=\mathrm{rank}(\mathbf{A})$; for smaller $k$, a residual of rank $\mathrm{rank}(\mathbf{A})-k$ remains. This is also the cleanest way to say how the present framework differs from the rank-reduction line of Chu, Funderlic and Golub \cite{chu1995rank}. There, a factorization is \emph{built} by peeling rank-one pieces, and the factors emerge from the algorithm. Here the finished factorization is characterized algebraically and simultaneously. So one factor may be replaced by an equivalent and the mixing matrix repaired, under the criterion of Theorem~\ref{thm:replacement}, which reads neither $\mathbf{A}$ nor the closed form of Corollary~\ref{Theorem:generalized-meta-factorization}. Rank additivity in this sense is one of the rank equalities collected by Marsaglia and Styan \cite{marsaglia1974equalities}.
Independently of those conditions, every nonzero $\mathbf{A}$ admits a~finite rank-reducing process which decomposes it into a~sum of low-rank matrices. The result is the outer-product decomposition. To be more specific, consider the Wedderburn rank-one reduction process,
\begin{align}
\begin{aligned}
\mathbf{A}_1 &= \mathbf{A},
\\
\mathbf{A}_{j+1} &= \mathbf{A}_j -
g_j^{-1}\mathbf{A}_j\mathbf{u}_j\mathbf{v}^*_j\mathbf{A}_j,
\\
g_j &= \mathbf{v}^*_j\mathbf{A}_j\mathbf{u}_j \neq 0
\quad\text{for}\quad j = 1,\dots,r,\quad r=\mathrm{rank}(\mathbf{A}).
\end{aligned}
\end{align}
Each step lowers the rank by exactly one, so $\mathbf{A}_{r+1}=\mathbf{0}$ and the process terminates after $r$ iterations. The inner dimension of the resulting factorization is therefore $k=r$ and every $g_j$ used below is defined. Assemble the raw direction matrices
\begin{align*}
\mathbf{\Omega}_c=[\,\mathbf{u}_1\ \cdots\ \mathbf{u}_r\,]\in\mathbb{F}^{n\times r},
\qquad
\mathbf{\Omega}_r=[\,\mathbf{v}_1\ \cdots\ \mathbf{v}_r\,]\in\mathbb{F}^{m\times r}.
\end{align*}
The process gives rise to the following two factorizations:
\begin{align}\label{eq:wedderburn-closing}
\begin{aligned}
\mathbf{A} &= \mathbf{F}\mathbf{G}\mathbf{H}^*
= \bmat{\mathbf{A}_1\mathbf{u}_1,\dots,\mathbf{A}_k\mathbf{u}_k}
\bmat{g_1 & \\ & \ddots & \\ & & g_k}^{-1}
\bmat{\mathbf{v}^*_1\mathbf{A}_1\\\dots\\\mathbf{v}^*_k\mathbf{A}_k}
\\
&= \mathbf{A}\mathbf{\Omega}_c(\mathbf{\Omega}_r^*\mathbf{A}\mathbf{\Omega}_c)^{-1}\mathbf{\Omega}_r^*\mathbf{A}.
\end{aligned}
\end{align}
The first line of \eqref{eq:wedderburn-closing} uses the residual factors $\mathbf{A}_j\mathbf{u}_j$ and $\mathbf{v}_j^*\mathbf{A}_j$; the second uses the raw factors $\mathbf{A}\mathbf{\Omega}_c$ and $\mathbf{\Omega}_r^*\mathbf{A}$. Eliminating the compressed matrix $\mathbf{\Omega}_r^*\mathbf{A}\mathbf{\Omega}_c$ in this order gives the successive Schur-complement pivots $g_j$. These nonzero pivots show that it is nonsingular and that its determinant is $\prod_jg_j$, although it need not equal the diagonal matrix $\mathrm{diag}(g_1,\dots,g_k)$ in the first line. Proposition~\ref{prop:CUR-exactness} then gives the second equality. The corresponding factors coincide exactly when the directions are biconjugate, $\mathbf{v}_i^*\mathbf{A}\mathbf{u}_j=0$ for $i\neq j$ --- in particular whenever $k=1$. Both products equal $\mathbf{A}$ without this extra condition.

These raw directions must also be distinguished from the canonical witnesses in Corollary~\ref{thm:wedderburn}. Applying that corollary to the residual factors of the first line gives witnesses $\mathbf{A}^{+}\mathbf{F}$ and $(\mathbf{A}^{+})^*\mathbf{H}$ which reproduce those residual factors and yield the diagonal compressed matrix $\mathrm{diag}(g_1,\dots,g_k)$. The raw direction matrices instead yield the generally nondiagonal compressed matrix in the second line.

The generalized Nystr\"{o}m method produces high-quality low-rank approximations of $\mathbf{A}$ by exploiting the structure of the outer-product factorization and relaxing the rank-reduction conditions, defined by (\ref{eq:rank-reduction-conds}), with random sketch matrices $\mathbf{\Omega}_c$ and $\mathbf{\Omega}_r$. What survives the relaxation must be stated precisely, because the two equations of meta-factorization part company here. Write $\mathbf{M}=\mathbf{\Omega}_r^*\mathbf{A}\mathbf{\Omega}_c$. The projector equation holds if and only if $\mathrm{rank}(\mathbf{M})=k$; the reconstruction equation holds if and only if $\mathrm{rank}(\mathbf{M})=\mathrm{rank}(\mathbf{A})$. When $k$ exceeds $\mathrm{rank}(\mathbf{A})$ the second can hold while the first fails, and it fails by an exact amount: $\|\mathbf{Y}^*\mathbf{F}-\mathbf{I}_k\|_F=\|\mathbf{H}^*\mathbf{X}-\mathbf{I}_k\|_F=\sqrt{k-\mathrm{rank}(\mathbf{M})}\ge1$ for \emph{every} choice of sketches, so oversampling does not repair it. This is the sketch specialization in the $k>\mathrm{rank}(\mathbf{A})$ row of Proposition~\ref{prop:regimes}; it does not rule out normalized reconstruction with other full-column-rank containing factors. In this regime the core is singular, so exact CPQR cannot retain $\mathbf{G}=\mathbf{I}_k$, while an SVD retains its diagonal form with zero singular values. The exact residual \eqref{eq:PE-residual} is not in tension with Park and Nakatsukasa \cite{park2025accuracy}, who derive approximation bounds with oversampling and finite-precision bounds for a stabilized CUR \emph{approximation}. Their target dimension satisfies $k\le\operatorname{rank}(\mathbf{A})$. They enlarge either the row or the column sample, producing a rectangular raw core, whereas \eqref{eq:PE-residual} concerns the unrecompressed square projector identity at common inner dimension $k>\operatorname{rank}(\mathbf{A})$. Thus the two results measure different quantities in different regimes --- approximation error and stability there, exact algebraic feasibility here --- and both statements are true.

The section closes by relating the generalized Nystr\"{o}m method to the CUR factorization. Recall that the CUR factorization provides an approximation:
\begin{align}\label{eq:CUR}
\mathbf{A} \approx \mathbf{C}\mathbf{U}\mathbf{R},
\end{align}
where $\mathbf{C} = \mathbf{A}(:,J)$ is a subset of representative columns of $\mathbf{A}$, $\mathbf{R} = \mathbf{A}(I,:)$ is a subset of representative rows of $\mathbf{A}$, and the mixing matrix is $\mathbf{U}$ --- the classical CUR symbol, unrelated to the SVD and UTV factors called $\mathbf{U}$ in Section~\ref{section:How...}. The properties of the approximation depend critically on the rank of $\mathbf{U}$ and the selection of columns and rows, $J$ and $I$, as demonstrated by Sorensen and Embree \cite{sorensen2016deim}, Wang and Zhang \cite{wang2013improving}, Martinsson and Tropp \cite{martinsson2020randomized}, Hamm and Huang \cite{hamm2020perspectives}. The column/row-selection form was studied as a pseudoskeleton approximation by Goreinov, Tyrtyshnikov and Zamarashkin \cite{goreinov1997pseudoskeleton}, and the case in which \eqref{eq:CUR} holds with equality has been characterized by Hamm and Huang \cite{hamm2020perspectives}. Our aim is not to improve on those characterizations but to show that they are what the meta-factorization equations produce, and to identify the design parameters that separate them.

To reconstruct the CUR decomposition from the meta-factorization equations, consider the following special case of sampling matrices,
\begin{align}
\bar{\mathbf{\Omega}}_c = \mathbf{I}_{n}(:,J)
\quad\text{and}\quad
\bar{\mathbf{\Omega}}_r = \mathbf{I}_{m}(:,I).
\end{align}
Matrix
\begin{align}
\mathbf{F} = \mathbf{A}\bar{\mathbf{\Omega}}_c = \mathbf{C}
\end{align}
contains the original columns of $\mathbf{A}$. Similarly, the matrix
\begin{align}
\mathbf{H}^* = \bar{\mathbf{\Omega}}_r^*\mathbf{A} = \mathbf{R}
\end{align}
contains the original rows of $\mathbf{A}$. Call $\mathbf{B}\in\mathbb{F}^{m\times p}$ and $\mathbf{D}\in\mathbb{F}^{n\times q}$ \emph{admissible} for this instance when $\mathrm{rank}(\mathbf{B}^*\mathbf{C})=\mathrm{rank}(\mathbf{R}\mathbf{D})=k$, which by Theorem~\ref{Theorem:projector-eq-sol} is exactly what makes $\mathbf{Y}^*$ and $\mathbf{X}$ solve the projector equation. For an admissible $\mathbf{B}$ we obtain
\begin{align}
\mathbf{Y}^*= (\mathbf{B}^*\mathbf{A}\bar{\mathbf{\Omega}}_c)^+\mathbf{B}^*,
\end{align}
and for an admissible $\mathbf{D}$ we obtain
\begin{align}
\mathbf{X} = \mathbf{D}(\bar{\mathbf{\Omega}}_r^*\mathbf{A}\mathbf{D})^+.
\end{align}
By the rules of meta-factorization, the corresponding mixing matrix is
\begin{align}\label{eq:CUR-G}
\begin{aligned}
\mathbf{G} &=
(\mathbf{B}^*\mathbf{A}\bar{\mathbf{\Omega}}_c)^+\mathbf{B}^*
\mathbf{A}
\mathbf{D}(\bar{\mathbf{\Omega}}_r^*\mathbf{A}\mathbf{D})^+
\\
&=
(\mathbf{B}^*\mathbf{C})^+\mathbf{B}^*
\mathbf{A}
\mathbf{D}(\mathbf{R}\mathbf{D})^+.
\end{aligned}
\end{align}
Selecting
\begin{align}\label{eq:CUR-B-D}
\mathbf{B} = \mathbf{A}\bar{\mathbf{\Omega}}_c = \mathbf{C}
\quad&\text{and}\quad
\mathbf{D} = \mathbf{A}^*\bar{\mathbf{\Omega}}_r = \mathbf{R}^*
\end{align}
makes the last two conditions of \eqref{eq:CUR-family-cond} automatic --- note that these are equalities with $\mathrm{rank}(\mathbf{C})$ and $\mathrm{rank}(\mathbf{R})$, not with $k$, so this choice fails to be admissible in the sense above exactly when $\mathrm{rank}(\mathbf{C})<k$ or $\mathrm{rank}(\mathbf{R})<k$, in particular whenever $\mathrm{rank}(\mathbf{A})<k$ --- and immediately leads to the following well-known form of the CUR decomposition,
\begin{align}\label{eq:CUR-exact}
\begin{aligned}
\mathbf{A} &=
\mathbf{C}(\mathbf{C}^*\mathbf{C})^+\mathbf{C}^*
\mathbf{A}
\mathbf{R}^*(\mathbf{R}\mathbf{R}^*)^+\mathbf{R}
\\
&=
\mathbf{C}\mathbf{C}^+
\mathbf{A}
\mathbf{R}^+\mathbf{R},
\end{aligned}
\end{align}
which holds precisely on the rank locus identified by Hamm and Huang \cite{hamm2020perspectives}, Theorem~A: their conditions (i) $\mathrm{rank}(\mathbf{A}(I,J))=\mathrm{rank}(\mathbf{A})$, (ii) $\mathbf{A}=\mathbf{C}\mathbf{A}(I,J)^{+}\mathbf{R}$ and (iii) $\mathbf{A}=\mathbf{C}\mathbf{C}^{+}\mathbf{A}\mathbf{R}^{+}\mathbf{R}$ are equivalent. Our derivation reaches (iii). Proposition~\ref{prop:CUR-exactness} above records the equivalence of (i) and (ii), which the framework proves in two lines and which is exactly what the meta-factorization equations demand; the identity displayed here is (iii), the $\mathbf{B}=\mathbf{C}$, $\mathbf{D}=\mathbf{R}^{*}$ member of Corollary~\ref{cor:CUR-family}, whose conditions read $\mathrm{rank}(\mathbf{C})=\mathrm{rank}(\mathbf{R})=\mathrm{rank}(\mathbf{A})$.
So we have reconstructed the algebraic \emph{form} of the CUR decomposition, and the word is deliberate: as displayed, \eqref{eq:CUR-exact} is not an identity for arbitrary $I$ and $J$. Proposition~\ref{prop:CUR-exactness} supplies the exact condition under which it is one, and Corollary~\ref{cor:CUR-family} says what the design parameters $\mathbf{B}$ and $\mathbf{D}$ buy before they are specialized.

For the same selected columns and rows, the two designs now have a visible distinction in the structure of their hidden projectors. The choice (\ref{eq:CUR-B-D}), $\mathbf{B}=\mathbf{F}$ and $\mathbf{D}=\mathbf{H}$, makes $\mathcal{C}(\mathbf{B})=\mathcal{C}(\mathbf{F})$ and $\mathcal{C}(\mathbf{D})=\mathcal{C}(\mathbf{H})$, so the hidden projectors are the \emph{orthogonal} ones and the mixing matrix is $\mathbf{C}^+\mathbf{A}\mathbf{R}^+$. This gives the CUR decomposition. The choice (\ref{eq:Nystroem-B-D}), $\mathbf{B}=\mathbf{\Omega}_r$ and $\mathbf{D}=\mathbf{\Omega}_c$, gives the generalized Nystr\"{o}m form and generally \emph{oblique} projectors; special data can make them orthogonal. With coordinate sketches, its mixing matrix is the intersection pseudoinverse $\mathbf{W}^{+}$. At the sampling matrices $\bar{\mathbf{\Omega}}_c$ and $\bar{\mathbf{\Omega}}_r$ of this instance, Corollary~\ref{cor:CUR-family} shows that both are members of one two-design-matrix family, and identifies the exactness condition each of them inherits.

The following MATLAB code illustrates the concepts for a~random (naive) selection of columns and rows:
\begin{lstlisting}
    In      = eye(n);       Im      = eye(m);
    I       = randperm(m,k);J       = randperm(n,k);
    Q       = In(:,J);      P       = Im(:,I);
    F       = A*Q;          H       = A'*P;
    % EITHER: Nystrom-like method
    B       = P;            D       = Q;
    % OR:     CUR (with naive sampling)
    % B     = F;            D       = H;
    % If rank(B'*F) = rank(H'*D) = k, calculate Y and X by QR,
    [QY,RY] = qr(B'*F,0);   [QX,RX] = qr(H'*D,0);
    Y       = (RY\(QY'*B'))';   X   = (D/RX)*QX';
    % otherwise use pinv() instead of the two lines above:
    % Y     = (pinv(B'*F)*B')'; X   = D*pinv(H'*D);
    G       = Y'*A*X;
    Ar      = F*G*H';
\end{lstlisting}
The index sets $I$ and $J$ are drawn at random in this fragment, so the reconstruction it computes is exact only when the draw happens to land on the rank locus of Proposition~\ref{prop:CUR-exactness}. Otherwise \texttt{Ar} is an approximation. An implementation seeking exact CUR can check $\mathrm{rank}(\mathbf{A}(I,J))=\mathrm{rank}(\mathbf{A})$ and reselect if necessary; this requires $k\ge\mathrm{rank}(\mathbf{A})$. Numerical rank decisions require a tolerance appropriate to the data.

The constructions of this section are validations. The generalized Nystr\"{o}m method, the CUR decomposition and the outer-product decomposition all arise from one procedure run with different design parameters, and for the first two exactness reduces to the single rank condition of Proposition~\ref{prop:CUR-exactness}. Two increments came with them, both at corollary strength. Corollary~\ref{thm:wedderburn} rewrites the classical rank-reduction criterion as the direct test \eqref{eq:rank-reduction-closed} on the proposed factors and gives the sampling matrices explicitly. Corollary~\ref{cor:CUR-family} exhibits two published CUR identities as two choices of $(\mathbf{B},\mathbf{D})$ within a single family.
\section{Generalized right-hand sides and \texorpdfstring{$(N{+}1)$}{(N+1)}-potent certificates}
\label{section:Generalized...}

The projector equation uses the identity matrix as its right-hand side. What changes if we replace it by nonsingular matrices $\mathbf Z_c$ and $\mathbf Z_r$? They can be absorbed into the transforms $\mathbf Y^*$ and $\mathbf X$, leaving a normalized pair that satisfies the original equation. This change of variables does not enlarge the exact factorizations available for fixed outer factors $\mathbf F$ and $\mathbf H$. Exact reconstruction still requires those factors to contain the column and row spaces of $\mathbf A$.

The section separates this normalization from what happens when the generalized transforms are used without rescaling. Their products with the outer factors need not be projectors. Finite-order right-hand sides instead give $(N{+}1)$-potent operators. Under the span inclusions, their repeated application gives a periodic reconstruction. A further condition determines when they preserve the core at every step. We first establish the normalization and reconstruction conditions, then derive that condition for finite-order pairs, including singular cores. Singular right-hand sides are outside the scope of the normalization.

Figure~\ref{fig:gauge} shows the change of variables. Proposition~\ref{prop:gauge-reduction} distinguishes the resulting normalized core from the canonical choice $\mathbf M=\mathbf F^+\mathbf A(\mathbf H^*)^+$. They agree whenever the span inclusions hold, but need not agree otherwise. Here $\mathbf M$ is a core in the fixed outer-factor coordinates, not the compressed matrix of Proposition~\ref{prop:regimes}(c).
\begin{figure}[!ht]
\centering
\begin{tikzpicture}[x=1cm,y=1cm,font=\normalsize,
    gen/.style={draw=black!45,line width=0.5pt,fill=black!5,
                minimum width=4.45cm,minimum height=0.86cm,font=\normalsize},
    nrm/.style={draw=black!60,line width=0.8pt,fill=black!13,
                minimum width=4.45cm,minimum height=1.05cm,font=\normalsize},
    panel/.style={font=\normalsize},
    map/.style={gray!70,line width=0.6pt,->}]

\node[gen] (g1) at (2.30,2.30)
  {$\mathbf{Y}^{*}\mathbf{F}=\mathbf{Z}_c,\ \ \mathbf{H}^{*}\mathbf{X}=\mathbf{Z}_r$};
\node[gen] (g2) at (2.30,1.00)
  {$\mathbf{Y}^{*}\mathbf{F}=\mathbf{Z}'_c,\ \ \mathbf{H}^{*}\mathbf{X}=\mathbf{Z}'_r$};
\node[gen] (g3) at (2.30,-0.30)
  {$\mathbf{Y}^{*}\mathbf{F}=\mathbf{Z}''_c,\ \ \mathbf{H}^{*}\mathbf{X}=\mathbf{Z}''_r$};
\node[panel,anchor=south] at (2.30,3.00) {nonsingular right-hand sides};

\node[nrm] (n) at (13.70,1.00)
  {$\mathbf{Y}^{*}\mathbf{F}=\mathbf{H}^{*}\mathbf{X}=\mathbf{I}_k$};
\node[panel,anchor=south] at (13.70,3.00) {the normalized equation};

\draw[map] (g1.east) -- (n.west);
\draw[map] (g2.east) -- (n.west);
\draw[map] (g3.east) -- (n.west);
\node[panel,align=center,fill=white,inner sep=2.5pt] at (8.00,2.60)
  {$\widehat{\mathbf{Y}}^{*}=\mathbf{Z}_c^{-1}\mathbf{Y}^{*}$,\\[1pt]
   $\widehat{\mathbf{X}}=\mathbf{X}\mathbf{Z}_r^{-1}$};
\node[panel,align=center,fill=white,inner sep=2.5pt] at (8.00,-0.55)
  {bijective, for each\\ fixed pair};

\node[panel,anchor=north,align=center] at (8.00,-1.40)
  {with $\mathbf{F}$ and $\mathbf{H}$ held fixed, every exact factorization reached on the left\\
   is one the normalized equation already reaches};
\end{tikzpicture}
\caption{For fixed $\mathbf F,\mathbf H$ and each fixed nonsingular pair $(\mathbf Z_c,\mathbf Z_r)$, normalization is a bijective change of variables on the certifying pair. After normalization, exact reconstruction holds under the same span inclusions as for the original projector equation. Changing the outer-factor spaces is a separate operation, studied in Proposition~\ref{prop:replacement-rigidity}.}
\label{fig:gauge}
\end{figure}
\par\noindent\begin{minipage}{\linewidth}
\setlength{\parindent}{2em}
\setlength{\parskip}{5.5pt}
\begin{proposition}[normalization]\label{prop:gauge-reduction}
Let $k\ge1$, fix $\mathbf A\in\mathbb F^{m\times n}$, $\mathbf F\in\mathbb F^{m\times k}$ and $\mathbf H\in\mathbb F^{n\times k}$, and suppose $\mathbf Y^*\mathbf F=\mathbf Z_c$ and $\mathbf H^*\mathbf X=\mathbf Z_r$, where $\mathbf Z_c,\mathbf Z_r\in GL_k(\mathbb F)$. Define
\begin{align*}
\widehat{\mathbf Y}^*=\mathbf Z_c^{-1}\mathbf Y^ *
&\quad\text{and}\quad 
\widehat{\mathbf X}=\mathbf X\mathbf Z_r^{-1},
\\
\widehat{\mathbf G}=\widehat{\mathbf Y}^*\mathbf A\widehat{\mathbf X}
&\quad\text{and}\quad
\mathbf M=\mathbf F^+\mathbf A(\mathbf H^*)^+.
\end{align*}
For each fixed pair $(\mathbf Z_c,\mathbf Z_r)$, this change of variables is a bijection from the generalized certifying pairs to the solutions of $\widehat{\mathbf Y}^*\mathbf F=\mathbf H^*\widehat{\mathbf X}=\mathbf I_k$.

The identity $\mathbf A=\mathbf F\widehat{\mathbf G}\mathbf H^*$ holds if and only if the span inclusions \eqref{eq:span-inclusions} hold. Under these inclusions, every certifying pair and every nonsingular $(\mathbf Z_c,\mathbf Z_r)$ give the same normalized core, $\widehat{\mathbf G}=\mathbf M$.

If either inclusion fails, no core reconstructs $\mathbf A$ with these outer factors. As the normalized certifying pair varies, $\widehat{\mathbf G}$ is nonconstant and unbounded in any matrix norm. Its value at $\widehat{\mathbf Y}^*=\mathbf F^+$, $\widehat{\mathbf X}=(\mathbf H^*)^+$ is $\mathbf M$, but this value need not persist at other pairs.
\end{proposition}
\end{minipage}\par
\begin{proof}
The nonsingular products $\mathbf Y^*\mathbf F$ and $\mathbf H^*\mathbf X$ force $\mathbf F$ and $\mathbf H$ to have full column rank. Thus $\mathbf F^+\mathbf F=\mathbf H^*(\mathbf H^*)^+=\mathbf I_k$.

Substitution gives $\widehat{\mathbf Y}^*\mathbf F=\mathbf H^*\widehat{\mathbf X}=\mathbf I_k$. The inverse map is $\mathbf Y^*=\mathbf Z_c\widehat{\mathbf Y}^*$, $\mathbf X=\widehat{\mathbf X}\mathbf Z_r$.

Under the span inclusions, $\mathbf F\widehat{\mathbf Y}^*$ acts as the identity on the columns of $\mathbf A$, and $\widehat{\mathbf X}\mathbf H^*$ acts as the identity on its rows. Hence $\mathbf F\widehat{\mathbf G}\mathbf H^*=\mathbf A$. Conversely, any representation $\mathbf A=\mathbf F\mathbf G\mathbf H^*$ implies the two inclusions. Multiplying this representation by $\mathbf F^+$ and $(\mathbf H^*)^+$ gives the unique core $\mathbf G=\mathbf M$.

To examine failure of a span inclusion, put $\mathbf Q_c=\mathbf I_m-\mathbf F\mathbf F^+$ and $\mathbf Q_r=\mathbf I_n-\mathbf H\mathbf H^+$. Every normalized pair has the form
$$
\widehat{\mathbf Y}^*=\mathbf F^++\mathbf U\mathbf Q_c,\qquad
\widehat{\mathbf X}=(\mathbf H^*)^++\mathbf Q_r\mathbf V,
$$
where $\mathbf U\in\mathbb F^{k\times m}$ and $\mathbf V\in\mathbb F^{n\times k}$ are arbitrary. Indeed, these formulas satisfy the two inverse equations, and a given normalized pair is recovered by taking $\mathbf U=\widehat{\mathbf Y}^*$ and $\mathbf V=\widehat{\mathbf X}$. Set
$$
\mathbf B=\mathbf F^+\mathbf A\mathbf Q_r,\qquad
\mathbf C=\mathbf Q_c\mathbf A(\mathbf H^*)^+,\qquad
\mathbf D=\mathbf Q_c\mathbf A\mathbf Q_r.
$$
Expanding gives
\begin{align*}
\widehat{\mathbf G}&=\mathbf M+\mathbf B\mathbf V+\mathbf U\mathbf C+\mathbf U\mathbf D\mathbf V,\\
\mathbf A&=\mathbf F\mathbf M\mathbf H^*+\mathbf F\mathbf B+\mathbf C\mathbf H^*+\mathbf D.
\end{align*}
The two inclusions hold exactly when $\mathbf B=\mathbf C=\mathbf D=\mathbf0$: the forward implication follows from $\mathbf Q_c\mathbf A=\mathbf A\mathbf Q_r=\mathbf0$, and the reverse implication follows from the second expansion. If $\mathbf B\ne\mathbf0$, set $\mathbf U=\mathbf0$ and choose $\mathbf V=t\mathbf V_0$ with $\mathbf B\mathbf V_0\ne\mathbf0$. If $\mathbf C\ne\mathbf0$, set $\mathbf V=\mathbf0$ and choose $\mathbf U=t\mathbf U_0$ with $\mathbf U_0\mathbf C\ne\mathbf0$. If both vanish, then $\mathbf D\ne\mathbf0$, and $k\ge1$ allows $\mathbf U_0,\mathbf V_0$ with $\mathbf U_0\mathbf D\mathbf V_0\ne\mathbf0$. Take $\mathbf U=t\mathbf U_0$, $\mathbf V=\mathbf V_0$. In every case $\widehat{\mathbf G}=\mathbf M+t\mathbf K$ for a nonzero $\mathbf K$, which proves nonconstancy and unboundedness. The canonical pair corresponds to $\mathbf U=\mathbf V=\mathbf0$.
\end{proof}

For a concrete failure of the span inclusions, take $\mathbf A=\mathbf I_2$ and $\mathbf F=\mathbf H=\mathbf e_1$. The normalized pairs $\widehat{\mathbf Y}^*=[\,1\ \ t\,]$ and $\widehat{\mathbf X}=[\,1\ \ s\,]^{\mathsf T}$ give $\widehat{\mathbf G}=1+ts$ and $\mathbf M=1$. Thus the canonical core value also occurs at some noncanonical pairs, when $ts=0$, but no choice reconstructs $\mathbf I_2$: the product is always $\operatorname{diag}(1+ts,0)$. The choice $t=s=1$ is also produced by Theorem~\ref{Theorem:projector-eq-sol} with $\mathbf B=\mathbf D=[\,1\ \ 1\,]^{\mathsf T}$.

\paragraph{When the unnormalized core is exact.}
For any generalized pair, distinguish the raw core $\mathbf G_0=\mathbf Y^*\mathbf A\mathbf X$ from the normalized core $\widehat{\mathbf G}=\widehat{\mathbf Y}^*\mathbf A\widehat{\mathbf X}$. Reversing the normalization gives
$$
\mathbf G_0=\mathbf Z_c\widehat{\mathbf G}\mathbf Z_r.
$$
This identity holds whether or not the span inclusions hold. Under those inclusions, Proposition~\ref{prop:gauge-reduction} gives $\widehat{\mathbf G}=\mathbf M$ and $\mathbf A=\mathbf F\mathbf M\mathbf H^*$. Since the outer factors have full rank $k$, cancellation then yields
$$
\mathbf F\mathbf G_0\mathbf H^*=\mathbf A
\quad\Longleftrightarrow\quad
\mathbf Z_c\mathbf M\mathbf Z_r=\mathbf M.
$$
Thus normalization always recovers the exact core under the span inclusions, whereas using the raw core requires the additional equality on the right. We call pairs satisfying that equality the \emph{stabilizer} of $\mathbf M$: they leave the core unchanged. Even in one dimension the equality can fail. With $\mathbf A=\mathbf F=\mathbf H=1$, $\mathbf Z_c=2$ and $\mathbf Z_r=3$, the normalized core is $1$ and the raw core is $6$.

For nonsingular $\mathbf M$, the stabilizer condition determines $\mathbf Z_r=\mathbf M^{-1}\mathbf Z_c^{-1}\mathbf M$ from $\mathbf Z_c$. Singular cores can allow more freedom. For $\mathbf M=\mathrm{diag}(1,0)$, direct multiplication gives all stabilizing pairs as
$$
\mathbf Z_c=\begin{bmatrix}a&b\\0&d\end{bmatrix},
\qquad
\mathbf Z_r=\begin{bmatrix}a^{-1}&0\\c&e\end{bmatrix},
\qquad a,d,e\ne0,
$$
with $b,c$ arbitrary. There are five parameters over $\mathbb F$, compared with four for a nonsingular $2\times2$ core. The pair $\mathbf Z_c=\mathrm{diag}(2,1)$, $\mathbf Z_r=\mathrm{diag}(1/2,3)$ is one example.

Proposition~\ref{prop:gauge-reduction} keeps the outer factors fixed. Changing an outer factor is a separate operation. Under the hypotheses of Proposition~\ref{prop:replacement-rigidity}, accepted replacements at $k=r$ are changes of basis within the same column space. For $r<k<m$, accepted candidates of rank $k$ have $(k-r)(m-k)$ additional column-space parameters, which rescaling a fixed certifying pair cannot reach. At $k=m$, the accepted candidates of rank $k$ again form a single change-of-basis orbit.

\paragraph{Recursive reconstruction.}
The choice now concerns the right-hand sides $\mathbf Z_c$ and $\mathbf Z_r$. A nonsingular matrix has \emph{finite order} if some positive power is the identity; its order is the least such exponent. We may choose both matrices to have finite order and take a positive common multiple $N$ of their orders. This is an additional restriction, useful when repeated application should return to the starting point. Normalization itself also allows matrices of infinite order, such as $2\mathbf I_k$.

For the recursive construction considered in Appendix~A of the earlier version, suppose $\mathbf F$ and $\mathbf H$ have full column rank and the span inclusions hold. Set
$$
\mathbf Y^*=\mathbf Z_c\mathbf F^+,\qquad
\mathbf X=(\mathbf H^*)^+\mathbf Z_r,\qquad
\mathbf M=\mathbf F^+\mathbf A(\mathbf H^*)^+.
$$
For every integer $j\ge1$,
$$
(\mathbf F\mathbf Y^*)^j=\mathbf F\mathbf Z_c^j\mathbf F^+,
\qquad
(\mathbf X\mathbf H^*)^j=(\mathbf H^*)^+\mathbf Z_r^j\mathbf H^*.
$$
Thus
\begin{align}\label{eq:recursive-reconstruction}
\mathbf A_j=(\mathbf F\mathbf Y^*)^j\mathbf A(\mathbf X\mathbf H^*)^j
=\mathbf F\mathbf Z_c^j\mathbf M\mathbf Z_r^j\mathbf H^*,\qquad j\ge0.
\end{align}
If $\mathbf Z_c^N=\mathbf Z_r^N=\mathbf I_k$, then $\mathbf A_{j+N}=\mathbf A_j$ and $\mathbf A_{pN}=\mathbf A$ for every integer $p\ge0$. Here $\mathbf A_0=\mathbf A$; the least period divides $N$ and may be smaller. Exact reconstruction at every step requires only $\mathbf Z_c\mathbf M\mathbf Z_r=\mathbf M$, by cancellation of the full-rank outer factors. This can hold even for infinite-order pairs: $\mathbf Z_c=2\mathbf I_k$ and $\mathbf Z_r=\tfrac12\mathbf I_k$ leave every core unchanged.

\paragraph{Eigenmodes and the discrete Fourier transform.}
Finite order also characterizes the operator type: $(\mathbf F\mathbf Y^*)^{N+1}=\mathbf F\mathbf Y^*$ exactly when $\mathbf Z_c^N=\mathbf I_k$ under the nonsingularity hypotheses above. These are classical potent matrices \cite{baksalary2013revisitation,mccloskey1984characterizations}; at $N=2$ they are tripotent.

An eigenmode is a component of a vector on which $\mathbf Z$ acts by multiplication by a single eigenvalue. If $\mathbf Z^N=\mathbf I_k$, the possible eigenvalues are the $N$-th roots of unity. The polynomial $x^N-1$ has simple roots, so $\mathbf Z$ is diagonalizable over $\mathbb C$. With $\omega=e^{2\pi\mathrm i/N}$, define
\begin{align}\label{eq:cyclic-idempotents}
\mathbf\Phi_j=\frac1N\sum_{\ell=0}^{N-1}\omega^{-j\ell}\mathbf Z^\ell,
\qquad j=0,\ldots,N-1.
\end{align}
These matrices extract the components: $\mathbf v_j=\mathbf\Phi_j\mathbf v$ satisfies $\mathbf Z\mathbf v_j=\omega^j\mathbf v_j$, and $\mathbf v=\sum_j\mathbf v_j$. Indeed, on an eigenvector with eigenvalue $\omega^a$, the sum in \eqref{eq:cyclic-idempotents} is $1$ for $a=j$ and $0$ otherwise. Consequently $\mathbf\Phi_i\mathbf\Phi_j=\delta_{ij}\mathbf\Phi_j$, $\sum_j\mathbf\Phi_j=\mathbf I_k$, and $\mathbf Z=\sum_j\omega^j\mathbf\Phi_j$. These projections need not be orthogonal.

For example, the swap $\mathbf S=\left[\begin{smallmatrix}0&1\\1&0\end{smallmatrix}\right]$ has order $2$. It separates $\mathbf v=(x,y)^{\mathsf T}$ into
$$
\mathbf v_0=\tfrac12\begin{bmatrix}x+y\\x+y\end{bmatrix},\qquad
\mathbf v_1=\tfrac12\begin{bmatrix}x-y\\y-x\end{bmatrix}.
$$
The first component is unchanged by swapping the coordinates; the second changes sign. Their sum recovers the original vector.

The Fourier connection is literal: $\mathbf v_j$ is the $j$th coefficient of the vector-valued sequence $\mathbf v,\mathbf Z\mathbf v,\ldots,\mathbf Z^{N-1}\mathbf v$, using the negative-exponent discrete Fourier transform (DFT) divided by $N$. Conversely,
$$
\mathbf Z^\ell\mathbf v=\sum_{j=0}^{N-1}\omega^{j\ell}\mathbf v_j.
$$
The fast Fourier transform (FFT) is an algorithm for computing a DFT. Thus an FFT along the iterate index $\ell$ computes all components $\mathbf v_j$ once that sequence is available. For a cyclic shift, these eigenmodes are the usual Fourier modes; they diagonalize circulant matrices and turn cyclic convolution into separate scalar multiplications, as used in periodic signal processing \cite[Sections~4.3--4.4]{strangCSE}. For arbitrary dense $\mathbf Z$, forming the sequence of iterates has its own cost; the Fourier identity alone supplies no fast factorization.

For the core, apply this decomposition on both sides. Write $\mathbf\Phi_{c,a}$ and $\mathbf\Phi_{r,b}$ for the families associated with $\mathbf Z_c$ and $\mathbf Z_r$, and set $\mathbf M_{ab}=\mathbf\Phi_{c,a}\mathbf M\mathbf\Phi_{r,b}$. Then $\mathbf M=\sum_{a,b}\mathbf M_{ab}$, and the map $\mathbf M\mapsto\mathbf Z_c\mathbf M\mathbf Z_r$ multiplies this block by $\omega^{a+b}$. The following proposition states which blocks can therefore remain unchanged and quantifies the change for unitary right-hand sides. For real data, the mode decomposition is taken in the complexification.

\begin{figure}[!t]
\centering
\begin{tikzpicture}[x=1cm,y=1cm,font={\normalfont\normalsize},
    panel/.style={font={\normalfont\normalsize}},
    cell/.style={draw=black!40,line width=0.4pt,minimum size=1.0cm,
                 inner sep=2pt,anchor=center,font={\normalfont\normalsize}},
    keep/.style={draw=black!65,line width=0.9pt,fill=black!16,
                 minimum size=1.0cm,inner sep=2pt,anchor=center,font={\normalfont\normalsize}}]

\begin{scope}[xshift=3.60cm]
  \draw[black!40,line width=0.5pt] (0,0) circle (1.50);
  \draw[black!30,line width=0.4pt] (-1.50,0) -- (1.50,0);
  \draw[black!30,line width=0.4pt] (0,-1.50) -- (0,1.50);
  \foreach \a/\lab/\pos in {0/{$1$}/right, 90/{$\omega$}/above,
                            180/{$\omega^2$}/left, 270/{$\omega^3$}/below}{
     \path (0,0) ++(\a:1.50) coordinate (rt\a);
     \fill (rt\a) circle (1.7pt);
     \node[panel,\pos] at (rt\a) {\lab};}
  \node[panel,anchor=north,align=center] at (0,-2.62)
       {(a) $\mathbf{Z}_c^{N}=\mathbf{I}_k$: the spectrum lies\\
        on the $N$-th roots of unity};
\end{scope}

\begin{scope}[xshift=10.40cm,yshift=1.50cm]
  \foreach \a in {0,1,2,3}{
    \foreach \b in {0,1,2,3}{
      \pgfmathtruncatemacro{\keepit}{mod(\a+\b,4)==0 ? 1 : 0}
      \ifnum\keepit=1
        \node[keep] at (\b,-\a) {$\mathbf{M}_{\a\b}$};
      \else
        \node[cell] at (\b,-\a) {$\mathbf{M}_{\a\b}$};
      \fi}}
  \foreach \b in {0,1,2,3}{
    \node[panel] at (\b,0.80) {$\b$};}
  \foreach \a in {0,1,2,3}{
    \node[panel] at (-0.80,-\a) {$\a$};}
  \node[panel,anchor=south] at (1.50,1.14) {$b$};
  \node[panel,anchor=east] at (-1.08,-1.50) {$a$};
  \node[panel,anchor=north,align=center] at (1.50,-4.12)
       {(b) $\mathbf{Z}_c\mathbf{M}\mathbf{Z}_r=\omega^{s}\mathbf{M}$ allows only\\
        the blocks with $a+b\equiv s \pmod N$; here $s=0$, $N=4$};
\end{scope}
\end{tikzpicture}
\caption{For $\mathbf Z_c^N=\mathbf Z_r^N=\mathbf I_k$, the cyclic idempotents split any $k\times k$ core as $\mathbf M=\sum_{a,b}\mathbf M_{ab}$. Each block is multiplied by $\omega^{a+b}$ under $\mathbf M\mapsto\mathbf Z_c\mathbf M\mathbf Z_r$. The core is unchanged exactly when every block outside $a+b\equiv0\pmod N$ is zero. The shaded positions show the allowed support for $N=4$, including for singular cores.}
\label{fig:modes}
\end{figure}
\par\noindent\begin{minipage}{\linewidth}
\setlength{\parindent}{2em}
\setlength{\parskip}{5.5pt}
\begin{proposition}[finite order]\label{prop:finite-order}
In the setting of Proposition~\ref{prop:gauge-reduction}, fix an integer $N\ge1$ and write $\omega=e^{2\pi\mathrm{i}/N}$. Then
$$
(\mathbf{F}\mathbf{Y}^*)^{N+1}=\mathbf{F}\mathbf{Y}^*
\quad\text{if and only if}\quad
\mathbf{Z}_c^{N}=\mathbf{I}_k.
$$
The row-side identity $(\mathbf{X}\mathbf{H}^*)^{N+1}=\mathbf{X}\mathbf{H}^*$ holds exactly when $\mathbf{Z}_r^N=\mathbf{I}_k$. In particular, $\mathbf{F}\mathbf{Y}^*$ is idempotent exactly when $\mathbf{Z}_c=\mathbf{I}_k$.

Suppose $\mathbf{Z}_c^{N}=\mathbf{Z}_r^{N}=\mathbf{I}_k$, and let $\mathbf{\Phi}_{c,a}$, $\mathbf{\Phi}_{r,b}$ and $\mathbf{M}_{ab}$ be as above. Then, for every $k\times k$ matrix $\mathbf{M}$, singular included, and every integer $s$,
\begin{align}\label{eq:mode-selection}
\begin{split}
\mathbf{Z}_c\mathbf{M}\mathbf{Z}_r=\omega^{s}\mathbf{M}
\quad&\text{if and only if}\quad
\mathbf{M}_{ab}=\mathbf{0}
\quad\text{whenever}\ \ a+b\not\equiv s\pmod{N}.
\end{split}
\end{align}
If these finite-order right-hand sides are also unitary, the defect $\mathbf{\Delta}=\mathbf{Z}_c\mathbf{M}\mathbf{Z}_r-\mathbf{M}$ satisfies
$$
\|\mathbf{\Delta}\|_F^2=\sum_{a,b}|\omega^{a+b}-1|^2\|\mathbf{M}_{ab}\|_F^2.
$$
\end{proposition}
\end{minipage}\par
\begin{proof}
From $\mathbf{Y}^*\mathbf{F}=\mathbf{Z}_c$, induction gives $(\mathbf{F}\mathbf{Y}^*)^{j+1}=\mathbf{F}\mathbf{Z}_c^{j}\mathbf{Y}^*$ for every $j\ge0$, so $\mathbf{Z}_c^{N}=\mathbf{I}_k$ gives $(\mathbf{F}\mathbf{Y}^*)^{N+1}=\mathbf{F}\mathbf{Y}^*$. Conversely, nonsingularity of $\mathbf{Z}_c=\mathbf{Y}^*\mathbf{F}$ makes $\mathbf{F}$ of full column rank and $\mathbf{Y}^*$ of full row rank, so $\mathbf{F}\mathbf{Z}_c^{N}\mathbf{Y}^*=\mathbf{F}\mathbf{Y}^*$ cancels on both sides to $\mathbf{Z}_c^{N}=\mathbf{I}_k$. The case $N=1$ is the idempotence statement.

To prove the selection rule, expand both factors. This gives $\mathbf{Z}_c\mathbf{M}\mathbf{Z}_r-\omega^{s}\mathbf{M}=\sum_{a,b}(\omega^{a+b}-\omega^{s})\mathbf{M}_{ab}$, and multiplying that defect by $\mathbf{\Phi}_{c,a}$ on the left and $\mathbf{\Phi}_{r,b}$ on the right leaves the single term $(\omega^{a+b}-\omega^{s})\mathbf{M}_{ab}$. The two identities together give \eqref{eq:mode-selection}. For nonsingular $\mathbf{M}$ the $s=0$ case solves to $\mathbf{Z}_r=\mathbf{M}^{-1}\mathbf{Z}_c^{-1}\mathbf{M}$, which is therefore the invertible case of \eqref{eq:mode-selection} rather than a separate condition.

For the norm identity, note that when $\mathbf{Z}_c$ and $\mathbf{Z}_r$ are unitary they are normal, so the $\mathbf{\Phi}_j$ are Hermitian idempotents and the modes are orthogonal in the Frobenius inner product, and the preceding expansion at $s=0$ reads $\mathbf{\Delta}=\sum_{a,b}(\omega^{a+b}-1)\mathbf{M}_{ab}$. The identity follows.
\end{proof}
Read at $s=0$ with $\mathbf{M}$ the normalized core, \eqref{eq:mode-selection} describes completely the pairs of order dividing $N$ that stabilize $\mathbf{M}$, singular $\mathbf{M}$ included; every finite-order stabilizing pair is reached at $N$ the least common multiple of the two orders. Pairs of infinite order are not described, and for nonsingular $\mathbf{M}$ the condition reduces to $\mathbf{Z}_r=\mathbf{M}^{-1}\mathbf{Z}_c^{-1}\mathbf{M}$. Equation \eqref{eq:mode-selection} is a selection rule on the modes of $\mathbf{M}$, not a condition on its rank, and it is an identity for an arbitrary $\mathbf{M}$; Figure~\ref{fig:modes} draws the grid it selects from. Under the span inclusions \eqref{eq:span-inclusions}, where the normalized core is $\mathbf{M}=\mathbf{F}^{+}\mathbf{A}(\mathbf{H}^*)^{+}$ and $\mathbf{A}=\mathbf{F}\mathbf{M}\mathbf{H}^*$, the reconstruction splits mode by mode exactly on the stabilizer: if $\mathbf{Z}_c\mathbf{M}\mathbf{Z}_r=\mathbf{M}$, then
\begin{align}\label{eq:mode-reconstruction}
\mathbf{A}=\sum_{a=0}^{N-1}\mathbf{F}\,\mathbf{\Phi}_{c,a}\mathbf{M}\mathbf{\Phi}_{r,-a}\,\mathbf{H}^{*},
\end{align}
with $-a$ read modulo $N$, and conversely \eqref{eq:mode-reconstruction} forces $\mathbf{Z}_c\mathbf{M}\mathbf{Z}_r=\mathbf{M}$, since $\mathbf{F}$ and $\mathbf{H}^*$ have full rank $k$ here.

For each real right-hand side, the conjugate eigenspaces indexed by $a$ and $N-a$ determine a real invariant subspace. In a suitable real basis of that subspace, the restricted operator is a direct sum of $2\times2$ rotation blocks. This real change of basis need not be orthogonal. The modes $a=0$ and, for even $N$, $a=N/2$ correspond to the real eigenvalues $1$ and $-1$.

The defect $\mathbf\Delta=\mathbf Z_c\mathbf M\mathbf Z_r-\mathbf M$ records how much the core changes in one step. Each block $\mathbf M_{ab}$ changes by $(\omega^{a+b}-1)\mathbf M_{ab}$; blocks with $a+b\equiv0$ do not change. For unitary right-hand sides, these changes are orthogonal, so their squared Frobenius norms add as in Proposition~\ref{prop:finite-order}. At $N=2$, the blocks $\mathbf M_{01}$ and $\mathbf M_{10}$ change sign, giving
$$
\|\mathbf\Delta\|_F^2=4\bigl(\|\mathbf M_{01}\|_F^2+\|\mathbf M_{10}\|_F^2\bigr).
$$
Without unitarity, the blocks need not be orthogonal. The per-block bound $\|\mathbf M_{ab}\|\le\|\mathbf\Phi_{c,a}\|\,\|\mathbf\Delta\|\,\|\mathbf\Phi_{r,b}\|/|\omega^{a+b}-1|$ still holds in any submultiplicative norm when $a+b\not\equiv0$.

An aggregate identity can instead be expressed in eigenbasis coordinates, as follows. Write $\mathbf Z_c=\mathbf U\boldsymbol\Lambda\mathbf U^{-1}$ and $\mathbf Z_r=\mathbf V\boldsymbol\Gamma\mathbf V^{-1}$, with diagonal entries $\lambda_i$ and $\gamma_j$. For $\widehat{\mathbf M}=\mathbf U^{-1}\mathbf M\mathbf V$,
\begin{align}\label{eq:finite-order-coordinate-defect}
\|\mathbf U^{-1}\mathbf\Delta\mathbf V\|_F^2
=\sum_{i,j}|\lambda_i\gamma_j-1|^2|\widehat{M}_{ij}|^2.
\end{align}
This follows entry by entry from $\mathbf U^{-1}\mathbf\Delta\mathbf V=\boldsymbol\Lambda\widehat{\mathbf M}\boldsymbol\Gamma-\widehat{\mathbf M}$. For nonunitary eigenbases this is a coordinate norm; comparison with the ambient Frobenius norm introduces the norms of the basis changes and their inverses. Sharper bounds in terms of eigenbasis conditioning are not pursued here.

\paragraph{What stabilization lets us do.}
Here \emph{stabilizing} means leaving the core exactly unchanged. With fixed full-rank outer factors and the span inclusions, this is precisely the requirement that the unnormalized reconstruction still gives the same $\mathbf A$. For any proposed finite-order pair, \eqref{eq:mode-selection} is a complete test: every block whose two eigenvalues fail to multiply to $1$ must vanish. It applies to singular cores as well. Taking $N$ to be a positive common multiple of the two orders covers every finite-order pair; infinite-order pairs remain outside this Fourier description.

For the same swap $\mathbf S$ on both sides,
$$
\mathbf S\begin{bmatrix}a&b\\c&d\end{bmatrix}\mathbf S
=\begin{bmatrix}d&c\\b&a\end{bmatrix}.
$$
Thus the core is preserved exactly when $a=d$ and $b=c$: exchanging the two row coordinates and the two column coordinates leaves their encoded relationships unchanged. This includes singular cores, for example when $a=b=c=d=1$.

The decomposition also extracts the part of any core that has a prescribed finite symmetry:
$$
\overline{\mathbf M}=\frac1N\sum_{\ell=0}^{N-1}\mathbf Z_c^\ell\mathbf M\mathbf Z_r^\ell
=\sum_{a+b\equiv0\,({\rm mod}\,N)}\mathbf M_{ab}.
$$
The equality follows by summing the geometric series for each block, and shifting the summation index gives $\mathbf Z_c\overline{\mathbf M}\mathbf Z_r=\overline{\mathbf M}$. For the swap, this averages the two diagonal entries and the two off-diagonal entries. Reconstructing $\mathbf F\overline{\mathbf M}\mathbf H^*$ enforces that symmetry in the chosen core coordinates; it preserves the original data exactly when the original core was already invariant. A second averaging leaves $\overline{\mathbf M}$ unchanged, so this map is a projection onto invariant cores. When both right-hand sides are unitary, its adjoint averages the inverse powers, which run through the same cycle; the projection is therefore orthogonal in the Frobenius inner product.

The contribution of this section is the explicit connection between generalized certifying pairs and preservation of a fixed factorization. Proposition~\ref{prop:gauge-reduction} separates a change of certifying coordinates, which reaches the same exact factorizations, from an unnormalized action on the core. For a supplied finite-order pair, Proposition~\ref{prop:finite-order} gives an exact preservation test, including singular cores, and quantifies the defect when the right-hand sides are unitary. Normalization, cyclic idempotents and Fourier decomposition are classical ingredients; the potent-matrix setting is studied in \cite{baksalary2013revisitation,mccloskey1984characterizations}. Their application here identifies which core components can remain under the prescribed action and how changes in those components affect reconstruction. This test checks a proposed pair; it does not construct all stabilizing pairs. Changing the outer-factor spaces is the separate replacement problem of Section~\ref{section:Replacement}, and singular right-hand sides remain outside the normalization result.

\paragraph{Reconstruction and the reverse order law.}
Two-sided reconstruction determines how the left and right maps act together on the data. It need not determine either map separately: a change on one side can be compensated on the other. This distinction matters when using a reconstruction identity to test a proposed pseudoinverse. For $\mathbf W\in\mathbb F^{m\times n}$ and a candidate $\mathbf K\in\mathbb F^{n\times m}$, consider
$$
(\mathbf K\mathbf W)\,\mathbf W^*\,(\mathbf W\mathbf K)=\mathbf W^*.
$$
This has the form of the meta-factorization reconstruction equation, with data matrix $\mathbf W^*$ and left and right maps $\mathbf K\mathbf W$ and $\mathbf W\mathbf K$. It says that their combined action preserves $\mathbf W^*$. The candidate $\mathbf K=\mathbf W^+$ satisfies it, but so does $\mathbf K=-\mathbf W^+$: both maps change sign, and the signs cancel. Thus exact reconstruction alone cannot identify the Moore--Penrose inverse.

K\k{e}dzierski's Theorem~5.1 \cite{kedzierski2024rol} describes every solution of this equation and makes the connection with periodic reconstruction precise. Suppose $r=\mathrm{rank}(\mathbf W)>0$ and take the compact SVD $\mathbf W=\mathbf U_r\boldsymbol\Sigma\mathbf V_r^*$. In these singular-vector coordinates, for each nonzero singular value $\sigma$, the corresponding block of $\mathbf K$ is $\sigma^{-1}\mathbf J_\sigma$, where $\mathbf J_\sigma^2=\mathbf I$. Collect the involutions $\mathbf J_\sigma$ into the block-diagonal matrix $\mathbf J$; then $\mathbf J^2=\mathbf I_r$ and $\mathbf J\boldsymbol\Sigma=\boldsymbol\Sigma\mathbf J$. The two maps have the form
$$
\mathbf K\mathbf W=\mathbf V_r\mathbf J\mathbf V_r^*,
\qquad
\mathbf W\mathbf K=\mathbf U_r\mathbf J\mathbf U_r^*.
$$
Thus each map returns a vector in its range after two applications and vanishes on the orthogonal complement. Its square is the projector onto that range. The ambient maps are tripotent and are idempotent exactly when $\mathbf J=\mathbf I_r$. The classification and the results invoked below are K\k{e}dzierski's.

To identify the preceding meta-factorization explicitly, take the data to be $\mathbf W^*$ and set
$$
\mathbf F=\mathbf V_r,\qquad \mathbf H=\mathbf U_r,\qquad
\mathbf M=\boldsymbol\Sigma,\qquad
\mathbf Y^*=\mathbf J\mathbf V_r^*,\qquad \mathbf X=\mathbf U_r\mathbf J.
$$
The two right-hand sides are now $\mathbf Z_c=\mathbf Z_r=\mathbf J$, and $\mathbf J\boldsymbol\Sigma\mathbf J=\boldsymbol\Sigma$ is exactly the core-stabilization condition above. This is the $N=2$ case of \eqref{eq:recursive-reconstruction}:
$$
(\mathbf K\mathbf W)^j\mathbf W^*(\mathbf W\mathbf K)^j
=\mathbf V_r\mathbf J^j\boldsymbol\Sigma\mathbf J^j\mathbf U_r^*
=\mathbf W^*,\qquad j\ge0.
$$
The data therefore stay fixed at every step, even when the individual maps change signs. In the mode-selection rule, the opposite-sign core blocks vanish and the equal-sign blocks remain. A common candidate $\mathbf K$ restricts the two sides to the same involution commuting with $\boldsymbol\Sigma$; the earlier finite-order construction allows independent right-hand sides. When $\mathbf W=\mathbf0$, reconstruction holds for every $\mathbf K$ and both maps are zero, so no positive-dimensional core is needed.

There is also freedom that neither map can detect. One may add a matrix $\mathbf E$ with $\mathbf W\mathbf E=\mathbf0$ and $\mathbf E\mathbf W=\mathbf0$, which maps $\mathcal N(\mathbf W^*)$ into $\mathcal N(\mathbf W)$ and vanishes on $\mathcal C(\mathbf W)$. For example, $\mathbf W=\operatorname{diag}(1,0)$ and $\mathbf K=\operatorname{diag}(1,t)$ give the same two maps $\operatorname{diag}(1,0)$ for every $t$, although $\mathbf K=\mathbf W^+$ only at $t=0$. Even requiring the two maps to be projectors leaves this undetected component. Identifying an inverse therefore requires information about the candidate itself as well as its action in the reconstruction equation.

The sign possibility also occurs when the candidate is built by reversing a chain of three pseudoinverses. Take, over $\mathbb Q$,
\begin{align}\label{eq:rol-counterexample}
\mathbf{A}=\begin{bmatrix}0&0\\1&0\end{bmatrix},
\quad
\mathbf{B}=\begin{bmatrix}-1&-1\\1&2\end{bmatrix},
\quad
\mathbf{C}=\begin{bmatrix}0&0\\1&1\end{bmatrix}.
\end{align}
The product and the candidate proposed by the reverse order law are
$$
\mathbf W=\mathbf A\mathbf B\mathbf C
=\begin{bmatrix}0&0\\-1&-1\end{bmatrix},
\qquad
\mathbf X=\mathbf C^+\mathbf B^+\mathbf A^+
=\begin{bmatrix}0&1/2\\0&1/2\end{bmatrix}
=-\mathbf W^+.
$$
Consequently,
$$
\mathbf W\mathbf X=-\mathbf W\mathbf W^+,
\qquad
\mathbf X\mathbf W=-\mathbf W^+\mathbf W.
$$
Each map reverses the sign on its one-dimensional range and vanishes on the orthogonal complement. Both are tripotent and Hermitian, but neither is idempotent. In the two-sided reconstruction the two sign changes cancel, giving $(\mathbf X\mathbf W)\mathbf W^*(\mathbf W\mathbf X)=\mathbf W^*$. By contrast, $\mathbf W\mathbf X\mathbf W=-\mathbf W$, so the inverse residual is $\mathbf W\mathbf X\mathbf W-\mathbf W=-2\mathbf W\ne\mathbf0$. The factorization therefore passes the reconstruction test while its candidate fails the reverse order law. The negative multiples of the orthogonal projectors remain Hermitian, so Hermiticity alone cannot certify the required projector identities.

The additional outer-inverse equation resolves this gap by an existing result: K\k{e}dzierski's Corollary~5.2 \cite{kedzierski2024rol} states that every solution of $\mathbf X\mathbf W\mathbf W^*\mathbf W\mathbf X=\mathbf W^*$ satisfies $\mathbf X\mathbf W\mathbf X=\mathbf W^+$. Hence
$$
\bigl[\mathbf X\mathbf W\mathbf W^*\mathbf W\mathbf X=\mathbf W^*
\ \text{and}\ \mathbf X\mathbf W\mathbf X=\mathbf X\bigr]
\quad\text{if and only if}\quad \mathbf X=\mathbf W^+.
$$
The reverse implication follows from the Penrose equations. This conclusion applies to arbitrary compatible $\mathbf W,\mathbf X$, independently of how many factors define them. It is a consequence of the cited result, rather than an open problem or a new theorem of this paper.

\section{Summary}
Meta-factorization relates the outer-factor spaces, the transforms that certify them, and the mixing matrix that reconstructs the data. For full-column-rank outer factors whose spaces contain the column spaces of $\mathbf A$ and $\mathbf A^*$, the projector and reconstruction equations specify how these choices fit together. The $(\mathbf B,\mathbf D)$ parametrization generates every certifying pair, and the worked SVD, CPQR, UTV, generalized Nystr\"{o}m, CUR and pseudoinverse constructions show how additional structural conditions select familiar factorizations. This is the organizing contribution: one description exposes both their common projector structure and the choices that distinguish them.

The replacement results turn that description into an explicit calculation on stored factors. Given a candidate left factor and its transform, with the retained right factor of full column rank, Theorem~\ref{thm:replacement} uses $\mathbf E_c\mathbf G=\mathbf0$ to decide whether the induced update $\check{\mathbf G}=\mathbf S_c\mathbf G$ preserves the matrix. A failed test still supplies its exact factored residual. With the candidate's pseudoinverse as transform, the repaired product is a closest reconstruction within its column space in every unitarily invariant norm. The rectangular extension permits a change in stored width. These results express classical consistency and projection principles in terms of the supplied replacement and stored factors, with their different hypotheses made explicit.

The accepted-subspace description explains how much a successful replacement may change. When both incumbent outer factors have full column rank and rank-$k$ candidates carry $\{1\}$-inverse transforms, the accepted column spaces are precisely the $k$-dimensional spaces containing $\mathcal C(\mathbf A)$. At $k=r=\mathrm{rank}(\mathbf A)$, replacement is a change of basis. For $r<k<m$, the accepted spaces have $(k-r)(m-k)$ parameters over the underlying field beyond that orbit; this extra column-space freedom vanishes at $k=m$. The framework also gives the exact residual for the specified square sketch constructions, the ordered UTV completion rule, a two-design-matrix CUR family, and explicit sampling matrices for the classical rank-reduction condition.

Generalized certifying equations clarify a different freedom. Nonsingular right-hand sides change certifying coordinates while reaching the same exact factorizations for fixed outer factors. When those right-hand sides have finite order, unnormalized reconstruction is periodic. Classical cyclic components identify exactly when a supplied pair preserves the core at every step, including singular cores. The reverse-order example shows why such compensated actions can reconstruct the matrix without identifying its Moore--Penrose inverse; the classification and its resolution are the cited results of K\k{e}dzierski.

The next computational step is a reliable numerical acceptance rule. It requires error bounds for the computed mismatch and a principled choice among candidate subspaces. The identities and worked examples specify the exact algebra such a rule must respect, including replacements that preserve the matrix while their certifying projectors have arbitrarily large norm.


\section*{CRediT author statement}

\textbf{Micha\l{} P. Karpowicz:} Conceptualization, Methodology, Formal analysis, Investigation, Software, Validation, Visualization, Writing --- original draft, Writing --- review \& editing, Project administration.

\section*{Acknowledgment}

The first version of this paper was written in the tumultuous times of the COVID-19 pandemic, when many plans had to be changed. Nevertheless, I would like to thank Prof. Gilbert Strang for his enthusiasm and helpful remarks during our remote discussions that covered my postponed stay at MIT. I also thank Prof. Michael Saunders for his support and review of the manuscript, Prof. Yuji Nakatsukasa for helpful discussions, and Prof. Alex Townsend for important references. Finally, this paper would not have been written without the support of my wife, Dr. Inez Okulska, who helped me turn a jungle of thoughts into a paper that I hope reads well despite the many equations it contains.

\section*{Declaration of competing interest}

The author declares that he has no known competing financial interests or personal relationships that could have appeared to influence the work reported in this paper.

\section*{Funding}

This research did not receive any specific grant from funding agencies in the public, commercial, or not-for-profit sectors.

\section*{Data and code availability}
No external datasets were used. The mathematical results are established by the proofs in the paper. The ancillary material provides executable MATLAB/Octave factorization demonstrations and network diagnostics, together with Python checks for exact network quantities, rectangular replacements and compact residuals. The network checker verifies the displayed cores, residual radicals, resistances and Kirchhoff index in exact arithmetic. The accompanying README gives dependencies, commands and verification limits. The broader MATLAB supplement is available at \url{https://github.com/mkarpowi/mft}.

The ancillary \textsc{Lean}~4 development checks fixed-width repair and residual identities, together with span and change-of-basis arguments under explicit inverse hypotheses. It includes a pinned toolchain, a dependency lock and a publication-checking script. The publication gate was rerun for the supplied sources: 34 modules, 203 theorems, no admitted steps, and only the standard axioms \texttt{propext}, \texttt{Classical.choice}, and \texttt{Quot.sound}. This certifies the formal statements, not every assertion of this revised text. The printed proofs supply the remaining links to rank assumptions and the geometric, norm and optimality results. The rectangular extension and compact residual calculation are supported by printed arguments and computational checks, and are outside the formal target. The ancillary README and statement map give the precise correspondence, identify surrogate statements and record the coverage limits.

\section*{Use of generative AI}
The author used Claude (Anthropic), through Claude Code, and Codex (OpenAI) during the development and revision of this work. Their uses included drafting and editing mathematical exposition, critical review of arguments, counterexample search, assistance with proof development and \textsc{Lean} formalization, preparation of TikZ figures, and generation, revision and testing of supporting code. The mathematical and computational claims are supported by the printed arguments and the documented executable checks, within their stated scope. The author takes full responsibility for the final text, mathematical claims, references, figures and shared code.

\section*{Version note}
This update retains the constructive exposition of the earlier meta-factorization paper and adds factor-replacement results, their geometry and reproducible examples. It makes the rank and structural hypotheses explicit, expands the effective-subspace description with a diagram, and develops nonsingular and finite-order generalized equations. The closing reverse-order discussion is corrected using an existing result of K\k{e}dzierski.

\end{document}